\documentclass[reqno,final,letterpaper]{amsproc}%
\usepackage{amsmath}
\usepackage{amssymb}
\usepackage{amsfonts}
\usepackage{showkeys}%
\usepackage{graphicx}
\providecommand{\U}[1]{\protect\rule{.1in}{.1in}}
\theoremstyle{plain}

\newtheorem{corollary}{Corollary}

\newtheorem{definition}{Definition}

\newtheorem{lemma}{Lemma}

\newtheorem{problem}{Problem}
\newtheorem{proposition}{Proposition}
\newtheorem{remark}{Remark}

\newtheorem{theorem}{Theorem}
\numberwithin{equation}{section}
\begin{document}
\title[Minimal Capacity]{Sets of Minimal Capacity and\\Extremal Domains}
\author[H. Stahl]{Herbert R Stahl}
\curraddr{Beuth Hochschule/FB II; Luxemburger Str. 10; 13 353 Berlin; Germany}
\email{HerbertStahl@aol.com}
\thanks{The research has been supported by the grant STA 299/13-1 der DFG}

\begin{abstract}
Let $f$ be a function meromorphic in a neighborhood of infinity. The central
problem in the present investigation is to find the largest domain
$D\subset\overline{\mathbb{C}}$ to which the function $f$ can be extended in a
meromorphic and single-valued manner. 'Large' means here that the complement
$\overline{\mathbb{C}}\setminus D$ is minimal with respect to (logarithmic)
capacity. Such extremal domains play an important role in Pad\'{e} approximation.

In the paper a unique existence theorem for extremal domains and their
complementary sets of minimal capacity is proved. The topological structure of
sets of minimal capacity is studied, and analytic tools for their
characterization are presented; most notable are here quadratic differentials
and a specific symmetry property of the Green function in the extremal domain.
A local condition for the minimality of the capacity is formulated and
studied. Geometric estimates for sets of minimal capacity are given.

Basic ideas are illustrated by several concrete examples, which are also used
in a discussion of the principal differences between the extremality problem
under investigation and some classical problems from geometric function theory
that possess many similarities, which for instance is the case for
Chebotarev's Problem.

\end{abstract}
\maketitle

\section{\label{s1}Introduction}

\qquad We assume that $f$ is a function meromorphic in a neighborhood of
infinity, and consider domains $D\subset\overline{\mathbb{C}}$ to which the
function $f$\ can be extended in a meromorphic and single-valued manner. The
basic problem of our investigation is to find the domain with a complement of
minimal (logarithmic) capacity. It will be shown that for any function $f$
that is meromorphic at infinity such a domain exists and is essentially
unique. The domain is called extremal, and its complement is called the
minimal set (or\ the set of minimal capacity). Formal definitions are given in
the Sections \ref{s2} and \ref{s3}.\smallskip

Extremal domains play an important role in rational approximation, and there
especially in the convergence theory of Pad\'{e} approximants (cf.
\cite{GoncharRakhmanov87}, \cite{Nuttall80}, \cite{NuttallSingh77},
\cite{Nuttall90}, \cite{Stahl86a}, \cite{Stahl86b}, \cite{Stahl87a},
\cite{Stahl89}, \cite{Stahl96}, \cite{Stahl97}, \cite{BakerGravesMorris}
Chapter 6). Variants of the concept will also be useful in other areas of
rational approximation and the theory of orthogonal polynomials.\smallskip

Several elements of the material in the present article have already been
studied in \cite{Stahl85a}, \cite{Stahl85b}, and \cite{Stahl85c}. Results from
there will be revisited, proofs will be redone, and the whole concept will be
extended and reformulated.\smallskip%
\begin{figure}
[ptb]
\begin{center}
\includegraphics[
height=2.412in,
width=3.5163in
]%
{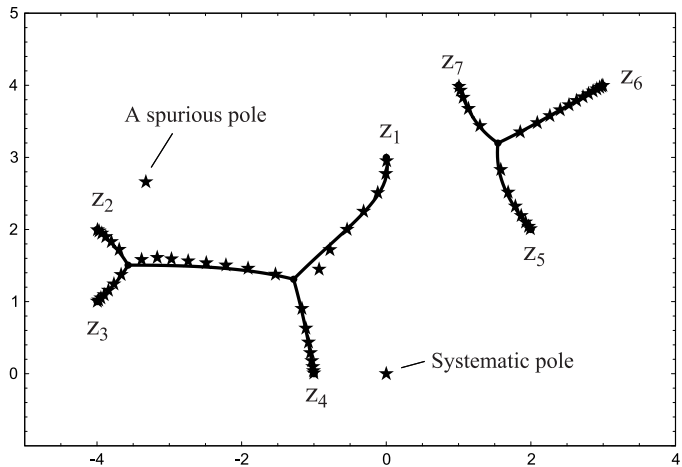}%
\caption{The poles of the Pad\'{e} approximant $[63/62]_{f}$ to the function
(\ref{f1a}) are represented by stars, and the associated minimal set
$K_{0}(f,\infty)$\ is represented by $8$ unbroken lines.}%
\label{fig0}%
\end{center}
\end{figure}

\subsection{\label{s11}A Concrete Example}

\qquad As an illustration of the role played by extremal domains in the theory
of Pad\'{e} approximation, we consider a concrete example. Let $f$ be the
algebraic function defined by
\begin{equation}
f(z):=\sqrt[4]{\prod\nolimits_{j=1}^{4}(1-z_{j}/z)}+\sqrt[3]{\prod
\nolimits_{j=5}^{7}(1-z_{j}/z)}\label{f1a}%
\end{equation}
with $7$ branch points $z_{1},\ldots,z_{7}$ that have been chosen rather
arbitrarily, but with the intention to get an evenly spread out configuration.
The seven values are given in (\ref{f65c}), further below, but their location
can readily be read from Figure \ref{fig0}.

The rather simple construction of the function $f$ makes it easy to understand
all possible meromorphic and single-valued continuations of $f$. Indeed, $f$
possesses a single-valued continuation throughout a domain $D\subset
\overline{\mathbb{C}}$ if, and only if, $\infty\in D$ and if each of the two
sets $\{z_{1},\ldots,z_{4}\}$ and $\{z_{5},z_{6},z_{7}\}$ of branch points is
connected in the complement $\overline{\mathbb{C}}\setminus D$.

The union of the $8$ arcs in Figure \ref{fig0} form the set of minimal
capacity for the function $f$, which we denote by $K_{0}(f,\infty)$, and by
$D_{0}(f,\infty):=\overline{\mathbb{C}}\setminus K_{0}(f,\infty)$ we denote
the extremal domain. Their definition and details about the calculation of the
minimal set $K_{0}(f,\infty)$ will be given in Section \ref{s2} and in the
discussion of Example \ref{e65} in Section \ref{s6}, further below.

Let $[63/62]_{f}$\ be the Pad\'{e} approximant of numerator and denominator
degree $63$\ and $62$, respectively, to the function $f$ developed at
infinity. In Figure \ref{fig0}, the poles of this approximant are represented
by stars. For any $n\in\mathbb{N}$\ the Pad\'{e} approximant $[n+1/n]_{f}=p/q$
is defined by the relation
\begin{equation}
f(z)q(\frac{1}{z})-p(\frac{1}{z})=\text{O}(z^{-2n-2})\text{ \ \ as
\ \ }z\rightarrow\infty\label{f1b}%
\end{equation}
with $p$ and $q$ polynomials of degree at most $n+1$ and $n$, respectively. An
comprehensive introduction to Pad\'{e} approximation can be found in
\cite{BakerGravesMorris}.

The connection between Pad\'{e} approximation and the minimal set
$K_{0}(f,\infty)$ will be established in the next theorem, which covers
functions of type (\ref{f1a}). It has been proved in \cite{Stahl97} (cf. also
\cite{BakerGravesMorris} Theorem 6.6.9), and is given here in a somewhat
shortened and specialized form.

\begin{theorem}
\label{t1a} For $n\rightarrow\infty$, the Pad\'{e} approximants $[n+1/n]_{f}$
converge to the function (\ref{f1a}) in capacity in the extremal domain
$D_{0}(f,\infty)\subset\overline{\mathbb{C}}$ associated with $f$, and this
convergence is optimal in the sense that it does not hold throughout any
domain $\widetilde{D}\subset\overline{\mathbb{C}}$ with $\operatorname*{cap}%
(\widetilde{D}\setminus D_{0}(f,\infty))>0$.
\end{theorem}

Theorem \ref{t1a} shows that extremal domains are convergence domains for
Pad\'{e} approximants, and this is also the case for our concrete example. In
Figure \ref{fig0} we observe that $61$\ out of $63$ poles of the Pad\'{e}
approximant $[63/62]_{f}$ are distributed very nicely along the $8$ arcs that
form the minimal set $K_{0}(f,\infty)=\overline{\mathbb{C}}\setminus
D_{0}(f,\infty)$. They are asymptotically distributed in accordance to the
equilibrium distribution on the minimal set $K_{0}(f,\infty)$ (cf.
\cite{Stahl97}, Theorem 1.8), and they mark the places, where we don't have convergence.

There are two poles that step out of line, and each one by a different reason:
One of them lies close to the origin, where it approximates the simple pole of
the function $f$ at the origin. Because of its correspondence to a pole of
$f$, it is called systematic.

The other one, which lies at $z=-3.35+2.66\,i$, does not correspond to a
singularity of the function $f$, and does obviously also not belong to any of
the chains of poles along the arcs in $K_{0}(f,\infty)$. Such poles are called
spurious in the theory of Pad\'{e} approximation. Spurious poles always appear
in combination with a nearby zero of the approximant. These pairs of poles and
zeros are close to cancellation. They are a phenomenon that unfortunately
cannot be ignored in Pad\'{e} approximation (cf. \cite{Stahl98},
\cite{Stahl97a}, or \cite{BakerGravesMorris} Chapter 6). Convergence in
capacity is compatible with the possibility of such spurious poles.

The convergence in capacity in Theorem \ref{t1a} implies that almost all poles
of the Pad\'{e} approximants $[n+1/n]_{f}$ have to leave the extremal domain
$D_{0}(f,\infty)$; they cluster on the minimal set $K_{0}(f,\infty)$. That
they do this in a rather regular way is shown in Figure \ref{fig0}. The
picture does not change much for other values of $n$ only that the location,
and possibly also the number of spurious poles may be different in each
case.\smallskip

If one wants to summarize the somewhat complicated convergence theory for
diagonal Pad\'{e} approximants in a short sentence one can say that extremal
domains are for Pad\'{e} approximants what discs are for power series.

\subsection{\label{s12}The Outline of the Manuscript}

\qquad In the next two Sections \ref{s2} and \ref{s3}, two alternative formal
definitions are given for the extremality problem under investigation. In the
second approach, the role of the function $f$ is taken over by a concrete
Riemann surface $\mathcal{R}$ over $\overline{\mathbb{C}}$. Both formulations
are equivalent.\smallskip

Illustrative examples are discussed in Section \ref{s6}, but before that in
the two Sections \ref{s4} and \ref{s5}, general results about minimal sets and
extremal domains are formulated and discussed. All proofs are postponed to
later sections.\smallskip

In Section \ref{s7}, a local version of the extremality problem is formulated
and discussed. After that in Section \ref{s8}, the extremality problem is
compared with some classical problems from geometric functions theory. For
such problems there exists a broad range of tools and techniques, as for
instance, boundary and inner variational methods, methods of extremal length,
and techniques connected with quadratic differentials (cf. \cite{Pommerenke75}%
, \cite{Goluzin}, \cite{Kuzmina82}, \cite{Kuzmina97}). Some of these ideas
will play a role in our investigation. We shall use a solution of one of these
problems as building block in one of our proofs.\smallskip

Practically, no proofs are given in the Sections \ref{s2} - \ref{s5} and
\ref{s7}; they all are all postponed to the Sections \ref{s9} and \ref{s10}.
In Section \ref{s110}, several auxiliary results from potential theory and
geometric function theory are assembled, of which some have been modified
quite substantially in order to fit their purpose in the present
paper.\smallskip

\subsection{\label{s13}Some Special Aspects}

\qquad It is a typical feature of the approach chosen in the present article
that a general existence and uniqueness proof is put at the beginning of the
analysis. This strategy has the advantage of giving great methodological
liberty in later proofs of special properties. At these later stages, the
knowledge of unique existence offers a free choice between different methods
and techniques from the tool boxes of geometric function theory; and because
of the uniqueness it is always clear that one is dealing with the same well
defined object. The prize to be paid for this strategy is a rather abstract
and somewhat heavy machinery for the uniqueness proof. The main tools there
are potential-theoretic in nature.\smallskip

It has been mentioned, and hopefully also illustrated by the introductory
example (\ref{f1a}), that extremality with respect to the logarithmic capacity
arises in a very natural way in connection with diagonal Pad\'{e}
approximants. In rational approximation also other types of capacity are of
interest, as for instance, condenser capacity or capacities in external
fields, which become relevant in connection with rational interpolants (cf.
\cite{Stahl96a}) or with essentially non-diagonal Pad\'{e} approximants. The
specific form of tools and methods in the present analysis should be helpful
for such potential generalizations.\medskip

\section{\label{s2}Basic Definitions and Unique Existence}

\qquad In the present section we introduce basic definitions and formulate a
theorem about the unique existence of a solution of the extremality problem.

Throughout the whole paper, we assume that $f$ is a function meromorphic in a
neighborhood of infinity, and denote its meromorphic extensions by the same
symbol $f$. By $\operatorname*{cap}(\cdot)$ we denote the (logarithmic) capacity.

\subsection{\label{s21}The Definition of Problem $(f,\infty)$}

\begin{definition}
\label{d21a}A domain $D\subset\overline{\mathbb{C}}$ is called admissible for
Problem $(f,\infty)$ if

\begin{itemize}
\item[(i)] $\infty\in D$, and if

\item[(ii)] $f$ has a single-valued meromorphic continuation throughout $D$.
\end{itemize}

By $\mathcal{D}(f,\infty)$ we denote the set of all admissible domains $D$ for
Problem $(f,\infty)$. A compact set $K\subset\mathbb{C}$\ is called admissible
for Problem $(f,\infty)$ if it is the complement $\overline{\mathbb{C}%
}\setminus D$ of an admissible domain $D\in\mathcal{D}(f,\infty)$. By
$\mathcal{K}(f,\infty)$ we denote the set of all admissible compact sets $K$
for Problem $(f,\infty)$.\smallskip
\end{definition}

Instead of meromorphic continuations, one could also consider analytic
continuations in condition (ii) of Definition \ref{d21a} without essentially
changing the whole concept. This later option has been taken in
\cite{Stahl85a}, \cite{Stahl85b}, and \cite{Stahl85c}. Meromorphic
continuations have been chosen here because of their natural affiliation with
rational approximation.\smallskip

\begin{definition}
\label{d21b}A compact set $K_{0}=K_{0}(f,\infty)\subset\mathbb{C}$ is called
minimal (or more lengthy: a set of minimal capacity with respect to Problem
$(f,\infty)$) if the following three conditions are satisfied:

\begin{itemize}
\item[(i)] $K_{0}\in\mathcal{K}(f,\infty)$.

\item[(ii)] We have
\begin{equation}
\operatorname*{cap}(K_{0})=\inf_{K\in\mathcal{K}(f,\infty)}\operatorname*{cap}%
(K).\label{f21a}%
\end{equation}

\item[(iii)] We have $K_{0}\subset K_{1}$ for all $K_{1}\in\mathcal{K}%
(f,\infty)$\ that satisfy condition (ii) with $K_{0}$\ replaced by $K_{1}$.
\end{itemize}

The domain $D_{0}(f,\infty):=\overline{\mathbb{C}}\setminus K_{0}(f,\infty) $
is called extremal with respect to Problem $(f,\infty)$ (or short: extremal domain).

By $\mathcal{K}_{0}(f,\infty)$ we denote\ the set of all admissible compact
sets $K$ of minimal capacity, i.e., all sets $K\in\mathcal{K}(f,\infty)$\ that
satisfy condition (ii), but not necessarily condition (iii), and by
$\mathcal{D}_{0}(f,\infty)$ the set of all admissible domains $D\in
\mathcal{D}(f,\infty)$ such that $\overline{\mathbb{C}}\setminus
D\in\mathcal{K}_{0}(f,\infty)$.
\end{definition}

With the introduction of the set of admissible domains $\mathcal{D}(f,\infty)$
and the definition of the extremal domain $D_{0}(f,\infty)$ together with its
complementary minimal set $K_{0}(f,\infty)$, Problem $(f,\infty)$ is fully
defined. The problem depends solely on the function $f$ given in neighborhood
of infinity.\smallskip

The point infinity plays a very special role for the function $f$ and also in
the definition of the (logarithmic) capacity, which is reflected in condition
(i) of Definition \ref{d21a}. This special role is the reason why the symbol
$\infty$ has been used besides of $f$ for the designation of Problem
$(f,\infty)$.\smallskip

\subsection{\label{s22}Unique Existence}

\qquad One of the central results in the present paper is the following
existence and uniqueness theorem.\smallskip

\begin{theorem}
[Unique Existence Theorem]\label{t22a} For any function $f$, which is
mero\-morphic in a neighborhood of infinity, there uniquely exists a minimal
set $K_{0}(f,\infty)$ and correspondingly a unique extremal Domain
$D_{0}(f,\infty)$ with respect to Problem $(f,\infty)$.\smallskip
\end{theorem}

Among the three conditions in Definition \ref{d21b}, condition (ii) is most
important, and condition (iii) plays only an auxiliary role. The situation
becomes evident by the next proposition.\smallskip

\begin{proposition}
\label{p22a} Elements of the set $\mathcal{K}_{0}(f,\infty)$\ differ at most
in a set of capacity zero, and we have
\begin{equation}
K_{0}(f,\infty)=\bigcap_{K\in\mathcal{K}_{0}(f,\infty)}K.\label{f22a}%
\end{equation}

\end{proposition}

The concept of extremal domains is most interesting if the function $f$ has
branch points. In the absence of branch points, the concept becomes in a
certain sense trivial, as the next proposition shows.\smallskip

\begin{proposition}
\label{p22b} If the function $f$ of Problem $(f,\infty)$ possesses no branch
poi\-nts, then the extremal domain $D_{0}(f,\infty)$ coincides with the
Weierstrass\ domain $W_{f}\allowbreak\subset\overline{\mathbb{C}}$ for
meromorphic continuation of the function $f$ starting at $\infty$.\smallskip
\end{proposition}

In Section \ref{s6} we shall discuss several concrete examples of functions
$f$ together with their extremal domains $D_{0}(f,\infty)$ and minimal sets
$K_{0}(f,\infty)$. These examples should give more substance to the formal
definitions in the present section.\smallskip

Several classical extremality problems from geometric function theory that are
defined by purely geometric constraints are reviewed in Section \ref{s8}.
There exist similarities with Problem $(f,\infty)$, but there are also
essential differences. The intention of the selection of examples in Section
\ref{s6} has been to illustrate these differences.\smallskip

\section{\label{s3}An Alternative Definition}

\qquad In the present section a definition of the extremality problem is given
that is equivalent to Problem $(f,\infty)$, but the role of the function $f$
is taken over by a Riemann surface $\mathcal{R}$. Of course,
single-valuedness, or better its absence, lies at the heart of the idea of a
Riemann surface, and so the alternative approach may shed light on the
geometric background of Problem $(f,\infty)$. Since in all later sections,
with the only exception of Subsection \ref{s42}, only Problem $(f,\infty)$
will be used as reference point, the alternative definition in the present
section can be skipped in a first reading.\smallskip

Let $\mathcal{R}$ be a Riemann surface over $\overline{\mathbb{C}}$, not
necessarily unbounded, and let $\pi:\mathcal{R}\longrightarrow\overline
{\mathbb{C}}$ be its canonical projection. We assume that $\infty\in
\pi(\mathcal{R})$.

\subsection{\label{s31}The Definition of Problem $(\mathcal{R},\infty^{(0)})
$}

\begin{definition}
\label{d31a} Let $\infty^{(0)}\ $be a point on the Riemann
surface$\ \mathcal{R}$ with $\pi(\infty^{(0)})=\infty$. Then a domain
$D\subset\mathcal{R}$ is called admissible for Problem $(\mathcal{R}%
,\infty^{(0)})$ if the following two conditions are satisfied:

\begin{itemize}
\item[(i)] $\infty^{(0)}\in D$.

\item[(ii)] The domain $D$ is planar (also called schlicht), i.e., $\pi
\mid_{D}$ is univalent, or in other words, we have $\operatorname*{card}%
\left(  (\pi^{-1}\circ\pi)(\{\zeta\})\cap D\right)  =1$ for all $\zeta\in D$.
\end{itemize}

By $\mathcal{D}(\mathcal{R},\infty^{(0)})$ we denote the set of all admissible
domains $D\subset\mathcal{R}$ for Problem $(\mathcal{R},\infty^{(0)}%
)$.\smallskip
\end{definition}

\begin{definition}
\label{d31b} A compact set $K\subset\mathbb{C}$ is admissible for Problem
$(\mathcal{R},\infty^{(0)})$ if it is of the form $K:=\overline{\mathbb{C}%
}\setminus\pi(D)$ with $D\in\mathcal{D}(\mathcal{R},\infty^{(0)})$.

By $\mathcal{K}(\mathcal{R},\infty^{(0)})$ we denote the set of all admissible
compact sets $K\subset\mathbb{C}$ for Problem $(\mathcal{R},\infty^{(0)}%
)$.\smallskip
\end{definition}

Notice that in contrast to admissible domains $D\in\mathcal{D}(f,\infty)$, now
admissible domains $D\in\mathcal{D}(\mathcal{R},\infty^{(0)})$ are subdomains
of the Riemann surface $\mathcal{R}$, while the admissible compact sets
$K\in\mathcal{K}(\mathcal{R},\infty^{(0)})$ remain to be subsets of
$\mathbb{C}$ like it has been the case in Definition \ref{d21a}.

Analogously to Definition \ref{d21b},\ we define the minimal set and the
extremal domain for Problem $(\mathcal{R},\infty^{(0)})$ as follows.\smallskip

\begin{definition}
\label{d31c} A compact set $K_{0}=K_{0}(\mathcal{R},\infty^{(0)}%
)\subset\overline{\mathbb{C}}$ is called minimal with respect to Problem
$(\mathcal{R},\infty^{(0)})$ if the following three conditions are satisfied:

\begin{itemize}
\item[(i)] $K_{0}\in\mathcal{K}(\mathcal{R},\infty^{(0)})$.

\item[(ii)] We have
\begin{align}
\operatorname*{cap}(K_{0})  & =\inf_{K\in\mathcal{K}(\mathcal{R},\infty
^{(0)})}\operatorname*{cap}(K)\label{f31a1}\\
& =\inf_{D\in\mathcal{D}(\mathcal{R},\infty^{(0)})}\operatorname*{cap}%
(\overline{\mathbb{C}}\setminus\pi(D)).\label{f31a2}%
\end{align}

\item[(iii)] We have $K_{0}\subset K_{1}$ for all $K_{1}\in\mathcal{K}%
(\mathcal{R},\infty^{(0)})$\ that satisfy assertion (ii) with $K_{0}%
$\ replaced by $K_{1}$.
\end{itemize}

A domain $D_{0}\in\mathcal{D}(\mathcal{R},\infty^{(0)})$ that satisfies
$\overline{\mathbb{C}}\setminus\pi(D_{0})=K_{0}(\mathcal{R},\infty^{(0)})$ is
called extremal with respect to Problem $(\mathcal{R},\allowbreak\infty
^{(0)})$, and it is denoted by $D_{0}(\mathcal{R},\infty^{(0)})$.

By $\mathcal{K}_{0}(\mathcal{R},\infty^{(0)})$ we denote\ the set of all
compact sets $K$ that satisfy the two conditions (i) and (ii), but not
necessarily condition (iii).\smallskip
\end{definition}

\subsection{\label{s32}Unique Existence and Equivalence}

\qquad For any Riemann surface $\mathcal{R}$ over $\overline{\mathbb{C}}$,
there exists a meromorphic function $f$ such that $\mathcal{R}=\mathcal{R}%
_{f}$ is the natural domain of definition of $f$. On the other hand, the
meromorphic continuation of a given function $f$, which is meromorphic in a
neighborhood of infinity, defines a Riemann surface $\mathcal{R}_{f}$ over
$\overline{\mathbb{C}}$ that contains a point $\infty^{(0)}\in\mathcal{R}_{f}$
with $\pi(\infty^{(0)})=\infty$, and this surface $\mathcal{R}_{f}$ is the
natural domain of definition for the function $f$.\ From these observations we
can conclude that the two Problems $(f,\infty)$ and $(\mathcal{R}_{f}%
,\infty^{(0)})$ are equivalent.

It is an immediate consequence of the equivalence of both problems that the
existence and uniqueness of a solution to Problem $(f,\infty)$ formulated in
Theorem \ref{t22a} carries over to Problem $(\mathcal{R},\infty^{(0)})$.
Details are formulated in the next theorem.\smallskip

\begin{theorem}
\label{t32a} (i) For any Riemann surface $\mathcal{R}$ over $\overline
{\mathbb{C}}$\ with $\infty^{(0)}\in\mathcal{R}$ and $\pi(\infty^{(0)}%
)=\infty$, there uniquely exists a minimal set $K_{0}=K_{0}(\mathcal{R}%
,\allowbreak\infty^{(0)})\subset\mathbb{C}$ for Problem $(\mathcal{R}%
,\infty^{(0)})$, and correspondingly, there also uniquely exists an extremal
domain $D_{0}=D_{0}(\mathcal{R},\allowbreak\infty^{(0)})\allowbreak
\subset\mathcal{R}$.

(ii) Let the Riemann surface $\mathcal{R}=\mathcal{R}_{f}$ be the natural
domain of definition for the function $f$, and let $f$\ be assumed to be
meromorphic in a neighborhood of infinity.\ Then the two extremal domains
$D_{0}(f,\infty)$\ and $D_{0}(\mathcal{R}_{f},\infty^{(0)})$\ of the
Definitions \ref{d21b} and \ref{d31c}, respectively, are identical up to the
canonical projection $\pi:\mathcal{R}_{f}\longrightarrow\overline{\mathbb{C}}%
$, i.e., we have
\begin{equation}
D_{0}(f,\infty)=\pi\left(  D_{0}(\mathcal{R}_{f},\infty^{(0)})\right)
.\label{f32a}%
\end{equation}
Further, we have
\begin{equation}
K_{0}(f,\infty)=K_{0}(\mathcal{R}_{f},\infty^{(0)}).\label{f32b}%
\end{equation}

\end{theorem}

\begin{proof}
We assume that the function $f$ has the Riemann surface $\mathcal{R}%
=\mathcal{R}_{f}$ as its natural domain of definition and that the function
element of $f$ at the point $\infty^{(0)}\in\mathcal{R}_{f}$, $\pi
(\infty^{(0)})=\infty$, is identical with the function $f$ at $\infty
\in\overline{\mathbb{C}}$.

It immediately follows from the two Definitions \ref{d21a} and \ref{d31a} that
for each domain $\widetilde{D}\in\mathcal{D}(\mathcal{R}_{f},\infty^{(0)})$ we
have $\pi(\widetilde{D})\in\mathcal{D}(f,\infty)$, and conversely, for each
domain $D\in\mathcal{D}(f,\infty)$ there exists an admissible domain
$\widetilde{D}\in\mathcal{D}(\mathcal{R}_{f},\infty^{(0)})$ with
$\pi(\widetilde{D})=D$.

After these preparations, the theorem is an immediate consequence of the
correspondence between the two sets $\mathcal{D}(\mathcal{R}_{f},\infty
^{(0)})$ and $\mathcal{D}(f,\infty)$ together with the two Definitions
\ref{d21b}, \ref{d31c}, and Theorem \ref{t22a}.\medskip
\end{proof}

The equivalence of the two Problems $(f,\infty)$ and $(\mathcal{R}_{f}%
,\infty^{(0)})$ allows us to opt freely for one of the two approaches. In the
present investigation we carry out the analysis in the framework of Problem
$(f,\infty)$. However, in applications it is sometimes favorable to start from
a Riemann surface $\mathcal{R}$. This approach will also give the intuitive
background for the discussion of concrete examples in Section \ref{s6}%
.\medskip

\section{\label{s4}Topological Properties}

\qquad Extremal problems in geometric function theory often lead to
topologically simply structured and smooth solutions. In the next two sections
it will be shown that a similar situation can be observed in our present investigations.

In Subsection \ref{s41} we address topological properties of the minimal set
$K_{0}(f,\allowbreak\infty)$, and corresponding results for the minimal set
$K_{0}(\mathcal{R},\allowbreak\infty^{(0)})$ associated with Problem
$(\mathcal{R},\infty^{(0)})$ are given in Subsection \ref{s42}.\smallskip

\subsection{\label{s41}Topological Properties of the Set $K_{0}(f,\infty)$}

\qquad The main result in the present section is a structure theorem for the
minimal set $K_{0}(f,\infty)$. As usual, the function $f$ is assumed to be
meromorphic in a neighborhood of infinity.\smallskip

\begin{theorem}
[Structure Theorem]\label{t41a} Let the function $f$ be meromorphic in a
neighborhood of infinity, and let $K_{0}=K_{0}(f,\infty)$\ be the minimal set
for Problem $(f,\infty)$. There exist two sets $E_{0},E_{1}\subset\mathbb{C}%
$\ and a family $\left\{  J_{j}\right\}  _{j\in I}$ of open and analytic
Jordan arcs such that
\begin{equation}
K_{0}(f,\infty)=E_{0}\cup E_{1}\cup\bigcup_{j\in I}J_{j},\label{f41a}%
\end{equation}
and the components in (\ref{f41a}) have the following properties:

\begin{itemize}
\item[(i)] We have $\partial E_{0}\subset\partial D_{0}(f,\infty)$, and at
each point $z\in\partial E_{0}$ the meromorphic continuation of the function
$f$\ has a non-polar singularity for at least one approach out of $D_{0}%
=D_{0}(f,\infty)$. The set $E_{0}\subset K_{0}$ is compact and
polynomial-convex, i.e., $\overline{\mathbb{C}}\setminus E_{0}$ is connected.

\item[(ii)] At each point $z\in E_{1}$ the function $f$\ has meromorphic
continuations\ out of $D_{0}$ from all possible sides, and these continuations
lead to more than $2$\ different function elements at the point $z$. The set
$E_{1}$ is discrete in $\overline{\mathbb{C}}\setminus E_{0}$.

\item[(iii)] All Jordan arcs $J_{j}$, $j\in I$, are contained in
$\overline{\mathbb{C}}\setminus(E_{0}\cup E_{1})$, they are pair-wise
disjoint, the function $f$ has meromorphic continuations to each point $z\in
J_{j}$, $j\in I$, from both sides of $J_{j}$ out of $D_{0}$, and these
continuations lead to $2$\ different function elements at each point $z\in
J_{j}$, $j\in I$.
\end{itemize}

The properties (i), (ii), and (iii) fully characterize all components on the
right-hand side of (\ref{f41a}).\smallskip
\end{theorem}

\begin{remark}
\label{r41a}The family of Jordan arcs $\left\{  J_{j}\right\}  _{j\in I}$ and
also the set $E_{1}$\ in (\ref{f41a})\ is empty if, and only if, all possible
meromorphic continuations of the function $f$ are single-valued, i.e., if the
function $f$\ has no branch points. This situation has already been addressed
in Proposition \ref{p22b}.\smallskip
\end{remark}

It follows from Theorem \ref{t41a} that the boundary $\partial D_{0}%
(f,\infty)$ is smooth everywhere on $\partial D_{0}(f,\infty)\setminus
(\partial E_{0}\cup E_{1})$. More information about this aspect is given in
the next theorem.\smallskip

\begin{theorem}
\label{t41b} The set $K_{0}(f,\infty)\setminus E_{0}$ is locally connected,
and only a finite number ($>2$) of arcs $J_{j}$, $j\in I$, meets at each point
of the set $E_{1}$.\smallskip
\end{theorem}

In the next section (cf. Remark \ref{r52b}), we shall see that the arcs
$J_{j}$ that meet at a point $z\in E_{1}$ form a regular star at $z$.\medskip

Before we close the present subsection, we will discuss the two influences
that determine the structure of the minimal set $K_{0}(f,\infty)$ in an
informal way.\smallskip

The principle of minimal capacity of the set $K_{0}(f,\infty)$ implies that
the extremal domain $D_{0}(f,\infty)$ is as large as possible, and
consequently it extends up to the natural boundary of the function $f$ (see
also Definition \ref{d71a0} in Subsection \ref{s71}, further below). On the
other hand, the requirement of single-valuedness of the function $f$ in
$D_{0}(f,\infty)$ can in general only be avoided by cuts in the complex plane
$\overline{\mathbb{C}}$; these cuts separate different branches of the
function $f$.

Both aspects, maximal extension and the principle of single-valuedness, find a
specific balance in the topological structure of the minimal set
$K_{0}(f,\infty)$. On one hand, there is the compact subset $E_{0}\subset
K_{0}(f,\infty)$, where on $\partial E_{0}$ meromorphic extensions of the
function $f$ find a natural boundary. On the other hand, there is the part
$K_{0}(f,\infty)\setminus E_{0}$ of $K_{0}(f,\infty)$, which essentially
consists of analytic Jordan arcs $J_{j}$, $j\in I$, which cut $\overline
{\mathbb{C}}\diagdown E_{0}$ in such a way that different branches of the
function $f$ are separated. They can be chosen with much liberty, and
therefore optimization is possible. This optimization is done according to the
principle of minimal capacity. We shall see in Section \ref{s5}, and more
specifically in Section \ref{s7}, how a balance between forces leads to a
state of equilibrium that determines the Jordan arcs $J_{j}$, $j\in
I$.\smallskip

\subsection{\label{s42}Topological Properties of the Set $K_{0}(\mathcal{R}%
,\infty^{(0)})$}

\qquad From Theorem \ref{t32a} we know that the two Problems $(f,\infty)$ and
$(\mathcal{R},\infty^{(0)})$ have equivalent solutions if there exists an
appropriate relationship between the Riemann surface $\mathcal{R}$\ and the
function $f$. As a consequence of this equivalence, there exists a description
of the topological properties of the set $K_{0}(\mathcal{R},\infty^{(0)})$
that corresponds to that given in Theorem \ref{t41a}. However, now the
function $f$ is no longer available, and its role has to be taken over by
properties of the Riemann surface $\mathcal{R}$.\smallskip

Let $\mathcal{R}$ be a Riemann surface over $\overline{\mathbb{C}}$. By
$\partial D$ and $\overline{D}$ we denote the boundary and the closure of a
domain $D\subset\mathcal{R}$ in $\mathcal{R}$. Further, we denote the set of
all branch points of $\mathcal{R}$ by $Br(\mathcal{R})\subset\mathcal{R}$, and
the relative boundary of the Riemann surface $\mathcal{R}$ over $\overline
{\mathbb{C}}$\ by $\partial\mathcal{R}$. We set $\widetilde{\mathcal{R}%
}:=\mathcal{R}\cup\partial\mathcal{R}$. If the Riemann surface $\mathcal{R}%
$\ is compact, then we have $\partial\mathcal{R=\emptyset}$.

The canonical projection $\pi:\mathcal{R}\longrightarrow\overline{\mathbb{C}}$
can be extended continuously to a projection $\widetilde{\pi}:\widetilde
{\mathcal{R}}\longrightarrow\overline{\mathbb{C}}$. We continue to denote the
boundary of a domain $D$ in$\ \widetilde{\mathcal{R}}$ by the same symbol
$\partial D$ as has been done in $\mathcal{R}$.

After this preparations, we are ready to formulate the analog of Theorem
\ref{t41a} for Problem $(\mathcal{R},\infty^{(0)})$.\smallskip

\begin{theorem}
\label{t42a} Let $D_{0}=D_{0}(\mathcal{R},\infty^{(0)})\subset\mathcal{R}$ and
$K_{0}=K_{0}(\mathcal{R},\infty^{(0)})\subset\overline{\mathbb{C}}$\ be the
uniquely existing extremal domain and minimal set, respectively, for Problem
$(\mathcal{R},\infty^{(0)})$. Like in Theorem \ref{t41a}, there exist two sets
$E_{0},E_{1}\subset\mathbb{C}$\ and a family $\left\{  J_{j}\right\}  _{j\in
I}$ of analytic, open Jordan arcs in $\mathbb{C}$ such that representation
(\ref{f41a}) holds true with $K_{0}(f,\infty)$ replaced by $K_{0}%
(\mathcal{R},\infty^{(0)})$, i.e., we have
\begin{equation}
K_{0}(\mathcal{R},\infty^{(0)})=E_{0}\cup E_{1}\cup\bigcup_{j\in I}%
J_{j}.\label{f42a0}%
\end{equation}

In the new situation, the components $E_{0},E_{1}$, and $\left\{
J_{j}\right\}  _{j\in I}$ in (\ref{f42a0})\ can be characterized by the
following properties:

\begin{itemize}
\item[(i)] The boundary $\partial E_{0}$\ of the compact set $E_{0}\subset
K_{0}$ is equal to
\begin{equation}
\widetilde{\pi}((\partial D_{0}\cap\partial\mathcal{R})\cup(Br(\mathcal{R}%
)\cap\overline{D_{0}})),\label{f42a}%
\end{equation}
and the set $E_{0}$\ is the polynomial-convex hull of $\partial E_{0}$. (For a
definition, see Definition \ref{d111b} in Subsection \ref{s1101}, further below).

\item[(ii)] The set $E_{1}\subset K_{0}$ is equal to
\begin{equation}
E_{1}:=\left\{  \text{ }z\in K_{0}\setminus E_{0}\text{ }\right|  \left.
\text{ }\operatorname*{card}(\pi^{-1}(\{z\})\cap\partial D_{0})>2\text{
}\right\} \label{f42b}%
\end{equation}
with $\pi$ being the canonical projection of $\mathcal{R}$ and not that of
$\widetilde{\mathcal{R}}$. The set $E_{1}\subset K_{0}$ is discrete in
$\overline{\mathbb{C}}\setminus E_{0}$.

\item[(iii)] If $I\neq\emptyset$, then $K_{0}\setminus(E_{0}\cup E_{1})$ is
the disjoint union of the analytic Jordan arcs $J_{j}$, $j\in I$. For each
point $z\in J_{j}$, $j\in I$, we have
\begin{equation}
\operatorname*{card}(\pi^{-1}(\{z\})\cap\partial D_{0})=2.\label{f42c}%
\end{equation}

\end{itemize}
\end{theorem}

\section{\label{s5}Analytic Characterizations}

\qquad We now come to analytic characterizations of the Jordan arcs $J_{j}$,
$j\in I$, in the minimal set $K_{0}(f,\infty)$ for Problem $(f,\infty)$. One
method is based on quadratic differentials, and a related one involves the
$S-$property (symmetry-property) of the extremal domain $D_{0}(f,\infty) $. In
the last subsection we consider the special case that the set $E_{0}$ in
Theorem \ref{t41a} is finite, which leads to the interesting special case of
rational quadratic differentials.

All results in the present section are formulated in the framework of Problem
$(f,\infty)$. Their transfer to Problem $(\mathcal{R},\infty^{(0)}) $ is
easily possible with the tools presented in Section \ref{s3} and Subsection
\ref{s42}.\smallskip

\subsection{\label{s51}The $S-$Property}

\qquad A characteristic property of the extremal domain $D_{0}=D_{0}%
(f,\infty)$ for Problem $(f,\infty)$ is a specific behavior of the Green
function $g_{D_{0}}(\cdot,\infty)$ on the Jordan arcs $J_{j}$, $j\in I$, in
$K_{0}(f,\infty)$ that have been introduced in (\ref{f41a}) of Theorem
\ref{t41a}. For a definition of the Green function we refer to Subsection
\ref{s1103}, further below.\smallskip

\begin{theorem}
\label{t51a} Under the assumptions made in Theorem \ref{t41a}, we have
\begin{equation}
\frac{\partial}{\partial n_{+}}g_{D_{0}}(z,\infty)=\frac{\partial}{\partial
n_{-}}g_{D_{0}}(z,\infty)\text{ \ for all \ }z\in J_{j}\text{, }j\in
I\text{,}\label{f51a}%
\end{equation}
with $\partial/\partial n_{+}$ and $\partial/\partial n_{-}$ denoting the
normal derivatives to both sides of the arcs $J_{j}$, $j\in I$, that have been
introduced in (\ref{f41a}) of Theorem \ref{t41a}.\smallskip
\end{theorem}

The symmetric boundary behavior (\ref{f51a}) of the Green function $g_{D_{0}%
}(\cdot,\infty)$ is called the $S-$property of the extremal domain
$D_{0}(f,\allowbreak\infty)$. In Section \ref{s7}, below, it will be shown
that the $S-$property can be interpreted as a local condition for the
minimality (\ref{f21a}) in Definition \ref{d21b}.

While in Theorem \ref{t51a} we get the $S-$property as a consequence of the
minimality (\ref{f21a}) in Definition \ref{d21b}, it will be proved in Theorem
\ref{t73a} in Subsection \ref{s73} that the $S-$property is even equivalent to
the minimality (\ref{f21a}). As a consequence of this further going result it
follows that the $S-$property can also be used as an alternative
characterization of the extremal domain $D_{0}(f,\infty)$.\smallskip

Notice that $I\neq\emptyset$ in (\ref{f51a}) implies $\operatorname*{cap}%
(K_{0})>0$, and consequently, in this case, the Green function $g_{D_{0}%
}(\cdot,\infty)$ in (\ref{f51a})\ exists in a proper sense (cf. Subsection
\ref{s1103}, further below).\ If on the other hand, we have $I=\emptyset$,
then relation (\ref{f51a}) is void.

From Theorem \ref{t41a} we now know that the arcs $J_{j}$, $j\in I$, are
analytic. Hence, the Green function $g_{D_{0}}(\cdot,\infty)$\ has harmonic
continuations across each arc $J_{j}$ from both sides (cf. Subsection
\ref{s1103}), and consequently the normal derivatives in (\ref{f51a}) exist
for each $z\in J_{j}$, $j\in I$.\smallskip

\subsection{\label{s52}Quadratic Differentials}

\qquad The $S-$property can be described in an equivalent way by quadratic
differentials. We say that a smooth arc $\gamma$ with parametrization
$z:[0,1]\longrightarrow\overline{\mathbb{C}}$ is a trajectory of the quadratic
differential $q(z)dz^{2}$ if we have
\begin{equation}
q(z(t))\overset{\bullet}{z}(t)^{2}<0\text{ \ \ \ for all \ \ }t\in
(0,1).\label{f52a}%
\end{equation}

We note that there exists an associated family of orthogonal trajectories,
which are defined by the same relation (\ref{f52a}), but with an inequality
showing in the other direction. As general reference to quadratic
differentials and their trajectories we use \cite{Strebel84} or
\cite{Jensen75}. Some of its local properties are assembled in Subsection
\ref{s1105}, further below.\medskip

\begin{theorem}
\label{t52a}Let $D_{0}=D_{0}(f,\infty)$, $E_{0},E_{1}\subset\overline
{\mathbb{C}}$,\ and $\left\{  J_{j}\right\}  _{j\in I}$ be the objects
introduced in Theorem \ref{t41a}, and let $g_{D_{0}}(\cdot,\infty)$ be the
Green function in $D_{0}$. Then the Jordan arcs $J_{j}$, $j\in I$, are
trajectories of the quadratic differential $q(z)dz^{2}$ with $q$ defined by
\begin{equation}
q(z):=\left(  2\frac{\partial}{\partial z}g_{D_{0}}(z,\infty)\right)
^{2},\label{f52b}%
\end{equation}
where $\partial/\partial z=\frac{1}{2}\left(  \partial/\partial x-i\,\partial
/\partial y\right)  $ is the usual complex differentiation. The function $q$
has a meromorphic (single-valued) continuation throughout the domain
$\overline{\mathbb{C}}\setminus E_{0}^{\prime}$ with $E_{0}^{\prime} $
denoting the sets of cluster points of $E_{0}$. Near infinity we have
\begin{equation}
q(z)=\frac{1}{z^{2}}+\text{O}(z^{-3})\text{ \ \ as \ \ }z\rightarrow
\infty.\label{f52c}%
\end{equation}
The function $q$ has at most simple poles in isolated points of $E_{0}$, and
it is analytic throughout $\overline{\mathbb{C}}\setminus E_{0}$.\medskip
\end{theorem}

It is not difficult to verify that the meromorphy of the function $q$ in
$\overline{\mathbb{C}}\setminus E_{0}^{\prime}$ is equivalent to the
$S-$property (\ref{f51a}).\smallskip

The local structure of the trajectories of quadratic differentials can rather
easily be understood and described (for more details see Subsection
\ref{s1105}, further below). Of special interest are neighborhoods of poles
and zeros of the function $q$ in (\ref{f52b}).\smallskip

\begin{remark}
\label{r52b} Since we know from Theorem \ref{t52a} that all Jordan arcs
$J_{j}$, $j\in I$, are trajectories of a quadratic differential $q(z)dz^{2}$
that is meromorphic in $\overline{\mathbb{C}}\setminus E_{0}^{\prime}$,\ it
follows from the local structure of the trajectories that all Jordan arcs
$J_{j}$, $j\in I$, that end at an isolated point $z$ of$\ E_{0}\cup E_{1}$
form a regular star at this point.\smallskip
\end{remark}

\subsection{\label{s53}Rational Quadratic Differentials}

\qquad The description of the Jordan arcs $J_{j}$, $j\in I$, as trajectories
of a quadratic differential $q(z)dz^{2}$ is especially constructive if the
function $q$ in (\ref{f52b}) is rational. This is the case if the set $E_{0}$
from Theorem \ref{t41a} is finite. Algebraic functions $f$ are prototypical
examples for this situation.\smallskip

For the formulation of the main result in this direction, we need the notion
of bifurcation points in $K_{0}(f,\infty)$, the associated bifurcation index,
and the notion of critical points of the Green function $g_{D_{0}}%
(\cdot,\infty)$.\smallskip

\begin{definition}
\label{d53a} Let the objects $K_{0}=K_{0}(f,\infty)$, $E_{0},E_{1}%
\subset\overline{\mathbb{C}}$,\ and $\left\{  J_{j}\right\}  _{j\in I}$ be and
ones as in the Theorems \ref{t41a} or \ref{t52a}. For each isolated point
$z\in E_{1}\cup E_{0}$, the bifurcating index $i(z)$ is\ the number of
different Jordan arcs $J_{j}$, $j\in I$, that end at this point $z$.\smallskip
\end{definition}

If $z$ is an isolated point of $K_{0}=K_{0}(f,\infty)$, then $z$ lies
necessarily in $E_{0}$, and by definition we have $i(z)=0$ since $z$ has no
contact to any arc in $K_{0}$. Such isolated points can exist; they are
generated by isolated, essential singularities of the function $f$ that are no
branch points.\smallskip

\begin{definition}
\label{d53b} Let $D_{0}=D_{0}(f,\infty)$ be the extremal domain, and assume
that $\operatorname*{cap}(K_{0}(f,\infty))>0$. By $E_{2}\subset D_{0}$ we
denote the set of all critical points of the Green function $g_{D_{0}%
}(z,\infty)$, and for each $z\in E_{2}$ we denote the order of the critical
point $z$ by $j(z),$ i.e., for $z\in E_{2}$, we have
\begin{equation}
\frac{\partial^{l}}{\partial z^{l}}g_{D_{0}}(z,\infty)\left\{
\begin{array}
[c]{lll}%
=0\smallskip & \text{ \ for \ } & l=1,\ldots,j(z)\\
\neq0 & \text{ \ for \ } & l=j(z)+1.
\end{array}
\right. \label{f53a}%
\end{equation}
If $\operatorname*{cap}(K_{0}(f,\infty))=0$, then we set $E_{2}=\emptyset
$.\smallskip
\end{definition}

The sets $E_{1}$ and $E_{2}$\ are always discrete in $\overline{\mathbb{C}%
}\setminus E_{0}$, while the set $E_{0}$ can be a mixture of isolated and
cluster points. Because of this later possibility, it was necessary to
distinguish the set $E_{0}^{\prime}$ of cluster points from the original set
$E_{0}$ in Theorem \ref{t52a}. The set $E_{1}\cup E_{0}\setminus E_{0}%
^{\prime}$\ contains all isolated points of $E_{1}\cup E_{0}$. We have
$E_{0}^{\prime}=\emptyset$ if and only if $E_{0}$ is finite.\smallskip

\begin{proposition}
\label{p53a} If $E_{0}$ is a finite set, then the sets $E_{1}$ and $E_{2}%
$\ are necessarily also finite.\smallskip\smallskip
\end{proposition}

After these preliminaries, we are ready to formulate the central results of
the present subsection.\smallskip

\begin{theorem}
\label{t53a} We use the same notations as in the Theorems \ref{t41a} and
\ref{t52a}, and assume that the set $E_{0}$ is finite. Then the function $q$
in (\ref{f52b}) is rational, and we have the explicit representation
\begin{equation}
q(z)=\prod_{v\in E_{0}\cup E_{1},\text{ }i(v)>0}(z-v)^{i(v)-2}\prod_{v\in
E_{2}}(z-v)^{2\,j(v)}.\smallskip\label{f53b}%
\end{equation}

\end{theorem}

Notice that there always exist points $z\in E_{0}$ with $i(z)=1$, which
implies that $q$ always is a broken rational function. Actually, this
assertion follows already from (\ref{f52c}) in Theorem \ref{t52a}, and further
we deduce from (\ref{f52c}) that the denominator degree of $q$ is exactly $2$
degrees larger than its numerator degree.\smallskip

The explicit formula (\ref{f53b}) for $q$ can be very helpful for the
numerical calculation of the analytic Jordan arcs $J_{j}$, $j\in I$, in
$K_{0}(f,\infty)$. If the points of the sets $E_{0}$, $E_{1}$, and $E_{2}$
have been determined, then most of the work is done, and one can calculate the
Jordan arcs $J_{j}$, $j\in I$, by solving a differential equation that is
based on (\ref{f52a}), (\ref{f52b}), and (\ref{f53b}). This procedure has, for
instance, also been used for the calculation of the arcs in the minimal sets
$K_{0}(f_{j},\infty)$, $j=1,\ldots,5$, in the Examples \ref{e61} - \ref{e65}
that follow next. The critical part of the job is the calculation of the zeros
of the function $q$ in (\ref{f53b}). More information about this topic can be
found at the end of the discussion of Example $f_{3}$ in Subsection
\ref{e63}.\medskip

\section{\label{s6}Examples}

\qquad In the present section we consider five specially chosen algebraic
functions $f=f_{1},\ldots,f_{5}$, and discuss for each of them the solution of
Problem $(f,\infty)$. Typically, we calculate and plot the minimal set
$K_{0}(f,\infty)$, discuss particular features of its shape, and identify the
sets $E_{0}$, $E_{1}$, $E_{2}$, and the family of Jordan arcs $J_{j}$, $j\in
I$, that have been introduced in Theorem \ref{t41a} and in Definition
\ref{d53b}.\ Also the quadratic differential $q(z)dz^{2}$ from Theorem
\ref{t53a} is identified for each case.

Some of the examples depend on one or two parameters; and variations of these
parameters will be done in order to understand the mechanisms that lead to
special features of the minimal set $K_{0}(f,\infty)$. Of special interest are:

\begin{itemize}
\item[a)] The connectivity of the minimal set $K_{0}(f,\infty)$ together with
the question of how it changes under variations of the function $f$.

\item[b)] The identification of active versus inactive branch points of $f$.
It turns out that in general not all branch points of the function $f$ play an
active role in the determination of the minimal set $K_{0}(f,\infty)$, and for
the calculation of $K_{0}(f,\infty)$ it is important to know already in
advance which of them are active and which ones remain passive.
\end{itemize}

The presentation and discussion of the five examples demands comparatively
much space, and there has been some hesitation to include all the material.
But it is hoped that the expenses on space and efforts are counterbalanced by
an improved understanding of the definitions and results presented in the last
four sections.\smallskip

\subsection{\label{e61}Example $f_{1}$}

\qquad As a first, and in most aspects rather trivial example, we consider the
function
\begin{equation}
f_{1}(z):=\frac{1}{\sqrt{z^{2}-1}},\label{f61a}%
\end{equation}
which often appears in approximation theory, and has been included here as a
warm-up exercise.

Clearly, the function has branch points at $-1$\ and $1$. Therefore, the set
$\mathcal{D}(f_{1},\infty)$ of admissible domains for Problem $(f_{1},\infty)$
from Definition \ref{d21a}\ consists of all domains $D\subset\overline
{\mathbb{C}}$ such that $\infty\in D$ and that the two points $-1$ and $1$ are
connected in the complement $K=\overline{\mathbb{C}}\setminus D$. The uniquely
existing extremal domain of Theorem \ref{t22a} is given by
\begin{equation}
D_{0}(f_{1},\infty)=\overline{\mathbb{C}}\setminus\lbrack-1,1],\label{f61b}%
\end{equation}
and the minimal set by $K_{0}(f_{1},\infty)=[-1,1]$. As sets $E_{0}$, $E_{1}
$, $E_{2}$,\ and arcs $J_{j}$, $j\in I$, introduced in Theorem \ref{t41a} and
in Definition \ref{d53b}, we have $E_{0}=\{-1,1\}$, $E_{1}=E_{2}=\emptyset$,
$I=\{1\}$, and $J_{1}=(-1,1)$. Solution (\ref{f61b}) is a consequence of the
monotonicity of $\operatorname*{cap}(\cdot)$ under projections onto straight
lines (cf., Lemma \ref{l111c} in Subsection \ref{s1101}, further below). The
single arc $J_{1}=(-1,1)$ in $K_{0}(f_{1},\infty)$\ is a trajectory of the
quadratic differential
\begin{equation}
\frac{1}{z^{2}-1}dz^{2},\label{f61c}%
\end{equation}
i.e., it satisfies the relation
\begin{equation}
\frac{1}{z^{2}-1}dz^{2}<0,\label{f61d}%
\end{equation}
and (\ref{f61c}) corresponds to Theorem \ref{t53a}\smallskip.

\subsection{\label{e62}Example $f_{2}$}

\qquad Next, we consider a function $f_{2}$ that depends on a parameter
$\varphi$. For $\varphi\in\left(  0,\pi/2\right)  $, we define $\left(
\varphi_{j}\right)  _{j=1,\ldots,4}:=(\varphi,\pi-\varphi,\pi+\varphi
,2\pi-\varphi)$, $z_{j}:=\exp(i\,\varphi_{j})$, $j=1,\ldots,4$, $P_{4}%
(z):=\prod_{j=1}^{4}(1-z_{j}/z)$, and then we define the function $f_{2} $ as
\begin{equation}
f_{2}(z):=\sqrt[2]{P_{4}(z)}\label{f62a}%
\end{equation}
with a choice of the sign of the square root in (\ref{f62a}) so that
$f_{2}(\infty)=1$. The function $f_{2}$ has the four branch points
$z_{1},\ldots,z_{4}$, and it is symmetric with respect to the real and the
imaginary axis. The symmetries lead to corresponding symmetries of the minimal
set $K_{0}(f_{2},\infty)$ and the extremal domain $D_{0}(f_{2},\infty)$ for
each $\varphi\in\left(  0,\pi/2\right)  $.%
\begin{figure}
[ptb]
\begin{center}
\includegraphics[
trim=0.000000in 0.000000in 0.000000in -0.058759in,
height=4.1297cm,
width=12.0441cm
]%
{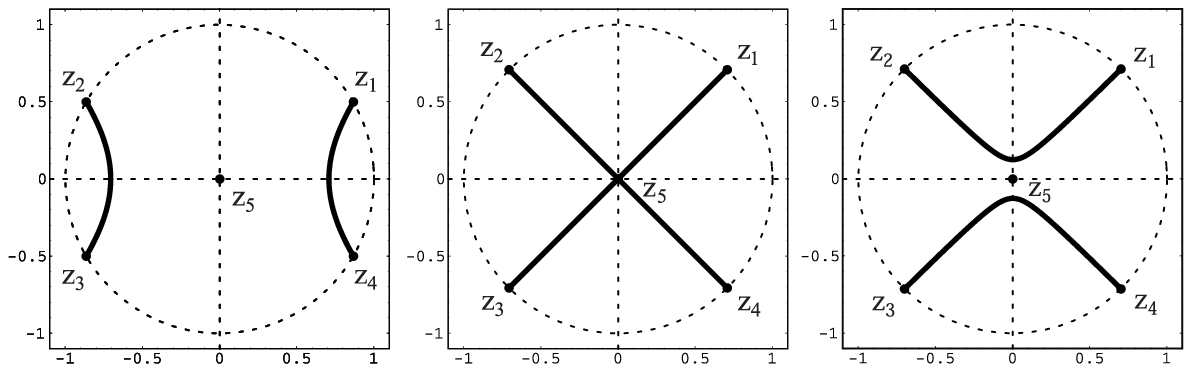}%
\caption{Three examples of minimal sets $K_{0}(f_{2},\infty)$ for Problem
$(f_{2},\infty)$ with $f_{2}$ defined in (\ref{f62a}). The three windows
corresponed to the three parameter values $\varphi=\pi/6$, $\varphi=\pi/4$,
and $\varphi=101\pi/400$, respectively.}%
\label{fig1}%
\end{center}
\end{figure}

For each $\varphi\in\left(  0,\pi/2\right)  $, the set $\mathcal{D}%
(f_{2},\infty)$ of admissible domains introduced in Definition \ref{d21a}
consists of all domains $D\subset\overline{\mathbb{C}}$ such that $\infty\in
D$ and that at least two disjoint pairs of the four branch points
$z_{1},\ldots,z_{4}$ are connected in $K=\overline{\mathbb{C}}\setminus D$. It
is not necessary that all four points $z_{1},\ldots,z_{4}$ are connected, nor
that a specific combination of pairs has to be connected in $K=\overline
{\mathbb{C}}\setminus D$.

From the uniqueness of the minimal set $K_{0}(f_{2},\infty),$ which has been
proved in Theorem \ref{t22a}, it follows that from the variety of
connectivities that are possible for the set $K\in\mathcal{K}(f_{2},\infty) $
and a given fixed parameter value $\varphi$, a specific one is selected as the
minimal set $K_{0}(f_{2},\infty)$.

The shape and the connectivity of the minimal set $K_{0}(f_{2},\infty)$
depends on the parameter $\varphi$, and we distinguish the three cases
$0<\varphi<\pi/4$, $\varphi=\pi/4$, and $\pi/4<\varphi<\pi/2$, which we will
label as cases $a$, $b$, and $c$, respectively. In the three windows of Figure
\ref{fig1}, the three cases are represented by the minimal sets $K_{0}%
(f_{2},\infty)$ for the parameter values $\varphi=\pi/6,$ $\varphi=\pi/4$, and
$\varphi=101\pi/400$, respectively. The value $\varphi=101\pi/400$ has been
chosen to be close to the critical value $\varphi=\pi/4$. The picture in the
third window gives an impression of the metamorphosis of the set $K_{0}%
(f_{2},\infty)$ when $\varphi$\ approaches and then crosses the critical value
$\varphi_{0}=\pi/4$.

In the two cases $a$\ and $c$, the minimal set $K_{0}(f_{2},\infty)$\ consists
of two components. We have $E_{0}=\{z_{1},\ldots,z_{4}\}$, $E_{1}=\emptyset$,
$E_{2}=\{0\}$, $I=\{1,2\}$, and the two analytic Jordan arcs $J_{1}$ and
$J_{2}$ in $K_{0}(f_{2},\infty)$\ which connect the two pairs of branch points
$\{z_{1},z_{4}\}$ and $\{z_{2},z_{3}\}$ in case $a$ and the two pairs
$\{z_{1},z_{2}\}$ and $\{z_{3},z_{4}\}$ in case $c$.

The case $b$ corresponds to the single parameter value $\varphi=\pi/4$. Here,
all four branch points $z_{1},\ldots,z_{4}$ are connected in $K_{0}%
(f_{2},\infty)$; the set is a continuum. We have $E_{0}=\{z_{1},\ldots
,z_{4}\}$, $E_{1}=\{0\}$, $E_{2}=\emptyset$, $I=\{1,\ldots,4\}$, and the four
Jordan arcs $J_{1},\ldots,J_{4}$ in $K_{0}(f_{2},\infty)$\ are the four
segments $(0,z_{j})$, $j=1,\ldots,4$.

The two Jordan arcs $J_{1}$ and $J_{2}$ in the two cases $a$ and $c$, and also
the $4$ Jordan arcs $J_{1},\ldots,\allowbreak J_{4}$ in case $b,$\ are
trajectories of the quadratic differential
\begin{equation}
\frac{z^{2}}{\prod_{j=1}^{4}(z-z_{j})}dz^{2}.\label{f62b}%
\end{equation}

Taking advantage of the symmetry of the function $f_{2}$, one can show that
for each $\varphi\in\left(  0,\pi/2\right)  \setminus\{\pi/4\}$ the two arcs
$J_{1}$ and $J_{2}$ are sections of an hyperbole. Indeed, it is not difficult
to verify that the mapping $z\mapsto z^{2}$ maps the two arcs $J_{1}$ and
$J_{2}$ onto one straight segment, which proves this last assertion.\smallskip

\subsection{\label{e63}Example $f_{3}$}

\qquad The third example is very similar to the second one, only that now the
forth root is taken instead of the square root in (\ref{f62a}). We use the
same definitions for $\varphi$,$\ \varphi_{j}$, $z_{j}$, $j=1,\ldots,4 $, and
$P_{4}$ as in Example \ref{e62}, and define function $f_{3}$\ as
\begin{equation}
f_{3}(z):=\sqrt[4]{P_{4}(z)}.\label{f63a}%
\end{equation}
The branch of the root $\sqrt[4]{\cdot}$ is chosen so that $f_{3}(\infty)=1 $.
Although the basic structure of the two functions $f_{3}$\ and $f_{2}$\ is
very similar, there exist decisive differences with respect to their
meromorphic continuability. For each parameter value $\varphi\in\left(
0,\pi/2\right)  $, the set $\mathcal{D}(f_{3},\infty)$ of admissible domains
for Problem $(f_{3},\infty)$ consists of all domains $D\subset\overline
{\mathbb{C}}$ such that $\infty\in D$ and all four branch points $z_{1}%
,\ldots,z_{4}$ are connected in the complementary set $K=\overline{\mathbb{C}%
}\setminus D$.%
\begin{figure}
[ptb]
\begin{center}
\includegraphics[
height=2.0903in,
width=4.5186in
]%
{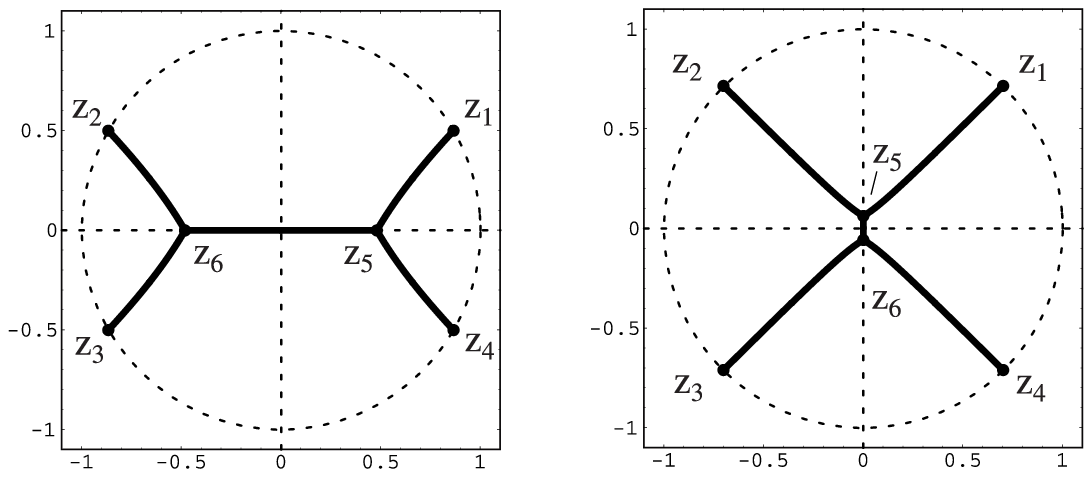}%
\caption{Two examples of minimal sets $K_{0}(f_{3},\infty)$ for Problem
$(f_{3},\infty)$ with $f_{3}$ defined in (\ref{f63a}). The two windows
correspond to the parameter values $\varphi=\pi/6$ and $\varphi=101\pi/400$,
respectively.}%
\label{fig2}%
\end{center}
\end{figure}

As in Example \ref{e62}, we distinguish three cases $a$, $b$, and $c$, which
are again defined by $0<\varphi<\pi/4$, $\varphi=\pi/4$, and $\pi
/4<\varphi<\pi/2$, respectively. In all three cases, the minimal set
$K_{0}(f_{3},\infty)$ is connected, and the extremal domain $D_{0}%
(f_{3},\infty)$ is simply connected. However, the minimal set $K_{0}%
(f_{3}\allowbreak\infty)$ is of a somewhat different structure in each of the
three cases.

In case $b,$ the two functions $f_{2}$ and $f_{3}$ have an identical extremal
domain $D_{0}(f_{3},\infty)$ and an identical minimal set $K_{0}(f_{3}%
,\infty)=K_{0}(f_{2},\infty)$. The minimal set has already been shown in the
middle window of Figure \ref{fig1}.

For the two other cases $a$ and $c$, two representatives of the minimal sets
$K_{0}(f_{3},\allowbreak\infty)$ are shown in Figure \ref{fig2}. The two cases
are represented by the same two parameter values $\varphi=\pi/6$ and
$\varphi=101\,\pi/400$ as already used before in Figure \ref{fig1}. A new
phenomenon now is the appearance of two bifurcation points in $K_{0}%
(f_{3},\infty)$, which are denoted by $z_{5}$ and $z_{6}$ in Figure \ref{fig2}.

In the two cases $a$\ and $c$, we have $E_{0}=\{z_{1},\ldots,z_{4}\}$,
$E_{1}=\{z_{5},z_{6}\}$, $E_{2}=\emptyset$, $I=\{1,\ldots,5\}$, and the five
open analytic Jordan arcs $J_{1},\ldots,J_{5}$ in $K_{0}(f_{3},\infty)
$\ connect the six points $z_{1},\ldots,z_{6}$ as shown in Figure \ref{fig2}.
These five Jordan arcs $J_{1},\ldots,\allowbreak J_{5}$, and also the four
arcs in case $b$, are trajectories of the quadratic differential
\begin{equation}
\frac{(z-z_{5})(z-z_{6})}{\prod_{j=1}^{4}(z-z_{j})}dz^{2}.\label{f63b}%
\end{equation}

Notice that in case $b$, we have $z_{5}=z_{6}=0$. In the two other cases, we
always have $z_{5}=-z_{6}\neq0$. From a practical point of view the
calculation of the two bifurcation points $z_{5}$ and $z_{6}$ is the main work
and causes the main difficulties for the calculation of the arcs $J_{1}%
,\ldots,\allowbreak J_{5}$. We want to take a closer look on this
problem.\smallskip

The form of the quadratic differential (\ref{f63b}) already suggests that
elliptic integrals should play a role in the analytic determination of the
bifurcation points $z_{5}$ and $z_{6}$. Indeed, with the machinery presented
in \cite{Lowien73}, \cite{Blerch82}, or \cite{Pirl69}, it is not too difficult
to formulate conditions that allow to determine the points $z_{5}$ and $z_{6}%
$. We reproduce the main elements of the procedure for case $a$, i.e., for the
case $\varphi\in\left(  0,\pi/4\right)  $, and define the function
\begin{equation}
g(a,x):=\left\vert \sqrt{\frac{x-a}{x\,(x^{2}-2x\cos(2\varphi)+1)}}\right\vert
\text{ \ for \ }a\in(0,1)\text{, \ }x\in\mathbb{R.}\label{f63c}%
\end{equation}
The improper elliptical integral
\begin{equation}
I(a):=\lim_{c\rightarrow+\infty}\left[  \int_{a}^{c}g(a,x)dx-\int_{-c}%
^{0}g(a,x)dx\right] \label{f63d}%
\end{equation}
is strictly monotonic for $a\in(0,1)$, and we have $I(0)>0$ and $I(1)<0$.
Consequently, there uniquely exists $a_{0}\in(0,1)$ with $I(a_{0})=0$. The two
bifurcation points $z_{5}$ and $z_{6}$ are then given by
\begin{equation}
z_{5}=z_{5}(\varphi)=+\sqrt{a_{0}},\text{ \ \ }z_{6}=-z_{5}.\label{f63e}%
\end{equation}
For the special parameter value $\varphi=\pi/6$, for which the corresponding
minimal set $K_{0}(f_{3},\infty)$ is shown in the first window of Figure
\ref{fig2}, we get%

\begin{equation}
a_{0}=0.231584\text{, \ \ }z_{5}=0.481232\text{, \ and \ }z_{6}%
=-0.481232.\label{f63f}%
\end{equation}

In a derivation of the expressions (\ref{f63c})\ and (\ref{f63d}), one has in
a first step to transform the minimal set $K_{0}(f_{3},\infty)$\ by the
mapping $z\mapsto z^{2}$ into a continuum that connects the three points $0$,
$e^{i2\varphi}$, and $e^{-i2\varphi}$.

After the reduction to a three-point problem, one can apply results that have
been proved in \cite{Kuzmina68} (see also \cite{Kuzmina82}, Theorem 1.5). In
\cite{Kuzmina82}, Theorem 1.5, the value $a_{0}$ is expressed as the solution
of a system of four equations that involve Jacobi elliptical functions and
theta functions. We have not investigated whether the approach is numerically
easier to handle than the equation $I(a_{0})\overset{!}{=}0$, which is based
on (\ref{f63d}). In any case, the level of difficulties that arise already in
this rather simply structured case of function $f_{3}$ gives an idea of the
type of difficulties that arise if one has to determine the points in the set
$E_{1}$ (and $E_{2}$) in a more general situation. In the next two examples
these points have been calculated by a numerical method that has been
developed by the author on an ad-hoc basis. It is based on a geometrical
approach. The method will be published in a separate paper. Further comments
about the numerical side of the problem will be made in Subsection \ref{s83},
further below.

\subsection{\label{e64}Example $f_{4}$}

\qquad In the fourth example, we consider a modification of the function
$f_{3}$, which itself has already been a modification of function $f_{2}$. We
use again the definitions $\varphi$,$\ \varphi_{j}$, $z_{j}$, $j=1,\ldots,4$,
and $P_{4}$ from Example \ref{e62}, and define the new function $f_{4}$\ as
\begin{equation}
f_{4}(z):=\sqrt[2]{\sqrt[2]{P_{4}(z)}-c}.\label{f64a}%
\end{equation}
In addition to the former parameter $\varphi$, there is now a second parameter
$c$, which may assume arbitrary complex values $c\in\mathbb{C}$, but we shall
consider only special situations. We discuss complex values of $c$ that lie
near the origin, and in addition real values of $c$ in the interval $(0,1)$.
The signs of the inner and outer square root in (\ref{f64a}) are assumed to be
chosen in such a way that both roots are positive for $z=\infty$ and $c=0$. In
case of $c=0$, the two functions $f_{4}$ and $f_{3}$\ are identical.\smallskip

The study of the function $f_{4}$ and its associated minimal set $K_{0}%
(f_{4},\infty)$ will be more complex and involved than that of the last two
examples, which in some sense have been preparations of the present example.
Our main interest will be concentrated on the following three questions:

\begin{itemize}
\item[1)] It is not difficult to see that for almost all parameter
constellations the function $f_{4}$ has $8$ branch points. But not all of them
will always play an active role in the determination of the minimal set
$K_{0}(f_{4},\infty)$, some of them are hidden away from $K_{0}(f_{4},\infty)$
somewhere on a 'lower' sheet of the Riemann surface $\mathcal{R}_{f_{4}}$ that
is defined by $f_{4}$. In the terminology of Section \ref{s3}, we can say that
these inactive branch points on $\mathcal{R}_{f_{4}}$ stay away from the
extremal domain $D_{0}(\mathcal{R}_{f_{4}},\infty^{(0)})\subset\mathcal{R}%
_{f_{4}}$. The first question in our discussion is therefore:\ Which of the
branch points of $f_{4}$ are 'active' and which ones are 'inactive' for a
given parameter constellation?

\item[2)] We have already seen in Example \ref{e62} that the connectivity of
the minimal set $K_{0}(f_{4},\infty)$ can change. Motivated by this
experience, the second question will be: What is the connectivity of the
minimal set $K_{0}(f_{4},\infty)$ for a given parameter constellation, and how
does it change with variations of the parameter values?

\item[3)] At the end of the last example we have discussed in some detail the
difficulties to find the points of the set $E_{1}$. In general these points
are bifurcation points of the minimal set, and these points are crucial for
the quadratic differential (\ref{f53b}) in Theorem \ref{t53a}. The third
question is therefore: How do the bifurcation points of the minimal set
$K_{0}(f_{4},\infty)$ depend on the parameter values, and at which parameter
constellations\ do these points merge or split up?
\end{itemize}

The function $f_{4}$ has in general eight branch points; four of them are
identical with those of the two functions $f_{2}$\ and $f_{3}$, and they will
be denoted again by $z_{1},\ldots,z_{4}$. These four branch points do not
depend on the parameter $c$.

For every parameter $\varphi\in\lbrack0,\pi/2)$ there exists a whole region of
parameter values $c$ such that only these four 'old' branch points
$z_{1},\ldots,z_{4}$ of $f_{4}$ appear in the minimal set $K_{0}(f_{4}%
,\infty)$, and in these cases they are the only branch points that play an
active role in the determination of $K_{0}(f_{4},\infty)$. All other branch
points will be called 'inactive'.\smallskip%
\begin{figure}
[ptb]
\begin{center}
\includegraphics[
height=2.0695in,
width=2.4491in
]%
{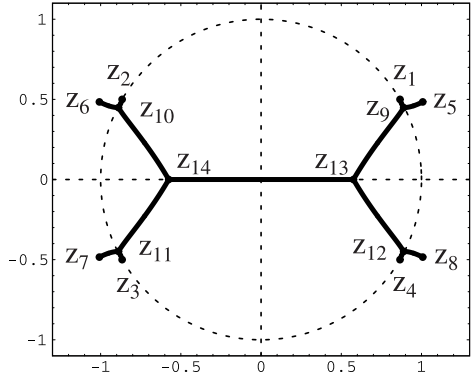}%
\caption{The minimal set $K_{0}(f_{4},\infty)$ for Problem $(f_{4},\infty)$
with $f_{4}$ defined in (\ref{f64a}) with parameter values $\varphi=\pi/6$ and
$c=\sqrt{0.4}$.}%
\label{fig3}%
\end{center}
\end{figure}

Throughout the discussion, we keep the parameter $\varphi=\pi/6$ fixed, which
implies that all minimal sets $K_{0}(f_{4},\infty)$ that will be considered
during our discussion should be compared with the set $K_{0}(f_{3},\infty)$ in
the first window of Figures \ref{fig2}.\smallskip

In a first step we choose
\begin{equation}
c=r\,e^{it}\text{ \ \ with \ \ }t\in\lbrack0,2\pi)\text{ \ \ and
\ \ }r>0\text{ \ small,}\label{f64b}%
\end{equation}
and see what happens. If $|c|>0$ is small, then the four new branch points
$z_{5},\ldots,z_{8}$ of the function $f_{4}$ lie close to the four old branch
points $z_{1},\ldots,z_{4}$. In Figure \ref{fig3} the situation is shown for
the parameter values $\varphi=\pi/6$ and $c=\sqrt{0.4}$. Of course,
$\sqrt{0.4}$ is not very small, however, smaller values of $|c|$\ lead to
configurations that are difficult to plot.\smallskip%
\begin{figure}
[t]
\begin{center}
\includegraphics[
height=4.0845in,
width=4.0378in
]%
{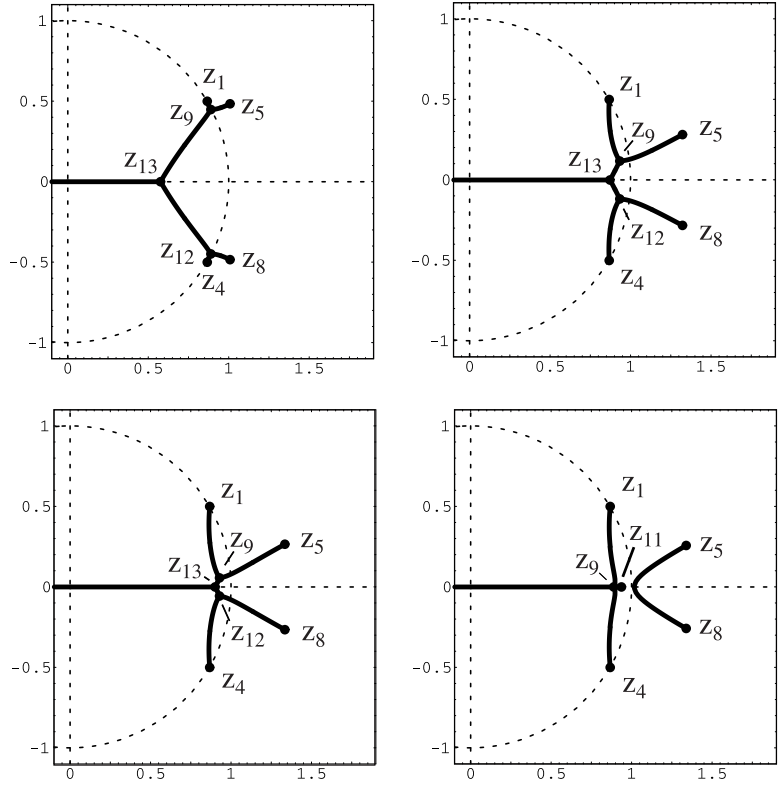}%
\caption{Four examples of minimal sets $K_{0}(f_{4},\infty)$ for Problem
$(f_{4},\infty)$ with $f_{4}$ defined in (\ref{f64a}) with the parameter
$\varphi=\pi/6$ fixed and parameter values $c=\sqrt{0.4},\sqrt{0.7}%
,\sqrt{0.705},$ and $\sqrt{0.715}$\ from top row left to bottom row right.}%
\label{fig4}%
\end{center}
\end{figure}

While in (\ref{f64b}) the parameter $t$ runs through $[0,2\pi)$, each one of
the four new branch points $z_{5},\ldots,z_{8}$ encircles two times the
corresponding old branch point $z_{1},\ldots,z_{4}$.

The interesting point is now that the four new branch points $z_{5}%
,\ldots,z_{8}$ are elements of the minimal set $K_{0}(f_{4},\infty)$ only on
one half of their twofold circular path. On the other half, they become
'inactive', i.e., they are hidden away on another sheet of the Riemann surface
$\mathcal{R}_{f_{4}}$. In this later case, the set $K_{0}(f_{4},\infty)$
contains only the four branch points $z_{1},\ldots,\allowbreak z_{4}$, and
consequently, it is identical with the minimal set $K_{0}(f_{3},\infty)$,
which has been shown in the first window of Figure \ref{fig2}.

It has already been said that in Figure \ref{fig3}, the minimal set
$K_{0}(f_{4},\infty)$ is shown for the parameter values $\varphi=\pi/6$ and
$c=\sqrt{0.4}$. This is a parameter constellation in which all eight branch
points $z_{1},\ldots,z_{8}$ are active. In contrast to this, the parameter
constellation $\varphi=\pi/6$ and $c=-\sqrt{0.4}$, which corresponds to
$t=\pi$\ in (\ref{f64b}), leads to a minimal set $K_{0}(f_{4},\infty)$ that
contains only the four old branch points $z_{1},\ldots,z_{4}$, and it is
therefore identical with the minimal set $K_{0}(f_{3},\infty)$ shown in the
first window of Figure \ref{fig2}.

Studying the minimal set $K_{0}(f_{4},\infty)$ for $|c|$ small, gives a good
illustration of the phenomenon of active and inactive branch points. Of
course, an extension of such a discussion to arbitrary values of
$c\in\mathbb{C}$ would be possible, but it become rather
complicated.\smallskip

Next, we consider Problem $(f_{4},\infty)$ for the six specially chosen real
parameter values $c=\sqrt{0.4},\allowbreak\sqrt{0.7},\allowbreak\sqrt
{0.705},\allowbreak\sqrt{0.715},\sqrt{0.74},\sqrt{0.76}$ and keep again
$\varphi=\pi/6$\ fixed. The selected values should be seen as representatives
for the general situation of $c\in(0,1)$. The discussion will show why the
specific selection is interesting.\smallskip

There exists numerical evidence (but no analytic proof, so far) that at the
critical parameter value $c_{0}=\sqrt[4]{1/2}$, the minimal set $K_{0}%
(f_{4},\infty)$ changes its connectivity. It is obvious that there exists
$c_{0}\in(0,1)$ which is equal to, or lies close to $\sqrt[4]{1/2}$ such that
for $0\leq c\leq c_{0}$ the set $K_{0}(f_{4},\infty)$ is connected, and for
$c_{0}<c<1$ it is disconnected. In the disconnected case, it consists of three
components. For $c\rightarrow c_{0}-0$,\ in each of the two half-planes
$\{\operatorname*{Re}(z)\lessgtr0\}$ three bifurcation points of $K_{0}%
(f_{4},\infty)$ merge and form a new bifurcation point of order five in each
of the two half-planes.

In Figure \ref{fig4}, the sequence of four minimal sets $K_{0}(f_{4},\infty) $
is shown for the parameter values we have $c=\sqrt{0.4},\allowbreak\sqrt
{0.7},\allowbreak\sqrt{0.705},\allowbreak\sqrt{0.715}$. The sequence shows the
metamorphosis of the set $K_{0}(f_{4},\infty)$ while the parameter $c$ crosses
the critical value $c_{0}=\sqrt[4]{1/2}=\sqrt{0.707...} $. In the four windows
the set $K_{0}(f_{4},\infty)$\ is shown only for the right half-plane.

In the first three windows of Figure \ref{fig4}, the minimal set $K_{0}%
(f_{4},\infty)$ is connected, and there are three bifurcation points
$z_{9},z_{12},$ and $z_{13}$, each of order $3$, which then merge to a single
bifurcation point when $c$ reaches the critical value $c_{0}$. At that moment,
the new bifurcation point is of order $5$.

When the critical value $c_{0}$ has been passed, then the minimal set
$K_{0}(f_{4},\infty)$ is disconnected, as shown in the fourth window of Figure
\ref{fig4}. There remains a bifurcation point $z_{9}$ of order $3$, and as a
new phenomenon, we have a critical point of the Green function $g_{D_{0}%
}(\cdot,\infty)$, $D_{0}=D_{0}(f_{4},\infty)$, at $z_{11}$.\smallskip%
\begin{figure}
[ptb]
\begin{center}
\includegraphics[
height=1.9597in,
width=3.9704in
]%
{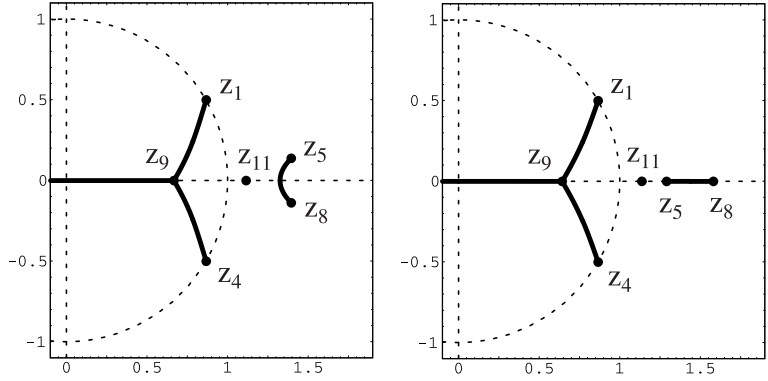}%
\caption{Two examples of minimal sets $K_{0}(f_{4},\infty)$ for Problem
$(f_{4},\infty)$ with $f_{4} $ defined in (\ref{f64a}) with parameter
$\varphi=\pi/6$ and $c=\sqrt{0.74}$ in the left window and $\varphi=\pi/6$ and
$c=\sqrt{0.76}$ in the right window.}%
\label{fig4b}%
\end{center}
\end{figure}

Another interesting parameter value is $c_{1}=\sqrt{3/4},$ since at the
parameter constellation $\varphi=\pi/6$ and $c_{1}=\sqrt{3/4}$ two pairs of
branch points of the function $f_{4}$ collapse to simple zeros of $f_{4}%
$.\ These two simple zeros are located at $\pm\sqrt{2}$.

In Figure \ref{fig4b}, the transition process at the critical value
$c_{1}=\sqrt{3/4}$ is represented by the two parameter values $c=\sqrt{0.74}$
and $c=\sqrt{0.76}$. One can see how the concerned components of $K_{0}%
(f_{4},\infty)$ change their shape from a type of vertical arcs to horizontal
slits.\smallskip

We conclude the discussion of Example \ref{e64} by assembling informations
about the sets $E_{0}$, $E_{1}$, $E_{2}$, and the arcs $J_{j}$, $j\in I$,
introduced in Theorem \ref{t41a} and in Definition \ref{d53b}. This is done
for the six parameter constellations of the two Figures \ref{fig4}\ and
\ref{fig4b}. In addition we also give the quadratic differential $q(z)dz^{2}$
from Theorem \ref{t53a}. This information corresponds to the whole set
$K_{0}(f_{4},\infty)$, while in the Figures \ref{fig4}\ and \ref{fig4b} only
restrictions to the right half-plane have been plotted.\smallskip

For the three parameter values $c=\sqrt{0.4},\allowbreak\sqrt{0.7}%
,\allowbreak\sqrt{0.705}$ the minimal set $K_{0}(f_{4},\allowbreak\infty)$ is
connected, and with respect to $E_{0}$, $E_{1}$, $E_{2}$, $J_{j}$, $j\in I $,
and the quadratic differential $q(z)dz^{2}$ we have identical structures.

We have $E_{0}=\{z_{1},\ldots,z_{8}\}$, $E_{1}=\{z_{9},\ldots,z_{14}\}$, and
$E_{2}=\emptyset$. All eight branch points $z_{1},\ldots,z_{8}$ of $f_{4}$ are
active, there are six bifurcation points $z_{9},\ldots,\allowbreak z_{14}$ and
$13$ Jordan arcs $J_{j}$, $j\in I=\{1,\ldots,13\}$. In accordance to Theorem
\ref{t53a}, all $13$ arcs $J_{j}$, $j\in I$, are trajectories of the quadratic
differential
\begin{equation}
\frac{\prod_{j=9}^{14}(z-z_{j})}{\prod_{j=1}^{8}(z-z_{j})}dz^{2}.\label{f64c}%
\end{equation}

For the three parameter values $c=\sqrt{0.715},\allowbreak\sqrt{0.74}%
,\allowbreak\sqrt{0.76}$, which correspond to the fourth window in Figure
\ref{fig4} and the two windows in Figure \ref{fig4b}, the minimal set
$K_{0}(f_{4},\infty)$ consists of three components. The sets $E_{0}$, $E_{1}$,
$E_{2}$, $J_{j}$, $j\in I$, and the quadratic differential $q(z)dz^{2}$ are of
the same structure in all three cases. There are two bifurcation points
$z_{9}$, $z_{10}$, and the Green function $g_{D_{0}}(\cdot,\infty),$
$D_{0}=D_{0}(f_{4},\infty),$\ has two critical points $z_{11}$ and $z_{12}$.

Thus, we have $E_{0}=\{z_{1},\ldots,z_{8}\}$, $E_{1}=\{z_{9},z_{10}\}$, and
$E_{2}=\{z_{11},z_{12}\}$. There are $7$ Jordan arcs $J_{j}$, $j\in
I=\{1,\ldots,7\}$, and these arcs are trajectories of the quadratic
differential
\begin{equation}
\frac{\prod_{j=9}^{10}(z-z_{j})\prod_{j=11}^{12}(z-z_{j})^{2}}{\prod_{j=1}%
^{8}(z-z_{j})}dz^{2}.\label{f64d}%
\end{equation}
\smallskip

\subsection{\label{e65}Example $f_{5}$}

\qquad As a last example, we come back to the algebraic function (\ref{f1a}),
which has already been used in the Introduction for a demonstration of the
connection between Pad\'{e} approximation and sets of minimal capacity. This
function is now denoted as $f_{5}$, and it has been defined in (\ref{f1a}) as
\begin{equation}
f_{5}(z):=\sqrt[4]{\prod\nolimits_{j=1}^{4}(1-z_{j}/z)}+\sqrt[3]%
{\prod\nolimits_{j=5}^{7}(1-z_{j}/z)}\label{f65a}%
\end{equation}
with the $7$ branch points that have been chosen as
\begin{align}
z_{1}  & =1+3\,i,\text{ \ \ \ }z_{2}=-4+2\,i,\text{ \ \ \ }z_{3}=-4+i,\text{
\ \ \ }z_{4}=0+2\,i,\nonumber\\
z_{5}  & =2+2\,i,\text{ \ \ \ }z_{6}=3+4\,i,\text{ \ \ \ \ \ }z_{7}%
=1+4\,i.\label{f65b}%
\end{align}

The choice of the branch points was in principle arbitrary, but it reflects
the intension to avoid symmetries in the minimal set $K_{0}(f_{5},\infty)$ of
a sort that has been very dominant in the $3$\ Examples \ref{e62} -
\ref{e64}.\smallskip

From the structure of function $f_{5}$, we conclude that the set
$\mathcal{D}(f_{5},\infty)$ of admissible domains for Problem $(f_{5},\infty)$
introduced in Definition \ref{d21a}\ consists of all domains $D\subset
\overline{\mathbb{C}}$ such that $\infty\in D$ and that the elements of each
of the two subsets of branch points $\{z_{1},\ldots,z_{4}\}$ and
$\{z_{5},z_{6},z_{7}\}$ are connected in the complement $K=\overline
{\mathbb{C}}\setminus D$.\smallskip%
\begin{figure}
[ptb]
\begin{center}
\includegraphics[
height=2.821in,
width=4.1096in
]%
{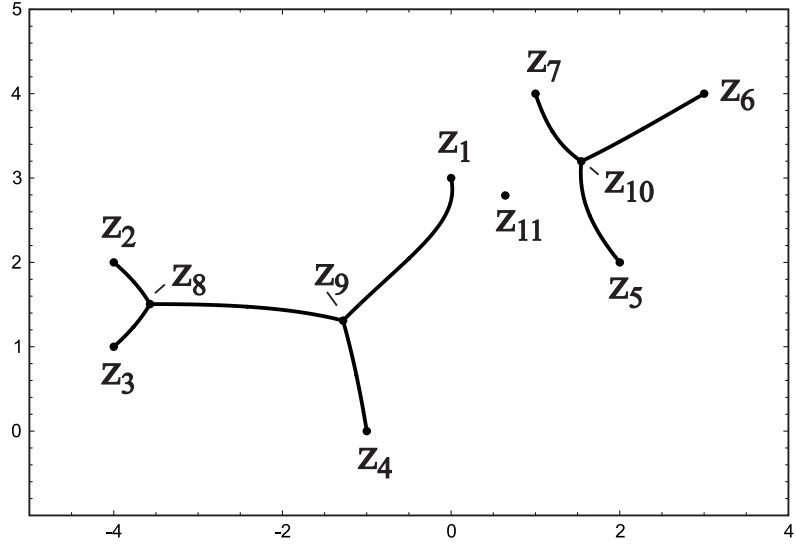}%
\caption{The minimal sets $K_{0}(f_{5},\infty)$ for Problem $(f_{5},\infty)$
with $f_{5} $ defined by (\ref{f65a}) and (\ref{f65b}).}%
\label{fig5}%
\end{center}
\end{figure}

It turns out that the minimal set $K_{0}(f_{5},\infty)$ consists of two
components, and that indeed each of them connects one of the two sets
$\{z_{1},\ldots,z_{4}\}$ and $\{z_{5},z_{6},z_{7}\}$. The set $K_{0}%
(f_{5},\infty)$ is shown in Figure \ref{fig5}. It has three bifurcation
points, which are denoted by $z_{8},z_{9},z_{10}$, and the Green function
$g_{D_{0}}(\cdot,\infty)$ in the extremal domain $D_{0}=D_{0}(f_{5},\infty)$
possesses exactly one critical point, which is denoted by $z_{11}$ in Figure
\ref{fig5}.

While the $7$ branch points $z_{1},\ldots,z_{7}$ in (\ref{f65b}) can be
considered as input to the problem, the location of the four other points
$z_{8},\ldots,z_{11}$ has to be determined by the criterion of minimality of
the set $K_{0}(f_{5},\infty)$. The calculation of these four points has been
done numerically, and their values are
\begin{align}
z_{8}  & =-3.57021+1.50570\,i,\nonumber\\
z_{9}  & =-1.28112+1.30991\,i,\nonumber\\
z_{10}  & =\;\text{\ \ }1.54341+3.19816\,i,\label{f65c}\\
z_{11}  & =\;\text{\ \ }0.64231+2.79311\,i.\nonumber
\end{align}

The $8$ Jordan arcs $J_{j}$, $j\in I=\{1,\ldots,8\}$, in $K_{0}(f_{5},\infty)$
are trajectories of the quadratic differential
\begin{equation}
\frac{(z-z_{11})^{2}\prod_{j=8}^{10}(z-z_{j})}{\prod_{j=1}^{7}(z-z_{j})}%
dz^{2}.\label{f65d}%
\end{equation}
The sets $E_{0}$, $E_{1}$, $E_{2}$ introduced in Theorem \ref{t41a} and in
Definition \ref{d53b} are now $E_{0}=\{z_{1},\ldots,z_{7}\}$, $E_{1}%
=\{z_{8},z_{9},z_{10}\}$, and $E_{2}=\{z_{11}\}$.\medskip

\subsection{\label{s66}Some General Remarks}

\qquad The main motivation for the selection and presentation of the $5$
Examples \ref{e61} - \ref{e65} was to illustrate the variety of topological
structures that are possible for the minimal set $K_{0}(f,\infty)$. Naturally,
such examples should be kept simple, but even for the comparatively simply
structured functions $f_{1},\ldots,f_{5}$\ in the Examples \ref{e61} -
\ref{e65}, the shape and the connectivity of the minimal set $K_{0}(f,\infty)$
has not always been clear at the outset of the analysis.

Naturally, the situation becomes more technical and much more difficult to
handle if the function $f$ becomes more complex, and especially, if it is no
longer algebraic. As a consequence, the set $E_{0}$ may no longer be finite.
For general functions $f$ it is very difficult to predict shape and
connectivity of the minimal set $K_{0}(f,\infty)$. One way to get some
information and a rough idea in this respect is to calculate poles of Pad\'{e}
approximants to the function $f$. This, by the way, has been done in the study
of the function (\ref{f1a}) in the Introduction, and the result in Figure
\ref{fig0} should be compared with Figure \ref{fig5}.

A critical task\ for the numerical calculation of the Jordan arcs $J_{j}$,
$j\in I$, in the minimal set $K_{0}(f,\infty)$ is the calculation of the zeros
in the quadratic differential (\ref{f53b}) in Theorem \ref{t53a}. For this
purpose we have developed a numerical procedure, which has been used in the
analysis of the Examples \ref{e63} - \ref{e65}. More details about this topic
will be given in Subsection \ref{s83}, further blow.\medskip

\section{\label{s7}A Local Criterion and Geometric Estimates}

\qquad The $S-$property (symmetry property), which has already been introduced
and considered in Subsections \ref{s51}, will again take central stage in the
first three subsections. We start with a definition of this property that
characterises the whole domain, and will then show that it is a local
condition for the minimality (\ref{f21a}) in Definition \ref{d21b}. As a
somewhat surprising result in Subsection \ref{s73}, we shall formulate a
theorem in which it is proved that the $S-$property is also sufficient for the
global minimality (\ref{f21a}) in Definition \ref{d21b}. In the fourth
subsections several inclusion relations for the minimal set $K_{0}(f,\infty) $
are presented that can be helpful in many practical situations.\smallskip

\subsection{\label{s71}A General Definition of the $S-$Property}

\qquad In Theorem \ref{t51a} of Subsection \ref{s51} the $S-$property
(\ref{f51a}) appears as an important characteristic of the extremal domain
$D_{0}(f,\infty)$ and its complementary minimal set $K_{0}(f,\infty)$. In the
present subsection we define the $S-$property for arbitrary admissible domains
$D\in\mathcal{D}(f,\infty)$. We start with an auxiliary definition.\smallskip

\begin{definition}
\label{d71a0}An admissible domain $D\in\mathcal{D}(f,\infty)$ for Problem
$(f,\infty)$ is called elementarily maximal if for every point $z\in\partial
D$ one of the following two assertions holds true.

\begin{itemize}
\item[(i)] There exists at least one meromorphic continuation of the function
$f$ out of the domain $D$ that has a non-polar singularity at $z$.

\item[(ii)] There exist at least two meromorphic continuations of the function
$f$ out of $D$ that lead to two non-identical function elements at
$z$.\smallskip
\end{itemize}
\end{definition}

It is immediate that if an admissible domain $D\in\mathcal{D}(f,\infty)$ is
not elementarily maximal, then the domain $D$ can be enlarged in a straight
forward way without leaving the class $\mathcal{D}(f,\infty)$ of admissible
domains. Hence, the elementarily maximal domains are the maximal elements in
$\mathcal{D}(f,\infty)$ with respect to ordering by inclusion. We formulate
this statement as a proposition.\smallskip

\begin{proposition}
\label{p71a}The elementarily maximal domains of Definition \ref{d71a0} are the
maximal elements in $\mathcal{D}(f,\infty)$ with respect to ordering by
inclusion.\smallskip
\end{proposition}

From the Structure Theorem \ref{t41a} in Subsection \ref{s41}, we easily
deduce that the extremal domain $D_{0}(f,\infty)$ is elementarily maximal, but
of course, there exist many other maximal elements in $\mathcal{D}(f,\infty
)$.\ Often it is helpful, and in most situations also possible, to assume
without loss of generality that an arbitrarily chosen admissible domain
$D\in\mathcal{D}(f,\infty)$ is elementarily maximal.

After these preliminaries we come to the definition of the $S-$property of a
domain.\smallskip

\begin{definition}
\label{d71a}We say that an admissible domain $D\in\mathcal{D}(f,\infty)$
possesses the $S-$property (symmetry property) with respect to Problem
$(f,\infty)$ if its complement $K=\overline{\mathbb{C}}\setminus D$ is of the
form
\begin{equation}
K=E_{0}\cup E_{1}\cup\bigcup_{j\in I}J_{j}\label{f71a2}%
\end{equation}
and

\begin{itemize}
\item[(i)] assertion (i) of Definition \ref{d71a0}\ holds true for every
$z\in\partial E_{0}$,

\item[(ii)] assertion (ii) of Definition \ref{d71a0}\ holds true for every
$z\in K\setminus E_{0}$,

\item[(iii)] all $J_{j}$, $j\in I$, are open, analytic Jordan arcs,

\item[(iv)] the set $E_{1}\subset K\setminus E_{0}$ is discrete in
$\mathbb{C}\setminus E_{0}$, each point $z\in E_{1}$ is the end point of at
least three different arcs of $\{J_{j}\}_{j\in I}$, and

\item[(v)] if $I\neq\emptyset$, then we have
\begin{equation}
\frac{\partial}{\partial n_{+}}g_{D}(z,\infty)=\frac{\partial}{\partial n_{-}%
}g_{D}(z,\infty)\text{ \ for all \ }z\in J_{j},j\in I\label{f71a1}%
\end{equation}
with $\partial/\partial n_{+}$ and $\partial/\partial n_{-}$ denoting the
normal derivatives to both sides of the arcs $J_{j}$, $j\in I$. By
$g_{D}(\cdot,\infty)$\ we denote the Green function in $D$.\smallskip
\end{itemize}
\end{definition}

If $I\neq\emptyset$, then it is immediate that $\operatorname*{cap}\left(
\partial D\right)  >0$, and consequently the Green function $g_{D}(z,\infty)$
exists in this case in a proper way (see Subsection \ref{s1103}, further
below). From identity (\ref{f71a1}) one can deduce that the Jordan arcs
$J_{j}$, $j\in I$, are analytic. Hence, the analyticity assumed in assertion
(iii) of Definition \ref{d71a} is implicitly also contained in assertion (v).

Because of the two assertions (i) and (ii) in Definition \ref{d71a}, a domain
$D\in\mathcal{D}(f,\infty)$ with the $S-$property is also elementarily maximal
in the sense of Definition \ref{d71a0}.

With Definition \ref{d71a} and the Structure Theorem \ref{t41a}, we can
rephrase Theorem \ref{t51a} in Subsection \ref{s51} as follows: The extremal
domain $D_{0}(f,\infty)$ possesses the $S-$property. In Subsection \ref{s73},
below, we shall see that also the reversed conclusion holds true, i.e., if an
admissible domain $D\in\mathcal{D}(f,\infty)$ possesses the $S-$property, then
it is identical with the extremal domain $D_{0}(f,\infty)$ of Definition
\ref{d21b}.\smallskip

\subsection{\label{s72}A Local Extremality Condition}

\qquad\ In the present subsection we show that Hadarmard's boundary variation
formula for the Green function implies that the $S-$property of Definition
\ref{d71a} is a local condition for the minimality of $\operatorname*{cap}%
\left(  \overline{\mathbb{C}}\setminus D\right)  $, i.e., $\operatorname*{cap}%
\left(  \overline{\mathbb{C}}\setminus D\right)  $ assumes a (local) minimum
under local variations of the boundary of an admissible domain $D\in
\mathcal{D}(f,\infty)$ that $D$ possesses the $S-$property.

We start with the introduction of some notations that are needed for the setup
of the boundary variation for Hadamard's variation formula (for a very
readable introduction to this topic we recommend the appendix of
\cite{Courant50}). Let $D\subset\overline{\mathbb{C}}$ be a domain with
$\infty\in D$, assume that in $\partial D$ there exists a smooth, open Jordan
arc $\gamma\subset\partial D$, and assume further that the domain $D$ lies
only on one side of $\gamma$. By $n(v)\in\mathbb{T}$, we denote the normal
vector on $\gamma$ at the point $v\in\gamma$ that shows into $D$. Let
$v_{0}\in\gamma$ be fixed, $\varphi\geq0$ a smooth function defined on
$\gamma$ with compact support. We assume that the support of $\varphi$ is
small and that it contains the point $v_{0}\in\gamma$ in its interior, and
choose $\varepsilon\in\mathbb{R}$ with $|\varepsilon|>0$ small.

With these definitions we introduce a variation $\widetilde{D}$ of the domain
$D$ by moving each boundary point $v\in\gamma$ along the vector $\varepsilon
\varphi(v)\,n(v)$. If $|\varepsilon|>0$ is sufficiently small, then the new
domain $\widetilde{D}$ is well defined. Hadamard's variation formula for the
Green function $g_{D}(z,w)$ under this type of variation of the domain
$D$\ says that
\begin{align}
& g_{\widetilde{D}}(z,\infty)-g_{D}(z,\infty)=\label{f72a1}\\
& \text{ \ \ \ \ \ \ \ \ \ \ \ \ \ }\frac{\varepsilon}{2\pi}\int_{\gamma}%
\frac{\partial}{\partial n}g_{D}(v,\infty)\frac{\partial}{\partial n}%
g_{D}(v,z)\varphi(v)ds_{v}+\text{O}(\varepsilon^{2})\nonumber
\end{align}
for $|\varepsilon|\rightarrow0$, where $\partial/\partial n$ denotes the
normal derivative and $ds$ the line element on $\gamma$. The Landau symbol
O$(\cdot)$ holds uniformly for $z$\ varying on a compact subset of
$\widetilde{D}\cap D$.\smallskip

From (\ref{f72a1}) and the connection between the logarithmic capacity and the
Green function (cf. Lemma \ref{l113b} in Subsection \ref{s1103}, further
below), we then get
\begin{equation}
\frac{\operatorname*{cap}(\overline{\mathbb{C}}\setminus\widetilde{D}%
)}{\operatorname*{cap}(\overline{\mathbb{C}}\setminus D)}-1=\frac{\varepsilon
}{2\pi}\int_{\gamma}\left(  \frac{\partial}{\partial n}g_{D}(v,\infty)\right)
^{2}\varphi(v)ds_{v}+\text{O}(\varepsilon^{2})\label{f72b1}%
\end{equation}
for $|\varepsilon|\rightarrow0$, which shows that Hadamard's variation formula
(\ref{f72a1}) gives us an explicit expression for the first order variation of
$\operatorname*{cap}(\overline{\mathbb{C}}\setminus D)$ under local variations
of a smooth piece $\gamma$ of the boundary $\partial D$.\smallskip

Let us now assume that the domain $D\subset\overline{\mathbb{C}}$ contains a
smooth, open Jordan arc $J\subset\partial D$ with the property that on both
sides of $J$ there are only points of $D$, i.e., $J$ is a cut in some larger
domain. As before, by $n(v)\in\mathbb{T}$ we denote the normal vector to $J$
at a point $v\in J$, and assume that all normal vectors $n(v)\in\mathbb{T}$,
$v\in J$, show towards the same side of $J$. Again, by $\varphi\geq0$ we
denote a smooth function on $J$ with compact support, and assume that the
support is contained in a neighborhood of $v_{0}\in J$. The parameter
$\varepsilon\in\mathbb{R}$ with $|\varepsilon|>0$ plays the same role as before.

With the definitions just introduced, we first define a variation
$J_{\varepsilon}$\ of the arc $J$. The new arc $J_{\varepsilon}$ results from
moving each point $v\in J$ along the vector $\varepsilon\varphi(v)\,n(v)$. For
$|\varepsilon|>0$ sufficiently small, $J_{\varepsilon}$ is well defined and
again a smooth Jordan arc. The variation $D_{\varepsilon}$ of the domain $D$
is then defined by replacing the arc $J$ by $J_{\varepsilon}$. This type of
variation changes the boundary $\partial D$ only locally,\ but the changes
take place in two subregions of $D$. The two pieces of the boundary $\partial
D$ that correspond to the arc $J$\ are moved in opposite directions. Because
of this variation at two places in $D$ in opposite directions, we deduce from
(\ref{f72b1}) that
\begin{align}
& \frac{\operatorname*{cap}(\overline{\mathbb{C}}\setminus D_{\varepsilon}%
)}{\operatorname*{cap}(\overline{\mathbb{C}}\setminus D)}-1=\label{f72c1}\\
& \;\;\;\;\;\;\frac{\varepsilon}{2\pi}\int_{J}\left[  \left(  \frac{\partial
}{\partial n_{+}}g_{D}(v,\infty)\right)  ^{2}-\left(  \frac{\partial}{\partial
n_{-}}g_{D}(v,\infty)\right)  ^{2}\right]  \varphi(v)ds_{v}+\text{O}%
(\varepsilon^{2})\nonumber
\end{align}
for $|\varepsilon|\rightarrow0$, where $\partial/\partial n_{+}$ and
$\partial/\partial n_{-}$ denote the normal derivatives to both sides of $J$.
From (\ref{f72c1}) and the fact that the support of the function $\varphi
$\ can be chosen as small as we want, we can conclude rather immediately that
the symmetry (\ref{f71a1}) in Definition \ref{d71a} of the $S-$property is
equivalent to the vanishing of the first order variation of
$\operatorname*{cap}(\overline{\mathbb{C}}\setminus D)$. It follows
immediately from assertion (ii) in Definition \ref{d71a} and the local
character of the variation of the arc $J\subset\partial D$ that the resulting
variational domain $D_{\varepsilon}$ of the original domain $D$ belongs again
to $\mathcal{D}(f,\infty)$ if $|\varepsilon|>0$ is small. The conclusion of
our discussion is formulated in the next theorem.\smallskip

\begin{theorem}
\label{t72a} Let the complement $K=\overline{\mathbb{C}}\setminus D$ of an
admissible domain $D\in\mathcal{D}(f,\infty)$ be of the form (\ref{f71a2})
with two sets $E_{0}$, $E_{1}$, and the family of arcs $\{J_{j}\}$, $j\in I$,
that satisfy the assertions (i) - (iv) in Definition \ref{d71a}. Then the
symmetry condition (\ref{f71a1}) holds for every $z\in J_{j}$, $j\in I$, if,
and only if, the first order variations of $\operatorname*{cap}(\overline
{\mathbb{C}}\setminus D)$ vanish for all local variations of these arcs
$J_{j}$ done as just described, i.e., if we have
\begin{equation}
\lim_{|\varepsilon|\rightarrow\infty}\frac{1}{\varepsilon}(\operatorname*{cap}%
(\overline{\mathbb{C}}\setminus D_{\varepsilon})-\operatorname*{cap}%
(\overline{\mathbb{C}}\setminus D))=0\label{f72d1}%
\end{equation}
for all such variations.$\smallskip$
\end{theorem}

\subsection{\label{s73}$S-$Property and Uniqueness}

\qquad In the light of Theorem \ref{t72a}, the result of Theorem \ref{t51a}
can no longer surprise since we now know that the symmetry$\ $pro\-perty
(\ref{f51a}) in Theorem \ref{t51a} is a necessary condition for the minimality
(\ref{f21a}) in Definition \ref{d21b}. The interesting and perhaps somewhat
surprising result in the present section is the next theorem, in which the
last conclusion is reversed; it is shown that the $S-$property is also
sufficient for the minimality (\ref{f21a}) in Definition \ref{d21b}.\smallskip

\begin{theorem}
\label{t73a} If an admissible domain $D\in\mathcal{D}(f,\infty)$ possesses the
$S-$property in the sense of Definition \ref{d71a}, then $D$ is identical with
the extremal domain $D_{0}(f,\infty)$ of Definition \ref{d21b}.\smallskip
\end{theorem}

Since we know from Theorem \ref{t22a} that the extremal domain $D_{0}%
(f,\infty)$ is unique, we can deduce a uniqueness result for the extremal
domain $D_{0}(f,\infty)$ from the $S-$property as a corollary to Theorem
\ref{t73a}.\smallskip

\begin{corollary}
\label{c73a}The $S-$property of an admissible domain $D\in\mathcal{D}%
(f,\infty)$ determines uniquely the extremal domain $D_{0}(f,\infty)$ of
Problem $(f,\infty)$.\smallskip
\end{corollary}

The interpretation of the $S-$property as a local condition for the minimality
(\ref{f21a}) in Definition \ref{d21b} is interesting in itself, but it is also
interesting for several applications in rational approximation. Hadarmard's
variation formula (\ref{f72b1}), on the other hand, is not very helpful as a
tool for proofs of the two important Theorems \ref{t22a} and \ref{t41a} since
it requires the knowledge of smoothness of the arcs $J_{j}$, $j\in I$, in the
boundary $\partial D$. However, this property is known only when most of the
groundwork for the proofs has already been done.\medskip

\subsection{\label{s74}Geometric Estimates}

\qquad The minimal set $K_{0}(f,\infty)$ of Problem $(f,\allowbreak\infty)
$\ is in general not convex. The rather trivial Example \ref{e61} is perhaps
the only case, where we have convexity. However, convexity can give rough, and
sometimes also quite helpful, geometric estimates for the minimal set
$K_{0}(f,\infty)$. Some results in this direction are contained in the next theorem.

\begin{theorem}
\label{t74a} Let $K_{0}(f,\infty)$ be the minimal set for Problem $(f,\infty
)$, and let further $E_{0}\subset K_{0}(f,\infty)$ be the compact set that has
been introduced in the Structure Theorem \ref{t41a}, i.e., $\partial E_{0}%
$\ contains all non-polar singularities of the function $f$\ that can be
reached by meromorphic continuations of the function $f$ out of the extremal
domain $D_{0}(f,\infty)$.

(i) \ \ For any convex compact set $K\subset\mathbb{C}$ with the property that
the function $f$ has a single-valued meromorphic continuation throughout
$\overline{\mathbb{C}}\setminus K$, we have
\begin{equation}
K_{0}(f,\infty)\subset K.\label{f74a1}%
\end{equation}

(ii) \ Let $\operatorname*{Co}(E_{0})$ denote the convex hull of $E_{0}$. Then
we have
\begin{equation}
K_{0}(f,\infty)\subset\operatorname*{Co}(E_{0}).\label{f74a2}%
\end{equation}

(iii) Let $K\subset\mathbb{C}$ be a convex compact set, $E\subset
\mathbb{C}\setminus K$ a set of capacity zero that is closed in $\mathbb{C}%
\setminus K$, and assume that the function $f$ has a single-valued meromorphic
continuation throughout $\overline{\mathbb{C}}\setminus(K\cup E)$ . Then we
have
\begin{equation}
K_{0}(f,\infty)\subset K\cup E.\label{f74a3}%
\end{equation}

(iv) \ There uniquely exist two sets $K_{\min}\subset\mathbb{C}$ and $E_{\min
}\subset\mathbb{C}\setminus K_{\min}$ with the same properties as assumed in
assertion (iii) for the pairs of sets $\{K,E\}$ such that these sets are
minimal with respect to inclusion among all pairs $\{K,E\}$\ that satisfy the
assumptions of assertion (iii), and we have
\begin{equation}
K_{0}(f,\infty)\subset K_{\min}\cup E_{\min}.\label{f74a4}%
\end{equation}

(v) \ Let $Ex(K_{\min})$ denote the set of extreme points of the convex set
$K_{\min}$ from assertion (iv). Then we have
\begin{equation}
Ex(K_{\min})\cup E_{\min}\subset E_{0}.\label{f74a5}%
\end{equation}
\smallskip
\end{theorem}

\section{\label{s8}Geometrically Defined Extremality Problems}

\qquad Extremality problems are a classical topic in geometric function
theory, and among the different versions that are studied there we also find
the kind of problems that are concerned with sets of minimal capacity. In the
present section our interest concentrates on extremality problems that are
defined purely by geometrical conditions since these problems have strong
similarities with Problem $(f,\infty)$. But there also exist significant
differences, which, of course, are the interesting aspects for our discussion.

In order to make this discussion more concrete, and also for later use in
proofs, further below, we formulate two classical problems of the geometrical
type. The first one is presented in two versions.\smallskip

\subsection{\label{s81}Two Classical Problems}

\begin{problem}
\label{p81a}\textit{(Chebotarev's Problem) Let finitely many points }%
$a_{1},\ldots,\allowbreak a_{n}\in\mathbb{C}$\textit{\ be given. Find a
continuum }$K\subset\mathbb{C}$\textit{\ with} the property that
\begin{equation}
a_{j}\in K\text{ \ \ \ \ \textit{for} \ \ }j=1,\ldots,n,\label{f81a}%
\end{equation}
\textit{and further that the logarithmic capacity }$\operatorname*{cap}\left(
K\right)  $\textit{\ is minimal among all continua }$K\subset\mathbb{C}%
$\textit{\ that satisfy (\ref{f81a}).}\smallskip
\end{problem}

Problem \ref{p81a} can be refined in a way that brings it closer to situations
that could be observed in the Examples \ref{e63}, \ref{e64}, and \ref{e65} in
Section \ref{s6}.\smallskip

\begin{problem}
\label{p81b}\textit{Let }$m$\textit{\ sets }$A_{i}\subset\mathbb{C}$\textit{,
}$i=1,\ldots,m$\textit{, of finitely many points }$a_{ij}\in A_{i}$\textit{,
}$j=1,\ldots,n_{i}$\textit{, }$i=1,\ldots,m$,\textit{\ be given. Find }%
$m$\textit{\ continua }$K_{1},\ldots,K_{m}\subset\mathbb{C}$\textit{\ with}
the property that
\begin{equation}
a_{ij}\in K_{i}\text{ \ \ \ \ \textit{for} \ \ }i=1,\ldots,m,\text{\ }%
j=1,\ldots,n_{i},\label{f81b}%
\end{equation}
\textit{and further that the logarithmic capacity }$\operatorname*{cap}\left(
K_{1}\cup\ldots\cup K_{m}\right)  $\textit{\ is minimal among all continua
}$K_{1},\ldots,K_{m}\subset\mathbb{C}$\textit{\ that satisfy (\ref{f81b}%
).}\smallskip
\end{problem}

It is evident that Problem \ref{p81b} has many similarities to the Problems
$(f,\infty)$ in the Examples \ref{e61} - \ref{e65} in Section \ref{s6}.
However, these examples also illustrate some of the essential differences.
Especially, there is the question about 'active' versus 'inactive' branch
points and also the question about the connectivity of the minimal set
$K_{0}(f,\infty)$. Such questions don't exist for the classical problems,
since there they are part of the setup of the problem. In Problem $(f,\infty)$
it is in general not possible to have answers to such questions in advance;
the answers are part of the solution and not part of the definition as in
Problem \ref{p81a} and \ref{p81b}.

The functions $f$ in the Examples \ref{e61} - \ref{e65} are rather simple and
transparent representatives for the functions possible in Problem $(f,\infty
)$. In the case of a more complex analytic function $f$, the minimal set
$K_{0}(f,\infty)$ can be very complicated.

From a certain point of view, the two Problems \ref{p81a} and \ref{p81b} can
be seen as special cases of Problems $(f,\infty)$, one has only to choose the
function $f$ in an appropriate way. We exemplify this argument for Problem
\ref{p81b}. Let $f_{1}$ be defined as
\begin{equation}
f_{1}(z):=\sum_{i=1}^{m}\prod_{j=1}^{n_{i}}\left[  1-\frac{a_{ij}}{z}\right]
^{1/n_{i}},\label{f82a}%
\end{equation}
then it is immediate that the minimal set $K_{0}(f_{1},\infty)$ from Theorem
\ref{t22a} is the unique solution of Problem \ref{p81b}.\smallskip

As a second example for a purely geometrically defined extremality problem we
consider the following one:\smallskip

\begin{problem}
\label{p81c}\textit{Let two disjoint, finite sets of points }$a_{1}%
,\ldots,\allowbreak a_{n}\in\overline{\mathbb{C}}$\textit{\ and }$b_{1}%
,\ldots,\allowbreak b_{m}\in\overline{\mathbb{C}}$\textit{\ be given. Find two
continua }$K,V\subset\overline{\mathbb{C}}$\textit{\ with the property that}
\begin{equation}
a_{j}\in K\text{ \ \ \textit{for}\ \ }j=1,\ldots,n,\text{ \ \textit{\ \ \ }%
}b_{i}\in V\text{ \ \ \textit{for}\ \ }i=1,\ldots,m,\label{f81c}%
\end{equation}
\textit{and further that the condenser capacity }$\operatorname*{cap}\left(
K,V\right)  $\textit{\ is minimal among all pairs of continua }$K,V\subset
\overline{\mathbb{C}}$\textit{\ that satisfy (\ref{f81c}).}\smallskip
\end{problem}

For a definition of the condenser capacity we refer to \cite{SaTo} Chapter
II.5. or \cite{Bagby}. Problem \ref{p81c} has been included here because of
two reasons: its solution will be used as an important element in one of the
proofs further below, and secondly, it is perhaps the simplest example of its
kind with non-unique solutions. In this respect, it underlines the importance
and relevance of the uniqueness part in Theorem \ref{t22a}. More about this
second aspect follows in the next subsection.\smallskip

\subsection{\label{s82}Some Methodological and Historic Remarks}

\qquad Problem \ref{p81a} has apparently been mention for the first time in a letter by Chebotarev to G. P\'{o}lya (see
\cite{Polya29}). The existence and uniqueness of a solution for this problem has been proved already shortly
afterwards in 1930 by H. Gr\"{o}tzsch \cite{Groetzsch30} with his famous strip method. In \cite{Groetzsch30} one can also
find a description of the analytic arcs in the minimal set by quadratic differentials, although the presentation has
been done in a different language. In about the same time of \cite{Groetzsch30}, M.A. Lavrentiev has formulated and
studied Problem \ref{p81a} in \cite{Lavrentiev30} and \cite{Lavrentiev34} in an equivalent but somewhat different
setting.

A comprehensive review of methods and results relevant for the Problems
\ref{p81a}, \ref{p81b}, and \ref{p81c} can be found in the two long survey
papers \cite{Kuzmina82}, \cite{Kuzmina97}. We also mention in this respect the
textbooks \cite{Goluzin} and \cite{Pommerenke75}.\smallskip

In the introduction to the present section it has been mentioned that a wide
range of extremality problems has been studied in geometric function theory.
There exists a correspondingly broad variety of methods (different types of
variational methods, the methods of extremal length, quadratic differentials,
etc.) for the analysis of such problems. For our purpose the survey papers
\cite{Kuzmina82} and \cite{Kuzmina97} have provided a good coverage of the
relevant literature.

In our proofs we shall need only properties of the solution of a special case
of Problem \ref{p81c} (see Definition \ref{d101a} in Subsection \ref{s1011},
further below). In this problem the two sets $A=\{a_{1},\ldots,a_{n}\}$ and
$B=\{b_{1},\ldots,b_{n}\}$ consist of points which are reflections of each
other on the unit circle $\partial\mathbb{D}$, i.e., we assume that
$b_{j}=1/\overline{a}_{j}$ for $j=1_{1},\ldots,n$. Under this assumption,
Problem \ref{p81c} can be seen as a hyperbolic version of Problem \ref{p81a}.
Indeed, the set $A$ of the given $n$ points is assumed to be contained in
$\mathbb{D}$ and the logarithmic capacity $\operatorname{cap}(K)$ in Problem
\ref{p81a}\ is replaced by the hyperbolic capacity of $K\subset\mathbb{D}$
(see Subsection \ref{s1011}, further below).

Our last topic in the present subsection is concerned with the possibility of
non-unique solutions to Problem \ref{p81c}. We start with some remarks about
Teichm\"{u}ller's\ problem, which practically is the most special situation of
Problem \ref{p81c}. If in Problem \ref{p81c} both sets $A$ and $B$ consist of
only $2$ points, then with the help of a Moebius transformation one can show
that without loss of generality $3$\ of the $4$ points can be chosen in a
standardized way, which usually is done so that $A=\{-1,1\}$ and
$B=\{b,\infty\}$ with $b$ being an arbitrary point in $\mathbb{C}%
\setminus\{-1,1\}$. Under these assumptions, the minimal condenser capacity
$\operatorname*{cap}\left(  K,V\right)  $ of Problem \ref{p81c} depends only
on single complex variable $b$. The minimality problem in this special form
has been suggested by O. Teichm\"{u}ller in \cite{Teichmueller38}, and it
carries today his name. Its solution and the study of its properties has
attracted some research interest (cf., \cite{Kuzmina82}, Chapter 5.2, for a
survey); we mention here only the very recent publication \cite{Heikkala06},
where a numerical method for an efficient calculation of $\operatorname*{cap}%
\left(  K,V\right)  $ in dependence of $b\in\mathbb{C}\setminus\{-1,1\}$ has
been developed and studied.

For\ our\ discussion,\ the\ cases\ with\ $b\in\left(  -1,1\right)  $\ are\ of
special\ interest,\ since Teichm\"{u}ller's problem has non-unique solutions
exactly for the parameter values $b\in\left(  -1,1\right)  $. We consider the
symmetric case $b=0$.

If in Problem \ref{p81c}, we choose $n=m=2$, $\left\{  a_{1},a_{2}\right\}
=\left\{  -1,1\right\}  $, and $\left\{  b_{1},b_{2}\right\}  =\left\{
0,\infty\right\}  $, then it is not too difficult to verify by symmetry
considerations that there exist at least two different solutions $(K,V)$. The
first one is given by $K:=\{$ $e^{it}$ $|$ $\pi\leq t\leq2\pi$ $\}$ and
$V:=\{$ $z$ $|$ $0\leq\operatorname*{Im}(z)\leq\infty$, $\operatorname*{Re}%
(z)=0$ $\}$, while the second one is its symmetric counterpart $\widetilde
{K}:=\{$ $e^{it}$ $|$ $0\leq t\leq\pi$ $\}$ and $\widetilde{V}:=\{$ $z$ $|$
$-\infty\leq\operatorname*{Im}(z)\leq0$, $\operatorname*{Im}(z)=0$ $\}$. This
counterexample to uniqueness underlines that the uniqueness part of Theorems
\ref{t22a} cannot be trivial.

The proof of uniqueness of the solution to Problem $(f,\infty)$ is contained
in Subsection \ref{s93}, and it has demanded some new ideas and concepts. A
review of the uniqueness question for the general case of Problem \ref{p81c}
is contained in Chapter 5.4 of \cite{Kuzmina82}.\smallskip

\subsection{\label{s83}The Numerical Calculation of the Set $K_{0}(f,\infty)
$}

\qquad From Theorem \ref{t41a} we have a general knowledge of the structure of
the minimal set $K_{0}(f,\infty)$, and we know that there uniquely exist two
compact sets $E_{0}$, $E_{1}$, and a family of Jordan arcs $J_{j}$, $j\in I$,
which are trajectories of a certain quadratic differential, and the union of
these objects forms the set $K_{0}(f,\infty)$ in\ (\ref{f41a}) of Theorem
\ref{t41a}. In each concrete case of a function $f$ that is not as simple as
that in Example \ref{e61}, the determination of $E_{0}$, $E_{1},$ and $J_{j}$,
$j\in I$, is a difficult and tricky task, and there is no general method at
hand that can be applied in all situations.

The situation is different in the more special case of Theorem \ref{t53a},
where we have a rational quadratic differential $q(z)dz^{2}$, which can be
used for the calculation of the Jordan arcs $J_{j}$, $j\in I$. In this more
special situation, only two critical tasks have to be done: The first one
consists in finding the set of branch points of the function $f$ in Problem
$(f,\infty)$ that play an active role in the determination of the set
$K_{0}(f,\infty)$; part of this first task is also the determination of the
connectivity of the set $K_{0}(f,\infty)$. The second critical task is the
calculation of the zeros in the quadratic differential (\ref{f53b}) in Theorem
\ref{t53a}. This second task appears in a similar form if one wants to solve
Problem \ref{p81b}, and therefore it has found already earlier attention in
the literature. Some results in this direction have been reviewed in the
discussion at the end of Example \ref{e63}.\smallskip

In the analysis of the Examples \ref{e62} - \ref{e65} in Section \ref{s6}, the
second task has been solved with the help of a numerical procedure that has
been developed in an ad-hoc manner by the author. Details of the procedure
will be published elsewhere.\medskip

\section{\label{s9}Proofs I}

\qquad In the present section we prove Theorem \ref{t22a} together with the
accompanying Propositions \ref{p22a}, \ref{p22b}, and Theorem \ref{t32a}.
Thus, we are primarily dealing with a proof of the unique existence of a
solution to the Problems $(f,\infty)$. Like in Theorem \ref{t22a}, we assume
throughout the section \ that the function $f$ is meromorphic in the
neighborhood of infinity.

\subsection{\label{s91}Meromorphic Continuations Along Arcs}

\qquad The continuation of a function element along a given arc $\gamma$ is
basic for any technique of meromorphic continuations. In the present
subsection we introduce special sets of arcs and curves, and define on them a
homotopy relation that is adapted to our special needs in later proofs. Toward
the end of the subsection in Proposition \ref{p91a},\ we prove a
characterization of the domains in $\mathcal{D}(f,\infty)$ in terms of these
newly introduced tools, i.e., a characterization of admissible domains for
Problem $(f,\infty)$.

As a general notational convention, we denote the impression of a curve or an
arc by the same symbol as use for the curve or the arc itself.

\begin{definition}
\label{d91a} By $\Gamma=\Gamma(f,\infty)$ we denote the set of all Jordan
curves $\gamma$ with the following two properties:

\begin{itemize}
\item[(i)] We have $\infty\in\gamma$.

\item[(ii)] There exists a point $z\in\gamma\setminus\{\infty\}$, called
separation point of $\gamma$, such that the curve $\gamma$ is broken down into
the two closed partial arcs $\gamma^{-}$ and $\gamma^{+}$ connecting the two
points $z$\ and $\infty$. The function $f$ is assumed to possess meromorphic
continuations along each of the two arcs, and these two arcs are not
identical, i.e., we have $\gamma=\gamma^{+}-\gamma^{-}$ and $\gamma^{+}%
\cap\gamma^{-}=\{z,\infty\}$. ('Closed' means here the arc contains its end points).
\end{itemize}

\noindent We assume that each Jordan curve $\gamma\in\Gamma$\ has a
parametrization of the form
\begin{equation}
\gamma:\left[  -1,1\right]  \longrightarrow\overline{\mathbb{C}}\label{f91a}%
\end{equation}
with $\gamma(-1)=\gamma(1)=\infty$ and $\gamma(0)=z$.\smallskip
\end{definition}

From (\ref{f91a}), we have the parametrization
\begin{equation}
\gamma^{+}:\left[  1,0\right]  \longrightarrow\overline{\mathbb{C}%
}\text{,\ \ \ }\gamma^{-}:\left[  -1,0\right]  \longrightarrow\overline
{\mathbb{C}}\text{ }\label{f91b}%
\end{equation}
for the two partial arcs $\gamma^{-}$ and $\gamma^{+}$.\smallskip

Whether a Jordan curves $\gamma$ with $\infty\in\gamma$ belongs to $\Gamma$
depends on the function $f$. A necessary and sufficient condition can be
formulated as follows: We have $\gamma\in\Gamma$\ if, and only if, the two
meromorphic continuations of $f$ that start at $\infty$ and follow $\gamma$ in
the two different directions cover the whole curve $\gamma$. We emphasize that
the two continuations may hit non-polar singularities somewhere on the curve
$\gamma$, but this is only allowed to happen after the separation point has
already been passed.\smallskip

Throughout the present section we assume that the separation point
$z=z_{\gamma}\in\gamma\in\Gamma$ is chosen in an appropriate way, and we give
details only if necessary.\smallskip

In the next definition the set $\Gamma$ is divided into two subclasses.

\begin{definition}
\label{d91b} A Jordan curve $\gamma\in\Gamma=\Gamma(f,\infty)$ with partial
arcs $\gamma^{-}$ and $\gamma^{+}$ belongs to the subclass $\Gamma_{0}%
=\Gamma_{0}(f,\infty)\subset\Gamma$ if the meromorphic continuations of the
function $f$ along the two arcs $\gamma^{-}$ and $\gamma^{+}$ lead to the same
function element at the separation point $z$ of $\gamma$. If, on the other
hand, these continuations lead to two different function elements at $z$, then
the curve $\gamma$ belongs to the subclass $\Gamma_{1}=\Gamma_{1}%
(f,\infty)\subset\Gamma$.
\end{definition}

It is immediate that the two subsets $\Gamma_{0}$ and $\Gamma_{1}$ are
disjoint, and we have $\Gamma=\Gamma_{0}\cup\Gamma_{1}$.\smallskip

On the set $\Gamma$ we define a homotopy relation that fits our special needs.
Two elements $\gamma_{0},\gamma_{1}\in\Gamma$ are considered to be homotopic
if the two pairs $\{\gamma_{0}^{-},\gamma_{1}^{-}\}$ and $\{\gamma_{0}%
^{+},\gamma_{1}^{+}\}$ of partial arcs are homotopic in the usual sense, and
if in addition property (ii) in Definition \ref{d91a} is carried over from one
to the other Jordan curve $\gamma_{0}$ and $\gamma_{1}$ in a continuous
manner. More formally, we have the next definition.

\begin{definition}
\label{d91c} Two Jordan curves $\gamma_{0},\gamma_{1}\in\Gamma$ with partial
arcs $\gamma_{j}^{\pm}$, $j=0,1$, and separation points $z_{j}$, $j=0,1$, are
called homotopic (written $\sim$) if there exists a continuous function
$h:\left[  -1,1\right]  \times\left[  0,1\right]  \longrightarrow
\overline{\mathbb{C}}$ with the two following two properties:

\begin{itemize}
\item[(i)] For $j=0,1$, we have
\begin{equation}
\gamma_{j}(t)=h(t,j),\text{ \ \ \ \ \ }t\in\left[  -1,1\right]  .\label{f31c}%
\end{equation}

\item[(ii)] For each $s\in\left(  0,1\right)  $ a Jordan curve $\gamma_{s}$ is
defined by
\begin{equation}
\gamma_{s}:=h(\cdot,s):\left[  -1,1\right]  \longrightarrow\overline
{\mathbb{C}},\label{f31d}%
\end{equation}
and each $\gamma_{s}$ belongs to $\Gamma$ with separation point $\gamma
_{s}(0)$.
\end{itemize}

\noindent The equivalence class of $\gamma\in\Gamma$\ with respect to the
homotopy relation $\sim$ is denoted by $\left[  \gamma\right]  $.
\end{definition}

\begin{lemma}
\label{l91a} The splitting of the set $\Gamma$ into the two subclasses
$\Gamma_{0}$ and $\Gamma_{1}$ of Definition \ref{d91b} is compatible with the
homotopy relation of Definition \ref{d91c}.
\end{lemma}

\begin{proof}
The conclusion of the lemma is immediate.\medskip
\end{proof}

The ring domain $R\subset\overline{\mathbb{C}}$ and the continuum
$V\subset\mathbb{C}$\ in the next lemma will be used at several places in the
sequel. We say that $R$\ is a ring domain in $\overline{\mathbb{C}}$\ if
$\overline{\mathbb{C}}\setminus R$\ consists of two components.

\begin{lemma}
\label{l91b} For any $\gamma_{0}\in\Gamma=\Gamma(f,\infty)$ there exists a
ring domain $R\subset\overline{\mathbb{C}}$ with $\gamma_{0}\subset R $, for
which the following five assertions hold true:

\begin{itemize}
\item[(i)] The Jordan curve $\gamma_{0}$ separates the two components $A_{1} $
and $A_{2}$ of $\overline{\mathbb{C}}\setminus R$.

\item[(ii)] Any Jordan curve $\gamma\subset R$ with $\infty\in\gamma$ that
separates the two components $A_{1}$ and $A_{2}$ of $\overline{\mathbb{C}%
}\setminus R$ belongs to $\Gamma$.

\item[(iii)] Any $\gamma\in\Gamma$ with $\gamma\subset R$\ belong to
$\gamma\in\lbrack\gamma_{0}]$, i.e., we have $\gamma\sim\gamma_{0}$ in the
sense of Definition \ref{d91c}.

\item[(iv)] If a Jordan curve $\gamma\subset R$ separates the two components
$A_{1}$ and $A_{2}$ of $\overline{\mathbb{C}}\setminus R$, then any Jordan
curve $\widetilde{\gamma}\subset R$ with $\infty\in\widetilde{\gamma}$, which
is homotopic to $\gamma$\ in $R$\ (in the usual sense), belongs to $\Gamma$.

\item[(v)] If $\gamma_{0}\in\Gamma_{1}=\Gamma_{1}(f,\infty)$, then every
admissible compact set $K\in\mathcal{K}(f,\infty)$ contains a continuum
$V\subset\mathbb{C}$ that cross-sects $R$, i.e., we have
\begin{equation}
V\cap A_{j}\neq\emptyset\text{ \ for\ \ }j=1,2\label{f91c1}%
\end{equation}
with $A_{1}$ and $A_{2}$ the two components of $\overline{\mathbb{C}}\setminus
R $. The set $\mathcal{K}(f,\infty)$ has been introduced in Definition
\ref{d21a}.
\end{itemize}
\end{lemma}

\begin{proof}
Let $U^{-}$ and $U^{+}$ be two open and simply connected neighborhoods of the
partial arcs $\gamma_{0}^{-}$ and $\gamma_{0}^{+}$ of $\gamma_{0}$, and let
the function $f$ possess meromorphic continuations throughout\ $U^{-} $ and
$U^{+}$. Let further $z_{0}=\gamma_{0}(0)$ denote the separation point of
$\gamma_{0}$. By using $\varepsilon-$neighborhoods of $\gamma_{0}^{-}$ and
$\gamma_{0}^{+}$, one can easily show that there exists a ring domain
$R\subset\overline{\mathbb{C}}$ and an open disk $U_{0} $ with $z_{0}$ as its
centre such that
\begin{align}
& \overline{U}_{0}\subset U^{-}\cap U^{+}\text{, \ \ }\gamma_{0}\subset
R\subset U^{-}\cup U^{+},\label{f91d1}\\
& R\cap\partial U_{0}\text{ \ has exactly two components, and}\label{f91d2}\\
& R\setminus\overline{U}_{0}\text{ \ is a simply connected domain.}%
\label{f91d3}%
\end{align}
Assertion (i) immediately follows from the construction of the ring domain $R
$ if the $\varepsilon-$neighborhoods of $\gamma_{0}^{-}$ and $\gamma_{0}^{+}$
are chosen sufficiently narrow.

Assertion (ii) follows from the following two facts: (a) any Jordan curve
$\gamma$ in $R$ that separates the two components $A_{1}$ and $A_{2}$\ is
homotopic to $\gamma_{0}$ in the usual sense, and (b) $\gamma$\ will intersect
with $U_{0}$\ because of (\ref{f91d2}) and (\ref{f91d3}). From the last
assertion, it follows that we can choose a separation point $z$ for $\gamma$
anywhere in $\gamma\cap U_{0}$.

The assertions (iii) and (iv) are obvious completions of assertion (ii), and
they follow rather immediately from the construction of $R$\ in (\ref{f91d1}),
(\ref{f91d2}), and (\ref{f91d3}).

For the proof of assertion (v) we assume that $K$ is an arbitrary element of
$\mathcal{K}(f,\infty)$, i.e., $\overline{\mathbb{C}}\setminus K$ is an
admissible domain for Problem $(f,\infty)$ as introduced in Definition
\ref{d21a}, and further we assume that $\gamma_{0}\in\Gamma_{1}$.

We considered the open set $R\setminus K$. From $\gamma_{0}\in\Gamma_{1}$ and
assertion (iv) it follows that
\begin{equation}
\gamma\cap K\neq\emptyset\label{f91d4}%
\end{equation}
for all Jordan curves $\gamma\subset R$ that separate $A_{1}$ from $A_{2}$.
Indeed, if (\ref{f91e1}) were false for some Jordan curve $\gamma$, then this
curve could be modified near infinity in $R\setminus K$ into a Jordan curve
$\widetilde{\gamma}\subset R\setminus K$ that is homotopic to $\gamma$ in $R$
and $\infty\in\widetilde{\gamma}$. From assertion (iv) we then know that
$\widetilde{\gamma}\in\Gamma$. Since $R\setminus K\subset\overline{\mathbb{C}%
}\setminus K\in\mathcal{D}(f,\infty)$, we know from Definition \ref{d21a} that
the function $f$ has a single-valued meromorphic continuation along the whole
curve $\widetilde{\gamma}$, which implies that $\widetilde{\gamma}\in
\Gamma_{0}$. On the other hand, from the assumption $\gamma_{0}\in\Gamma_{1}$
we deduce with assertion (iii) that also $\gamma_{1}\in\Gamma_{1}$. This last
contradiction proves (\ref{f91d4}).

Assertion (v) then follows from (\ref{f91d4}) and the next Lemma \ref{l91c}.
The lemma is of independent interest for several applications at other places
in the sequel.\medskip
\end{proof}

\begin{lemma}
\label{l91c} Let $R\subset\overline{\mathbb{C}}$ be a ring domain, $A_{1}$ and
$A_{2}$ the two components of $\overline{\mathbb{C}}\setminus R$, and let
$K\subset\mathbb{C}$ be a compact set. There exists a continuum $V\subset K$
with
\begin{equation}
V\cap A_{j}\neq\emptyset\text{ \ for\ \ }j=1,2\label{f91e1}%
\end{equation}
if and only if
\begin{equation}
\gamma\cap K\neq\emptyset\label{f91e2}%
\end{equation}
for every Jordan curve $\gamma\subset R$ that separates $A_{1}$ from $A_{2}$.
\end{lemma}

\begin{proof}
Let us first assume that there exists a Jordan curve $\gamma$ with the given
properties for which (\ref{f91e2}) is false, and let $O_{1}$ and $O_{2} $ be
the interior and the exterior domain of $\gamma$. Then for any continuum
$V\subset K$ satisfying (\ref{f91e1}) we would have the contradiction that
$V\subset O_{1}\cup O_{2}$ and $V\cap O_{j}\neq\emptyset$ for $j=1,2$.

Next, we assume that (\ref{f91e2}) holds true, and set $B_{0}:=\overline
{R}\cap K$. Let $C_{j}\subset B_{0}$, $j\in I$, be the family of all
components of $B_{0}$ that are disjoint from at least one of the two sets
$A_{1}$ or $A_{2}$. The set $I$ is denumerable, we assume $I\subset\mathbb{N}%
$, and define
\begin{equation}
B_{n}:=\overline{B_{0}\setminus\bigcup\nolimits_{j\in I,\text{ }j\leq n}C_{j}%
}\text{ \ \ for \ \ }n=1,2,\ldots\label{f91e3}%
\end{equation}
The assumption of (\ref{f91e2}) implies that also
\begin{equation}
\gamma\cap B_{n}\neq\emptyset\text{ \ \ for \ \ }n>0\label{f91e4}%
\end{equation}
and for every Jordan curve $\gamma\subset R$ that separates $A_{1}$ from
$A_{2}$. Indeed, if there would exist an exceptional Jordan curve $\gamma$,
then $\gamma$ could be modified into a Jordan curve $\widetilde{\gamma}\subset
R\setminus B_{0}$ that is homotopic to $\gamma$ in $R$, which then would
contradict (\ref{f91e2}).

From (\ref{f91e4}) we deduce that
\begin{equation}
B_{\infty}:=\bigcap\nolimits_{n\in\mathbb{N}}B_{n}\neq\emptyset.\label{f91e5}%
\end{equation}

The set $B_{\infty}$ contains only components that intersect simultaneously
both sets $A_{1}$ and $A_{2}$, which proves the existence of a continuum
$V\subset B_{\infty}\subset K$ satisfying (\ref{f91e1}).\medskip
\end{proof}

The following proposition has been the main reason and motivation for the
introduction of the sets $\Gamma$, $\Gamma_{0}$, and $\Gamma_{1}$ of Jordan
curves in the Definitions \ref{d91a} and \ref{d91b}.\smallskip

\begin{proposition}
\label{p91a} Let $\Gamma=\Gamma(f,\infty)$, $\Gamma_{0}$, $\Gamma_{1}%
\subset\Gamma$ be the sets of Jordan curves introduced in the two Definitions
\ref{d91a} and \ref{d91b}, and let $\mathcal{D}(f,\infty)$ be the set of
admissible domains for Problem $(f,\infty)$ introduced in Definition
\ref{d21a}.

A domain $D\subset\overline{\mathbb{C}}$ with $\infty\in D$ belongs to
$\mathcal{D}(f,\infty)$ if, and only if, the following two assertions hold true:

\begin{itemize}
\item[(i)] The function $f$ has a meromorphic continuation along each closed
Jordan arc $\gamma$ in $D$ that starts at $\infty$.

\item[(ii)] For each Jordan curve $\gamma\in\Gamma_{1}$ we have $\gamma
\cap(\overline{\mathbb{C}}\setminus D)\neq\emptyset$.
\end{itemize}
\end{proposition}

\begin{proof}
Assertion (i) ensures that the function $f$ has a meromorphic continuation to
each point of the domain $D$, and assertion (ii) guarantees that these
continuations are single-valued. Hence, the two assertions (i) and (ii) imply
$D\in\mathcal{D}(f,\infty)$.

The other direction of the proof follows also rather immediately from the two
Definitions \ref{d21a} and \ref{d91b}. If $D\in\mathcal{D}(f,\infty)$, then
clearly assertions (i) holds true; and if there would exist $\gamma_{1}%
\in\Gamma_{1}$ with $\gamma_{1}\subset D$, then this would contradict the
assumption in Definitions \ref{d21a} that the meromorphic continuation of the
function $f$ in $D$ is single-valued.\medskip
\end{proof}

\subsection{\label{s92}The Existence of a Domain in $\mathcal{D}_{0}%
(f,\infty)$}

\qquad In Definition \ref{d21b}, the set of all admissible domains with a
complement of minimal capacity has been denoted by $\mathcal{D}_{0}(f,\infty
)$. In the present subsection we prove that $\mathcal{D}_{0}(f,\infty)$ is not
empty.\smallskip

\begin{proposition}
\label{p92a}We have $\mathcal{D}_{0}(f,\infty)\neq\emptyset$.\smallskip
\end{proposition}

The basic structure of the proof of Proposition \ref{p92a} is simple and
straightforward: A minimizing sequence of admissible compact sets $K_{n}%
\in\mathcal{K}(f,\infty)$, $n\in\mathbb{N}$, is chosen in such a way that in
the limit the minimality condition (\ref{f21a}) in Definition \ref{d21b} is
satisfied. The transition to the limit situation is done in the frame work of
potential theory. It is shown that after some plausible corrections the
resulting domain is admissible for Problem $(f,\infty)$. However, in the
practical realization a number of technical hurdles have to be overcome; the
whole proof is broken down in several consecutive steps, which are presented
as lemmas.\smallskip

In a first step, we deal with the very special situation that we have the
value zero in the minimality (\ref{f21a}) of Definition \ref{d21b}.\smallskip

\begin{lemma}
\label{l92a}If in (\ref{f21a}) of Definition \ref{d21b} we have
\begin{equation}
\inf_{K\in\mathcal{K}(f,\infty)}\operatorname*{cap}(K)=0,\label{f92a}%
\end{equation}
then the subclass $\Gamma_{1}(f,\infty)$ of Jordan curves introduced in
Definition \ref{d91b} is empty.
\end{lemma}

\begin{proof}
For an indirect proof we assume $\Gamma_{1}=\Gamma_{1}(f,\infty)\neq\emptyset
$. Let $\gamma_{0}$ be an element of $\Gamma_{1}$, and let further
$R\subset\overline{\mathbb{C}}$ be a ring domain with $\gamma_{0}\subset R$ as
introduced in Lemma \ref{l91b}. From assertion (iv) of Lemma \ref{l91b}\ it
follows that for every admissible compact set $K\in\mathcal{K}(f,\infty)$
there exists a continuum $V\subset K$ that intersects $R$, i.e., we have
\begin{equation}
V\cap A_{j}\neq\emptyset\text{ \ for\ \ }j=1,2\label{f92a1}%
\end{equation}
and $A_{1}$, $A_{2}$ the two components of $\overline{\mathbb{C}}\setminus R$.
From the lower estimate (\ref{f111b2}) for the capacity given in Lemma
\ref{l111a}, further below, we then conclude that
\begin{equation}
\operatorname*{cap}(K)\geq\operatorname*{diam}(V)/4\geq\operatorname*{dist}%
(A_{1},A_{2})/4\text{.}\label{f92a2}%
\end{equation}
Since the right-hand side of (\ref{f92a2}) is independent of $V$\ and the
choice of $K$,\ the estimate (\ref{f92a2}) contradicts (\ref{f92a}). Thus, we
have proved that $\Gamma_{1}=\emptyset$.\medskip
\end{proof}

In Lemma \ref{l92a} a special case of Proposition \ref{p22b} has been
addressed, and we have the following corollary.\medskip

\begin{corollary}
\label{c92a}If condition (\ref{f92a}) is satisfied, then all meromorphic
continuations of the function $f$ are single-valued, and consequently, the
extremal domain $D_{0}=D_{0}(f,\infty)$ of Definition \ref{d21b} is the
Weierstrass\ domain $W_{f}\subset\overline{\mathbb{C}}$ for meromorphic
continuation of the function $f$ starting at $\infty$.
\end{corollary}

\begin{proof}
It follows immediately from Definition \ref{d91b} that $\Gamma_{1}=\emptyset$
is equivalent to the single-valuedness of all meromorphic continuations of $f$
in $\overline{\mathbb{C}}$, and consequently we have $D_{0}(f,\infty)=W_{f}%
$.\medskip
\end{proof}

Thanks to Lemma \ref{l92a}, we can now assume without loss of generality for
the remainder of the present subsection that
\begin{equation}
\inf_{K\in\mathcal{K}(f,\infty)}\operatorname*{cap}(K)=:c_{0}>0.\label{f92b}%
\end{equation}

We select a sequence of admissible compact sets $K_{n}\in\mathcal{K}%
(f,\infty)$, $n\in\mathbb{N}$, such that
\begin{equation}
\lim_{n\rightarrow\infty}\operatorname*{cap}(K_{n})=c_{0}.\label{f92c}%
\end{equation}

\begin{lemma}
\label{l92b} There exists $r>0$ such that we can assume without loss of
generality that the sequence $\{K_{n}\}$\ in (\ref{f92c})\ satisfies
\begin{equation}
K_{n}\subset\left\{  \text{\thinspace}|z|\leq r\text{\thinspace}\right\}
\text{ \ \ \ for all \ \ }n\in\mathbb{N}\text{.}\label{f92d}%
\end{equation}

\end{lemma}

\begin{proof}
Let $r>1$ be such that $f$ is meromorphic in $\left\{  \text{\thinspace
}|z|>r-1\text{\thinspace}\right\}  $. For any admissible compact set
$K\in\mathcal{K}(f,\infty)$, we denote by $\widetilde{K}$ the radial
projection of $K$ onto the disk $\left\{  \text{\thinspace}|z|\leq
r\text{\thinspace}\right\}  $ as defined in (\ref{f111d1}) of Subsection
\ref{s1101}, further below. It is not difficult to verify that because of
$\overline{\mathbb{C}}\setminus K\in\mathcal{D}(f,\infty)$ we also
have$\ \widetilde{D}:=\overline{\mathbb{C}}\setminus\widetilde{K}%
\in\mathcal{D}(f,\infty)$. One has only to check the conditions in Definition
\ref{d21a}.

From Lemma \ref{l111c} in Subsection \ref{s1101}, it then follows that
$\operatorname*{cap}(\widetilde{K})\leq\operatorname*{cap}(K)$. Hence, any
compact set $K_{n}$ in (\ref{f92c}), which does not satisfy (\ref{f92d}), can
be replaced by $\widetilde{K}_{n}$, and because of (\ref{f92b}), the limit
(\ref{f92c}) remains unchanged under such modifications.\medskip
\end{proof}

In the sequel we assume that the inclusions (\ref{f92d}) hold true for all
compact sets $K_{n}\in\mathcal{K}(f,\infty)$, $n\in\mathbb{N}$, in
(\ref{f92c}).\medskip

Let $\omega_{n}$ be the equilibrium distribution of the compact set $K_{n}$,
$n\in\mathbb{N}$, and let further $g_{n}=g_{D_{n}}(\cdot,\infty)$ be the Green
function in the domain $D_{n}$ (for definitions of $\omega_{n}$ and $g_{n}$
see Section \ref{s110}, further below). As explained in Subsection
\ref{s1104}, there always exists an infinite subsequence $N\subset\mathbb{N}$
such that the weak$^{\ast}$ limit
\begin{equation}
\omega_{n}\overset{\ast}{\longrightarrow}\omega_{0}\text{ \ \ as
\ \ }n\rightarrow\infty\text{, }n\in N.\label{f92e1}%
\end{equation}
exists. Since inclusion (\ref{f92d}) has been assumed to hold true for the
sequence $\{K_{n}\}$, we have
\begin{equation}
\operatorname*{supp}(\omega_{0})\subset\left\{  \text{\thinspace}|z|\leq
r\text{\thinspace}\right\}  ,\label{f92e2}%
\end{equation}
and $\omega_{0}$\ is a probability measures.

Using representation (\ref{f113b1}) of Lemma \ref{l113b}\ for the Green
function $g_{n}$ we have
\begin{equation}
g_{n}=-p(\omega_{n};\cdot)-\log\operatorname*{cap}\left(  K_{n}\right)
\label{f92e21}%
\end{equation}
with $p(\omega_{n};\cdot)$ denotes the logarithmic potential of $\omega_{n}$,
which has formerly been defined in Subsection \ref{s1102}, further below. From
limit (\ref{f92e1}) and the Lower Envelope Theorem \ref{t112a} of potential
theory (cf. Subsection \ref{s1102}, further below) we then conclude that
\begin{equation}
\limsup_{N}g_{n}\,\leq\,g_{0}:=-p(\omega_{0};\cdot)-\log\left(  c_{0}\right)
\label{f92e3}%
\end{equation}
with the constant $c_{0}$\ introduced in (\ref{f92b}). In (\ref{f92e3}),
equality holds\ quasi everywhere in $\mathbb{C}$ (for a definition of
\textquotedblright quasi everywhere\textquotedblright\ see Definition
\ref{d111a}, further below). It follows from (\ref{f92e1}) and (\ref{f92e2})
that outside of $\left\{  \text{\thinspace}|z|\leq r\text{\thinspace}\right\}
$ we have a proper limit in (\ref{f92e3}) instead of the limes superior, and
equality holds there instead of the inequality stated in (\ref{f92e3}). In
$\left\{  \text{\thinspace}|z|>r\text{\thinspace}\right\}  $, the limit in
(\ref{f92e3}) holds locally uniformly.

\begin{definition}
\label{d92a}We define
\begin{equation}
\widetilde{K}_{0}:=\overline{\left\{  \text{ }z\in\mathbb{C}\text{ }%
|\text{\ }g_{0}(z)=0\text{\ }\right\}  },\label{f92e4}%
\end{equation}
and $\widetilde{D}_{0}:=\overline{\mathbb{C}}\setminus\widetilde{K}_{0}$.
\end{definition}

The two sets $\widetilde{D}_{0}$ and $\widetilde{K}_{0}$ will become building
blocks for extremal domains and minimal sets of Problem $(f,\infty) $, but
several modifications and special considerations have to be made before the
construction can be finished.\medskip

We note that the two sets $\widetilde{D}_{0}$ and $\widetilde{K}_{0}$, like
the measure $\omega_{0}$ and the function $g_{0}$, depend on the subsequence
$N\subset\mathbb{N}$ used in the limit (\ref{f92e1}).

\begin{lemma}
\label{l92c}We have
\begin{equation}
\operatorname*{cap}(\widetilde{K}_{0})\,\leq\,c_{0}=\inf_{D\in\mathcal{D}%
(f,\infty)}\operatorname*{cap}(\overline{\mathbb{C}}\setminus D),\label{f92e6}%
\end{equation}
$\widetilde{K}_{0}\subset\left\{  \text{\thinspace}|z|\leq r\text{\thinspace
}\right\}  ,$ and $\widetilde{D}_{0}$\ is a domain.
\end{lemma}

\begin{proof}
The function $g_{0}$\ introduced in (\ref{f92e3})\ is subharmonic in
$\mathbb{C} $, which implies that the set $\widetilde{D}_{0}$ introduced in
Definition \ref{d92a}\ is a domain.

The inclusion $\widetilde{K}_{0}\subset\left\{  \text{\thinspace}|z|\leq
r\text{\thinspace}\right\}  $\ is an immediate consequence of (\ref{f92e2}).

It remains to prove (\ref{f92e6}). From (\ref{f92e3}) it follows that
$g_{0}\geq0$ everywhere in $\mathbb{C}$. Since the logarithmic potential of a
finite measure is continuous quasi everywhere in $\mathbb{C}$ (cf. the
introductory paragraphs of Subsection \ref{s1102}, further below), we conclude
from (\ref{f92e4}) that
\begin{equation}
g_{0}(z)=0\text{ \ \ for quasi every \ }z\in\widetilde{K}_{0}.\label{f92f1}%
\end{equation}

We can assume without loss of generality that $\operatorname*{cap}%
(\widetilde{K}_{0})>0$ since otherwise (\ref{f92e6}) is trivially true. From
Lemma \ref{l112a} in Subsection \ref{s1102} we then know that the equilibrium
distribution $\widetilde{\omega}_{0}$ on $\widetilde{K}_{0}$ is of finite
energy. Hence, we can use the Principle of Domination from Theorem \ref{t112b}
in Subsection \ref{s1102} for a comparison of the function $g_{0}%
=-p(\omega_{0};\cdot)-\log\left(  c_{0}\right)  $ from (\ref{f92e3}) with the
Green function $g_{\overline{\mathbb{C}}\setminus\widetilde{K}_{0}}%
(\cdot,\infty)=-p(\widetilde{\omega}_{0};\cdot)-\log\operatorname*{cap}%
(\widetilde{K}_{0})$. In the last equation, we have used representation
(\ref{f113b1}) from Lemma \ref{l113b} in Subsection \ref{s1103}. With the
Principle of Domination we deduce from (\ref{f92f1}) that
\begin{equation}
g_{\overline{\mathbb{C}}\setminus\widetilde{K}_{0}}(z,\infty)\geq
\,g_{0}(z)\text{ \ \ for all \ \ }z\in\overline{\mathbb{C}}.\label{f92f2}%
\end{equation}
Comparing both side in (\ref{f92f2}) near infinity yields the inequality
\begin{equation}
\log\operatorname*{cap}(\widetilde{K}_{0})\leq\log c_{0},\label{f92f3}%
\end{equation}
which proves (\ref{f92e6}).\medskip
\end{proof}

It will turn out in (\ref{f92i2}) and (\ref{f92i3}), further below, that in
(\ref{f92e6}) we always have equality. Because of (\ref{f92f2}), this means
that we have $\omega_{0}=\widetilde{\omega}_{0}$, and therefore the measure
$\omega_{0}$ has no mass points outside of $\widetilde{K}_{0}$.\medskip

In the next lemma, we see that $\overline{\mathbb{C}}\setminus\widetilde
{K}_{0} $ is indeed an important building block for an admissible domain with
a minimal capacity, i.e., an element of $\mathcal{D}_{0}(f,\infty)$. The
result should be seen in relation to Proposition \ref{p91a}.

\begin{lemma}
\label{l92d}We have $\gamma\cap\widetilde{K}_{0}\neq\emptyset$ for every
Jordan curve $\gamma\in\Gamma_{1}$.
\end{lemma}

\begin{proof}
Let $\gamma_{0}$ be an arbitrary element of $\Gamma_{1}$ with $\Gamma
_{1}=\Gamma_{1}(f,\infty)$ introduced in Definition \ref{d91b}, and let
further $R\subset\overline{\mathbb{C}}$ be a ring domain with $\gamma
_{0}\subset R$ as introduced in Lemma \ref{l91b}. Since $\gamma_{0}\in
\Gamma_{1}$, we know from assertion (iv) in Lemma \ref{l91b} that for every
$n\in\mathbb{N}$ there exists a continuum $V_{n}\subset K_{n}$ which
cross-sects the ring domain $R$, i.e., we have
\begin{equation}
V_{n}\cap A_{j}\neq\emptyset\text{ \ \ for \ }j=1,2\label{f92f4}%
\end{equation}
with $A_{1}$ and $A_{2}$ the two components of $\overline{\mathbb{C}}\setminus
R $.

Using Lemma \ref{l114b} from Subsection \ref{s1104} together with the
assumptions (\ref{f92c}), (\ref{f92d}), and (\ref{f92e1}), we conclude from
(\ref{f92f4}) that there exists a continuum $V\subset\left\{  \text{\thinspace
}|z|\leq r\text{\thinspace}\right\}  $ that satisfy (\ref{f114f12}) and
(\ref{f114c13}) in Lemma \ref{l114b}, i.e., we have
\begin{equation}
V\cap A_{j}\neq\emptyset\text{ \ \ for \ }j=1,2,\text{ \ and}\label{f92f6}%
\end{equation}%
\begin{equation}
g_{0}(z)=0\text{ \ \ for \ \ }z\in V.\label{f92f5}%
\end{equation}
From (\ref{f92e4}) and (\ref{f92f5}), we then conclude that $V\subset
\widetilde{K}_{0}$. Because of (\ref{f92f6}), we also know from Lemma
\ref{l91c} that $\gamma_{0}\cap V\neq\emptyset$, and consequently, we have
shown that $\gamma_{0}\cap\widetilde{K}_{0}\neq\emptyset$.\medskip
\end{proof}

In the proof of Lemma \ref{l92d}, Lemma \ref{l114b} from Subsection
\ref{s1104}, further below, has played a key role. The lemma will be essential
at several other places in the sequel, and since its proof in Subsection
\ref{s1104} is based on Carath\'{e}odory's Theorem about kernel convergence,
we can say that the proof of the last lemma and also that of the other results
is essentially based on Carath\'{e}odory's fundamental Theorem. This theorem
gives also importance to the use of the continua $V$\ that have already
appeared in the two Lemmas \ref{l91b} and \ref{l91c}.\smallskip

As a corollary of Lemma \ref{l92d}, we deduce that any meromorphic
continuation of the function $f$ in the domain $\widetilde{D}_{0}$ is
single-valued. Thus, one of the two conditions in Proposition \ref{p91a} for a
characterization of an admissible domain is satisfied by $\widetilde{D}_{0} $.
What we still have not investigated is the question whether the function $f$
can be meromorphically continued to every point of $\widetilde{D}_{0}$, or how
large the set of exceptional points in $\widetilde{D}_{0}$ can be if such a
continuation is not possible. We start the investigation with the following
definition.\smallskip

\begin{definition}
\label{d92b}Let $\widetilde{E}_{0}\subset\widetilde{D}_{0}$ be the set of all
points $z\in\widetilde{D}_{0}$ that satisfy the following two conditions:

\begin{itemize}
\item[(i)] There exists a Jordan arc $\gamma$ in $\widetilde{D}_{0}$ with
initial point $\infty$ and end point $z$, and the function $f$ has a
meromorphic continuation along $\gamma\setminus\{z\}$.

\item[(ii)] At the point $z$, the meromorphic continuation of $f$ along
$\gamma$\ has a non-polar singularity.\smallskip
\end{itemize}
\end{definition}

The next lemma is an immediate consequence of Lemma \ref{l92d}.

\begin{lemma}
\label{l92e}Let $z_{0}\in\widetilde{D}_{0}$, and let $\gamma_{1},$ $\gamma
_{2}$\ be two Jordan arcs that both satisfy condition (i) in Definition
\ref{d92b} with $z_{0}$\ as end point. Then condition (ii) of Definition
\ref{d92b} is either simultaneously satisfied, or simultaneously not satisfied
by the two arcs $\gamma_{1}$ and $\gamma_{2}$.
\end{lemma}

\begin{proof}
Let us assume that $\gamma_{1}$ and $\gamma_{2}$ are two Jordan arcs, which
both satisfy condition (i) in Definition \ref{d92b} with $z_{0}$\ as end
point, let $\gamma_{1}$ satisfy also condition (ii), but $\gamma_{2}$ not.
Then after some modifications, if necessary, the composition $\gamma
:=\gamma_{1}-\gamma_{2}$ is a Jordan curve in $\widetilde{D}_{0}$. A
separation point in the sense of Definition \ref{d91a} can be chosen on
$\gamma_{1}$ in the neighborhood of $z_{0}$, and we then have $\gamma\in
\Gamma$ with $\Gamma=\Gamma(f,\infty)$ introduced in Definition \ref{d91a}. It
follows from the assumptions made with respect to $\gamma_{1}$ and $\gamma
_{2}$ together with Definition \ref{d91a} that $\gamma\in\Gamma_{1}$, but this
would contradict Lemma \ref{l92d}.\medskip
\end{proof}

The next lemma is a preparation of a proof of the result that
$\operatorname*{cap}(\widetilde{E}_{0})=0$ for the set $\widetilde{E}_{0}$
that has been introduced in Definition \ref{d92b}. Not only the formulation
but also the proof this lemma is rather technical.

\begin{lemma}
\label{l92f}Let $D_{1}$ be a simply connected and bounded domain with
$\overline{D}_{1}\subset\widetilde{D}_{0}$. Then there exists $n_{1}%
\in\mathbb{N}$ such that
\begin{equation}
\widetilde{E}_{0}\cap\overline{D}_{1}\subset K_{n}\text{ \ for all\ \ }n\geq
n_{1},\text{ }n\in N\label{f92g}%
\end{equation}
with $N\subset\mathbb{N}$ the subsequence used in the limit (\ref{f92e1}).
\end{lemma}

\begin{proof}
With the assumptions made with respect to the domain $D_{1}$ it is rather
immediate that there exists a ring domain $R\subset\widetilde{D}_{0}$ such
that for one of the two components $A_{1}$ and $A_{2}$ of $\overline
{\mathbb{C}}\setminus R$, say $A_{1}$, we have
\begin{equation}
\overline{D}_{1}\subset A_{1}\subset\widetilde{D}_{0}\text{ \ and \ \ }%
\infty\in R.\label{f92g1}%
\end{equation}
In addition to the domain $D_{1}$, we define the domain $D_{2}:=A_{1}\cup
R\subset\widetilde{D}_{0}$.

In a first step of the proof we show that there exists $n_{1}\in\mathbb{N}$
such that for each $n\geq n_{1}$, $n\in N$, there exists at least one Jordan
curve
\begin{equation}
\gamma\in\Gamma_{0}=\Gamma_{0}(f,\infty)\text{ \ \ with \ }\gamma\subset
R\setminus K_{n}.\label{f92g2}%
\end{equation}

The proof of (\ref{f92g2}) will be given indirectly. For a negation of
(\ref{f92g2}) we consider the following two cases a and b:

Case a: There exists an infinite subsequence $N_{1}\subset N$ such that for
each $n\in N_{1}$ and each Jordan curve $\gamma\subset R$ that separates
$A_{1}$ from $A_{2}$ we have
\begin{equation}
\gamma\cap K_{n}\neq\emptyset.\label{f92g3}%
\end{equation}

Case b: There exists an infinite subsequence $N_{2}\subset N$ such that for
each $n\in N_{2}$ there exists at least one Jordan curve
\begin{equation}
\gamma\in\Gamma_{1}=\Gamma_{1}(f,\infty)\text{ \ \ with \ }\gamma\subset
R\setminus K_{n}.\label{f92g4}%
\end{equation}
The sets $\Gamma_{0}$ and $\Gamma_{1}$ have been introduced in Definition
\ref{d91b}.

If we have disproved Case a, then it follows from assertion (ii) of Lemma
\ref{l91b} that there exists $\gamma\in\Gamma$ with $\gamma\subset R\setminus
K_{n}$ for every $n\in N$\ sufficiently large, and consequently, either
(\ref{f92g2}) or (\ref{f92g4}) holds true for the particular Jordan curve
$\gamma$. If also Case b is disproved, then it follows from assertion (iii) in
Lemma \ref{l91b} that for every $n\in N$ sufficiently large there exists a
Jordan curve $\gamma$ which satisfies (\ref{f92g2}), and we have accomplished
the first step of the proof.\smallskip

In order to disprove Case a, we observe that from Lemma \ref{l91c} together
with (\ref{f92g3}), it follows that for each $n\in N_{1}$ there exists a
continuum $V_{n}\subset K_{n}$ with
\begin{equation}
V_{n}\cap A_{j}\neq\emptyset\text{ \ \ for \ \ }j=1,2\text{.}\label{f92g5}%
\end{equation}
As in the relations (\ref{f92f6}) and (\ref{f92f5}) in the proof of Lemma
\ref{l92d}, we deduce from (\ref{f92g5}) with the help of Lemma \ref{l114b} in
Subsection \ref{s1104}, further below, that there exists a continuum
$V\subset\widetilde{K}_{0}$ with $V\cap R\neq\emptyset$, but the existence of
$V$ contradicts the assumption $R\subset\widetilde{D}_{0}$. Hence, Case a is
disproved.\smallskip

In order to disprove Case b, we observe that from the existence of the Jordan
curve $\gamma$\ in (\ref{f92g4}) and assertion (iv) in Lemma \ref{l91b} it
follows that for each $n\in N_{2}$ there exists a continuum $V_{n}\subset
K_{n}$ that satisfies (\ref{f92g5}). With the same arguments as just used
after (\ref{f92g5}), we come to the same conclusion that there exists
$V\subset\widetilde{K}_{0}$ with $V\cap R\neq\emptyset$, which again
contradicts $R\subset\widetilde{D}_{0}$, and thus, also Case b is
disproved.\smallskip

As already said before, with a disproof of the two Cases a and b, we have
shown that there exists $n_{1}\in\mathbb{N}$ such that for every $n\geq n_{1}%
$, $n\in N$ there exists a Jordan curve $\gamma$ which satisfies
(\ref{f92g2}).\smallskip

In a second step, we prove indirectly the relations (\ref{f92g}) for $n\geq
n_{1}$, $n\in N$. Let us assume that $(\widetilde{E}_{0}\cap\overline{D}%
_{1})\setminus K_{n_{0}}\neq\emptyset$ for a certain $n_{0}\geq n_{1}$,
$n_{0}\in N$. Then there exists a point
\begin{equation}
z_{0}\in(\widetilde{E}_{0}\cap\overline{D}_{1})\setminus K_{n_{0}%
}.\label{f92g6}%
\end{equation}

Let in accordance with Definition \ref{d92b} $\gamma_{1}\subset\widetilde
{D}_{0}$\ be a Jordan arc with initial point $\infty$ and end point $z_{0}$
such that the two conditions (i) and (ii) of Definition \ref{d92b} are satisfied.

Since the function $f$ is single-valued and meromorphic in $D_{n_{0}%
}=\overline{\mathbb{C}}\setminus K_{n_{0}}\in\mathcal{D}(f,\infty)$, there
exists a Jordan arc $\widetilde{\gamma}_{2}\subset D_{n_{0}}$ with initial
point $\infty$ and end point $z_{0}$ and $f$ is meromorphic along
$\widetilde{\gamma}_{2}$.

We know from (\ref{f92g2}) that in $R\setminus K_{n_{0}}=R\cap D_{n_{0}}$
there exists a Jordan curve $\gamma_{0}$ with $\infty\in\gamma_{0}$, and along
$\gamma_{0}$ the function $f$ has a single-valued meromorphic continuation.
Hence, we can modify the arc $\widetilde{\gamma}_{2}$ in such a way that the
modified Jordan arc $\gamma_{2}$ coincides with $\widetilde{\gamma}_{2}$ after
its last contact with $\gamma_{0}$, but the whole Jordan arc $\gamma_{2}$ is
contained in $D_{n_{0}}\cap D_{2}\subset D_{n_{0}}\cap\widetilde{D}_{0}$, and
it connects $\infty$\ with $z_{0}$. Clearly, the new Jordan arc $\gamma_{2}$
satisfies condition (i) of Definition \ref{d92b}, but it does not satisfy
condition (ii). Hence, the two Jordan arcs $\gamma_{1}$ and $\gamma_{2}$
contradict Lemma \ref{l92e}. This contradiction disproves the existence of
$z_{0}$ in (\ref{f92g6}), and completes the proof of lemma.\medskip
\end{proof}

The proof of Lemma \ref{l92f} has been very technical since in its background
logic we were confronted with the possibility that in each admissible compact
set $K_{n}$, $n\in N$, different non-polar singularities of the function $f$
may be 'active' or 'inactive'. Illustrations for this phenomenon have been
given in the examples of Section \ref{s6}. In the situation of Lemma
\ref{l92f}, it has turned out that with the selection of the subsequence
$N\subset\mathbb{N}$ in (\ref{f92e1}) all relevant choices in this respect
have been fixed by the set $\widetilde{K}_{0}$.\medskip

\begin{lemma}
\label{l92g}We have
\begin{equation}
\operatorname*{cap}(\widetilde{E}_{0})=0\label{f92h1}%
\end{equation}
for the set $\widetilde{E}_{0}$ introduced in Definition \ref{d92b}.
\end{lemma}

\begin{proof}
For an indirect proof we assume that
\begin{equation}
\operatorname*{cap}(\widetilde{E}_{0})>0.\label{f92h2}%
\end{equation}
Using an exhaustion of the domain $\widetilde{D}_{0}\cap\left\{
\text{\thinspace}|z|\leq r\text{\thinspace}\right\}  $ by overlapping closed
and simply connected domains, one can show because of (\ref{f92h2}) that there
exists a simply connected domain $D_{1}$\ with $\overline{D}_{1}%
\subset\widetilde{D}_{0}\cap\left\{  \text{\thinspace}|z|\leq
r\text{\thinspace}\right\}  $ such that
\begin{equation}
\operatorname*{cap}(\widetilde{E}_{0}\cap\overline{D}_{1})>0.\label{f92h3}%
\end{equation}
The constant $r$ is the same as that in Lemma \ref{l92b}. From (\ref{f92h3})
and Lemma \ref{l92f}, we know that there exists $n_{1}\in\mathbb{N}$ such that
(\ref{f92g1}) holds true. Using Lemma \ref{l114a} from Subsection \ref{s1104},
further below, we deduce from (\ref{f92g1}) together with the assumptions
(\ref{f92c}) and (\ref{f92e1}) that we have
\begin{equation}
g_{0}(z)=0\text{ \ \ for quasi every \ }z\in\widetilde{E}_{0}\cap\overline
{D}_{1}\label{f92h4}%
\end{equation}
with $g_{0}$\ the function introduced in (\ref{f92e3}). From (\ref{f92h3}),
(\ref{f92h4}), and (\ref{f92e4}) in Definition \ref{d92a}, it then follows
that $\widetilde{E}_{0}\cap\widetilde{K}_{0}\neq\emptyset$, but this
contradicts Definition \ref{d92b}, and thus, the lemma is proved.\medskip
\end{proof}

With the two Definitions \ref{d92a} and \ref{d92b}, the Lemmas \ref{l92c},
\ref{l92d}, \ref{l92g}, and Proposition \ref{p91a} we are prepared to prove
Proposition \ref{p92a} and close the present subsection.\medskip

\textbf{Proof of Proposition \ref{p92a}.} In a first step, we deal with the
special case that (\ref{f92a}) holds true. We then know from Lemma \ref{l92a}
and its Corollary \ref{c92a} that the extremal domain $D_{0}=D_{0}%
(f,\infty)\in\mathcal{D}_{0}(f,\infty)$ of Definition \ref{d21b} exists and is
identical with the Weierstrass\ domain $W_{f}\subset\overline{\mathbb{C}}$ for
meromorphic continuations of the function $f$ starting at $\infty$.\smallskip

In the second step, we assume that the inequality in (\ref{f92b}) is
satisfied, and define
\begin{equation}
D_{0}:=\widetilde{D}_{0}\setminus\widetilde{E}_{0},\label{f92i1}%
\end{equation}
and show that $D_{0}\in\mathcal{D}_{0}(f,\infty)$.

The set $\widetilde{E}_{0}$ is identical to its polynomial-convex hull
$\widehat{E}_{0}$. Indeed, from Lemma \ref{l92d} and from Lemma \ref{l111d} in
Subsection \ref{s1101}, further below, we deduce that
\begin{equation}
\operatorname*{cap}(\widehat{E}_{0})=\operatorname*{cap}(\widetilde{E}%
_{0})=0.\label{f92i4}%
\end{equation}
From (\ref{f92i4}) it follows that $\widehat{E}_{0}$ can have no inner points,
and consequently, we have $\widehat{E}_{0}=\widetilde{E}_{0}$. This identity
together with (\ref{f92i1}) and Lemma \ref{l92c} implies that $D_{0} $ is a
domain.\smallskip

From Lemma \ref{l92d} we know that condition (ii) in Proposition \ref{p91a} is
satisfied. From (\ref{f92i1}) and Definition \ref{d92b} it follows that also
condition (i) in Proposition \ref{p91a} is satisfied. It therefore follows
from Proposition \ref{p91a} that $D_{0}$ is an admissible domain for Problem
$(f,\infty)$, i.e., $D_{0}\in\mathcal{D}(f,\infty)$.\smallskip

Since the capacity of a capacitable set does not change its value if a set of
capacity zero is added or subtracted (cf. Lemma \ref{l111b} in Subsection
\ref{s1101}, further below), we deduce from the two Lemmas \ref{l92c} and
\ref{l92g} that
\begin{equation}
\operatorname*{cap}(\overline{\mathbb{C}}\setminus D_{0})=\operatorname*{cap}%
(\overline{\mathbb{C}}\setminus\widetilde{D}_{0})=\operatorname*{cap}%
(\widetilde{K}_{0})\leq c_{0}\text{ }\label{f92i2}%
\end{equation}
with the constant $c_{0}$\ introduced in (\ref{f92b}). From the minimality
(\ref{f92b}) and the fact that $D_{0}\in\mathcal{D}(f,\infty)$, we conclude
that in (\ref{f92i2}) a proper inequality is not possible. Hence, we have
proved that
\begin{equation}
\operatorname*{cap}(\overline{\mathbb{C}}\setminus D_{0})=\inf_{D\in
\mathcal{D}(f,\infty)}\operatorname*{cap}(\overline{\mathbb{C}}\setminus
D),\label{f92i3}%
\end{equation}
which implies that $D_{0}\in\mathcal{D}_{0}(f,\infty)$, and the proof of
Proposition \ref{p92a} completed. $\blacksquare$\medskip

\subsection{\label{s93}Uniqueness up to a Set of Capacity Zero}

\qquad In the present subsection we prove that all admissible domains in
$\mathcal{D}_{0}(f,\infty)$ differ only in a set of capacity zero. In Section
\ref{s2}, this result has already been stated as the first part of Proposition
\ref{p22a}, and there the sets $\mathcal{D}_{0}(f,\infty)$ and $\mathcal{K}%
_{0}(f,\infty)$\ have also been introduced in Definition \ref{d21b}. We
formulate the result here as a proposition, which then will be proved at the
end of the subsection after several auxiliary results have been formulated and
proved.\smallskip

\begin{proposition}
\label{p93a}Sets in $\mathcal{K}_{0}(f,\infty)$ differ at most in a subset of
capacity zero.\smallskip
\end{proposition}

A key role in the proof of Proposition \ref{p93a} is played by a number of
special sets that are introduced in Definition \ref{d93a} below. Especially
the construction of the compact set $K_{%
\acute{}%
0}$ in (\ref{f93b6}) can be seen as a type of convex combination, which will
become more clear in Subsection \ref{s95}, further below.\smallskip

We start with the formal set-up for an indirect proof of Proposition
\ref{p93a} and assume contrary to the assertion of the proposition that there
exist at least two minimal compact sets $K_{1},K_{2}\in\mathcal{K}%
_{0}(f,\infty)$ that differ in a set of positive capacity, i.e., we assume
\begin{equation}
\operatorname*{cap}\left(  \left(  K_{1}\setminus K_{2}\right)  \cup\left(
K_{2}\setminus K_{1}\right)  \right)  >0\text{ \ \ for \ \ }K_{1},K_{2}%
\in\mathcal{K}_{0}(f,\infty)\text{.}\label{f93a1}%
\end{equation}
The corresponding admissible domains are defined as
\begin{equation}
D_{j}:=\overline{\mathbb{C}}\setminus K_{j}\in\mathcal{D}_{0}(f,\infty),\text{
\ \ }j=1,2.\label{f93a2}%
\end{equation}
Since we have assumed $K_{1},K_{2}\in\mathcal{K}_{0}(f,\infty),$ we know that
the minimality (\ref{f21a}) in Definition \ref{d21b} holds for both sets,
i.e., we have
\begin{equation}
\operatorname*{cap}(K_{1})=\operatorname*{cap}(K_{2})=\inf_{K\in
\mathcal{K}(f,\infty)}\operatorname*{cap}(K)=c_{0}\label{f93a3}%
\end{equation}
with $c_{0}$ the same constant as that introduced in (\ref{f92b}). The two
Green functions $g_{D_{j}}(\cdot,\infty)$ in the two domains $D_{j}%
,$\ $j=1,2,$ are denoted by
\begin{equation}
g_{j}:=g_{D_{j}}(\cdot,\infty),\text{ \ \ }j=1,2.\label{f93a4}%
\end{equation}
From Lemma \ref{l113c} in Subsection \ref{s1103}, further below, we know that
assumption (\ref{f93a1}) is equivalent to the assertion that the two Green
functions $g_{1}$ and $g_{2}$ are not identical. From the harmonicity of
$g_{1}-g_{2}$ in $D_{1}\cap D_{2}$, it then follows that
\begin{equation}
g_{1}(z)\neq g_{2}(z)\text{ \ \ for almost all \ }z\in D_{1}\cap
D_{2},\label{f93a5}%
\end{equation}
and equality holds in $D_{1}\cap D_{2}$ on piece-wise analytic arcs. These
arcs are part of the set $S_{0}$ that is formally defined in (\ref{f93b1}) in
the next definition. All sets introduced in Definition \ref{d93a} will appear
in subsequent lemmas that lead to the proof of Proposition \ref{p93a}\ at the
end of the present subsection.

\begin{definition}
\label{d93a}Under the assumptions (\ref{f93a1}) and (\ref{f93a3}) and with the
same notations as introduced in (\ref{f93a2}) and (\ref{f93a4}), we define the
sets $S_{0},K_{3},\widetilde{K}_{0},K_{10},K_{20},\allowbreak K_{0}%
,\allowbreak D_{0}\subset\overline{\mathbb{C}}$ in the following way:
\begin{align}
& S_{0}:=\overline{\left\{  \,z\in\overline{\mathbb{C}}\,\right.  \left\vert
\,g_{1}(z)=g_{2}(z)\,\right\}  },\medskip\label{f93b1}\\
& K_{3}:=\widehat{K_{1}\cup K_{2}},\medskip\label{f93b2}\\
& \widetilde{K}_{0}:=\widehat{S_{0}\cap K_{3}},\medskip\label{f93b3}\\
& K_{10}:=\left\{  \,z\in K_{1}\setminus\widetilde{K}_{0}\,\right.  \left\vert
\,g_{1}(z)>g_{2}(z)\,\right\}  ,\medskip\label{f93b4}\\
& K_{20}:=\left\{  \,z\in K_{2}\setminus\widetilde{K}_{0}\,\right.  \left\vert
\,g_{2}(z)>g_{1}(z)\,\right\}  ,\medskip\label{f93b5}\\
& K_{0}:=\widetilde{K}_{0}\cup K_{10}\cup K_{20},\medskip\label{f93b6}\\
& D_{0}:=\overline{\mathbb{C}}\setminus K_{0}.\label{f93b7}%
\end{align}
The polynomial-convex hull of a bounded set $S\subset\mathbb{C}$ is
denoted\ by\ $\widehat{S}$ (cf. Definition \ref{d111b} in Subsection
\ref{s1101}, further below).
\end{definition}

For the proof of Proposition \ref{p93a} the following strategy will be
applied: First, it is proved that the set $D_{0}$ introduced in (\ref{f93b7})
is a domain. Then it is shown that $D_{0}$ is an admissible domain, i.e.,
$D_{0}\in\mathcal{D}(f,\infty)$. After that in the final step, it is proved
that the assumptions (\ref{f93a1}) and (\ref{f93a3}) imply that
$\operatorname*{cap}(K_{0})<c_{0}$ for the compact set introduced in
(\ref{f93b6}). But such an estimate contradicts the minimality assumed in
(\ref{f93a3}). From a methodological point of view the last step is the most
interesting and also the most challenging one.\smallskip

We start with two lemmas in which topological and some potential-theoretic
properties of sets from Definition \ref{d93a} are investigated. The first
lemma is of a more preparatory character.\smallskip

\begin{lemma}
\label{l93a}We set
\begin{equation}
d:=g_{1}-g_{2},\label{f93c1}%
\end{equation}
and define the two sets
\begin{align}
& B_{+}:=\{\,z\in\overline{\mathbb{C}}\setminus S_{0}\,|\,g_{1}(z)>g_{2}%
(z)\,\}=\{\,z\in\overline{\mathbb{C}}\setminus S_{0}\,|\,d(z)>0\,\},\medskip
\label{f93c2}\\
& B_{-}:=\{\,z\in\overline{\mathbb{C}}\setminus S_{0}\,|\,g_{1}(z)<g_{2}%
(z)\,\}=\{\,z\in\overline{\mathbb{C}}\setminus S_{0}%
\,|\,d(z)<0\,\}.\label{f93c3}%
\end{align}
Both sets are open. The function $d$ is superharmonic in $B_{+}$ and
subharmonic in $B_{-}$.
\end{lemma}

\begin{proof}
Let $C\subset\overline{\mathbb{C}}$ be an arbitrary component of
$\overline{\mathbb{C}}\setminus S_{0}$. This component is broken down into the
two sets
\begin{equation}
C_{1}:=C\cap B_{+}\text{ \ and \ \ }C_{2}:=C\cap B_{-}\text{.}\label{f93c4}%
\end{equation}
Since the function $d$ is the difference of two Green functions, we know from
Lemma \ref{l113a1} in Subsection \ref{s1103}, further below, that $d$ is
continuous outside of a measurable set $A\subset\overline{\mathbb{C}}$ with
$\operatorname*{cap}(A)=0$. The definitions (\ref{f93c4}) together with the
continuity of $d$ then imply that
\begin{equation}
\partial C_{j}\cap C\subset A\text{ \ \ for \ }j=1,2\text{.}\label{f93c5}%
\end{equation}
Indeed, if we assume that $z\in\partial C_{1}\cap C$ and $z\notin A$, then it
follows from the continuity of $d$ in $C\setminus A$\ that $d(z)=0$, and
therefore $z\in S_{0}$, which contradicts the definition of the component $C$.
For $j=2$\ the same conclusion holds true.

Next, we show that we can have
\begin{equation}
C_{j}\setminus A\neq\emptyset\label{f93c6}%
\end{equation}
at most for one of the two possibilities $j=1,2$. Indeed, it follows from
$\operatorname*{cap}(A)=0$ and from Lemma \ref{l111e} in Subsection
\ref{s1101}, further below, that $C\setminus A$ is connected. If we assume
that (\ref{f93c6}) holds for both $j=1,2$, then it follows from the continuity
of $d$ in $C\setminus A$ that there exists $z\in C\setminus A$ with $d(z)=0$.
But this would imply that $z\in S_{0}$, which again contradicts the definition
of the component $C$. We assume without loss of generality that
\begin{equation}
C_{2}\subset A\text{ \ \ and \ }C_{1}\supset C\setminus A.\label{f93c7}%
\end{equation}
Let, as in Definition \ref{d113a2} of Subsection \ref{s1103}, further below,
$Rg(K_{1})$ denote the set $K_{1}$ minus all irregular points of $K_{1}$. From
(\ref{f93c7}) it follows that
\begin{equation}
C\cap\overline{Rg(K_{1})}=\emptyset.\label{f93c8}%
\end{equation}
Indeed, if there exists $z\in C\cap\overline{Rg(K_{1})}$, then we know from
part (iv) of Lemma \ref{l113a2} in Subsection \ref{s1103}, further below, that
$\operatorname*{cap}(C\cap K_{1})>0$, and further with Lemma \ref{l113a1} in
Subsection \ref{s1103} we conclude that also $\operatorname*{cap}(C\cap
Rg(K_{1}))>0$.

Since $g_{1}(z)=0$ for all $z\in Rg(K_{1})$, it follows that $d(z)\leq0$ for
$z\in Rg(K_{1})$ and therefore that $d(z)<0$ for $z\in C\cap Rg(K_{1})$. With
(\ref{f93c7}) this implies that $C\cap Rg(K_{1})\subset A$, and consequently,
we have $\operatorname*{cap}(C\cap Rg(K_{1}))=0$. Since the last conclusion
has led to a contradiction, (\ref{f93c8}) is proved.

From (\ref{f93c8}) and part (iii) of Lemma \ref{l113a2} in Subsection
\ref{s1103}, further below, we conclude that the function $d$ is superharmonic
in the component $C$. Indeed, from the Lemma \ref{l113a2} we know that $g_{1}
$ is harmonic in $C$, and on the other hand, $-g_{2}$ is superharmonic in $C$
(cf. Lemma \ref{l113e} in Subsection \ref{s1103}).

Since the minimum principle is valid for superharmonic functions, we conclude
from (\ref{f93c7}) and (\ref{f93c4}) that $d(z)>0$ for all $z\in C$, and
consequently, we have proved
\begin{equation}
C\subset B_{+}.\label{f93c9}%
\end{equation}

The conclusion (\ref{f93c9}) is conditional on the assumption made in
(\ref{f93c5}). The alternative choice in (\ref{f93c5}) would have led to a
reversed role for the two subsets $C_{+}$ and $C_{-}$, and we would have
proved that $C\subset B_{-}$, and further that the function $d$ is subharmonic
in $C$.

Putting both possibilities together, we have proved that each component of the
open set $\overline{\mathbb{C}}\setminus S_{0}$ is completely contained in one
of the two subsets $B_{+}$ or $B_{-}$. Hence, we have shown that these two
sets are open. Further, it has been shown that the function $d$ is
superharmonic (resp. subharmonic) in $B_{+}$ (resp. in $B_{-}$).\medskip
\end{proof}

\begin{lemma}
\label{l93b}(i) The two sets
\begin{equation}
B_{1}:=(K_{3}\setminus\widetilde{K}_{0})\cap B_{+}\text{ \ \ and \ \ }%
B_{2}:=(K_{3}\setminus\widetilde{K}_{0})\cap B_{-}\label{f93d1}%
\end{equation}
are disjoint, and they are open in $K_{3}$.

(ii) \ We have
\begin{align}
& B_{j}\setminus K_{j}=B_{j}\setminus K_{j0}\text{ \ \ for \ \ }j=1,2\text{,
\ and}\medskip\label{f93d2}\\
& \operatorname*{cap}(K_{10})=\operatorname*{cap}(K_{20})=0\text{.}%
\label{f93d3}%
\end{align}

(iii) Set $D_{3}:=\overline{\mathbb{C}}\setminus K_{3}$, then both sets
\begin{equation}
D_{3}\cup(B_{j}\setminus K_{j})\text{, \ \ }j=1,2\text{,}\label{f93d4}%
\end{equation}
are domains.

(vi) The set $D_{0}$\ from (\ref{f93b7}) is a domain, and we the
decomposition
\begin{equation}
D_{0}=D_{3}\cup(B_{1}\setminus K_{1})\cup(B_{2}\setminus K_{2}).\label{f93d5}%
\end{equation}

\end{lemma}

\begin{proof}
Assertion (i) is an immediate consequence of Lemma \ref{l93a}.

In order to prove assertion (ii), we observe that it follows from the
definition of $K_{10}$ in (\ref{f93b4}) together with (\ref{f93c2}) and
(\ref{f93d1}) that $K_{10}=B_{1}\cap K_{1}$, which implies (\ref{f93d2}) for
$j=1$.

From the introduction of $K_{10}$ in (\ref{f93b4}) and the definition of
irregular points in the Definitions \ref{d113a} and \ref{d113a2} in Subsection
\ref{s1103}, further below, it further follows that $K_{10}\subset Ir(K_{1})$,
and the first part of (\ref{f93d3}) therefore is a consequence of Lemma
\ref{l113a1} in Subsection \ref{s1103}. Assertion (ii) follows analogously for
$j=2$.

Since $\widetilde{K}_{0}$ is polynomial-convex by definition, each component
of $K_{3}\setminus\widetilde{K}_{0}$ is a subset of a component of
$\overline{\mathbb{C}}\setminus S_{0}$, which implies that $D_{3}\cup B_{j}$
is a domain for each $j=1,2$. Assertion (iii) then follows immediately from
the second part of Lemma \ref{l111e} in Subsection \ref{s1101}, further below,
together with (\ref{f93d2}) and (\ref{f93d3}).

For a proof of assertion (iv), we observe that
\begin{align}
D_{0}  & =\overline{\mathbb{C}}\setminus(\widetilde{K}_{0}\cup K_{10}\cup
K_{20})=(\overline{\mathbb{C}}\setminus K_{3})\cup(K_{3}\setminus\widetilde
{K}_{0})\setminus(K_{10}\cup K_{20})\nonumber\\
& =D_{3}\cup(B_{1}\setminus K_{1})\cup(B_{2}\setminus K_{2}).\label{f93d6}%
\end{align}
Indeed, the identities follow from the defining relations of the sets in
Definition \ref{d93a} together with (\ref{f93d2}). From (\ref{f93d6}) and
assertion (iii) of the present lemma, it follows that $D_{0}$ is a
domain.\medskip
\end{proof}

The main result in Lemma \ref{l93b} is the assertion that the set $D_{0}$ is a domain.

\begin{lemma}
\label{l93c}The domain $D_{0}$ introduced in (\ref{f93b7}) of Definition
\ref{d93a} is admissible for Problem $(f,\infty)$, i.e., we have $D_{0}%
\in\mathcal{D}(f,\infty)$ with $\mathcal{D}(f,\infty)$ introduced in
Definition \ref{d21a}.
\end{lemma}

\begin{proof}
The basis of the proof is Proposition \ref{p91a}.

In a first step we prove that assertion (i) of Proposition \ref{p91a} holds
true. Let $f_{j}$, $j=1,2$, be the two single-valued meromorphic continuations
of the function $f$ in Problem $(f,\infty)$ in the two admissible domains
$D_{j}$, $j=1,2$, in (\ref{f93a2}). In $D_{3}$ both functions are identical;
but beyond this domain, the situation is different since we have assumed that
the two domains $D_{1}$ and $D_{2}$ are not identical; the set $K_{3}$ has
been defined in (\ref{f93b2}). Using the two sets from (\ref{f93d1}), we
defined a new function $f_{0}$ as
\begin{equation}
f_{0}(z):=\left\{
\begin{array}
[c]{lcc}%
f_{1}(z)=f_{2}(z) & \text{ \ for \ }\smallskip & z\in D_{3},\\
f_{1}(z) & \text{ \ for \ }\smallskip & z\in B_{1}\setminus K_{1},\\
f_{2}(z) & \text{ \ for \ } & z\in B_{2}\setminus K_{2}.
\end{array}
\right. \label{f93e1}%
\end{equation}
It follows from the assertions (iii) and (iv) of Lemma \ref{l93b} that $f_{0}
$ is a meromorphic continuation of the function $f$ into $D_{0}$, which shows
that assertion (i) of Proposition \ref{p91a} holds true.

Assertion (ii) in Proposition \ref{p91a} is a direct consequence of assertions
(i) in Lemma \ref{l93b}. Hence, it follows from Proposition \ref{p91a} that
$D_{0}\in\mathcal{D}(f,\infty)$.\medskip
\end{proof}

We come now to the main part of the analysis in the present subsection, the
proof of the inequality $\operatorname*{cap}(K_{0})<c_{0}$ for the compact set
$K_{0}=\overline{\mathbb{C}}\setminus D_{0}$ introduced in (\ref{f93b6}) in
Definition \ref{d93a}. In this proof two auxiliary functions $h_{0}$ and
$h_{1}$ will play an important role; they are introduced in the next
definition.\smallskip

\begin{definition}
\label{d93b}With the sets introduced in Definition \ref{d93a} and the notation
$g_{1}$ and $g_{2}$ for the Green functions (\ref{f93a4}), we define two
functions $h_{0}$ and $h_{1}$ by
\begin{equation}
h_{0}(z):=\left\{
\begin{array}
[c]{lcl}%
\frac{1}{2}(g_{1}(z)+g_{2}(z))\smallskip & \text{ \ for \ }\smallskip &
z\in\overline{\mathbb{C}}\setminus K_{3},\\
\frac{1}{2}\left|  g_{1}(z)-g_{2}(z)\right|  \smallskip & \text{ \ for
\ }\smallskip & z\in K_{3}\setminus\operatorname*{Int}(\widetilde{K}_{0}),\\
0 & \text{ \ for \ } & z\in\operatorname*{Int}(\widetilde{K}_{0}),
\end{array}
\right. \label{f93f1a}%
\end{equation}
\begin{equation}
h_{1}(z):=\left\{
\begin{array}
[c]{lcl}%
\frac{1}{2}\left|  g_{1}(z)-g_{2}(z)\right|  \smallskip & \text{ \ for
\ }\smallskip & z\in\overline{\mathbb{C}}\setminus K_{3},\\
\frac{1}{2}(g_{1}(z)+g_{2}(z))\smallskip & \text{ \ for \ }\smallskip & z\in
K_{3}\setminus\operatorname*{Int}(\widetilde{K}_{0}),\\
\frac{1}{2}(\widehat{g_{1}(z)+g_{2}(z)}) & \text{ \ for \ } & z\in
\operatorname*{Int}(\widetilde{K}_{0}).
\end{array}
\right. \label{f93f1b}%
\end{equation}
In (\ref{f93f1b}), $\widehat{g_{1}+g_{2}}$ is the solution of the Dirichlet
problem in each component $C$ of the interior $\operatorname*{Int}%
(\widetilde{K}_{0})$ of $\widetilde{K}_{0}$ with $(g_{1}+g_{2})|_{\partial
\widetilde{K}_{0}}$ as boundary function.$\smallskip$
\end{definition}

In the next lemma a number of technical details are proved; they will be
needed in the subsequent lemma.$\smallskip$

\begin{lemma}
\label{l93d}If the assumptions (\ref{f93a1}) and (\ref{f93a3}) are satisfied,
then with the notations from the Definitions \ref{d93a}, \ref{d93b}, and Lemma
\ref{l93b}, the following assertions hold true:

(i) \ \ There exists a signed measure $\sigma_{0}$ of finite energy with
\begin{equation}
\operatorname*{supp}(\sigma_{0})\subset(K_{1}\cup K_{2})\setminus
\operatorname*{Int}(\widetilde{K}_{0})\label{f93f2a}%
\end{equation}
such that the function $h_{0}$ from (\ref{f93f1b}) has the representation
\begin{equation}
h_{0}=g_{0}(\cdot,\infty)+\int g_{0}(\cdot,v)d\sigma_{0}(v),\label{f93f2b}%
\end{equation}
where $g_{0}(\cdot,\cdot)$ is the Green function in the domain $D_{0}$. The
measure $\sigma_{0}$ is carried by the set
\begin{equation}
\Sigma_{0}:=(K_{1}\cup K_{2})\setminus\widetilde{K}_{0},\label{f93f2d}%
\end{equation}
and we have
\begin{equation}
\sigma_{0}\neq0.\label{f93f2c}%
\end{equation}

(ii) \ There exists a signed measure $\sigma_{1}$ of finite energy with
\begin{equation}
\operatorname*{supp}(\sigma_{1})\subset S_{0}\cup K_{3}\label{f93f3a}%
\end{equation}
such that the function $h_{1}$ from (\ref{f93f1b})\ has the representation
\begin{equation}
h_{1}=p(\sigma_{1};\cdot),\label{f93f3b}%
\end{equation}
where $p(\sigma_{1};\cdot)$ denotes the logarithmic potential of the measure
$\sigma_{1}$ as introduced in (\ref{f112a1}) in Subsection \ref{s1102},
further below. We have
\begin{equation}
\sigma_{1}(\mathbb{C})=0.\label{f93f3c}%
\end{equation}

(iii) With the notation (\ref{f93f2d}), we have
\begin{align}
& h_{0}(z)=h_{1}(z)\text{ \ \ for quasi every}\smallskip\text{ \ }z\in
\Sigma_{0},\label{f93f4a}\\
& h_{1}(z)=0\text{ \ \ \ \ \ \ \ for quasi every}\smallskip\text{ \ }z\in
S_{0}\setminus\operatorname*{Int}(K_{3}),\label{f93f4b}\\
& \sigma_{0}=-\sigma_{1}|_{K_{3}\setminus\widetilde{K}_{0}},\smallskip
\label{f93f4c}\\
& \sigma_{1}|_{\operatorname*{Int}(K_{3})}\leq0.\label{f93f4d}%
\end{align}

\end{lemma}

\begin{proof}
In the first part of the proof we show that the two functions $h_{0}$ and
$h_{1}$ introduced in (\ref{f93f1a}) and (\ref{f93f1b}) can be represented by
potentials with measures of finite energy. This is done by an investigation of
a sequence of auxiliary functions.

By $h_{2}$ we denote the function
\begin{equation}
h_{2}:=\frac{1}{2}(g_{1}+g_{2})=r_{2}+p(\sigma_{2};\cdot),\label{f93f5a}%
\end{equation}
where the last equality is a consequence of (\ref{f93a4}) together with
representation (\ref{f113b1}) for Green functions in Lemma \ref{l113b} in
Subsection \ref{s1103}, further below. It follows from (\ref{f93a3}) and
(\ref{f93a4}) together with Lemma \ref{l113b} that we have
\begin{equation}
r_{2}=-\log(c_{0})\text{ \ \ and \ \ }\sigma_{2}=-\frac{1}{2}(\omega
_{1}+\omega_{2})\leq0\label{f93f5b}%
\end{equation}
with $\omega_{j}$ the equilibrium distribution on $K_{j}$, $j=1,2$.

From Lemma \ref{l112d} together with Lemma \ref{l113b} in the two Subsections
\ref{s1102} and \ref{s1103}, further below, we know that the function
$\frac{1}{2}\left\vert g_{1}-g_{2}\right\vert $ can also be represented as a
potential. We denote this function by $h_{3}$; and we have
\begin{equation}
h_{3}:=\frac{1}{2}\left\vert g_{1}-g_{2}\right\vert =p(\sigma_{3}%
;\cdot)\label{f93f6a}%
\end{equation}
with $\sigma_{3}$ a signed measure of finite energy and
\begin{equation}
\operatorname*{supp}(\sigma_{3})\subset S_{0}\cup K_{1}\cup K_{2}%
.\label{f93f6b}%
\end{equation}
In (\ref{f93f6a}), there is no constant\ because of (\ref{f93a3}).

From Lemma \ref{l113a1} in Subsection \ref{s1103}, further below, we know that
the two Green functions $g_{1}$ and $g_{2}$, and consequently also the two
functions $h_{2}$ and $h_{3}$ are continuous quasi everywhere in $\mathbb{C}$.
Hence, it follows from (\ref{f93b1}) in Definition \ref{d93a}\ and
(\ref{f93f6a}) that
\begin{equation}
h_{3}(z)=0\text{ \ \ for quasi every \ \ }z\in S_{0}.\label{f93f6c}%
\end{equation}
Because of (\ref{f93b2}) in Definition \ref{d93a},\ the two Green functions
$g_{1}$ and $g_{2}$ are harmonic outside of $K_{3}$, and therefore we have
equality for every $z\in S_{0}\setminus K_{3}$ in (\ref{f93f6c}) without any exception.

Next, we use the balayage technique (cf. Definition \ref{d112a} in Section
\ref{s1102}, further below) for sweeping the masses of the two measures
$\sigma_{2}$ and $\sigma_{3}$ out of the open set $\operatorname*{Int}%
(\widetilde{K}_{0})$. The two resulting balayage measures are denoted by
$\sigma_{4}$ and $\sigma_{5}$, respectively. From part (i) of Definition
\ref{d112a} of the balayage applied to the measure $\sigma_{2}$, we get as
consequence that the new function $h_{4}$ is of the form
\begin{equation}
h_{4}:=r_{2}+p(\sigma_{4};\cdot)=\left\{
\begin{array}
[c]{lcl}%
\frac{1}{2}(\widehat{g_{1}+g_{2})} & \text{\ in }\smallskip &
\operatorname*{Int}(\widetilde{K}_{0}),\\
h_{2} & \text{\ quasi everywhere on }\smallskip & \partial\widetilde{K}_{0},\\
h_{2} & \ \text{in } & \overline{\mathbb{C}}\setminus\widetilde{K}_{0}.
\end{array}
\right. \label{f93f7a}%
\end{equation}
In (\ref{f93f7a}) we have used (\ref{f93f5a}). About the measure $\sigma_{4} $
we know that
\begin{align}
& \operatorname*{supp}(\sigma_{4})\subset\operatorname*{supp}(\sigma
_{2})\setminus\operatorname*{Int}(\widetilde{K}_{0})\text{,}\smallskip
\label{f93f7e}\\
& \sigma_{4}|_{\overline{\mathbb{C}}\setminus\widetilde{K}_{0}}=\sigma
_{2}|_{\overline{\mathbb{C}}\setminus\widetilde{K}_{0}}=-\frac{1}{2}%
(\omega_{1}+\omega_{2})|_{\overline{\mathbb{C}}\setminus\widetilde{K}_{0}%
}.\label{f93f7b}%
\end{align}
On the other hand,\ from the balayage of the measure $\sigma_{3}$, it follows
that the new function $h_{5}$ is of the form
\begin{equation}
h_{5}:=p(\sigma_{5};\cdot)=\left\{
\begin{array}
[c]{lcl}%
0\text{ \ \ }\smallskip & \text{\ in }\smallskip & \operatorname*{Int}%
(\widetilde{K}_{0}),\\
0\text{ \ \ }\smallskip & \text{\ quasi everywhere on }\smallskip &
\partial\widetilde{K}_{0},\\
h_{3} & \ \text{in } & \overline{\mathbb{C}}\setminus\widetilde{K}_{0}.
\end{array}
\right. \label{f93f7c}%
\end{equation}
For the derivation\ of (\ref{f93f7c}) identity (\ref{f93f6a}) has been used.
The balayage measure $\sigma_{5}$ satisfies
\begin{align}
& \operatorname*{supp}(\sigma_{5})\subset\operatorname*{supp}(\sigma
_{3})\setminus\operatorname*{Int}(\widetilde{K}_{0}),\smallskip\label{f93f7f}%
\\
& \sigma_{5}|_{\overline{\mathbb{C}}\setminus\widetilde{K}_{0}}=\sigma
_{3}|_{\overline{\mathbb{C}}\setminus\widetilde{K}_{0}}\text{.}\label{f93f7d}%
\end{align}

The function $\frac{1}{2}(\widehat{g_{1}+g_{2})}$ in the first line of
(\ref{f93f7a}) is the solution of the Dirichlet problem in each component of
$\operatorname*{Int}(\widetilde{K}_{0})$ with boundary function $h_{2}%
|_{\partial\widetilde{K}_{0}}$ (cf. (\ref{f112c2}) in Definition \ref{d112a}
in Section \ref{s1102}, further below). Analogously, in the first line of
(\ref{f93f7c}) the function $h_{5}=0$ is the solution of the Dirichlet problem
in each component of $\operatorname*{Int}(\widetilde{K}_{0})$ with boundary
function $h_{3}|_{\partial\widetilde{K}_{0}}$ since from (\ref{f93f6c}) we
know that $h_{3}(z)=0$ for quasi every $z\in\partial\widetilde{K}_{0}%
$.$\medskip$

The functions $h_{4}$ and $h_{5}$ and their associated measures $\sigma_{4}$
and $\sigma_{5}$\ are the building blocks for the presentations by potentials
for the two functions $h_{0}$ and $h_{1}$ introduced in Definition \ref{d93b}.
The proof of existence of such potentials is the main objectives in the
present analysis. In the next step this aim will be achieved by using a method
what pasting potentials together, which is described in Lemma \ref{l112e} of
Subsection \ref{s1102}, further below.$\smallskip$

Because of (\ref{f93b2}) in Definition \ref{d93a} we have
\begin{equation}
\partial K_{3}\subset K_{1}\cup K_{2}\subset K_{3},\label{f93f8a}%
\end{equation}
and consequently
\begin{equation}
h_{4}=h_{5}\text{ \ quasi everywhere on\ \ }\partial K_{3},\label{f93f8g}%
\end{equation}
which together with the representations (\ref{f93f7a}) and (\ref{f93f7c})
shows that the assumptions of Lemma \ref{l112e} in Subsection \ref{s1102} are
satisfied. The domain $D$ in Lemma \ref{l112e} is now $\operatorname*{Int}%
(K_{3})$.

From (\ref{f93f1a}), (\ref{f93f1b}), and the technique described in Lemma
\ref{l112e}, we deduce that there exists two signed measures $\widetilde
{\sigma}_{0}$ and $\sigma_{1}$ of finite energy such that
\begin{align}
& h_{0}(z)=r_{2}+p(\widetilde{\sigma}_{0};z)\text{ \ \ for quasi
every\smallskip\ \ }z\in\mathbb{C},\label{f93f8b}\\
& h_{1}(z)=p(\sigma_{1};z)\text{ \ \ \ \ \ \ \ for quasi every\smallskip
\ \ }z\in\mathbb{C}.\label{f93f8c}%
\end{align}
Thus, in (\ref{f93f8b}) and (\ref{f93f8c}) we have representations by
potentials for the piecewise defined functions $h_{0}$ and $h_{1}$, respectively.

Because of (\ref{f93f8a}) and the properties of the functions $h_{4}$ and
$h_{5}$ in the two sets $\operatorname*{Int}(K_{3})$ and $\overline
{\mathbb{C}}\setminus\operatorname*{Int}(K_{3})$, we have
\begin{align}
& \operatorname*{supp}(\widetilde{\sigma}_{0})\subset(K_{1}\cup K_{2}%
\cup\widetilde{K}_{0})\setminus\operatorname*{Int}(\widetilde{K}%
_{0}),\smallskip\label{f93f8d}\\
& \operatorname*{supp}(\sigma_{1})\subset(S_{0}\setminus\operatorname*{Int}%
(K_{3}))\cup(K_{1}\cup K_{2}\cup\widetilde{K}_{0})\setminus\operatorname*{Int}%
(\widetilde{K}_{0}).\label{f93f8e}%
\end{align}
From (\ref{f112h4}) in Lemma \ref{l112e} in Subsection \ref{s1102} together
with (\ref{f93f1b}), (\ref{f93f5a}), and (\ref{f93f7a}), it further follows
that
\begin{equation}
\sigma_{1}|_{\operatorname*{Int}(K_{3})}=\sigma_{4}|_{\operatorname*{Int}%
(K_{3})}\leq0.\label{f93f8f}%
\end{equation}
The inequality in (\ref{f93f8f}) is a consequence of (\ref{f93f7b}%
).$\smallskip$

In order to prove a relationship between the two measures $\widetilde{\sigma
}_{0}$ and $\sigma_{1}$, we observe that from (\ref{f93f1a}) and
(\ref{f93f1b}) in Definition \ref{d93b} it follows that
\begin{equation}
h_{0}(z)+h_{1}(z)=\max(g_{1}(z),g_{2}(z))\text{ \ \ for all \ \ }z\in
\overline{\mathbb{C}}\setminus\widetilde{K}_{0}.\label{f93f9a}%
\end{equation}
From (\ref{f93f9a}) and Lemma \ref{l93a} we deduce that $h_{0}+h_{1}$ is
harmonic in $\mathbb{C}\setminus S_{0}$. Indeed, on the set $B_{+}$ introduced
in Lemma \ref{l93a}, we have $h_{0}+h_{1}=g_{1}$. Since we know from Lemma
\ref{l93a} that the function $d=g_{1}-g_{2}$ is superharmonic in $B_{+}$, we
deduce that $g_{1}$ is harmonic in $B_{+}$. On the set $B_{-}$ in Lemma
\ref{l93a}, analogous considerations hold true.

From (\ref{f93f1a}) and (\ref{f93f1b}) in Definition \ref{d93b} together with
the two representations (\ref{f93f8b}), (\ref{f93f8c}), and the harmonicity of
$h_{0}+h_{1}$ in $\mathbb{C}\setminus S_{0}$, it follows that
\begin{equation}
\widetilde{\sigma}_{0}|_{\mathbb{C}\setminus S_{0}}=-\sigma_{1}|_{\mathbb{C}%
\setminus S_{0}},\label{f93f9b}%
\end{equation}
which is the relation between $\widetilde{\sigma}_{0}$ and $\sigma_{1}$ we
were looking for.$\smallskip$

From (\ref{f93b6}) in Definition \ref{d93a} and (\ref{f93d3}) in Lemma
\ref{l93b} we know that $\ K_{0}\ $\ and $\widetilde{K}_{0})$ differ only in a
set of capacity zero. Hence, from the defining property (\ref{f113a1}) for
Green functions, which has been stated at the beginning of Subsection
\ref{s1103}, we then conclude that
\begin{equation}
g_{0}(\cdot,v):=g_{\overline{\mathbb{C}}\setminus K_{0}}(\cdot,v)\equiv
g_{\overline{\mathbb{C}}\setminus\widetilde{K}_{0}}(\cdot,v)\text{ \ \ for all
\ }v\in D_{0}.\label{f93f10a}%
\end{equation}

In the next step, we investigate the relation between Green function
$g_{0}=g_{0}(\cdot,\infty)$ and the function $h_{0}$. From (\ref{f93f1a}) in
Definition \ref{d93b} together with (\ref{f93f7a}), (\ref{f93f6c}), and
(\ref{f93b3}), we conclude that
\begin{equation}
h_{0}(z)=0\text{ \ \ for quasi every \ \ }z\in\widetilde{K}_{0}%
.\label{f93f10b}%
\end{equation}

For the function $h_{0}$ we have representation (\ref{f93f8b}). We will now
show that if we sweep the measure $\widetilde{\sigma}_{0}$ out of the domain
$\overline{\mathbb{C}}\setminus\widetilde{K}_{0}$ by balayage, we arrive at
the Green function $g_{0}(\cdot,v)$. Let $\widehat{\sigma}_{0}$ be the
balayage measure on $\partial\widetilde{K}_{0}$ resulting from sweeping
$\widetilde{\sigma}_{0}$ out of $\overline{\mathbb{C}}\setminus\widetilde
{K}_{0}$, then it follows from Definition \ref{d112a}, part (ii), in
Subsection \ref{s1102} together with formula (\ref{f113e2}) in Lemma
\ref{l113e} in Subsection \ref{s1103} that
\begin{align}
& r_{2}+p(\widehat{\sigma}_{0};z)-\int_{\overline{\mathbb{C}}\setminus
\widetilde{K}_{0}}g_{0}(v,\infty)d\widetilde{\sigma}_{0}(v)\nonumber\\
& \text{ \ \ \ \ \ \ \ \ \ \ \ \ \ \ \ \ \ \ \ \ }=r_{2}+p(\widetilde{\sigma
}_{0};z)-\int_{\overline{\mathbb{C}}\setminus\widetilde{K}_{0}}g_{0}%
(z,v)d\widetilde{\sigma}_{0}(v)\label{f93f10c}\\
& \text{ \ \ \ \ \ \ \ \ \ \ \ \ \ \ \ \ \ \ \ \ }=h_{0}(z)-\int
_{\overline{\mathbb{C}}\setminus\widetilde{K}_{0}}g_{0}(z,v)d\widetilde
{\sigma}_{0}(v)=0\nonumber
\end{align}
for quasi every $z\in\widetilde{K}_{0}$. In (\ref{f93f10c}) the last equality
follows from (\ref{f93f10b} and the fact that $g_{0}(\cdot,v)=0$ quasi
everywhere on $\widetilde{K}_{0}$ for all $v\in\overline{\mathbb{C}}%
\setminus\widetilde{K}_{0}$.

From (\ref{f93f10c}) we deduce that
\begin{equation}
g_{0}(\cdot,\infty)=h_{0}-\int_{\overline{\mathbb{C}}\setminus\widetilde
{K}_{0}}g_{0}(\cdot,v)d\widetilde{\sigma}_{0}(v).\label{f93f10d}%
\end{equation}
Indeed, since $\operatorname*{supp}(\widehat{\sigma}_{0})\subset\widetilde
{K}_{0} $, the right-hand side of (\ref{f93f10d}) is harmonic in
$\mathbb{C}\setminus\widetilde{K}_{0}$, has an appropriate behavior at
infinity, and is equal to zero quasi everywhere on $\widetilde{K}_{0}$. Hence,
identity (\ref{f93f10d}) holds true since the right-hand side of
(\ref{f93f10d}) satisfies the defining property (\ref{f113a1}) in Subsection
\ref{s1103} for the Green function $g_{0}(\cdot,\infty)=g_{D_{0}}(\cdot
,\infty)$.$\smallskip$

After this somewhat lengthy preparations we are ready to verify the individual
statements of the lemma. We define
\begin{equation}
\sigma_{0}:=\widetilde{\sigma}_{0}|_{K_{3}\setminus\widetilde{K}_{0}%
}=\widetilde{\sigma}_{0}|_{\mathbb{C}\setminus S_{0}}=-\sigma_{1}%
|_{\mathbb{C}\setminus S_{0}}.\label{f93f10e}%
\end{equation}
The second equality in (\ref{f93f10e}) is a consequence of (\ref{f93f8d}), and
the last one is identical with (\ref{f93f9b}).$\smallskip$

Representation (\ref{f93f2b}) in the lemma follows directly from
(\ref{f93f10d}) and the introduction of the measure $\sigma_{0}$ in
(\ref{f93f10e}). That the set $\Sigma_{0}$ in (\ref{f93f2d}) is a carrier of
the measure $\sigma_{0}$ is a consequence of (\ref{f93f10e}) and
(\ref{f93f8d}).$\smallskip$

Assertion (ii) is proved by (\ref{f93f8c}) and (\ref{f93f8e}).$\smallskip$

Identity (\ref{f93f4a}) follows directly from (\ref{f93f1a}) and
(\ref{f93f1b}) in Definition \ref{d93b} and the fact that for the Green
functions we have $g_{j}=0$ quasi everywhere on $K_{j}$, $j=1,2$, (cf. the
defining property (\ref{f113a1}) in Subsection \ref{s1103}, further
below).$\smallskip$

Identity (\ref{f93f4b}) is a consequence of (\ref{f93f6c}) and the fact that
$h_{1}=h_{3}=\frac{1}{2}\left|  g_{1}-g_{2}\right|  $ quasi everywhere on
$\overline{\mathbb{C}}\setminus\operatorname*{Int}(K_{3})$.$\smallskip$

Identity (\ref{f93f4c}) follows from (\ref{f93f10e}) and (\ref{f93f8d}), and
at last identity (\ref{f93f4d}) is practically identical with (\ref{f93f8f}%
).$\smallskip$

Thus, only inequality (\ref{f93f2c}) remains to be verified, and this will be
done indirectly. Let us assume that $\sigma_{0}=0$. From (\ref{f93f10e}), we
then know that $\sigma_{1}|_{\mathbb{C}\setminus S_{0}}=0$. It then is a
consequence of (\ref{f93f8f}) that the potential $h_{1}=p(\sigma_{1};\cdot) $
is subharmonic in the domain $(\mathbb{C}\setminus S_{0})\cup
\operatorname*{Int}(K_{3})$. From (\ref{f93f1b}), (\ref{f93f6c}), and
(\ref{f93f7c}) we know that
\begin{equation}
p(\sigma_{1};z)=0\text{ \ \ \ for quasi every \ \ }z\in S_{0}\setminus
\operatorname*{Int}(K_{3}).\label{f93f11a}%
\end{equation}
Since $\sigma_{1}$ is of finite energy, it follows from (\ref{f93f11a}) and
the subharmonicity of $h_{1}=p(\sigma_{1};\cdot)$ in $(\mathbb{C}\setminus
S_{0})\cup\operatorname*{Int}(K_{3})$ that $h_{1}(z)\leq0$ for all
$z\in\mathbb{C}$. But the function $h_{1}$ is non-negative by definition, and
so we have shown that $h_{1}\equiv0$. But the last identity contradicts the
assumption (\ref{f93a5}), and therefore assertion (\ref{f93f2c}) is
proved.\medskip
\end{proof}

With the proof of Lemma \ref{l93d} all preparations are done for beginning the
last step in the indirect proof of Proposition \ref{p93a} which is the next
lemma.\smallskip

\begin{lemma}
\label{l93e}Under the assumptions (\ref{f93a1}) and (\ref{f93a3}), we have
\begin{equation}
\operatorname*{cap}(K_{0})<c_{0}\label{f93g}%
\end{equation}
with $K_{0}$ the compact set introduced in (\ref{f93b6}) of Definition
\ref{d93a}, and $c_{0}$\ the constant introduced in (\ref{f93a3}).
\end{lemma}

\begin{proof}
From (\ref{f93f1a}) in Definition \ref{d93b} and the assumption made in
(\ref{f93a3}) we know that
\begin{equation}
h_{0}(z)=-\log(c_{0})+\text{O}(z^{-1})\text{ \ \ as \ \ \ }z\rightarrow
\infty.\label{f93g1a}%
\end{equation}
As in Lemma \ref{l93d}, we abbreviate the Green function $g_{D_{0}}%
(\cdot,\cdot)$ by $g_{0}(\cdot,\cdot)$, and the special case $g_{0}%
(\cdot,\infty)$ by $g_{0}$. From the representation of Green functions in
Lemma \ref{l113b} in Subsection \ref{s1103}, further below, we know that
\begin{equation}
g_{0}(z)=-\log(\operatorname*{cap}(K_{0}))+\text{O}(z^{-1})\text{ \ \ as
\ \ \ }z\rightarrow\infty.\label{f93g1b}%
\end{equation}
Hence, we have
\begin{align}
\log\frac{\operatorname*{cap}(K_{0})}{c_{0}}  & =\left.  \left(
h_{0}(z)-g_{0}(z)\right)  \right\vert _{z=\infty}\nonumber\\
& =\int g_{0}(v,\infty)d\sigma_{0}(v),\label{f93g1c}%
\end{align}
where the last equation follows from (\ref{f93f2b}) in Lemma \ref{l93d}.

If we knew that $\sigma_{0}$ were a purely negative measure, then we could get
the desired estimate (\ref{f93g}) very easily from (\ref{f93g1c}). However, we
cannot exclude that the measure $\sigma_{0}$ contains a positive part.
Therefore, we have to go a more complicated way for getting an estimation for
the integral
\begin{equation}
I_{0}:=\int g_{0}(v,\infty)d\sigma_{0}(v).\label{f93g1d}%
\end{equation}
The technical results in Lemma \ref{l93d} will provides the basis for the analysis.

Using in (\ref{f93g1d}) representation (\ref{f93f2b}) from Lemma \ref{l93d}
leads us to the expression
\begin{equation}
I_{0}=\int h_{0}d\sigma_{0}-\int\int g_{0}(v,w)d\sigma_{0}(v)d\sigma
_{0}(w)=:I_{1}-I_{2}.\label{f93g2a}%
\end{equation}

From the positive definiteness of the Green function as a kernel in an energy
formula, which has been stated in Lemma \ref{l113d} in Subsection \ref{s1103},
further below, it follows together with (\ref{f93f2c}) in Lemma \ref{l93d}
that
\begin{equation}
I_{2}>0.\label{f93g2b}%
\end{equation}

In (\ref{f93g2a}), there only remains the integral $I_{1}=\int h_{0}%
d\sigma_{0}$ to be estimated. This will be done after some transformations.
First, we make the following general remark: From Lemma \ref{l93d} we know
that the two measures $\sigma_{0}$ and $\sigma_{1}$ are both of finite energy.
Because of Lemma \ref{l112a} in Subsection \ref{s1102}, further below, we have
$\sigma_{0}(S)=\sigma_{1}(S)=0$ for every measurable set $S\subset\mathbb{C}$
of capacity zero. Consequently, integrals with respect to the measure
$\sigma_{0}$ or $\sigma_{1}$ are equal if their integrands coincide quasi
everywhere on a carrier of $\sigma_{0}$ or $\sigma_{1}$, respectively.

As in (\ref{f93f2d}) in Lemma \ref{l93d}, we denote by $\Sigma_{0}$ the set
$(K_{1}\cup K_{2})\setminus\widetilde{K}_{0}$. Since $\Sigma_{0}$ is a carrier
of $\sigma_{0}$, we have
\begin{align}
I_{1}  & =\int h_{0}d\sigma_{0}\,=\,-\int_{\Sigma_{0}}h_{0}d\sigma
_{1}\label{f93g3a}\\
& =-\int_{\Sigma_{0}}h_{1}d\sigma_{1}\label{f93g3b}\\
& =-\int h_{1}d\sigma_{1}+\int_{\widetilde{K}_{0}\cap\operatorname*{Int}%
(K_{3})}h_{1}d\sigma_{1}.\label{f93g3c}%
\end{align}
Indeed, the second equality in (\ref{f93g3a}) follows from (\ref{f93f2d}) and
(\ref{f93f4c}) in Lemma \ref{l93d}, the equality in (\ref{f93g3b}) is a
consequence of (\ref{f93f4a}) in Lemma \ref{l93d}, and the equality in
(\ref{f93g3c}) follows from (\ref{f93f4b}) in Lemma \ref{l93d} and the fact
that $\Sigma_{0}\cup(\widetilde{K}_{0}\cap\operatorname*{Int}(K_{3}))$ is a
carrier of $\sigma_{1}|_{\operatorname*{Int}(K_{3})}$.

From Lemma \ref{l112b}, part (ii), in Subsection \ref{s1102}, together with
(\ref{f93f4c}) and (\ref{f93f2c}) from Lemma \ref{l93d}, we conclude that
\begin{equation}
\int h_{1}d\sigma_{1}=\int\int\log\frac{1}{|v-w|}d\sigma_{1}(v)d\sigma
_{1}(w)>0,\label{f93g4a}%
\end{equation}
i.e., we have used the positive definiteness of the logarithmic kernel.

Since $h_{1}\geq0$ by definition, it follows from (\ref{f93f4d}) in Lemma
\ref{l93d} that
\begin{equation}
\int_{\widetilde{K}_{0}\cap\operatorname*{Int}(K_{3})}h_{1}d\sigma_{1}%
\leq0.\label{f93g4b}%
\end{equation}

Putting (\ref{f93g1c}), (\ref{f93g2a}), (\ref{f93g2b}), (\ref{f93g3c}),
(\ref{f93g4a}), and (\ref{f93g4b}) together, we conclude that
\begin{equation}
\log\frac{\operatorname*{cap}(K_{0})}{c_{0}}<0,\label{f93g5a}%
\end{equation}
which proves (\ref{f93g}).\medskip
\end{proof}

With the proof of Lemma \ref{l93e}, the preparations of the proof of
Proposition \ref{p93a} are completed. Despite of the complexity of some of the
preparatory lemmas, the basic structure of the approach is straight forward.
It starts with assumption (\ref{f93a2}), i.e., the assumption that there exist
two essentially different admissible domains $D_{1}$\ and $D_{2}$ with
complements $K_{1}$\ and $K_{2}$ of minimal capacity. Based on this
assumption, a new admissible domain $D_{0}$ with a complement $K_{0}$ has been
constructed in Definition \ref{d93a}, and it has then been shown in the last
lemma that $\operatorname*{cap}(K_{0})$\ is smaller than possible.\smallskip

\textbf{Proof of Proposition \ref{p93a}.} The indirect proof of the
proposition has been prepared by assumption (\ref{f93a1}). The introduction of
the two sets $K_{0}$ and $D_{0}=\overline{\mathbb{C}}\setminus K_{0}$ in
(\ref{f93b6}) and (\ref{f93b7}) of Definition \ref{d93a} provide the basis for
the falsification of assumption (\ref{f93a1}).

Indeed, in Lemma \ref{l93c}, it has been shown that for the domain $D_{0}$ is
admissible for Problem $(f,\infty)$, i.e., $D_{0}\in\mathcal{D}(f,\infty)$,
and in Lemma \ref{l93e}, it then is proved that the newly constructed set
$K_{0}$ satisfies the inequality $\operatorname*{cap}(K_{0}%
)<\operatorname*{cap}(K)$ for all $K\in\mathcal{K}_{0}(f,\infty)$, which
contradicts the minimality (\ref{f21a}) in Definition \ref{d21b}. Hence,
assumption (\ref{f93a1}) is falsified, and Proposition \ref{p93a} is proved.
$\blacksquare$\bigskip

The construction of the two sets $K_{0}$ and $D_{0}$ in Definition \ref{d93a}
can be seen as a special case of a general type of set-theoretical convex
combination, which will be elaborated further in Definition \ref{d95a} in
Subsection \ref{s95}, below.\medskip

\subsection{\label{s94}The Unique Existence of an Extremal Domain}

\qquad In the present subsection we prove Theorem \ref{t22a} and the two
Propositions \ref{p22a} and \ref{p22b}, which are all three concerned with the
unique existence of an extremal domain for Problem $(f,\infty)$. With the two
Propositions \ref{p92a} and \ref{p93a} in the last two Subsections
\ref{s92}\ and \ref{s93}, the main work for these proofs has already been
done, we have only to put the different pieces together. We start with a
technical lemma.\smallskip

\begin{lemma}
\label{l94a}The two sets
\begin{align}
K_{0}  & :=\bigcap_{K\in\mathcal{K}_{0}(f,\infty)}K\text{ \ \ \ and}%
\label{f94a1}\\
D_{0}  & :=\bigcup_{D\in\mathcal{D}_{0}(f,\infty)}D\ \label{f94a2}%
\end{align}
are well defined, and we have
\begin{equation}
K_{0}\in\mathcal{K}_{0}(f,\infty)\text{ \ and \ \ }D_{0}\in\mathcal{D}%
_{0}(f,\infty)\label{f94a3}%
\end{equation}
with the two sets $\mathcal{K}_{0}(f,\infty)$ and $\mathcal{D}_{0}(f,\infty
)$\ introduced in Definition \ref{d21b}.
\end{lemma}

\begin{proof}
From Proposition \ref{p92a} we know that $\mathcal{K}_{0}(f,\infty
)\neq\emptyset$, hence, the sets $K_{0}$ and $D_{0}$ in (\ref{f94a1}) and
(\ref{f94a2}) are well defined,\ and $D_{0}$ is a domain with $\infty\in
D_{0}$.

In order to prove (\ref{f94a3}), we have only to show that
\begin{equation}
D_{0}\in\mathcal{D}(f,\infty),\label{f94b1}%
\end{equation}
since if we know that $K_{0}\in\mathcal{K}(f,\infty)$, then the minimality
condition (\ref{f21a}) in Definition \ref{d21b} follows immediately from
(\ref{f94a1}) together with the monotonicity of the capacity (cf. Subsection
\ref{s1101}). Relation (\ref{f94b1}) will be proved with the help of
Proposition \ref{p91a} of Subsection \ref{s91}; for this purpose we have to
show that the two assertions (i) and (ii) in Proposition \ref{p91a} hold true
for the domain $D_{0}$ and the compact set $K_{0}$, respectively.

We start with assertion (i) in Proposition \ref{p91a}. For every $z\in D_{0}$
there exists an admissible domain $D_{1}\in\mathcal{D}_{0}(f,\infty
)\subset\mathcal{D}(f,\infty)$ with $z\in D_{1}$. Since assertion (i) holds
true for $D_{1},$\ it holds true also for the larger domain $D_{0}$.

Next, we prove that also assertion (ii) in Proposition \ref{p91a} holds true
for the set $K_{0}$. Let $\gamma_{0}$ be an arbitrary Jordan curve of
$\Gamma_{1}=\Gamma_{1}(f,\infty)$. From assertion (ii) in Proposition
\ref{p91a} we know that
\begin{equation}
\gamma_{0}\cap K\neq\emptyset\label{f94b2}%
\end{equation}
\ for all $K\in\mathcal{K}_{0}(f,\infty)$. In order to prove that
(\ref{f94b2}) holds true also for the set $K_{0}$, we show in a first step
that (\ref{f94b2}) holds true for the intersection $K_{12}:=K_{1}\cap K_{2}$
of any two sets $K_{1},K_{2}\in\mathcal{K}_{0}(f,\infty)$. Indeed, let us
assume that
\begin{equation}
\gamma_{0}\cap K_{12}=\emptyset.\label{f94b3}%
\end{equation}
Let further $R\subset\overline{\mathbb{C}}\setminus K_{12}$ be a ring domain
as introduced in Lemma \ref{l91b} with $\gamma_{0}\subset R$. From Proposition
\ref{p92a} we know that
\begin{equation}
\operatorname*{cap}(K_{1}\setminus K_{2})=0.\label{f94b4}%
\end{equation}
Hence, the set $K_{1}\setminus K_{2}$ cannot intersect the whole ring $R$, and
consequently there exists a Jordan curve $\gamma_{1}\in\Gamma$ with
\begin{equation}
\gamma_{1}\subset R\setminus(K_{1}\setminus K_{2})=R\setminus K_{1}%
\label{f94b5}%
\end{equation}
that separates the two components $A_{1}$ and $A_{2}$ of $\overline
{\mathbb{C}}\setminus R$. The equality in (\ref{f94b5}) is a consequence of
$K_{1}=K_{12}\cup K_{1}\setminus K_{2}$ and $R\cap K_{12}=\emptyset$.

From Lemma \ref{l91b}, part (ii), we know that $\gamma_{1}\thicksim\gamma_{0}%
$. Hence, from the assumption $\gamma_{0}\in\Gamma_{1}$ we conclude also
$\gamma_{1}\in\Gamma_{1}$. On the other hand, it follows from (\ref{f94b5})
that $\gamma_{1}\cap K_{1}=\emptyset$, which then contradicts assertion (ii)
in Proposition \ref{p91a}, and with this falsification of (\ref{f94b3}) we
have proved that (\ref{f94b2}) holds true for $K_{12}$.

With the same argumentation as that applied to the intersection $K_{12}$ of
two elements from $\mathcal{K}_{0}(f,\infty)$, one can prove that relation
(\ref{f94b2}) holds true also for an intersection of finitely many elements
from $\mathcal{K}_{0}(f,\infty)$, i.e., we can prove that
\begin{equation}
\gamma_{0}\cap(K_{1}\cap\ldots\cap K_{m})\neq\emptyset\label{f94b6}%
\end{equation}
for an arbitrary $m\in\mathbb{N}$ and arbitrarily chosen sets $K_{j}%
\in\mathcal{K}_{0}(f,\infty)$, $j=1,\ldots,m$.

Let us now assume that relation (\ref{f94b2}) does not hold true for the set
$K_{0}$ of (\ref{f94a1}), i.e., we assume
\begin{equation}
\gamma_{0}\cap\bigcap_{K\in\mathcal{K}_{0}(f,\infty)}K=\emptyset.\label{f94b7}%
\end{equation}
An infinite intersection of compact sets can be empty only if already a finite
intersection is empty. However, such a possibility has been excluded in
(\ref{f94b6}). Hence, we have proved that (\ref{f94b2}) holds true for the set
$K_{0}$, and as a consequence, we have shown that assertion (ii) in
Proposition \ref{p91a} holds true for the set $K_{0}$.

Having verified the two assertions (i) and (ii) in Proposition \ref{p91a} for
$D_{0}$ and $K_{0}$, it follows from the proposition that the domain $D_{0}$
is admissible, i.e., (\ref{f94b1}) is proved, and the proof of the lemma is
completed.\medskip
\end{proof}

We now come to the three proofs of the central results from Section
\ref{s2}.\smallskip

\textbf{Proof of Theorem \ref{t22a}.} As minimal set $K_{0}=K_{0}(f,\infty)$
and as extremal domain $D_{0}=D_{0}(f,\infty)$ we choose the two sets
introduced in (\ref{f94a1}) and (\ref{f94a2}). It follows immediately from
Lemma \ref{l94a} that $K_{0}$ satisfies the three conditions (i), (ii), and
(iii) in Definition \ref{d21b}. Hence, the existence side of Theorem
\ref{t22a} is established.

Uniqueness then follows immediately from (\ref{f94a1}) in Lemma \ref{l94a}.
$\blacksquare$\medskip

\textbf{Proof of Proposition \ref{p22a}.} The proof is a combination of
Proposition \ref{p93a} and Lemma \ref{l94a}. The first half-sentence in
Proposition \ref{p22a} has been proved in Proposition \ref{p93a}, and the
second one is identical with (\ref{f94a1}) in Lemma \ref{l94a}. $\blacksquare
$\medskip

\textbf{Proof of Proposition \ref{p22b}.} If the function $f$ has no branch
points, then we have $\Gamma_{0}=\Gamma$ and $\Gamma_{1}=\emptyset$ in
Definition \ref{d91b}. Hence, assertion (i) in Proposition \ref{p91a} is
trivially true, and therefore $D_{0}=D_{0}(f,\infty)$ is the largest domain to
which the function $f$ can be meromorphically extended. Such a domain can be
denoted as the Weierstrass domain for meromorphically continuation of $f$ if
it is well-defined.

The domain $D_{0}$ is identical with the Weierstrass domain $W_{f}$ for
analytic continuation of the function $f$ plus all polar singularities of $f$
that can be reached from within $W_{f}$, and which can be added without
destroying the property that the resulting set is a domain. This completed the
proof of Proposition \ref{p22b}. $\blacksquare$\medskip

\subsection{\label{s95}A Convexity Property}

\qquad With the proof of Proposition \ref{p93a} the main task of Subsection
\ref{s93} had been done. However, in the present subsection,\ we will add an
extension to Definition \ref{d93a}. It has already been mentioned after the
proof Proposition \ref{p93a} at the end of Subsection \ref{s93} that the
construction of the set $K_{0}$ in Definition \ref{d93a} can be seen as a
special case of a whole family of set-theoretic convex-combinations of the two
sets $K_{1}$ and $K_{2}$ in Definition \ref{d93a}. The extended construction
leads to a whole continuum of sets $K_{h}$ with $h\in\left[  0,1\right]  $,
and for the capacity of these sets $K_{h}$ we get an interesting inequality
that generalizes inequality (\ref{f93g}) in Lemma \ref{l93e}. These extended
results are certainly of independent interest, but they are also needed in
Subsection \ref{s101}, below. The main properties of the new sets $K_{h}$,
$h\in\left[  0,1\right]  $, are proved in Theorem \ref{t95a}.\smallskip

\begin{definition}
\label{d95a}For two arbitrarily chosen admissible domains $D_{0},D_{1}%
\in\mathcal{D}(f,\allowbreak\infty)$ with corresponding compact sets
$K_{j}=\overline{\mathbb{C}}\setminus D_{j}\in\mathcal{K}(f,\infty)$, $j=0,1$,
that satisfy
\begin{equation}
\operatorname*{cap}(K_{j})>0\text{,\ \ \ \ }j=0,1,\label{f95a}%
\end{equation}
we define a family of domains $D_{h}\subset\overline{\mathbb{C}}$, $0\leq
h\leq1$, (which will turn out to be admissible domains) together with a family
of corresponding compact sets $K_{h}=\overline{\mathbb{C}}\setminus D_{h} $,
$0\leq h\leq1$, in a way that generalizes Definition \ref{d93a}. For $0\leq
h\leq1$ we define:
\begin{align}
& S_{h}:=\overline{\left\{  \,z\in\overline{\mathbb{C}}\,\right.  \left|
\,(1-h)\,g_{0}(z)=h\,g_{1}(z)\,\right\}  },\medskip\label{f95b1}\\
& K_{3}:=\widehat{K_{0}\cup K_{1}},\medskip\label{f95b2}\\
& \widetilde{K}_{h}:=\widehat{S_{h}\cap K_{3}},\medskip\label{f95b3}\\
& K_{0,h}:=\left\{  \,z\in K_{0}\,\right.  \left|  \,(1-h)\,g_{0}%
(z)>h\,g_{1}(z)\,\right\}  ,\medskip\label{f95b4}\\
& K_{1,h}:=\left\{  \,z\in K_{1}\,\right.  \left|  \,(1-h)\,g_{0}%
(z)<h\,g_{1}(z)\,\right\}  ,\medskip\label{f95b5}\\
& K_{h}:=\widetilde{K}_{h}\cup K_{0,h}\cup K_{1,h},\medskip\label{f95b6}\\
& D_{h}:=\overline{\mathbb{C}}\setminus K_{h}\label{f95b7}%
\end{align}
with $g_{j}=g_{D_{j}}(\cdot,\infty)$ the Green function in the domain $D_{j}$,
$j=0,1$.
\end{definition}

It is immediate that Definition \ref{d95a} is a generalization of Definition
\ref{d93a}. The role of the two input sets $K_{1}$and $K_{2}$ in Definition
\ref{d93a} is now played by the two sets $K_{0}$ and $K_{1}$, respectively.
The two set $K_{0}$ and $D_{0}$ in (\ref{f93b6}) and (\ref{f93b7}) of
Definition \ref{d93a} now correspond to the two sets $K_{1/2}$ and $D_{1/2}$,
respectively, in the new terminology of Definition \ref{d95a}.

Another generalization in Definition \ref{d95a} concerns the assumptions made
with respect to the two compact input sets $K_{0}$ and $K_{1}$. While in
Definition \ref{d93a} the input sets have been assumed to be of minimal
capacity, this assumption has been dropped without replacement in the extended definition.

The change of notation with respect to the input sets $K_{1}$and $K_{2}$ in
Definition \ref{d93a} into the sets $K_{1}$and $K_{2}$ in Definition
\ref{d95a} was necessary, and has the advantage that in the family of the
newly defined sets $K_{h}$, $h\in\left[  0,1\right]  $, the two special sets
$K_{0}$ and $K_{1}$ coincide with the two input sets $K_{0}$ and $K_{1}$ in
Definition \ref{d95a}, which can easily be verified.

In the next theorem we prove that the newly defined domains $D_{h}$,
$h\in\left[  0,1\right]  $, are all admissible for Problem $(f,\infty)$, and
most importantly, we prove that the functional $\log\operatorname*{cap}%
(K_{h})$ depends on the index $h\in\left[  0,1\right]  $ in a strictly convex
manner.\smallskip

\begin{theorem}
\label{t95a}(i) \ Under the assumptions of Definition \ref{d95a} we have
\begin{equation}
D_{h}\in\mathcal{D}(f,\infty)\text{ \ \ and \ \ }K_{h}\in\mathcal{K}%
(f,\infty)\text{ \ \ for all \ \ }h\in\left[  0,1\right] \label{f95c1}%
\end{equation}
with $\mathcal{K}(f,\infty)$\ and $\mathcal{D}(f,\infty)$ the sets introduced
in Definition \ref{d21a}.

(ii) \ If in addition to the assumptions of Definition \ref{d95a} we assume
that
\begin{equation}
\operatorname*{cap}\left(  \left(  K_{1}\setminus K_{0}\right)  \cup\left(
K_{0}\setminus K_{1}\right)  \right)  >0,\label{f95c4}%
\end{equation}
then we have
\begin{equation}
\log\operatorname*{cap}(K_{h})<(1-h)\,\log\operatorname*{cap}(K_{0}%
)+h\,\log\operatorname*{cap}(K_{1})\text{ \ \ for \ \ }0<h<1.\label{f95c2}%
\end{equation}
(iii) Under the assumptions of Definition \ref{d95a} we have the following
continuity property: For any $h_{0}\in\left[  0,1\right]  $ and for any open
set $U\subset\overline{\mathbb{C}}$\ with $K_{h_{0}}\subset U$, there exists a
neighborhood $V_{0}\subset\mathbb{R}$\ of $h_{0}$ such that
\begin{equation}
K_{h}\subset U\text{ \ \ \ for all \ \ }h\in V_{0}\cap\left[  0,1\right]
.\label{f95c3}%
\end{equation}

\end{theorem}

\begin{remark}
Assertion (iii) in Theorem \ref{t95a} means that in the Hausdorff metric the
compact sets $K_{h}$ depend continuously on the parameter $h\in\left[
0,1\right]  $.
\end{remark}

\begin{proof}
Definition \ref{d93a} has been the backbone of the proof of the essential
uniqueness of minimal sets in Proposition \ref{p93a} in Subsection \ref{s93}.
The important Lemma \ref{l93e} in Subsection \ref{s93}\ can be seen as a
special case of the convexity relation (\ref{f95c2}). It turns out that the
proof of Theorem \ref{t95a} is based on almost the same argumentations and
techniques as those applied in the proof of Proposition \ref{p93a}, therefore
we will now very closely follow the different stages of argumentations used in
Subsection \ref{s93}. As a consequence, we can shorten the proof of Theorem
\ref{t95a} considerably.

As a general policy, we will reformulate the content of lemmas and definitions
from Subsection \ref{s93} in such a way that it satisfies the needs of our new
situation, but we will use shortcuts and will not repeat all details. Often it
is only necessary to replace the difference $g_{1}-g_{2}$ of the two Green
functions $g_{1}$ and $g_{2}$ from Definition \ref{d93a} by the convex
combination $(1-h)\,g_{0}+h\,g_{1}$ of the two Green functions $g_{0}$ and
$g_{1}$ from Definition \ref{d95a}. This change is evidently suggested by
(\ref{f95b1}) in Definition \ref{d95a}.

Thus, for instance, the difference $d:=g_{1}-g_{2}$ in (\ref{f93c1}) will now
be replaced by
\begin{equation}
d(z):=((1-h)\,g_{0}+h\,g)(z)\text{ \ \ for \ \ }h\in\left[  0,1\right]  \text{
\ and \ }z\in\overline{\mathbb{C}}\text{.}\label{f95d1}%
\end{equation}

By using the same replacement repeatedly, one can transform all elements of
Definition \ref{d93a} into those of Definition \ref{d95a}. In the same way the
auxiliary definitions in the two Lemmas \ref{l93a} and \ref{l93b} can be
adapted to the new situation, and like in Lemma \ref{l93c}, one can prove
that
\begin{equation}
K_{h}\in\mathcal{K}(f,\infty)\text{ \ \ and \ \ }D_{h}\in\mathcal{D}%
(f,\infty)\text{ \ \ for all \ \ }h\in\left[  0,1\right] \label{f95d2}%
\end{equation}
with $K_{h}$ and $D_{h}$, $h\in\left[  0,1\right]  $, the sets introduced in
(\ref{f95b6}) and (\ref{f95b7}), respectively. The last conclusion proves
assertion (i) of Theorem \ref{t95a}.$\smallskip$

Analogously to (\ref{f93f1a}) and (\ref{f93f1b}) in Definition \ref{d93b}, we
now introduce the two auxiliary functions $h_{0}$ and $h_{1}$ by defining
\begin{equation}
h_{0}(z):=\left\{
\begin{array}
[c]{lcl}%
((1-h)\,g_{0}+h\,g_{1})(z)\smallskip & \text{ \ for \ }\smallskip &
z\in\overline{\mathbb{C}}\setminus K_{3},\\
\left|  (1-h)\,g_{0}-h\,g_{1}\right|  (z)\smallskip & \text{ \ for
\ }\smallskip & z\in K_{3}\setminus\operatorname*{Int}(\widetilde{K}_{h}),\\
0 & \text{ \ for \ } & z\in\operatorname*{Int}(\widetilde{K}_{h}),
\end{array}
\right. \label{f95d3a}%
\end{equation}
\begin{equation}
h_{1}(z):=\left\{
\begin{array}
[c]{lcl}%
\left|  (1-h)\,g_{0}-h\,g_{1}\right|  (z)\smallskip & \text{ \ for
\ }\smallskip & z\in\overline{\mathbb{C}}\setminus K_{3},\\
((1-h)\,g_{0}+h\,g_{1})(z)\smallskip & \text{ \ for \ }\smallskip & z\in
K_{3}\setminus\operatorname*{Int}(\widetilde{K}_{h}),\\
(\widehat{(1-h)\,g_{0}+h\,g_{1}})(z) & \text{ \ for \ } & z\in
\operatorname*{Int}(\widetilde{K}_{h})
\end{array}
\right. \label{f95d3b}%
\end{equation}
for $h\in\left[  0,1\right]  $\ with sets $K_{3}$\ and $\widetilde{K}_{h}%
$\ that have been defined in (\ref{f95b2}) and (\ref{f95b3}). In
(\ref{f95d3b}), the expression $\widehat{(1-h)\,g_{0}+h\,g_{1}}$ denotes the
solution of the Dirichlet problem in each component $C$ of the interior
$\operatorname*{Int}(\widetilde{K}_{h})$ of $\widetilde{K}_{h}$ with
$((1-h)\,g_{0}+h\,g_{1}))|_{\partial\widetilde{K}_{0}}$ as boundary function.
Both functions $h_{0}$ and $h_{1}$ depend on the parameter $h$.

Representations for the functions $h_{0}$ and $h_{1}$ can be proved in the
same way as has been done in Lemma \ref{l93d}. Like in (\ref{f93f2b}), we have
a representation for $h_{0}$ that now takes the form
\begin{equation}
h_{0}=g_{h}(\cdot,\infty)+\int g_{h}(\cdot,v)d\sigma_{0}(v)\label{f95d4}%
\end{equation}
with $g_{h}(\cdot,\cdot)$ the Green function of the domain $D_{h}$,
$h\in\left[  0,1\right]  $. For the measure $\sigma_{0}$ in (\ref{f95d4}) we
have
\begin{align}
& \text{ \ \ \ }\operatorname*{supp}(\sigma_{0})\subset(K_{0}\cup
K_{1})\setminus\operatorname*{Int}(\widetilde{K}_{h}),\label{f95d5}\\
& \sigma_{0}(\overline{\mathbb{C}}\setminus\Sigma_{0})=0\text{ \ \ for
\ \ }\Sigma_{0}:=(K_{1}\cup K_{2})\setminus\widetilde{K}_{0},\label{f95d6}%
\end{align}
and if $0<h<1$, then we have
\begin{equation}
\sigma_{0}\neq0\label{f95d7}%
\end{equation}
as a consequence of assumption (\ref{f95c4}). These conclusions can be proved
in exactly the same way as the corresponding assertions have been proved in
the proof of Lemma \ref{l93d}.

The assertions (ii) and (iii) of Lemma \ref{l93d} hold true also in the new
situation if one substitutes the sets $S_{0}$, $K_{1}$, $K_{2}$,
$\widetilde{K}_{0}$ by the sets $S_{h}$, $K_{0}$, $K_{1}$, $\widetilde{K}_{h}%
$. The sets $S_{h}$ and $\widetilde{K}_{h}$ have been introduced in Definition
\ref{d95a}. The proof of Lemma \ref{l93d} is quite long and involved, and the
same is true in the new situation if all details are taken into consideration.
Since everything can be done in practically the same way as before, we will
skip all details here.

In the new situation of Definition \ref{d95a} the convexity relation
(\ref{f95c2}) is the analog of the inequality (\ref{f93g}) in Lemma
\ref{l93e}, and its proof can be done in quite the same way as that of Lemma
\ref{l93d}. Like in (\ref{f93g1b}) and (\ref{f93g1c}), from (\ref{f95d4}) and
(\ref{f95d3a}) together with representation (\ref{f113b1}) for Green functions
in Lemma \ref{l113b} of Subsection \ref{s1103}, further below, it follows that
we have
\begin{align}
& \log\operatorname{cap}(K_{h})-\left[  (1-h)\,\log\operatorname{cap}%
(K_{0})+h\,\log\operatorname{cap}(K_{1})\right]  \medskip=\nonumber\\
& \text{ \ \ \ \ \ \ \ \ \ \ \ \ \ \ \ \ \ \ \ \ \ \ \ \ \ \ \ }=\left(
h_{0}(z)-g_{h}(z,\infty)\right)  |_{z=\infty}\label{f95d8}\\
& \text{ \ \ \ \ \ \ \ \ \ \ \ \ \ \ \ \ \ \ \ \ \ \ \ \ \ \ \ }=\int
g_{h}(v,\infty)d\sigma_{0}(v).\nonumber
\end{align}
In the verification of (\ref{f95d8}), representation (\ref{f113b1}) has
been\ applied to the Green functions $g_{0}$, $g_{1}$, and $g_{h}$. By
$g_{h}=g_{h}(\cdot,\cdot)$ we denote the Green function of the domain $D_{h} $
for $h\in\left[  0,1\right]  $.

In the same way as has been done in the proof of Lemma \ref{l93d}, it is then
shown that
\begin{equation}
\int g_{h}(v,\infty)d\sigma_{0}(v)<0\label{f95d9}%
\end{equation}
if, and only if, $0<h<1$. Together with (\ref{f95d8}), the last conclusion
proves assertion (ii) of Theorem \ref{t95a}.

It remains to prove assertion (iii), which will be done indirectly. We assume
that there exist $h_{0}\in\left[  0,1\right]  $ and an open set $U\subset
\mathbb{C}$ with $K_{h_{0}}\subset U$ such that there exist $h_{n}\in\left[
0,1\right]  $, $n\in\mathbb{N}$, with
\begin{equation}
h_{n}\rightarrow h_{0}\text{ \ \ as \ \ }n\rightarrow\infty\text{ \ \ and
\ \ }K_{h_{n}}\setminus U\neq\emptyset\text{ \ \ for all \ }n\in
\mathbb{N}\text{.}\label{f95e1}%
\end{equation}
Without loss of generality, we can then select $x_{n}\in K_{h_{n}}\setminus U
$, $n\in\mathbb{N}$, such that
\begin{equation}
x_{n}\rightarrow x_{0}\text{ \ \ as \ \ }n\rightarrow\infty\text{.}%
\label{f95e2}%
\end{equation}
From (\ref{f95b3}) - (\ref{f95b6}) it follows that
\begin{equation}
x_{0}\in K_{3}\setminus U.\label{f95e3}%
\end{equation}
We shall prove that assertion (\ref{f95e3}) is contradictory, which then
implies that assertion (iii) of the theorem holds true. For disproving
assertion (\ref{f95e3}) we distinguish three different cases.

\textbf{Case 1:} We assume that $x_{0}\notin K_{0}\cup K_{1}$. Then both Green
functions $g_{0}$ and $g_{1}$ are harmonic and continuous in a neighborhood of
$x_{0}$. As a consequence, it follows from $x_{n}\in K_{h_{n}}$ for
$n\in\mathbb{N}$ that also $x_{0}\in K_{h_{0}}$, which then disproves
assertion (\ref{f95e3}) since $K_{h_{0}}\subset U$.

We will give some more details of these last conclusions. From (\ref{f95b2}) -
(\ref{f95b6}) together with $x_{n}\in K_{h_{n}}$ and $x_{0}\notin K_{0}\cup
K_{1}$, we deduce that $x_{n}\in\widetilde{K}_{h_{n}}$ for $n\in\mathbb{N}$
sufficiently large. From (\ref{f95b1}) and (\ref{f95b3}) we then know that
\begin{equation}
(1-h_{n})\,g_{0}(x_{n})=h_{n}\,g_{1}(x_{n})\text{ \ \ for \ }n\in
\mathbb{N}\text{\ \ sufficiently large.}\label{f95e4}%
\end{equation}
From the continuity of $g_{0}$ and $g_{1}$ together with the limits in
(\ref{f95e1}) and (\ref{f95e2}) we deduce from (\ref{f95e4}) that
\begin{equation}
(1-h_{0})\,g_{0}(x_{0})=h_{0}\,g_{1}(x_{0}),\label{f95e5}%
\end{equation}
which implies that $x_{0}\in K_{h_{0}}$, as stated above.

\textbf{Case 2:} Let us now assume that $x_{0}\in K_{0}\cap K_{1}$. From
(\ref{f95b1}) - (\ref{f95b6}) it follows that $K_{0}\cap K_{1}\subset K_{h}$
for all $h\in\left[  0,1\right]  $. Consequently, we have $x_{0}\in K_{h_{0}}%
$, and because of $K_{h_{0}}\subset U$ this disproves (\ref{f95e3}).

\textbf{Case 3:} We assume that $x_{0}\in\left(  K_{0}\cup K_{1}\right)
\setminus\left(  K_{0}\cap K_{1}\right)  $. Because of symmetry, we can assume
without loss of generality that
\begin{equation}
x_{0}\in K_{0}\setminus K_{1}.\label{f95e6}%
\end{equation}
From (\ref{f95e6}) it follows that the Green function $g_{1}$ is positive and
harmonic in a neighborhood of $x_{0}$. Further, we can assume without loss of
generality that
\begin{equation}
x_{n}\notin K_{1}\text{ \ \ \ for all \ \ \ }n\in\mathbb{N}\text{.}%
\label{f95e7}%
\end{equation}
Because of (\ref{f95b1}) - (\ref{f95b6}), we can conclude from (\ref{f95e7})
and $x_{n}\in K_{h_{n}}$ that
\begin{equation}
(1-h_{n})\,g_{0}(x_{n})\geq h_{n}\,g_{1}(x_{n}),\label{f95e8}%
\end{equation}
which implies that
\begin{equation}
\liminf_{n\rightarrow\infty}d_{0}(x_{n})\geq0\label{f95e9}%
\end{equation}
for $d_{0}:=(1-h_{0})\,g_{0}-h_{0}\,g_{1}$. In the last conclusion we have
used the fact that the difference $g_{0}-g_{1}$ is bounded in a neighborhood
of $x_{0}$.

The function $d_{0}$ is subharmonic and non-constant in the neighborhood of
$x_{0}$.

If $d_{0}(x_{0})>0$, then it follows from (\ref{f95e6}) and (\ref{f95b4}) that
$x_{0}\in K_{0,h_{0}}\subset K_{h_{0}}$. If, on the other hand, $d_{0}%
(x_{0})=0$, then it follows from (\ref{f95b1}) and (\ref{f95b3}) that
$x_{0}\in\widetilde{K}_{h_{0}}\subset K_{h_{0}}$. If $d_{0}(x_{0})<0$, then it
follows from (\ref{f95e9}) that in each neighborhood of $x_{0}$ there exists a
point $\widetilde{x}_{0}$ with $d_{0}(\widetilde{x}_{0})=0$, which, because of
(\ref{f95b1}) and (\ref{f95b3}), again implies that $x_{0}\in\widetilde
{K}_{h_{0}}\subset K_{h_{0}}$. Hence, we have proved that $x_{0}\in K_{h_{0}}%
$, which disproves assertion (\ref{f95e3}) because of $K_{h_{0}}\subset U$.

Since (\ref{f95e3}) has been disproved for all three cases, the proof of
assertion (iii) of Theorem \ref{t95a} is completed.\medskip
\end{proof}

\section{\label{s10}Proofs II}

\qquad In the present section we prove all results that have been stated in
the Sections \ref{s4}, \ref{s5}, and \ref{s7}. In the first subsection we deal
with the special case of an algebraic function $f$. The results proved there
are of independent interest, and this is especially true for Theorem
\ref{t101b}\ towards the end of the subsection. But besides of that they are
also an essential preparation for the proofs of the main results from Sections
\ref{s4} and \ref{s5}, which will be given in Subsection \ref{s102}. Results
from Section \ref{s7} are proved in the last two subsections.\medskip

\subsection{\label{s101}Algebraic Functions}

\qquad The particularity of algebraic functions $f$ with respect to our
investigation is the fact that they possess only finitely many branch points
and no other types of non-polar singularities. As a consequence, the structure
of the minimal set $K_{0}(f,\infty)$ is in many respects special and also much
simpler to describe than this is the case in general. All functions $f$\ in
the Examples \ref{e61} - \ref{e65} of Section \ref{s6} are algebraic, and
these examples are illustrations of what we can expect on special results.
Since there exist only finitely many branch points, we have a direct
connection between Problem $(f,\infty)$ and a certain type of Problem
\ref{p81c}, which has already been discussed in Subsection \ref{s81}. Details
of the connection will be a major topic in the present subsection; another one
will be the role of rational quadratic differentials, which are in some sense
typical for Problem $(f,\infty)$ with $f$ being an algebraic function.\medskip

\subsubsection{\label{s1011}Sets of Minimal Hyperbolic Capacity}

\qquad In the present subsection we investigate a special case of Problem
\ref{p81c} from Subsection \ref{s81}.\smallskip

\begin{definition}
\label{d101a}\textit{Let }$A=\{a_{1},\ldots,\allowbreak a_{n}\}\subset
\mathbb{D}$\textit{\ be a set }of $n\geq2$ distinct points\textit{. The task
to find a continuum }$K\subset\mathbb{D}$\textit{\ with} the property that
\begin{equation}
a_{j}\in K\text{ \ \ \ \ \textit{for} \ \ }j=1,\ldots,n,\label{f101a1}%
\end{equation}
\textit{and that the condenser capacity }$\operatorname*{cap}\left(
K,\partial\mathbb{D}\right)  $\textit{\ is minimal among all continua
}$K\subset\mathbb{D}$\textit{\ that satisfy (\ref{f101a1}) is called Problem
}$(A,\mathbb{D})$\textit{. Its solution is denoted by }$K_{0}=K_{0}%
(A,\mathbb{D})$\textit{.}\smallskip
\end{definition}

For a definition of the condenser capacity $\operatorname*{cap}\left(
K,V\right)  $ with arbitrary compact sets $K,V\subset\overline{\mathbb{C}}$ we
refer to\ Chapter II.5 in \cite{SaTo} or to \cite{Bagby}. In the special with
$K\subset\mathbb{D}$ and $V:=\partial\mathbb{D}$, $\operatorname*{cap}\left(
K,\partial\mathbb{D}\right)  $ is also known as the hyperbolic capacity of $K
$ in $\mathbb{D}$. For more details see \cite{Tsuji59}. Because of this
terminology, the solution $K_{0}(A,\mathbb{D})$ of Problem $(A,\mathbb{D})$ is
also called set of minimal hyperbolic capacity, and Problem $(A,\mathbb{D})$
can be seen as the hyperbolic analogue of Chebotarev's Problem, which has been
discussed in Section \ref{s8} as Problem \ref{p81a}.

Problem $(A,\mathbb{D})$ can also be seen as a special case of Problem
\ref{p81c} from Section \ref{s81}. This connection is established in the next
proposition.\smallskip

\begin{proposition}
\label{p101a}A continuum $K_{0}\subset\mathbb{D}$ is a solution of Problem
$(A,\mathbb{D})$ if, and only if, the pair of sets $(K_{0}$, $K_{0}^{-1})$ is
a solution of Problem \ref{p81c} formulated with the two sets of points
$A=\{a_{1},\ldots,\allowbreak a_{n}\}\subset\mathbb{D}$ and $B=\{b_{1}%
,\ldots,\allowbreak b_{n}\}:=A^{-1}\subset\overline{\mathbb{C}}\setminus
\overline{\mathbb{D}}$, i.e., $b_{j}:=1/\overline{a}_{j}$ for $j=1,\ldots,n$.
By $S^{-1}$ we denote the reflection of a set $S$ on the unit circle
$\partial\mathbb{D}$.\smallskip
\end{proposition}

Proposition \ref{p101a} follows from Theorem 3.1 in \cite{Kuzmina82}, and the
relevant elements in its deduction are also assembled in Theorem \ref{t101a}, below.

The capacity $\operatorname*{cap}\left(  K,\partial\mathbb{D}\right)  $
depends only on the outer boundary of $K\subset\mathbb{D}$, and therefore we
have
\begin{equation}
\operatorname*{cap}\left(  K,\partial\mathbb{D}\right)  =\operatorname*{cap}%
(\widehat{K},\partial\mathbb{D)},\label{f101b1}%
\end{equation}
where $\widehat{K}$ denotes the polynomial-convex hull of $K$. If
$K\subset\mathbb{D}$\ is a continuum with $K=\widehat{K}$, then $\mathbb{D}%
\setminus K $\ is a ring domain, and in this special case $\operatorname*{cap}%
\left(  K,\partial\mathbb{D}\right)  $ is the reciprocal of the modulus of
this ring domain (cf. \cite{Bagby}). If $K$\ is not reduced to a single point,
then there exists $1<r<\infty$ and a bijective conformal map
\begin{equation}
\varphi:\mathbb{D}\setminus K\longrightarrow\{1<|z|<r\}\label{f101b2}%
\end{equation}
with $\varphi(1)=1$. The modulus of $\mathbb{D}\setminus K$ is then defined as
$\log(r),$ and consequently, we have $\operatorname*{cap}(K,\partial
\mathbb{D})=1/\log(r)$.

The function
\begin{equation}
p(z):=\left\{
\begin{array}
[c]{lcc}%
0\smallskip & \text{ \ for \ } & z\in K\\
\log|\varphi(z)|\smallskip & \text{ \ for \ } & z\in\mathbb{D}\setminus K\\
\log(r)=1/\operatorname*{cap}(K,\partial\mathbb{D}) & \text{ \ for \ } &
z\in\overline{\mathbb{C}}\setminus\mathbb{D}.
\end{array}
\right. \label{f101b3}%
\end{equation}
is known as the equilibrium potential of the condenser $(K,\partial
\mathbb{D})$. It is harmonic in $\mathbb{D}\setminus K$, and continuous
throughout $\overline{\mathbb{C}}$.

Problem $(A,\mathbb{D})$ has a unique solution $K_{0}=K_{0}(A,\mathbb{D}%
)\subset\mathbb{D}$. The continuum $K_{0}$ can be described very nicely by
critical trajectories of a quadratic differential. In the next theorem we
assemble these results together with other properties of the solution
$K_{0}(A,\mathbb{D})$, which will be important for our further investigations.
All results of the theorem have been proved in Chapter 3 of \cite{Kuzmina82}%
.\smallskip

\begin{theorem}
\label{t101a}(\cite{Kuzmina82}, Theorem 3.1) Let $A=\{a_{1},\ldots
,a_{n}\}\subset\mathbb{D}$ be a set of $n\geq2$ distinct points.

(i) There exists a continuum $K_{0}=K_{0}(A,\mathbb{D})\subset\mathbb{D}$,
which is the unique solution of Problem $(A,\mathbb{D})$ as introduced in
Definition \ref{d101a}.

(ii) There exist $n-2$ points $b_{1},\ldots,b_{n-2}\in\mathbb{D}$ such that
the continuum $K_{0}$ is the union of the closed critical trajectories of the
quadratic differential
\begin{equation}
q(z)\,dz^{2}\text{ \ \ with \ \ }q(z):=\frac{(z-b_{1})\ldots(z-b_{m-2}%
)(1-\overline{b}_{1}z)\ldots(1-\overline{b}_{m-2}z)}{(z-a_{1})\ldots
(z-a_{m})(1-\overline{a}_{1}z)\ldots(1-\overline{a}_{m}z)}\label{f101c1}%
\end{equation}
in $\mathbb{D}$. There exist only finitely many critical trajectories.

(iii) The equilibrium potential $p_{0}$ from (\ref{f101b3}) which
is\ associated with the extremal condenser $(K_{0},\partial\mathbb{D})$
satisfies relation
\begin{equation}
\left(  \frac{\partial}{\partial z}p_{0}(z)\right)  ^{2}=\frac{1}%
{4}\,q(z)\text{ \ \ \ \ for \ \ \ }z\in\overline{\mathbb{D}}\label{f101c2}%
\end{equation}
with $\partial/\partial z$ denoting complex differentiation. The potential
$p_{0}$ can be extended to a harmonic function in $\overline{\mathbb{C}%
}\setminus(K_{0}\cup K_{0}^{-1})$ by a reflection on the unit circle
$\partial\mathbb{D}$, and relation (\ref{f101c2}) holds true throughout
$\overline{\mathbb{C}}$ for the harmonic extension of $p_{0}$.

(iv) We have
\begin{equation}
p_{0}(z):=\left\{
\begin{array}
[c]{lcc}%
1/\operatorname*{cap}(K_{0},\partial\mathbb{D})\smallskip & \text{ \ for \ } &
z\in\overline{\mathbb{C}}\setminus\mathbb{D}\\
0 & \text{ \ for \ } & z\in K_{0}%
\end{array}
\right. \label{f101d1}%
\end{equation}
and
\begin{equation}
\frac{1}{2\,\pi}\oint_{\partial\mathbb{D}}\frac{\partial}{\partial n}%
p_{0}(\zeta)ds_{\zeta}=-1\label{f101d2}%
\end{equation}
with $\partial/\partial n$ the inward showing normal derivative on
$\partial\mathbb{D}$ and $ds$ the line element on $\partial\mathbb{D}%
$.\medskip
\end{theorem}

\subsubsection{\label{s1012}Problem $(f,\infty)$ for Algebraic Functions}

\qquad In the present subsection we study Problem $(f,\infty)$ for an
algebraic function $f$. We will shed light on the connection between this
problem and certain aspects of Problem $(A,\mathbb{D})$ from Definition
\ref{d101a}.\smallskip

Let $f$ be an algebraic function and assume that this function is meromorphic
at infinity. Algebraic functions have only finitely many singularities, and
the only non-polar singularities are branch points. Hence, in Problem
$(f,\infty)$, only a finite number of points is of critical relevance. By
$E_{0}\subset K_{0}(f,\infty)$ we denote the (finite) set of branch points of
the function $f$ that can be reached on the minimal set $K_{0}(f,\infty)$ by
meromorphic continuation of $f$ from within the extremal domain $D_{0}%
(f,\infty)$. In the discussion of the examples in Section \ref{s6}, this type
of branch points have been called the active branch points for the
determination of the minimal set $K_{0}(f,\infty)$. The sets $D_{0}(f,\infty)$
and $K_{0}(f,\infty)$\ have been introduced in Definition \ref{d21b}%
.\smallskip

\begin{lemma}
\label{l101a}Let $f$ be an algebraic function that is meromorphic at infinity.
Then the minimal set $K_{0}(f,\infty)$ for Problem $(f,\infty)$\ has only
finitely many components, which we denoted by $K_{1},\ldots,K_{m}$, i.e., we
have
\begin{equation}
K_{0}(f,\infty)=K_{1}\cup\ldots\cup K_{m}.\label{f1012a1}%
\end{equation}
Each component $K_{j}$, $j=1,\ldots,m$, contains at least two branch points of
$f$. The Green function $g_{D_{0}}(\cdot,\infty)$ in the extremal domain
$D_{0}=D_{0}(f,\infty)$\ has only finitely many critical points, and we have
\begin{equation}
g_{D_{0}}(z,\infty)=0\text{ \ \ \ for all \ \ \ }z\in K_{0}(f,\infty
).\label{f1012a3}%
\end{equation}

\end{lemma}

\begin{proof}
It follows from the two conditions (ii) and (iii) in Definition \ref{d21b}
that each component of $K_{0}(f,\infty)$ has to contain at least one non-polar
singularity of $f$, and the single-valuedness of $f$ in $D_{0}(f,\infty)$ then
further implies that each component of $K_{0}(f,\infty)$ has to contain at
least two branch points.

After this conclusion, all other assertions of the lemma follow directly from
of the finiteness of the set $E_{0}$ and the fact that a continuum has no
irregular points with respect to the Dirichlet problem (cf. Subsection
\ref{s1103}, further below).\medskip
\end{proof}

Based on Lemma \ref{l101a}, we divide the set $E_{0}$ into $m$ subsets
$E_{j}:=E_{0}\cap K_{j}$, $j=1,\ldots,m$. I.e., we have
\begin{equation}
E_{0}=E_{1}\cup\ldots\cup E_{m}\text{ with\ \ }E_{j}\subset K_{j}\text{ \ for
\ }j=1,\ldots,m\text{.}\label{f1012a2}%
\end{equation}
For $c>0$ we define the open set
\begin{equation}
U_{c}:=\left\{  \,z\in\mathbb{C}\,|\,g_{D_{0}}(z,\infty)(z)<c\,\right\}
\label{f1012b1}%
\end{equation}
with the Green function $g_{D_{0}}(\cdot,\infty)$ in $D_{0}=D_{0}(f,\infty)$.
Since $g_{D_{0}}(\cdot,\infty)$ has only finitely many critical points in
$D_{0}$, the open set $U_{c}$ consists of exactly $m$ components for $c>0$
sufficiently small. The number $m$\ is the same as that in (\ref{f1012a1}%
).\smallskip

\begin{lemma}
\label{l101b}Let the same assumptions hold true as in Lemma \ref{l101a}. Then
a constant $c_{0}>0$ can be chosen in such a way that the open set
$U_{0}:=U_{c_{0}}$ from (\ref{f1012b1}) has the following properties:

\begin{itemize}
\item[(i)] $\overline{U}_{0}$ contains no critical point of the Green function
$g_{D_{0}}(\cdot,\infty).$

\item[(ii)] $U_{0}$ consists of exactly $m$ components $U_{j}$, $j=1,\ldots
,m$, i.e.,
\begin{equation}
U_{0}=U_{1}\cup\ldots\cup U_{m},\label{f1012b2}%
\end{equation}
and we have
\begin{equation}
K_{j}\subset U_{j},\text{ \ \ for \ \ }j=1,\ldots,m\label{f1012b3}%
\end{equation}
with $K_{j}$ introduced in (\ref{f1012a1}).

\item[(iii)] Each component $U_{j}$, $j=1,\ldots,m$, in (\ref{f1012b2})\ is
simply connected, and $\partial U_{j}$ is an analytic Jordan curve.
\end{itemize}
\end{lemma}

\begin{proof}
All three assertions of the lemma are rather immediate. Because of
(\ref{f1012a3}), $U_{0}$ is an open neighborhood of $K_{0}(f,\infty)$, and we
can shrink $U_{0}$ as close to $K_{0}(f,\infty)$ as we wish. The first
assertions (i) follows from the fact that the Green function $g_{D_{0}}%
(\cdot,\infty)$ has only finitely many critical points in $D_{0}(f,\infty)$.

The two other assertions (ii) and (iii) follow then immediately from
(\ref{f1012a3}) for $c_{0}>0$ sufficiently small.\medskip
\end{proof}

In the next proposition we establish the connections between Problem
$(f,\infty)$ and problems of the type of Problem $(A,\mathbb{D})$. These
connections are the main topic in the present subsection.\smallskip

\begin{proposition}
\label{p101b}Let $f$ be an algebraic function that is meromorphic at infinity,
and let $K_{0}$ be the minimal set $K_{0}(f,\infty)$ of Definition \ref{d21b}
for Problem $(f,\infty)$. Let further the sets $K_{j}$, $E_{j}$, and $U_{j}$,
$j=1,\ldots,m$, be defined as in (\ref{f1012a1}), (\ref{f1012a2}), and
(\ref{f1012b2}), respectively, and let $\varphi_{j}:U_{j}\longrightarrow
\mathbb{D}$ be Riemann mapping functions, $j=1,\ldots,m$. We set
\begin{equation}
A_{j}:=\varphi_{j}(E_{j})\text{, \ }K_{0,j}:=\varphi_{j}(K_{j})\text{,
\ }\alpha_{j}:=\omega_{K_{0}}(K_{j})\text{, \ }j=1,\ldots,m\text{,}%
\label{f1012c1}%
\end{equation}
with $\omega_{K_{0}}$ denoting the equilibrium distribution on $K_{0}$ as
introduced in Subsection \ref{s1102}, further below. The following two
assertions hold true:

\begin{itemize}
\item[(i)] For each $j=1,\ldots,m$, the set $K_{0,j}$ is the minimal set
$K_{0}(A_{j},\mathbb{D})$ that solves Problem $(A_{j},\mathbb{D})$ from
Definition \ref{d101a}, which has been analyzed in Theorem \ref{t101a}.

\item[(ii)] For each $j=1,\ldots,m$, we have
\begin{equation}
g_{D_{0}}(z,\infty)(z)=\alpha_{j}(p_{0,j}\circ\varphi_{j})(z)\text{ \ \ for
\ \ \ }z\in U_{j}\text{,}\label{f1012c2}%
\end{equation}
where $p_{0,j}$\ is the equilibrium potential (\ref{f101b3}) for the extremal
condenser $(K_{0,j},\allowbreak\mathbb{D})$ of Problem $(A_{j},\mathbb{D}),$
and $g_{D_{0}}(\cdot,\infty)$ is the Green function in the extremal domain
$D_{0}(f,\infty)$.\smallskip
\end{itemize}
\end{proposition}

The practical significance of Proposition \ref{p101a} is that it shows a
possibility to transplant specific properties of the solutions $K_{0}%
(A_{j},\mathbb{D})$ of \ the Problems $(A_{j},\mathbb{D})$, $j=1,\ldots,m$, to
the solution of Problem $(f,\infty)$.\smallskip

\begin{proof}
We start with assertion (i), which will be proved indirectly. For this
purpose, we assume that at least one of the sets $K_{0,j}$, $j=1,\ldots,m$, is
not a minimal solution $K_{0}(A_{j},\mathbb{D})$ of Problem $(A_{j}%
,\mathbb{D})$. Without loss of generality we can assume that
\begin{equation}
K_{0,1}\neq\widetilde{K}_{0,1}:=K_{0}(A_{1},\mathbb{D}).\label{f1012d3}%
\end{equation}
We define
\begin{equation}
\widetilde{K}_{1}:=\varphi_{1}^{-1}(\widetilde{K}_{0,1})\text{, \ \ }%
\widetilde{K}_{0}:=(K_{0}\setminus K_{1})\cup\widetilde{K}_{1}\text{,
\ \ }\widetilde{D}_{0}:=\overline{\mathbb{C}}\setminus\widetilde{K}%
_{0},\label{f1012d4}%
\end{equation}
and show that the domain $\widetilde{D}_{0}$ is admissible for Problem
$(f,\infty)$, i.e.,
\begin{equation}
\widetilde{D}_{0}\in\mathcal{D}(f,\infty)\text{.}\label{f1012d5}%
\end{equation}
From (\ref{f1012d3}) we then deduce hat
\begin{equation}
\operatorname*{cap}(\widetilde{K}_{0})<\operatorname*{cap}\left(
K_{0}(f,\infty)\right)  .\label{f1012d1}%
\end{equation}

If (\ref{f1012d3}) and (\ref{f1012d5})\ are proved, then with (\ref{f1012d1})
we have a contradiction since, because of (\ref{f1012d5}), inequality
(\ref{f1012d1}) clearly contradicts the minimality (\ref{f21a}) in Definition
\ref{d21b} of the set $K_{0}=K_{0}(f,\infty)$. The contradiction shows that
assumption (\ref{f1012d3}) is false, and therefore assertion (i) is
proved.\smallskip

We start with the proof of (\ref{f1012d5}). Using Cauchy's formula, one can
rewrite the function $f$ as
\begin{equation}
f=f_{1}+\ldots+f_{m}\label{f1012d2}%
\end{equation}
with each $f_{j}$, $j=1,\ldots,m$, being meromorphic and single-valued in the
simply connected domain $\overline{\mathbb{C}}\setminus K_{j}$. Since the sets
$U_{1},\ldots,U_{m}$ are disjoint, we have only to consider the function
$f_{1}$ if we want to understand the changes in the global behavior of $f$
that are caused by the exchange of the sets $K_{1}$\ and $\widetilde{K}_{1}%
$\ that is defined by (\ref{f1012d4}).

From (\ref{f1012a2}), (\ref{f1012c1}), (\ref{f1012d3}), and (\ref{f1012d4}),
we know that both sets $K_{1}$ and $\widetilde{K}_{1}$ contain the same set
$E_{1}$ of branch points of the function $f$ on $K_{1}$. Consequently, the
function $f_{1}$ can be continued meromorphically throughout the whole domain
$\overline{\mathbb{C}}\setminus\widetilde{K}_{1}$. Since the domain
$\overline{\mathbb{C}}\setminus\widetilde{K}_{1}$ is simply connected, it
follows from the Monodromy Theorem that the continuation of $f_{1}$ is
single-valued in $\overline{\mathbb{C}}\setminus\widetilde{K}_{1}$, and
consequently, the function $f$ has also a single-valued meromorphic
continuation to the domain $\widetilde{D}_{0}$, which proves (\ref{f1012d5}%
).\smallskip

In order to prove (\ref{f1012d1}), we observe first that from the uniqueness
of the solution $\widetilde{K}_{0,1}:=K_{0}(A_{1},\mathbb{D})$ of Problem
$(A_{1},\mathbb{D})$, which has been established in Theorem \ref{t101a}, it
follows that
\begin{equation}
\operatorname*{cap}(\widetilde{K}_{0,1},\partial\mathbb{D)}%
<\operatorname*{cap}\left(  K_{0,1},\partial\mathbb{D}\right) \label{f1012e2}%
\end{equation}
with $K_{0,1}$ defined in (\ref{f1012c1}). Notice that $A_{1}\subset K_{0,1}$.
We shall now show that inequality (\ref{f1012e2}) implies (\ref{f1012d1}).

Indeed, let $p_{01}$ and $\widetilde{p}_{01}$ be the equilibrium potentials
(\ref{f101b3}) of the two condensers $(K_{0,1},\mathbb{D})$ and $(\widetilde
{K}_{0,1},\mathbb{D})$, respectively. From (\ref{f101b3}) it follows that
\begin{equation}
p_{01}(z)=\frac{1}{\operatorname*{cap}\left(  K_{0,1},\partial\mathbb{D}%
\right)  }\text{, \ \ }\widetilde{p}_{01}(z)=\frac{1}{\operatorname*{cap}%
(\widetilde{K}_{0,1},\partial\mathbb{D)}}\text{\ \ for \ }z\in\partial
\mathbb{D}.\label{f1012e1}%
\end{equation}
From the definition of the open set $U_{0}$ in Lemma \ref{l101b} together with
the properties of the Green function $g_{D_{0}}(\cdot,\infty)$, the definition
of the mapping $\varphi_{1}:U_{1}\longrightarrow\mathbb{D}$, the set $K_{0,1}%
$, and the number $\alpha_{1}$, which has been introduced in (\ref{f1012c1}),
we then deduce that the function
\begin{equation}
\overset{\vee}{p}_{01}:=\frac{1}{\alpha_{1}}g_{D_{0}}(\cdot,\infty
)\circ\varphi_{1}^{-1}\label{f1012e3}%
\end{equation}
has the following four properties: (i) The function $\overset{\vee}{p}_{01} $
is harmonic in $\mathbb{D}\setminus K_{0,1}$. (ii) We have $\overset{\vee}%
{p}_{01}(z)=0$ for all $z\in K_{0,1}$. (iii) We have
\begin{equation}
\overset{\vee}{p}_{01}(z)=\frac{c_{0}}{\alpha_{1}}\text{ \ \ for all \ \ }%
z\in\partial\mathbb{D}\label{f1012e4}%
\end{equation}
with the constant $c_{0}$ introduced in Lemma \ref{l101b}. (iv) We have
\begin{equation}
\frac{1}{2\pi}\int_{\partial\mathbb{D}}\frac{\partial}{\partial n}%
\overset{\vee}{p}_{01}ds=-1\label{f1012e4a}%
\end{equation}
because of the definition of $\alpha_{1}$ in (\ref{f1012c1}). The normal
derivative $\partial/\partial n$ on $\partial\mathbb{D}$ in (\ref{f1012e4a})
is assumed to be inwardly oriented.

From these four properties together with (\ref{f101b3}), it follows that in
$\mathbb{D}$ the function $\overset{\vee}{p}_{01}$ is identical with the
equilibrium potential $p_{01}$ of the condenser $(K_{0,1},\mathbb{D})$. From
the first equality in (\ref{f1012e1}) together with (\ref{f1012e4}), it then
follows that
\begin{equation}
\frac{1}{\operatorname*{cap}\left(  K_{0,1},\partial\mathbb{D}\right)  }%
=\frac{c_{0}}{\alpha_{1}}.\label{f1012e5}%
\end{equation}
Motivated by (\ref{f1012e1}) and (\ref{f1012e5}), we define
\begin{equation}
\widetilde{\alpha}_{1}:=\frac{\operatorname*{cap}(\widetilde{K}_{0,1}%
,\partial\mathbb{D)}}{\operatorname*{cap}\left(  K_{0,1},\partial
\mathbb{D}\right)  }\alpha_{1}<\alpha_{1},\label{f1012e6}%
\end{equation}
where the inequality is a consequence of (\ref{f1012e2}).

Next, we study the function
\begin{equation}
\overset{\vee}{g}_{0}(z):=\left\{
\begin{array}
[c]{ccc}%
\widetilde{\alpha}_{1}(\widetilde{p}_{01}\circ\varphi_{1})(z)\smallskip &
\text{ \ for \ } & z\in\overline{U}_{1}\\
g_{D_{0}}(z,\infty) & \text{ \ for \ } & z\in\overline{\mathbb{C}}%
\setminus\overline{U}_{1},
\end{array}
\right. \label{f1012e7}%
\end{equation}
which is basically a modification of the Green function $g_{D_{0}}%
(z\cdot,\infty)$ in the neighborhood $U_{1}$ of $K_{1}$. The function is
continuous in $\mathbb{C}$ since both partial functions in (\ref{f1012e7}) are
equal to $c_{0}$ on $\partial U_{1}$. Indeed, it follows from (\ref{f1012e6}),
(\ref{f1012e5}), and the second equation in (\ref{f1012e1}) that
\begin{equation}
g_{D_{0}}(z,\infty)=\,\overset{\vee}{g}_{0}(z)=c_{0}\text{ \ \ \ for all
\ \ }z\in\partial U_{1}.\label{f1012e8}%
\end{equation}
The function $\overset{\vee}{g}_{0}$ has the following four properties: (i) It
is harmonic in $\overline{\mathbb{C}}\setminus(\widetilde{K}_{0}\cup\partial
U_{1})$ because of the definitions made in (\ref{f1012d4}). (ii) It has smooth
normal derivatives from both sides of $\partial U_{1}$. (iii) We have
$\overset{\vee}{g}_{0}(z)=0$ for all $z\in\widetilde{K}_{0}$ because of the
first line in (\ref{f101b3}). (iv) Near infinity we have
\begin{equation}
\overset{\vee}{g}_{0}(z)=\log|z|+\log\frac{1}{\operatorname*{cap}\left(
K_{0}\right)  }+\text{o}(1)\text{ \ \ as \ \ }z\rightarrow\infty
,\label{f1012e9}%
\end{equation}
which follows from (\ref{f1012e7}) and Lemma \ref{l113b} in Subsection
\ref{s1103}, further below.

From the definition of $\overset{\vee}{g}_{0}$ in (\ref{f1012e7}) together
with the four properties of $\overset{\vee}{g}_{0}$ that have just been listed
and with the use of Lemma \ref{l113f} in Subsection \ref{s1103}, we deduce
that
\begin{equation}
g_{\widetilde{D}_{0}}(z,\infty)=\,\overset{\vee}{g}_{0}(z)+\int g_{\widetilde
{D}_{0}}(z,x)d\sigma(x),\text{ \ \ }z\in\overline{\mathbb{C}},\label{f1012e10}%
\end{equation}
where $\sigma$ is a signed measure with $\operatorname*{supp}(\sigma
)\subset\partial U_{1}$.\ From (\ref{f113d2}) in Lemma \ref{l113f}, we know
that the measure $\sigma$ is defined as the difference of the flux in
$\overset{\vee}{g}_{0}$ that comes to the Jordan curve $\partial U_{1}$ from
the both sites. Indeed, the total flux flowing into the set $\overline{U}_{1}$
from outside is equal to $\alpha_{1}$ because of the definition of $\alpha
_{1}$ in (\ref{f1012c1}). On the other hand, the flux coming from within
$U_{1}$ is equal to $\widetilde{\alpha}_{1}$ because of (\ref{f101d2}) and
(\ref{f1012e7}). Hence, from (\ref{f1012e6}) we conclude that
\begin{equation}
\sigma(\partial U_{1})=\alpha_{1}-\widetilde{\alpha}_{1}>0.\label{f1012e11}%
\end{equation}
Putting all partial results of the last paragraphs together, we arrive at the
following estimate:
\begin{align}
\log\operatorname*{cap}(K_{0})-\log\operatorname*{cap}(\widetilde{K}_{0})  &
=(g_{\widetilde{D}_{0}}(z,\infty)-\overset{\vee}{g}_{0}(z))|_{z=\infty
}\nonumber\\
& =\int g_{\widetilde{D}_{0}}(x,\infty)d\sigma(x)\label{f1012e12}\\
& =\int\overset{\vee}{g}_{0}(x)d\sigma(x)+\int g_{\widetilde{D}_{0}%
}(x,y)d\sigma(x)d\sigma(x)\nonumber\\
& >c_{0}(\alpha_{1}-\widetilde{\alpha}_{1})>0.\nonumber
\end{align}
Indeed, the first equality in (\ref{f1012e12}) follows from (\ref{f1012e9})
and representation (\ref{f113b1}) in Lemma \ref{l113b} in Subsection
\ref{s1103}, further below. The second one is a consequence of (\ref{f1012e10}%
) and the symmetry of the Green function with respect to both of its
arguments. The third one follows again from (\ref{f1012e10}). The first
inequality in (\ref{f1012e12}) is a consequence of (\ref{f1012e8}) and
(\ref{f1012e11}) together with\ the positive definiteness of the Green kernel
(cf. Lemma \ref{l113d} in Subsection \ref{s1103}, further below). The last
inequality follows again from (\ref{f1012e11}).

With the inequalities in (\ref{f1012e12}) we have proved (\ref{f1012d1}). It
has already been mentioned after (\ref{f1012d1}) that the proof of assertion
(i) is complete as soon as we have completed the deduction of (\ref{f1012d5})
and (\ref{f1012d1}).

The considerations made in (\ref{f1012e7}) with respect to the function
$\overset{\vee}{g}_{0}$ show that if each set $K_{0,j}$, $j=1,\ldots,m$, is
the unique solution of Problem $(A_{j},\mathbb{D})$, therefore, identity
(\ref{f1012c2}) holds true for each $j=1,\ldots,m$. Hence, assertion\ (ii) is
a consequence of assertion (i), and the proof of the whole Proposition
\ref{p101a} is complete.\medskip
\end{proof}

\subsubsection{\label{s1013}The Minimal Set for Algebraic Functions}

\qquad With Proposition \ref{p101a} and Theorem \ref{t101a}\ we are prepared
to prove a detailed description of the minimal set $K_{0}(f,\infty)$ for
Problem $(f,\infty)$ with an algebraic function $f$.

The next theorem covers most of the content in the main theorems in Section
\ref{s4} and \ref{s5}. Since $f$ is assumed to be an algebraic function, we
deal here only with a special version of Problem $(f,\infty)$, however, we
remark that at the present point the results of Section \ref{s4} and \ref{s5}
are still not proved, and more than that, the results in the next theorem will
later be used as intermediate steps in the general proofs.\medskip

\begin{theorem}
\label{t101b}Let $f$ be algebraic function that is not rational. We assume
that the function is meromorphic at infinity. Let further $K_{0}(f,\infty)$ be
the minimal set for Problem $(f,\infty)$.\smallskip

(a) \ The interior of $K_{0}(f,\infty)$ is empty, and there exist two finite
sets $E_{0}$, $E_{1}$, and a finite family of open and analytic Jordan arcs
$J_{j}$, $j\in I$, such that
\begin{equation}
K_{0}(f,\infty)=E_{0}\cup E_{1}\cup\bigcup_{j\in I}J_{j}.\label{f1013a1}%
\end{equation}
The components in (\ref{f1013a1}) correspond to those in Theorem \ref{t41a} of
Section \ref{s4}, but under the additional assumption that $f$\ is algebraic
we can give a more specific characterization:

\begin{itemize}
\item[(i)] The set $E_{0}$ is finite, and it consists of all branch points of
$f$ in $K_{0}(f,\infty)$ that can be reached by meromorphic continuation of
$f$ out of the extremal domain $D_{0}(f,\infty)$.

\item[(ii)] The set $E_{1}$ is finite, and it consists of all bifurcation
points of $K_{0}(f,\infty)$ that do not belong to $E_{0}$.

\item[(iii)] The family $\left\{  J_{j}\right\}  _{j\in I}$ of analytic Jordan
arcs is finite. All arcs $J_{j}$, $j\in I$, are pair-wise disjoint. The
function $f$ has meromorphic continuations across each arc $J_{j}$, $j\in I$,
from both sides. Each arc $J_{j}$, $j\in I$, is a trajectory of the quadratic
differential (\ref{f1013a2}) having end points that belong to $E_{0}\cup
E_{1}$, and all open trajectories of (\ref{f1013a2}) starting and ending at a
point of $E_{0}\cup E_{1}$ belong to the family $\left\{  J_{j}\right\}
_{j\in I}$.\smallskip
\end{itemize}

(b) \ The set $E_{0}$ contains at least $2$ points; we denote the points in
$E_{0}$ by $a_{1},\ldots,a_{n}$. There exist $n-2$ points $b_{1}%
,\ldots,b_{n-2}\in\mathbb{C}$ such that the Jordan arcs $J_{j}$, $j\in I$, are
trajectories of the quadratic differential
\begin{equation}
q(z)\,dz^{2}\text{ \ \ with \ \ }q(z):=\frac{(z-b_{1})\ldots(z-b_{n-2}%
)}{(z-a_{1})\ldots(z-a_{n})}\label{f1013a2}%
\end{equation}
Not all points of the set $B=\{b_{1},\ldots,b_{n-2}\}$ are necessarily
different, and not all of them are necessarily contained in $K_{0}(f,\infty)
$.\smallskip

(c) \ The minimal set $K_{0}(f,\infty)$ consists of finitely many components;
we denote their number by $m$. Each of these components contains at least two
elements of $E_{0}$. We have $E_{1}\subset B$. If $m>1$, then the Green
function $g_{D_{0}}(\cdot,\infty)$, $D_{0}=D_{0}(f,\infty)$, possesses
critical points of total order $m-1$, and each of these critical points
appears in the set $B$ with a frequency of twice its order.\medskip
\end{theorem}

\begin{remark}
\label{r101a}The Examples \ref{e61} - \ref{e64} in Section \ref{s6} belong to
the class of problems covered by Theorem \ref{t101b}. In the discussion of
these examples one finds concrete and explicit examples for the sets $E_{0}$,
$E_{1}$, for the families of Jordan arcs $\left\{  J_{j}\right\}  _{j\in I}$,
and also for the quadratic differentials (\ref{f1013a2}).

There exists strong similarities between the two Theorems \ref{t101b} and
\ref{t53a}, but we note that the later one has a somewhat different
orientation; it is focused only on the finiteness of the set $E_{0}$.\medskip
\end{remark}

\begin{proof}
We define
\begin{equation}
q(z):=\left(  \frac{\partial}{\partial z}g_{D_{0}}(z,\infty)\right)
^{2}\text{ \ \ for \ \ }z\in D_{0}=D_{0}(f,\infty)\label{f1013b1}%
\end{equation}
with $D_{0}(f,\infty)$ the extremal domain for Problem $(f,\infty)$. It is
immediate that $q$ is analytic in $D_{0}(f,\infty)$. From (\ref{f1013b1}) and
representation (\ref{f113b1}) for the Green function in Lemma \ref{l113b} of
Subsection \ref{s1103}, further below, we deduce that at infinity the function
$q$\ has the development
\begin{equation}
q(z)=z^{-2}+\text{O}(z^{-3})\text{ \ \ \ as \ \ \ }z\rightarrow\infty
.\label{f1013b2}%
\end{equation}

The function $q$\ is different from zero everywhere in $D_{0}(f,\infty
)\cap\mathbb{C}$ except at the critical points of the Green function
$g_{D_{0}}(\cdot,\infty)$, where it has zeros. It follows from (\ref{f1013b1})
that the order of each of these zeros is twice the order of the critical
point. Critical points and their order have been introduced in Definition
\ref{d53b} in Subsection \ref{s53}.

From Lemma \ref{l101a} we know that $K_{0}(f,\infty)$\ has only a finite
number of components, which we denote again by $K_{j}$, $j=1,\ldots,m$. A
combination of Proposition \ref{p101a} and Theorem \ref{t101a} shows that
$q$\ is meromorphic in a neighborhood of each component $K_{j}$,
$j=1,\ldots,m$. Hence, the function is meromorphic throughout $\overline
{\mathbb{C}}$, and consequently it is a rational function with all its poles
contained in $K_{0}(f,\infty)$.

For the deduction of more specific assertions we can without loss of
generality restrict our attention to individual components $K_{j}$ and open
neighborhoods $U_{j}$, $j=1,\ldots,m$, of these sets. Without loss of
generality we will choose $j=1$ in the sequel.

It follows from (\ref{f1012c2}) and (\ref{f1012c1}) in Proposition \ref{p101a}
together with assertion (ii) and (iii) of Theorem \ref{t101a} that all poles
of $q$ on $K_{1}$ are simple, and they have to belong to the set $E_{1}$ from
(\ref{f1012a2}), i.e., they have to be branch points of $f$ on $K_{1}$.

Further it follows especially from assertion (ii) of Theorem \ref{t101a} that
on $K_{1}$ the function $q$ has exactly two zeros less than it has poles on
$K_{1}$, where multiplicities of zeros have to be taken into account. The
zeros in question constitute the set $E_{1}\cap K_{1}$.

In the application of assertion (ii) of Theorem \ref{t101a} there may appear
cancellations of numerator and denominator factors in the function
(\ref{f101c1}). If none of such cancellations occurs, which can be seen as the
generic case, then every branch point of the function $f$ on $K_{1}$
corresponds to a simple pole of the function $q$ on $K_{1}$.

From what has been proved so far together with the definitions made in
(\ref{f1012c1}) and the identity (\ref{f1012c2}) of Proposition \ref{p101a},
we deduce from assertion (ii) of Theorem \ref{t101a} that all Jordan arcs
$J_{j} $, $j\in I$, that belong to $K_{1}$ are transformed by the conformal
map $\varphi_{1}$ of (\ref{f1012c1}) into a critical trajectory $\varphi
_{1}(J_{j})$ of the quadratic differential (\ref{f101c1}) of Theorem
\ref{t101a}, and the reverse conclusion holds also true. With these last
conclusions we have proved assertion (iii) of part (a) in the theorem for the
component $K_{1}$.

All conclusions that have been proved so far for the component $K_{1}$ hold
true in the same way on the other $m-1$ components $K_{j}$, $j=2,\ldots,m$, of
the minimal set $K_{0}(f,\infty)$, which proves practically all assertions of
the theorem.

We add that the number of zeros of the rational function $q$ on all $m$
component $K_{1},\ldots,K_{m}$ together with the $2(m-1)$ zeros at the
critical points of the Green function $g_{D_{0}}(\cdot,\infty)$ add up to
exactly two zeros less than the number of poles that $q$ has in $\mathbb{C}$.
This account reaffirms exactly the behavior of the function $q$ at infinity,
which is shown in development (\ref{f1013b2}).\medskip
\end{proof}

\subsection{\label{s102}Some Technical Results}

\qquad In the present subsection we prove some technical results which then
are needed in the remainder of the section in proofs of results from the
Sections \ref{s4}, \ref{s5}, and \ref{s7}. Most important are here the proofs
of the two Theorems \ref{t41a} and \ref{t73a}. In a first part of the
subsection, the results will be formulated together with related definitions;
proofs will then follow afterwards. Some of the results depend on rather
subtle topological assumptions.\smallskip

The first proposition is especially important for the proof of Theorem
\ref{t73a}.\smallskip

\begin{proposition}
\label{p102a}Let $D_{1},D_{2}\in\mathcal{D}(f,\infty)$ be two admissible
domains for Problem $(f,\infty)$. If we assume that $D_{1}$ possesses the
$S-$property as introduced in Definition \ref{d71a}, and if we assume further
that $D_{2}$ is elementarily maximal in the sense of Definition \ref{d71a0},
then we have either $D_{1}=D_{2}$ or
\begin{equation}
\operatorname{cap}(\partial D_{1})<\operatorname{cap}(\partial D_{2}%
).\label{f102a1}%
\end{equation}

\end{proposition}

Besides of Proposition \ref{p102a} we need a very similar result, which is not
related to admissible domains $D\in\mathcal{D}(f,\infty)$; instead it is based
on purely topological assumptions, which however comes to the same thing. We
prepare the formulation of the result by some definitions.\medskip

\begin{definition}
\label{d102a}Let $K,E\subset\mathbb{C}$ be two polynomials-convex and compact
sets with $\operatorname{cap}(K)>0$ and $E\subset K$. We say that the set
$K$\ possesses the $S-$property on the subset $K\setminus E$ if the following
two assertions are satisfied:

\begin{itemize}
\item[(i)] The set $K\setminus E$ is of the form
\begin{equation}
K\setminus E=E_{1}\cup\bigcup_{j\in I}J_{j}\label{f102a2}%
\end{equation}
with $E_{1}$ a discrete set in $K\setminus E$ and $\{J_{j}\}_{j\in I}$ a
family of smooth, open, and disjoint Jordan arcs $J_{j}$. Each point $z\in
E_{1}$ is an end point of at least three different arcs from $\{J_{j}\}_{j\in
I}$.

\item[(ii)] The Green function $g_{D}(\cdot,\infty)$ with $D:=\overline
{\mathbb{C}}\setminus K$ satisfies the symmetry relation
\begin{equation}
\frac{\partial}{\partial n_{+}}g_{D}(z,\infty)=\frac{\partial}{\partial n_{-}%
}g_{D}(z,\infty)\text{ \ for all \ }z\in J_{j},j\in I\label{f102a3}%
\end{equation}
with $\partial/\partial n_{+}$ and $\partial/\partial n_{-}$ denoting the
normal derivatives to both sides of the arcs $J_{j}$, $j\in I$.\smallskip
\end{itemize}
\end{definition}

It is obvious that there exist many parallels between the Definitions
\ref{d102a} and Definitions \ref{d71a} of Section \ref{s7}; the special aspect
of the new definition is the independence from Problem $(f,\infty)$ and the
admissible domains $D\in\mathcal{D}(f,\infty)$. On the other hand, it is
immediate that any admissible domain $D\in\mathcal{D}(f,\infty)$ with
$K:=\overline{\mathbb{C}}\setminus D$ that possesses the $S-$property in the
sense of Definition \ref{d71a} possesses also the $S-$property in the sense of
Definition \ref{d102a} on the subset $K\setminus E_{0}$, where $E_{0}$ is the
compact set from (\ref{f71a2}) and assertion (i) in Definition \ref{d71a}%
.\smallskip

Let $K_{1},K_{2},E\subset\mathbb{C}$ be three compact sets with $E\subset
K_{1}\cap K_{2}$. Two components $E_{1},E_{2}\subset E$ of $E$\ are said to
the connected in $K_{1}$ if both components are contained in the same
component of $K_{1}$. In this sense, $K_{1}$ defines a connectivity relation
on the components of $E$. We say (in the usual sense) that the connectivity of
$E$ in $K_{1}$ is coarser than the connectivity in $K_{2}$ if the
connectedness of two components $E_{1},E_{2}\subset E$ in $K_{2}$ implies
their connectedness in $K_{1}$.\medskip

\begin{definition}
\label{d102b}Let $K_{1},K_{2},E\subset\mathbb{C}$ be three polynomial-convex
and compact sets with $\operatorname{cap}(K_{1})>0$\ and $E\subset K_{1}\cap
K_{2}$, and let us assume that $K_{1}$ possesses the $S-$property in the sense
of Definition \ref{d102a} on $K_{1}\setminus E$ with $\{J_{j}\}_{j\in I}$
denoting the family of Jordan arcs introduced in (\ref{f102a2}). We say that
the connectivity of $E$ in $K_{1}$ is minimally coarser than the connectivity
of $E$ in $K_{2}$ if the following to assertions hold true:

\begin{itemize}
\item[(i)] The connectivity of $E$ in $K_{1}$ is coarser than that in $K_{2} $.

\item[(ii)] If $\widetilde{K}_{1}$ is the compact set that results from
dropping one of the arcs $J_{j}$, $j\in I$, from $K_{1}$, then assertion (i)
holds no longer true with $\widetilde{K}_{1}$ replacing $K_{1}$.\medskip
\end{itemize}
\end{definition}

\begin{remark}
\label{r102a}Since the arcs $J_{j}$, $j\in I$, in (\ref{f102a2})\ are assumed
to be open, one can drop any arc $J_{j_{0}}$, $j_{0}\in I$, from $K$, and the
remaining set $\widetilde{K}=K\setminus J_{j_{0}}$ is still compact and
polynomial-convex, but of course, the connectivity defined by $\widetilde{K}$
is finer than that defined by $K$.\medskip
\end{remark}

\begin{proposition}
\label{p102b}Let $K_{1},K_{2},E\subset\mathbb{C}$ be three polynomial-convex
and compact sets with $\operatorname{cap}(K_{j})>0$, $j=1,2$,\ and $E\subset
K_{1}\cap K_{2}$. Let us assume further that $K_{1}$ possesses the
$S-$property in the sense of Definition \ref{d102a} on $K_{1}\setminus E$ and
that the connectivity of $E$ in $K_{1}$ is minimally coarser than the
connectivity in $K_{2}$. Then at least one of the two assertions
\begin{equation}
K_{1}\subset K_{2}\text{ \ \ \ or \ \ }\operatorname{cap}(K_{1}%
)<\operatorname{cap}(K_{2})\label{f102a4}%
\end{equation}
holds true.\medskip
\end{proposition}

The proof of Proposition \ref{p102b} will follow in the footsteps of
Proposition \ref{p102a} only that at its beginning there are differences
because of the different type of assumptions in Proposition \ref{p102a}%
.\medskip

In the next proposition a bridge is built between the set-up of Proposition
\ref{p102b} and the world of Problem $(f,\infty)$ with its admissible domains
$D\in\mathcal{D}(f,\infty)$. The assumptions of the proposition are rather
technical, but they are constructed in such a way that they fit well to the
situation in the proof of Theorem \ref{t41a}, further below, where Proposition
\ref{p102b} is needed in a crucial way.\medskip

\begin{proposition}
\label{p102c}Let the admissible domain $D\in\mathcal{D}(f,\infty)$ be
elementarily maximal in the sense of Definition \ref{d71a0} with
$\operatorname{cap}(\partial D)>0$. Set $K:=\overline{\mathbb{C}}\setminus D$,
and let $E_{0}\subset K$ denote the minimal compact and polynomial-convex set
with the property that $\partial E_{0}$ contains all points for which
assertion (i) of Definition \ref{d71a0} holds true.

Let $U\subset\mathbb{C}$ be an open set with $E_{0}\subset U$, set
$E:=\overline{U}\cap K$, and assumed that there exists a polynomial-convex and
compact set $K_{1}\subset\mathbb{C}$ satisfying the following assertions:

\begin{itemize}
\item[(i)] $\operatorname{cap}(K_{1})>0$\ and $E\subset K_{1}$.

\item[(ii)] $K_{1}$ possesses the $S-$property in the sense of Definition
\ref{d102a} on $K_{1}\setminus E$.

\item[(iii)] The connectivity of $E$ in $K_{1}$ is minimally coarser than the
connectivity of $E$ in $K$.

\item[(iv)] We have $K_{1}\setminus K\neq\emptyset$.
\end{itemize}

If these assumptions are satisfied, then there exists an admissible domain
$\widetilde{D}\in\mathcal{D}(f,\infty)$ with
\begin{equation}
\operatorname{cap}(\partial\widetilde{D})<\operatorname{cap}(\partial
D).\label{f102a5}%
\end{equation}

\end{proposition}

We now come to a proposition, which will play a technical role in the proof of
Theorem \ref{t41a}.\medskip

\begin{proposition}
\label{p102d}Let $K,E\subset\mathbb{C}$ be two polynomial-convex and compact
sets with $\operatorname{cap}(K)>0$\ and $E\subset K$. We assumed that $K$
possesses the $S-$property in the sense of Definition \ref{d102a} on
$K\setminus E$, and by $E_{1}\subset K\setminus E$\ and $\{J_{j}\}_{j\in I}$
we denote the compact discrete set and the family of open Jordan arcs
introduced in (\ref{f102a2}). We set $D:=\overline{\mathbb{C}}\setminus K$ and
defined the function $q$ by
\begin{equation}
q(z):=\left(  2\frac{\partial}{\partial z}g_{D}(z,\infty)\right)
^{2}\ \ \ \ \ \text{for \ \ }z\in\overline{\mathbb{C}}\setminus
E\label{f102a6}%
\end{equation}
with $\partial/\partial z=\frac{1}{2}\left(  \partial/\partial x-i\,\partial
/\partial y\right)  $ the usual complex differentiation and $g_{D}%
(\cdot,\infty)$ the Green function the domain $D$.

The function $q$ is analytic in $\overline{\mathbb{C}}\setminus E$ as a
consequence of the assumed $S-$proper\-ty of $K$ on $K\setminus E$, it has a
zero at each point $z\in E_{1}$, the order of each of these zeros $z\in E_{1}
$ is equal to the number of different arcs from $\{J_{j}\}_{j\in I}$ that have
$z$ as their end point minus $2$, i.e., it is of order $i(z)-2$ with $i(z)$
denoting the bifurcations index introduced in Definition \ref{d53a}, the
function $q$ is different from zero in $\overline{\mathbb{C}}\setminus(E\cup
E_{1})$ except at the critical points of $g_{D}(\cdot,\infty)$ (cf. Definition
\ref{d53b}) and at infinity, further we have the estimate
\begin{equation}
\left|  q(z)\right|  \leq\frac{3}{\operatorname{dist}(z,E)^{2}}\left(
\log(3\,r)+\log\frac{1}{\operatorname{cap}(K)}\right)  \text{ \ for all
\ }z\in\{|z|\leq r\}\setminus E\label{f102a7}%
\end{equation}
and any $r>0$ sufficiently large so that $K\subset\{|z|\leq r\}$.\medskip
\end{proposition}

The most important part of Proposition \ref{p102d} is the estimate
(\ref{f102a7}). We underlined that this estimate depends only on
$\operatorname{cap}(K)$ and the set $E$, but not on the shape or extension of
the set $K$ or the complementary domain $D$.\medskip

It has been stated in Proposition \ref{p102d} that the $S-$property of a
compact set $K$ on $K\setminus E$ implies the analyticity of the function $q$
defined in (\ref{f102a6}), is a rather immediate conclusion of (\ref{f102a3})
and (\ref{f102a6}). It is interesting that also the reverse conclusion holds
true, which is formulated in the next lemma.\medskip

\begin{lemma}
\label{l102a}Let the function $q$ be defined by (\ref{f102a6}) with the same
notations as those introduced and use in Proposition \ref{p102d}, and assume
further that $q$ is analytic in $\overline{\mathbb{C}}\setminus E$. Then the
set $K$ possesses the $S-$property on $K\setminus E$ in the sense of
Definition \ref{d102a}.\medskip
\end{lemma}

We next come to the proofs of the four Propositions \ref{p102a} - \ref{p102d}
and of Lemma \ref{l102a}.\smallskip

\subsubsection{Proof of Proposition \ref{p102a}}

\qquad The proof of Proposition \ref{p102a} is rather involved and as an
essential tool the Dirichlet integral of a Green function is used.\smallskip

\begin{proof}
[Proof of Proposition \ref{p102a}]We assume that
\begin{equation}
D_{1}\neq D_{2},\label{f102b1}%
\end{equation}
and show then that this implies (\ref{f102a1}).

Set $K_{j}:=\overline{\mathbb{C}}\setminus D_{j}$, $j=1,2$, and denote by
$\widetilde{E}_{0,j}\subset K_{j}$, $j=1,2$, the two sets of points $z\in
K_{j}$ for which assertion (i) in Definition \ref{d71a} and in Definition
\ref{d71a0} are satisfied for the domains $D_{1}$ and $D_{2}$, respectively.
Let $E_{0,j}$ be the polynomial-complex hull of $\widetilde{E}_{0,j}$, i.e.,
\begin{equation}
E_{0,j}:=\widehat{\widetilde{E}_{0,j}},\,\ \ j=1,2.\label{f102b2}%
\end{equation}
Further, we defined
\begin{equation}
K_{3}:=\widehat{K_{1}\cup K_{2}}\text{ \ and \ }D_{3}:=\overline{\mathbb{C}%
}\setminus K_{3}.\label{f102b3}%
\end{equation}
Since the domain $D_{1}$ has been assumed to possess the $S-$property, we know
from (\ref{f71a2}) in Definition \ref{d71a} that $K_{1}$ can be represented in
the form
\begin{equation}
K_{1}=E_{1,0}\cup E_{1}\cup\bigcup_{j\in I}J_{j}\label{f102b4}%
\end{equation}
with the two sets $E_{1,0}$, $E_{1}$, and the family of Jordan arcs $J_{j}$,
$j\in I$, with properties as described in Definition \ref{d71a}.

In the next step we study some properties of the two sets $K_{1}\setminus
K_{2}$ and $K_{2}\setminus K_{1}$ that follow immediately from assumption
(\ref{f102b1}). We have
\begin{equation}
(\partial K_{3}\setminus K_{2})\cap E_{0,1}=\emptyset\text{ \ \ and
\ \ }(\partial K_{3}\setminus K_{1})\cap E_{0,2}=\emptyset.\label{f102b5}%
\end{equation}
Indeed, from (\ref{f102b3}) it follows that $\partial K_{3}\subset\partial
K_{1}\cup\partial K_{2}$. Since $D_{3}\subset D_{1}\cap D_{2}$, and since both
domains $D_{1}$ and $D_{2}$ are elementally maximal, $z\in\widetilde{E}%
_{0,1}\cap\partial K_{3}$ implies $z\in\widetilde{E}_{0,2}\cap\partial K_{3}$,
and vice versa. Consequently, we have
\begin{equation}
E_{0,j}\cap\partial K_{3}\subset K_{1}\cap K_{2}\cap\partial K_{3}\text{
\ \ for \ \ \ }j=1,2,\label{f102b6}%
\end{equation}
which proves (\ref{f102b5}).

We have
\begin{equation}
\operatorname{cap}(K_{2}\setminus K_{1})>0\text{ \ \ and \ \ }%
\operatorname{cap}(K_{1}\setminus K_{2})>0.\label{f102b7}%
\end{equation}
Here, we first prove that $\operatorname{cap}(K_{2}\setminus K_{1})>0$. This
will be done in an indirect way, and we assume for this purpose that
\begin{equation}
\operatorname{cap}(K_{2}\setminus K_{1})=0.\label{f102b8}%
\end{equation}

Let $f_{j}$ denote the meromorphic continuations of the function $f$ into the
domains $D_{j}$, $j=1,2,3$. With the same arguments as used in the proof of
Lemma \ref{l92a} in Subsection \ref{s92},\ we can show that assumption
(\ref{f102b8}) implies that all meromorphic continuations of the function
$f_{2}$ out of $D_{1}\setminus K_{2}$ into $D_{1}$\ lead to the same function
in each component of the open set $D_{1}\setminus K_{2}$. These functions then
are necessarily identical with the function $f_{1}$. Since the domain $D_{2}$
has been assumed to be elementarily maximal, it follows from assertion (ii) in
Definition \ref{d71a0} that $K_{2}\setminus K_{1}=\emptyset$, which implies
that $K_{2}\subset K_{1}$. As a consequence of the assumed $S-$property of the
domain $D_{1}$, we know that $D_{1}$ is also elementarily maximal (see the
assertions (i) and (ii) in Definition \ref{d71a}), and therefore $D_{2}\supset
D_{1}$ implies that $D_{1}=D_{2}$. This last conclusion contradicts
(\ref{f102b1}), and consequently we have proved $\operatorname{cap}%
(K_{2}\setminus K_{1})>0$ in (\ref{f102b7}). A proof of $\operatorname{cap}%
(K_{1}\setminus K_{2})>0$ in (\ref{f102b7}) can be done in exactly the same way.

From (\ref{f102b4}) and the fact that the two sets $\partial K_{3}\setminus
K_{2}$ and $E_{0,1}$ are disjoint, which has been proved in (\ref{f102b5}), we
immediately conclude that
\begin{equation}
\partial K_{3}\setminus K_{2}\subset E_{1}\cup\bigcup_{j\in I}J_{j}%
\label{f102b9}%
\end{equation}
i.e., $\partial K_{3}\setminus K_{2}$ is the union of open subarcs of arcs
from the family $\{J_{j}\}_{j\in I}$ together with points from $E_{1}$. The
dominant parts in this union are the open Jordan arcs since the set $E_{1}$ is
countable, and therefore we have $\operatorname{cap}(E_{1})=0$.

We now continue our investigation with further definitions. We set
\begin{equation}
V:=\overline{\operatorname{Int}(K_{3})\cap(K_{1}\setminus K_{2})},\text{
\ \ }\widetilde{D}_{2}:=D_{2}\setminus V\text{, \ \ }\widetilde{K}_{2}%
:=K_{2}\cup V=\mathbb{C}\setminus\widetilde{D}_{2}.\label{f102c1}%
\end{equation}

It is immediate that $\widetilde{D}_{2}$ is open, but it is not necessarily a
domain. By $g_{j}(\cdot,\cdot)$ we denote the Green functions $g_{D_{j}}%
(\cdot,\cdot)$ in the domains $D_{j}$, $j=1,2$. Because of (\ref{f102b7}),
these two Green functions exist in a proper sense (see Subsection \ref{s1103},
further below).

Next, we show that
\begin{equation}
K_{1}\cap\widetilde{D}_{2}\subset\bigcup_{j\in I}J_{j}\cap\partial
\operatorname{Int}(K_{3}),\label{f102c2}%
\end{equation}
or more precisely, we show that $K_{1}\cap\widetilde{D}_{2}$ consists only of
open subarcs of the arcs $J_{j}$ from (\ref{f102b9}) that are contained in
$\partial\operatorname{Int}(K_{3})$. Indeed, since $K_{1}$ possesses the
$S-$property of Definition \ref{d71a}, it follows from assertion (ii) in
Definition \ref{d71a} that $K_{1}\subset\overline{\operatorname{Int}(K_{3})}$.
It further follows from the definitions in (\ref{f102c1}) that $K_{1}%
\cap\widetilde{D}_{2}\subset\partial K_{3}\setminus K_{2}$. Because of
(\ref{f102b9}), it remains only to show that $K_{1}\cap\widetilde{D}_{2}\cap
E_{1}=\emptyset$. Let us assume that $z\in K_{1}\cap\widetilde{D}_{2}\cap
E_{1}$. From assertion (iv) in Definition \ref{d71a} of the $S-$property we
know that at least three different arcs of the family $\{J_{j}\}$ in
(\ref{f102b4})\ have $z$ as endpoint. Since $K_{1}\cap\widetilde{D}_{2}$ lies
in $\partial K_{3}\setminus K_{2}$, the meromorphic continuations of the two
functions $f_{1}$ and $f_{2}$ out of the domain $D_{3}$ are identical, and
therefore, at least one of the arcs ending at $z$ belongs to $V$; and
consequently, we have $z\in V$, which contradicts $z\in K_{1}\cap\widetilde
{D}_{2}$. Thus, (\ref{f102c2}) is proved.

A key role in the proof of the proposition is played by the function
$\widetilde{g}_{1}$, which is defined as
\begin{equation}
\widetilde{g}_{1}(z):=\left\{
\begin{array}
[c]{lll}%
g_{1}(z,\infty)\medskip & \text{ \ \ for \ \ } & z\in D_{3},\\
-g_{1}(z,\infty) & \text{ \ \ for \ \ } & z\in K_{3}\setminus K_{2}.
\end{array}
\right. \label{f102c3}%
\end{equation}
All discussions, so far, can be seen as preliminaries to an investigation of
properties of the function $\widetilde{g}_{1}$. In this connection the
$S-$property of the domain $D_{1}$ is very important since it implies that the
two pieces in the definition of the function $\widetilde{g}_{1}$ are harmonic
continuations of each other across the arcs in $K_{1}\cap\widetilde{D}_{2}$.

Indeed, from symmetry (\ref{f71a1}) in Definition \ref{d71a} together with
(\ref{f102c2}) and the remarks just after (\ref{f102c2}), we conclude that
$\widetilde{g}_{1}$ is harmonic in $\widetilde{D}_{2}$. Notice that the domain
$D_{1}$ is assume to possess the $S-$property.

The function $\widetilde{g}_{1}$ is superharmonic in $D_{2}$. Indeed, from
Lemma \ref{l113b} in Subsection \ref{s1103}, further below, we know that the
Green function $g_{1}(z,\infty)$ is subharmonic in $\mathbb{C}$. Since
$V\subset K_{3}\setminus K_{2}$, the superharmonicity follows directly from
(\ref{f102c3}).

From the defining identity (\ref{f113a1}) of the Green function in Subsection
\ref{s1103}, further below, together with (\ref{f102c3}) and the definition of
$V$ in (\ref{f102c1}), we conclude that
\begin{equation}
\widetilde{g}_{1}(z)=0\ \ \ \text{for quasi every}\ \ \ z\in V.\label{f102c4}%
\end{equation}

From the properties of $\widetilde{g}_{1}$ that have just been discussed
together with the Poison-Jensen Formula (cf. Theorem \ref{t113a} in Subsection
\ref{s1103}, further below) we get for $\widetilde{g}_{1}$ the representation
\begin{equation}
\widetilde{g}_{1}(z)=\widetilde{h}_{1}(z)+g_{2}(z,\infty)+g_{0}(z)\text{
\ \ for \ \ }z\in D_{2},\label{f102c5}%
\end{equation}
where $g_{2}(\cdot,\infty)$ is the Green function in $D_{2}$, and $g_{0}$ the
Green potential
\begin{equation}
g_{0}(z)=\int_{V}g_{2}(z,v)d\omega_{1}(v)\label{f102c6}%
\end{equation}
with $\omega_{1}$ the equilibrium distribution on $K_{1}$, and $V$ the set
from (\ref{f102c1}). The function $\widetilde{h}_{1}$ in (\ref{f102c5}) is the
solution of the Dirichlet problem in $D_{2}$ with boundary values
\begin{equation}
\widetilde{h}_{1}(z)=\widetilde{g}_{1}(z)\text{ \ \ for quasi every \ \ }%
z\in\partial K_{2}.\label{f102c7}%
\end{equation}
Identity (\ref{f102c5})\ can easily be verified by considering its values on
$\partial D_{2}$, on $V$, and near infinity.

From (\ref{f102b7}) and Lemma \ref{l113c} in Subsection \ref{s1103}, further
below, we deduce that the two Green functions $g_{1}(\cdot,\infty)$ and
$g_{2}(\cdot,\infty)$ are essentially different, and consequently, we have
\begin{equation}
g_{1}(z,\infty)>0\text{ \ \ for quasi every \ \ }z\in K_{2}\setminus
K_{1}.\label{f102c8}%
\end{equation}
From (\ref{f102c7}) and (\ref{f102c3}), we then conclude that
\begin{equation}
\widetilde{h}_{1}\neq0.\label{f102c9}%
\end{equation}

By definition, we have $g_{0}(z)\geq0$ for all $z\in\overline{D}_{2}$, and
$g_{0}(z)>0$\ for $z\in D_{2}$\ if, and only if, $\omega_{1}(V)>0$. However,
this last condition may in general not be satisfied; even the case
$V=\emptyset$ is possible.

In order to prove (\ref{f102a1}), we prove that the identity
\begin{align}
\log\frac{\operatorname{cap}(K_{2})}{\operatorname{cap}(K_{1})}  &
=D_{K_{2}\setminus K_{1}}(g_{1}(\cdot,\infty))+D_{D_{2}}(\widetilde{h}%
_{1})+2\,g_{0}(\infty)\label{f102d1}\\
& +\int_{V}\int_{V}g_{2}(v,w)d\omega_{1}(v)d\omega_{1}(w)\nonumber
\end{align}
holds true, where $D_{K_{2}\setminus K_{1}}(g_{1}(\cdot,\infty))$ and
$D_{D_{2}}(\widetilde{h}_{1})$ are Dirichlet integrals that have been
introduced in (\ref{f113f3}) in Subsection \ref{s1103}, further below.

Since we know from (\ref{f102c9}) that the harmonic function $\widetilde
{h}_{1}$ is not identical zero in $D_{2}$, it follows from the definition of
the Dirichlet integral in (\ref{f113f3}) that
\begin{equation}
D_{K_{2}\setminus K_{1}}(g_{1}(\cdot,\infty))\geq0\ \ \ \text{and
}\ \ D_{D_{2}}(\widetilde{h}_{1})>0.\label{f102d2}%
\end{equation}

We have already mentioned that the Green potential $g_{0}$ is always
non-negative. It follows from (\ref{f102c6}) and the positivity of the Green
function as kernel function (see Lemma \ref{l113d} in Subsection \ref{s1103},
further below) that
\begin{equation}
g_{0}(\infty)\geq0,\text{ \ \ \ }\int_{V}\int_{V}g_{2}(v,w)d\omega
_{1}(v)d\omega_{1}(w)\geq0,\label{f102d3}%
\end{equation}
and we have proper inequalities in both cases of (\ref{f102d3}) if, and only
if, $\omega_{1}(V)>0$.

When identity (\ref{f102d1}) is proved, then inequality (\ref{f102a1}) follows
immediately from identity (\ref{f102d1}) together with (\ref{f102d2}) and
(\ref{f102d3}). Hence, it remains only to prove that (\ref{f102d1}) holds
true, which will be done next.

From Lemma \ref{l113b}, Lemma \ref{l113g}, and Corollary \ref{c113g1} in
Subsection \ref{s1103}, further below, we deduce that
\begin{align}
\log\frac{\operatorname{cap}(K_{2})}{\operatorname{cap}(K_{1})}  & =\left(
g_{1}(\cdot,\infty)-g_{2}(\cdot,\infty)\right)  (\infty)\label{f102e1}\\
& =D_{D_{1,r}}(g_{1}(\cdot,\infty))-D_{D_{2,r}}(g_{2}(\cdot,\infty
))+\text{O}(\frac{1}{r})\nonumber
\end{align}
as $r\rightarrow\infty$,\ where $D_{j,r}$ denotes the bounded domain
\begin{equation}
D_{j,r}:=D_{j}\cap\{|z|<r\}\text{ \ \ for \ \ }j=1,2,\label{f102e3}%
\end{equation}
and $r>0$ so large that $K_{j}\subset\{|z|<r\}$ for $j=1,2,3$.

From the definition of the Dirichlet integral in (\ref{f113f3}), further
below, and the definition of the function $\widetilde{g}_{1}$ in
(\ref{f102c3}), we deduce that
\begin{align}
D_{D_{1,r}}(g_{1}(\cdot,\infty))  & =D_{K_{2}\setminus K_{1}}(g_{1}%
(\cdot,\infty))+D_{\widetilde{D}_{2,r}}(g_{1}(\cdot,\infty))\label{f102e2}\\
& D_{K_{2}\setminus K_{1}}(g_{1}(\cdot,\infty))+D_{\widetilde{D}_{2,r}%
}(\widetilde{g}_{1}).\nonumber
\end{align}
In (\ref{f102e2}), the open set $\widetilde{D}_{2,r}$ is defined as
$\widetilde{D}_{2,r}:=$ $\widetilde{D}_{2}\cap\{|z|<r\}$ in analogy to
(\ref{f102e3}). It has already been mentioned in (\ref{f102d2}) that
$D_{K_{2}\setminus K_{1}}(g_{1}(\cdot,\infty))\geq0$.\medskip

We will now have a closer look on the Dirichlet integral $D_{\widetilde
{D}_{2,r}}(\widetilde{g}_{1})$ in (\ref{f102e2}). From the representations
(\ref{f102c5}), (\ref{f102c6}), and equality (\ref{f102c4}), it follows with
the help of Lemma \ref{l113j} in Subsection \ref{s1103}, further below, that
\begin{equation}
D_{\widetilde{D}_{2,r}}(\widetilde{g}_{1})=D_{D_{2,r}}(\widetilde{h}_{1}%
+g_{2}(\cdot,\infty)+g_{0,r})+\int_{V}\int_{V}g_{2}(v,w)d\omega_{1}%
(v)d\omega_{1}(w)\label{f102e4}%
\end{equation}
with $\omega_{1}$ and $V$ defined like in (\ref{f102c6}). In (\ref{f102e4}),
the function $g_{0,r}$ denotes the solution of the Dirichlet problem in
$D_{2,r}$ with boundary values $g_{0,r}(z)=g_{0}(z)$ for quasi every
$z\in\partial D_{2,r}$. We know therefore from (\ref{f102c6}) that
\begin{equation}
g_{0,r}(z)=\left\{
\begin{array}
[c]{lll}%
0\medskip & \text{ \ \ for quasi every \ \ } & z\in\partial D_{2},\\
g_{0}(z) & \text{ \ \ for \ \ } & |z|=r.
\end{array}
\right. \label{f102e5}%
\end{equation}

Next, we investigate the Dirichlet integral $D_{D_{2,r}}(\widetilde{h}%
_{1}+g_{2}(\cdot,\infty)+g_{0,r})$ in (\ref{f102e4}) in more detail. Using the
notations introduced in (\ref{f102c5}), (\ref{f102e4}), and also in
(\ref{f113f2}), further below, we prove the identity
\begin{align}
& D_{D_{2,r}}(\widetilde{h}_{1}+g_{2}(\cdot,\infty)+g_{0,r})\nonumber\\
& \ \ \ \ =D_{D_{2,r}}(\widetilde{h}_{1})+D_{D_{2,r}}(g_{2}(\cdot
,\infty))+D_{D_{2,r}}(g_{0,r})+\label{f102e6}\\
& \text{ \ \ \ \ \ \ \ \ }+2\,D_{D_{2,r}}(\widetilde{h}_{1},g_{2}(\cdot
,\infty))+2\,D_{D_{2,r}}(\widetilde{h}_{1},g_{0,r})+2\,D_{D_{2,r}}%
(g_{0,r},g_{2}(\cdot,\infty)).\nonumber
\end{align}

Indeed, the positivity of the integrand in the Dirichlet integral implies
that
\begin{equation}
\lim_{r\rightarrow\infty}D_{D_{2,r}}(\widetilde{h}_{1})=D_{D_{2}}%
(\widetilde{h}_{1})>0,\label{f102e7}%
\end{equation}
and since $\widetilde{h}_{1}$ is harmonic and bounded in $D_{2}$, the
Dirichlet integral on the right-hand side of (\ref{f102e7}) is finite.

The Green potential $g_{0}$ is bounded near infinity, therefore it follows
from the definition of $g_{0,r}$ as a solution of a Dirichlet problem in
$D_{2,r}$ with boundary values (\ref{f102e5}) that
\begin{equation}
\lim_{r\rightarrow\infty}g_{0,r}(z)=0\text{ \ \ locally uniformly for
\ \ \ }z\in D_{2},\label{f102e8}%
\end{equation}
and the same conclusion holds also for the first derivatives in $\nabla
g_{0,r}$ and $r\rightarrow\infty$; these derivatives converge also to zero
locally uniformly in $D_{2}$. As a consequence, we have
\begin{equation}
\lim_{r\rightarrow\infty}D_{D_{2,r}}(g_{0,r})=0.\label{f102e9}%
\end{equation}
With the Cauchy-Schwartz inequality and the boundedness of $D_{D_{2,r}%
}(\widetilde{h}_{1})$ we deduce from (\ref{f102e9}) that we also have
\begin{equation}
\lim_{r\rightarrow\infty}D_{D_{2,r}}(\widetilde{h}_{1},g_{0,r}%
)=0.\label{f102e10}%
\end{equation}

The function $\widetilde{h}_{1}$ is harmonic in $D_{2}$, and therefore it
follows from Lemma \ref{l113h} in Subsection \ref{s1103}, further below, that
\begin{equation}
\lim_{r\rightarrow\infty}D_{D_{2,r}}(\widetilde{h}_{1},g_{2}(\cdot
,\infty))=0.\label{f102e11}%
\end{equation}

Since the Green potential $g_{0}$ is harmonic in $\{|z|>r\}$, we deduce from
(\ref{f102e5}) and Lemma \ref{l113i} in Subsection \ref{s1103}, further below,
that
\begin{equation}
D_{D_{2,r}}(g_{0,r},g_{2}(\cdot,\infty))=\frac{1}{2\pi}\int_{\{|z|=r\}}%
g_{0}(z)\frac{\partial}{\partial n}g_{2}(z,\infty)ds_{z}=g_{0}(\infty
)\label{f102e12}%
\end{equation}
for $r>0$\ sufficiently large.

From identity (\ref{f102e6}) together with (\ref{f102e7}), (\ref{f102e9}),
(\ref{f102e10}), (\ref{f102e11}), and (\ref{f102e12}), we get
\begin{equation}
\lim_{r\rightarrow\infty}\left(  D_{D_{2,r}}(\widetilde{h}_{1}+g_{2}%
(\cdot,\infty)+g_{0,r})-D_{D_{2,r}}(g_{2}(\cdot,\infty))\right)  =D_{D_{2}%
}(\widetilde{h}_{1})+2\,g_{0}(\infty).\label{f102e13}%
\end{equation}
Further, we then get from formula (\ref{f102e1}) together with (\ref{f102e2}),
(\ref{f102e4}), and (\ref{f102e13}) that identity (\ref{f102d1}) holds true,
which completes the proof of the proposition.\medskip
\end{proof}

\subsubsection{Proof of Proposition \ref{p102b}}

\qquad It has already been mentioned that the proof of Proposition \ref{p102b}
follows in the footsteps of that of Proposition \ref{p102a} only that we have
a modified opening since we have to start from a different type of
assumptions. But after the introduction of the function $\widetilde{g}_{1} $
in (\ref{f102c3}), we can use the argumentation of the last proof without any
change.\medskip

\begin{proof}
[Proof of Proposition \ref{p102b}]We assume
\begin{equation}
K_{1}\nsubseteq K_{2},\label{f102f1}%
\end{equation}
and show then that this implies $\operatorname{cap}(K_{1})<\operatorname{cap}%
(K_{2})$.

Like in the proof of Proposition \ref{p102a}, we set
\begin{equation}
K_{3}:=\widehat{K_{1}\cup K_{2}}\text{ \ \ and \ \ \ }D_{j}:=\overline
{\mathbb{C}}\setminus K_{j},j=1,2,3,\label{f102f2}%
\end{equation}
where $\widehat{\cdot}$ denotes the polynomial-convex hull (cf. Definition
\ref{d111b} in Subsection \ref{s1101}, further below). Since it has been
assumed that the compact set $K_{1}$ possesses the $S-$property on
$K_{1}\setminus E$ in the sense of Definition \ref{d102a}, in $K_{1}$ we have
the compact set $E_{1}$ and the family of Jordan arcs $\{J_{j}\}_{j\in I}%
$\ that have been introduced in (\ref{f102a2}) of Definition \ref{d102a}, and
from there we then know that
\begin{equation}
K_{1}=E\cup E_{1}\cup\bigcup_{j\in I}J_{j}.\label{f102f3}%
\end{equation}
From assumption (\ref{f102f1}), representation (\ref{f102f3}), assumption
$\operatorname{cap}(K_{1})>0$, $E\subset K_{2}$, and the fact that the
capacity of an open Jordan arc is positive (cf. Lemma \ref{l111a} in
Subsection \ref{s1101}, further below), it follows that
\begin{equation}
\operatorname{cap}(K_{1}\setminus K_{2})>0.\label{f102f4}%
\end{equation}

The boundary $\partial C$ of any component $C$ of the open set
\begin{equation}
O:=\operatorname{Int}(K_{3})\setminus K_{2}\label{f102f5}%
\end{equation}
contains elements of $\partial K_{1}$ and $\partial K_{2}$ because of the
assumed polynomial-convexity of $K_{1}$. We have
\begin{equation}
(K_{1}\setminus K_{2})\cap\partial O\subset K_{1}\setminus E,\label{f102f6}%
\end{equation}
and if we define the set $V$\ by
\begin{equation}
V:=\overline{\operatorname{Int}(K_{3})\cap(K_{1}\setminus K_{2})}%
,\label{f102f7}%
\end{equation}
then it follows from (\ref{f102f3}), (\ref{f102f6}), and (\ref{f102f7}) that
\begin{equation}
(K_{1}\setminus(K_{2}\cup V))\cap\partial O\subset\bigcup_{j\in I}%
J_{j}.\label{f102f8}%
\end{equation}
On the other hand, we have
\begin{equation}
K_{1}\setminus(K_{2}\cup V)\subset\partial O,\label{f102f9}%
\end{equation}
since otherwise any arc $J_{j}$, $j\in I$, in $K_{1}\setminus(K_{2}\cup
V\cup\partial O)$ could be removed from $K_{1}$ without separating any pair of
components of the set $E$ in the modified set $K_{1}$ if this pair is already
connected in $K_{2}$. But such a situation would contradict the assumption
that the connectivity of the set $E$ in $K_{1}$ is minimal coarser than the
connectivity in $K_{2}$ (cf. Definition \ref{d102b}).

Like in (\ref{f102c1}) in the proof of Proposition \ref{p102a}, we define
\begin{equation}
\widetilde{D}_{2}:=D_{2}\setminus V\text{ \ and\ \ }\widetilde{K}_{2}%
:=K_{2}\cup V=\mathbb{C}\setminus\widetilde{D}_{2}\label{f102f10}%
\end{equation}
with the compact set $V$ introduced in (\ref{f102f7}). It follows from
(\ref{f102f8}) and (\ref{f102f9}) that
\begin{equation}
K_{1}\cap\widetilde{D}_{2}\subset\bigcup_{j\in I}J_{j}\cap\partial
\operatorname{Int}(K_{3}).\label{f102f11}%
\end{equation}
Indeed, because of (\ref{f102f9}) we have $K_{1}\cap\widetilde{D}_{2}%
=K_{1}\setminus(K_{2}\cup V)\subset\partial O\setminus K_{2}\subset
\partial\operatorname{Int}(K_{3})$, and because of (\ref{f102f8}) together
with (\ref{f102f9}) we have $K_{1}\cap\widetilde{D}_{2}=K_{1}\setminus
(K_{2}\cup V)\subset\bigcup_{j\in I}J_{j}$.

After these preparations we can now follow the argumentation in the proof of
Proposition \ref{p102a} word-for-word. Exactly, like in (\ref{f102c3}) we
define the function $\widetilde{g}_{1}$ in the domain $D_{2}$. With the same
arguments as those used in the proof of Proposition \ref{p102a} after
(\ref{f102c3}) we then arrive after many intermediate steps at the conclusion
that
\begin{equation}
\operatorname{cap}(K_{1})<\operatorname{cap}(K_{2}),\label{f102f12}%
\end{equation}
which proves Proposition \ref{p102b}.\medskip
\end{proof}

\subsubsection{Proof of Proposition \ref{p102c}}

\qquad The proof of Proposition \ref{p102c} relies strongly on Definition
\ref{d95a} and the subsequent Theorem \ref{t95a} together with Proposition
\ref{p102b}.\medskip

\begin{proof}
[Proof of Proposition \ref{p102c}]We set
\begin{equation}
K_{0}:=K,\text{ \ \ \ \ }K_{3}:=\widehat{K_{0}\cup K_{1}},\text{ \ \ and
\ \ \ }D_{j}:=\overline{\mathbb{C}}\setminus K_{j},j=0,1,3.\label{f102g1}%
\end{equation}
In (\ref{f102g1}), $\widehat{\cdot}$ denotes the polynomial-convex hull (cf.
Definition \ref{d111b} in Subsection \ref{s1101}, further below). It is
immediate that the $D_{j}$ are domains.

From Proposition \ref{p102b} together with the assumptions (i) - (iv) of the
proposition, it follows that
\begin{equation}
\operatorname{cap}(K_{1})<\operatorname{cap}(K_{0}).\label{f102g2}%
\end{equation}

If we would know that $D_{1}=\overline{\mathbb{C}}\setminus K_{1}%
\in\mathcal{D}(f,\infty)$, then the proof of the proposition would be
completed with (\ref{f102g2}). However, in general we have $D_{1}%
\notin\mathcal{D}(f,\infty)$ since the meromorphic continuation of the
function $f $ out of the domain $D_{3}$ into $D_{1}$ may hit non-polar
similarities in $D_{1}\cap\operatorname{Int}(K_{3})$.

In order to overcome these difficulties we make use of a construction
introduced in Definition \ref{d95a}. We consider the whole family of compact
sets $K_{h}$, $h\in\left[  0,1\right]  $, with complementary domains
$D_{h}=\overline{\mathbb{C}}\setminus K_{h}$ that are defined by the relations
(\ref{f95b1}) through (\ref{f95b7}) of Definition \ref{d95a} starting from the
two domains $D_{0}$ and $D_{1}$ defined in (\ref{f102g1}), and with elements
of the proof of Theorem \ref{t95a} we then prove that there exist $h_{0}%
\in\left(  0,1\right)  $ such that
\begin{equation}
D_{h}\in\mathcal{D}(f,\infty)\text{ \ \ for all \ }0<h\leq h_{0}%
,\label{f102g3}%
\end{equation}
and further that
\begin{equation}
\log\operatorname*{cap}(K_{h})<(1-h)\,\log\operatorname*{cap}(K_{0}%
)+h\,\log\operatorname*{cap}(K_{1})\text{ \ \ for \ \ }0<h<1.\label{f102g4}%
\end{equation}

From (\ref{f102g4}) we deduce that
\begin{equation}
\log\operatorname*{cap}(K_{h_{0}})<\log\operatorname*{cap}(K_{0})-h_{0}%
\log\frac{\operatorname*{cap}(K_{0})}{\operatorname*{cap}(K_{1})}%
,\label{f102g5}%
\end{equation}
which together with (\ref{f102g2}) proves that
\begin{equation}
\operatorname{cap}(K_{h_{0}})<\operatorname{cap}(K_{0}).\label{f102g6}%
\end{equation}
If we set $\widetilde{D}:=D_{h_{0}}$, then the proposition follows from
(\ref{f102g6}) together with (\ref{f102g3}). Hence, it remains only to prove
that the two assertions (\ref{f102g3}) and (\ref{f102g4}) hold true.

The two assertions (\ref{f102g3}) and (\ref{f102g4}) have already been proved
as assertion (i) and (ii) in Theorem \ref{t95a}, but under partly different
assumptions. The difference is the following: In Theorem \ref{t95a}, it has
been assumed that both domains $D_{0}$ and $D_{1}$ belong to $\mathcal{D}%
(f,\infty)$, while in the present situation, we do not know whether $D_{1}%
\in\mathcal{D}(f,\infty)$. Instead, we now have the assumptions (i) - (iv), of
which (ii) and (iii) are the two most important ones.

In the a step we prove that (\ref{f102g3}) holds true. As in the proof of
Lemma \ref{l93c}, where an analogous result has been proved for Proposition
\ref{p91a}, the main tool will again be Proposition \ref{p91a}.

It follows from assumption (iii) and the assumption that $D_{0}\in
\mathcal{D}(f,\infty)$ together with Proposition \ref{p91a} that for every
Jordan curve $\gamma\in\Gamma_{1}$, with $\Gamma_{1}$ introduced in Definition
\ref{d91b}, and $\gamma\subset\overline{\mathbb{C}}\setminus E_{0}$, we have
\begin{equation}
\gamma\cap K_{0}\neq\emptyset\text{ \ \ and \ \ }\gamma\cap K_{1}\neq
\emptyset.\label{f102g7}%
\end{equation}
From the defining relations (\ref{f95b1}) - (\ref{f95b7}) in Definition
\ref{d95a} for the sets $K_{h}$, we further conclude that besides of
(\ref{f102g7}) we also have
\begin{equation}
\gamma\cap K_{h}\neq\emptyset\text{ \ \ for all \ \ \ }h\in\left(  0,1\right)
.\label{f102g8}%
\end{equation}
Details of the argumentation are the same as those in the proof of Lemma
\ref{l93a}.

From Proposition \ref{p91a} and (\ref{f102g8}) it is clear that for a proof of
(\ref{f102g3}) it remains only to show that assertion (i) in Proposition
\ref{p91a} holds true for $0\leq h\leq h_{0}$, i.e., we have to show that
there exists $h_{0}\in\left(  0,1\right)  $ such that the function $f$ has a
meromorphic continuation to each point of $D_{h}$ for $0\leq h\leq h_{0}$.

We defined the two sets $B_{jh}$, $j=0,1$, $h\in\left(  0,1\right)  $, by
\begin{align}
B_{0,h}  & :=\{\,z\in D_{h}\cap K_{3}\,|\,(1-h)g_{0}(z)>hg_{1}%
(z)\,\},\label{f102g9}\\
B_{1,h}  & :=\{\,z\in D_{h}\cap K_{3}\,|\,(1-h)g_{0}(z)<hg_{1}%
(z)\,\}.\label{f102g10}%
\end{align}
where $g_{j}$ denotes the Green function $g_{D_{j}}(\cdot,\infty)$, $j=0,1$,
in the same way as in Definition \ref{d95a}. In the same way as in the proof
of Lemma \ref{l93b}, it follows from (\ref{f95b4}), (\ref{f95b5}),
(\ref{f95b6}), and (\ref{f95b7}) that the two sets $D_{3}\cap B_{jn}$,
$j=0,1$, are domains, and we have
\begin{equation}
B_{jh}\subset D_{j}\text{ \ \ \ for \ \ }j=0,1\text{, }h\in\left(  0,1\right)
.\label{f102g11}%
\end{equation}
Since it has been assumed that $D_{0}\in\mathcal{D}(f,\infty)$, the function
$f$ processes a meromorphic continuation into the domain $D_{0}$, which we
denote by $f_{0}$. From (\ref{f102g11}) it follows that the function $f_{0}$
is defined throughout $D_{3}\cup B_{0,h}$.

An analogous argumentation is unfortunately not possible for the domain
$D_{1}$. Here, we have to follow a different path of argumentation. From the
assumed properties of the set $E_{0}\subset\overset{\circ}{E}$ it follows that
there exists a domain $D_{2}$ with the property that
\begin{equation}
D_{2}\supset D_{3},\text{ \ \ }K_{0}\setminus E_{0}\subset D_{2}%
,\label{f102g12}%
\end{equation}
and the function $f$ can meromorphically be continued to the domain $D_{2}$.
We denote this continuation by $f_{2}$.

In the next step we show that there exists $h_{0}\in\left(  0,1\right)  $ such
that
\begin{equation}
K_{h}\setminus\operatorname{Int}(E)\subset D_{2}\text{ \ \ for all
\ \ \ }0\leq h\leq h_{0}.\label{f102g13}%
\end{equation}
Let $\widetilde{U}_{0}\subset D_{2}$ be an open set consisting only of simply
connected components and assume that $K_{0}\setminus U\subset\widetilde{U}%
_{0}$, where $U$ is the open set introduced in the formulation of the
proposition. There exists an open set $U_{0}\subset\mathbb{C}$ consisting only
of simply connected components such that
\begin{equation}
K_{0}\subset U_{0}\text{, \ \ \ }U_{0}\cap D_{0}\subset\widetilde{U}%
_{0}\text{, \ and \ }U_{0}\cap\partial U\subset\widetilde{U}_{0}%
.\label{f102g14}%
\end{equation}
With exactly the same arguments as those applied in the proof of assertion
(iii) of Theorem \ref{t95a} we then show that there exists $h_{0}\in\left(
0,1\right)  $ such that
\begin{equation}
K_{h}\subset U_{0}\text{ \ \ for all \ \ \ }0\leq h\leq h_{0}.\label{f102g15}%
\end{equation}
Notice that in the proof of assertion (iii) of Theorem \ref{t95a} only
topological assumptions about the two sets $K_{0}$, $K_{1}$, and their
complementary domains $D_{0}$ and $D_{1}$ have been used. The inclusion
(\ref{f102g13}) follows directly from (\ref{f102g15}).

Since $U_{0}$ consists only of simply connected components, it follows from
(\ref{f102g13}), (\ref{f102g15}), and the definition of the set $B_{1,h}$ in
(\ref{f102g10}) that
\begin{equation}
B_{1,h}\subset D_{2}\text{ \ \ for \ \ \ }0\leq h\leq h_{0},\label{f102g16}%
\end{equation}
and consequently the functions $f_{2}$ is defined throughout the domain
$D_{3}\cup B_{1,h}$ for all $0\leq h\leq h_{0}$.

Since $D_{h}=D_{3}\cup B_{0,h}\cup B_{0,h}$, we have shown that the function
$f$ possesses are meromorphic continuation to each point $z\in D_{h}$ for
$0\leq h\leq h_{0}$. Hence, assumption (i) of Proposition \ref{p91a} holds
true for each $0\leq h\leq h_{0}$, and (\ref{f102g3}) then follows from
Proposition \ref{p91a}.\medskip

After the verification of (\ref{f102g3}), it remains only to prove that the
inequality (\ref{f102g4}) holds true. Here, we copy the corresponding proof of
(\ref{f95c2}) from the proof of Theorem \ref{t95a} word for word. The
condition (\ref{f95c4}) in Theorem \ref{t95a} follows from the two assumptions
(i) and (iv) in Proposition \ref{p102c}. A detailed argumentation for this
last conclusion has been given after (\ref{f102b7}). With the proof of
(\ref{f102g4}), the whole proof of Proposition \ref{p102c} is
completed.\bigskip
\end{proof}

\begin{proof}
[Proof of Proposition \ref{p102d}]It is rather immediate that the $S-$property
of $K$ on $K\setminus E$ implies that the function $q$ is analytic in
$\overline{\mathbb{C}}\setminus E$.

From the definition of $q$ in (\ref{f102a6}) it follows that level-lines of
the Green function $g_{D}(\cdot,\infty)$ are trajectories of the quadratic
differential $q(z)dz^{2}$ (for a definition of trajectories see (\ref{f52a})
in Section \ref{s52}). The arcs $J_{j}$, $j\in I$, in $K\setminus E$ are
critical trajectories of $q(z)dz^{2}$, and, of course, they are also
level-lines of $g_{D}(\cdot,\infty)$ corresponding to the value $0$. From the
local structure of trajectories, it follows that at each bifurcation point
$z\in\overline{\mathbb{C}}\setminus E$ of trajectories, we have a zero of
order $i(z)-2$, where $i(z)$ is the bifurcation index of Definition \ref{d53a}
(cf. \cite{Jensen75}, Chapter 8.2, or \cite{Strebel84}).

That the only zeros of $q$ in $\overline{\mathbb{C}}\setminus(E\cup E_{1})$
are critical points of the Green function $g_{D}(\cdot,\infty)$ in the sense
of Definition \ref{d53b} is an immediate consequence of (\ref{f102a6}), and
the same is also true for the double zero of $q$ at infinity.

We now come to the proof of the inequality (\ref{f102a7}). From (\ref{f102a6})
and we deduce that
\begin{align}
|q(z)|  & =4\frac{\partial}{\partial z}g_{D}(z,\infty)\overline{\frac
{\partial}{\partial z}g_{D}(z,\infty)}\nonumber\\
& =(\frac{\partial}{\partial x}-i\,\frac{\partial}{\partial y})g_{D}%
(z,\infty)(\frac{\partial}{\partial x}+i\,\frac{\partial}{\partial y}%
)g_{D}(z,\infty)\label{f102h1}\\
& =(\frac{\partial}{\partial x})g_{D}(z,\infty))^{2}+(\frac{\partial}{\partial
y})g_{D}(z,\infty))^{2}\nonumber
\end{align}
for $z\in\overline{\mathbb{C}}\setminus E$, and consequently we have the
estimate
\begin{align}
|q(z)|  & =\frac{1}{\pi d^{2}}\left\vert \iint_{\Delta(z,d)}q(\zeta)dm_{\zeta
}\right\vert \leq\frac{1}{\pi d^{2}}\iint_{\Delta(z,d)}|q(\zeta)|dm_{\zeta
}\nonumber\\
& =\frac{2}{d^{2}}D_{\Delta(z,d)}(g_{D}(\cdot,\infty))\label{f102h2}%
\end{align}
for every $z\in\mathbb{C}\setminus E$ with $0<d<\operatorname{dist}(z,E)$,
$\Delta(z,d):=\{\,\zeta\,|$ $|\zeta-z|\leq d\,\}$, $dm_{\zeta}$ the area
element at the point $\zeta\in\mathbb{C}$, and $D_{\ldots}(\cdot)$\ denotes
the Dirichlet integral introduced in (\ref{f113f3}) in Subsection \ref{s1103},
further below.

Let now $r>0$ be such that $K\subset\{|z|\leq r\}$, and let further
$z\in\{|z|\leq r\}\setminus E$. Then we have $\operatorname{dist}(z,E)<2r$,
and consequently $\Delta(z,d)\subset\{|z|\leq3r\}\setminus E$ for all
$0<d<\operatorname{dist}(z,E)$. From (\ref{f102h2}) and identity
(\ref{f113g1}) in Lemma \ref{l113g} in Subsection \ref{s1103}, further below,
we deduce that
\begin{align}
|q(z)|  & \leq\frac{2}{d^{2}}D_{\{|z|\leq3r\}\setminus E}(g_{D}(\cdot
,\infty))=\frac{2}{d^{2}}D_{\{|z|\leq3r\}\setminus K}(g_{D}(\cdot
,\infty))\nonumber\\
& =\frac{2}{d^{2}}\left(  \log(3\,r)+\log\frac{1}{\operatorname{cap}%
(K)}+\text{O}(\frac{1}{r})\right)  \text{ \ \ as \ \ r}\rightarrow
\infty.\label{f102h3}%
\end{align}
The equality in the first line of (\ref{f102h3}) is a consequence of the fact
that the set $K\setminus E$ is of planar Lebesgue measure zero, which follows
from the assumed $S-$property of $K$ on $K\setminus E$. The second equality in
(\ref{f102h3}) follows from (\ref{f113g1}) in Lemma \ref{l113g}. Since
$d<\operatorname{dist}(z,E)$ can be chosen arbitrarily, the inequality
(\ref{f102a7}) follows directly from (\ref{f102h3}).\bigskip
\end{proof}

\begin{proof}
[Proof of Lemma \ref{l102a}]Let $z\in J_{j}$, $j\in I$, the an arbitrary point
on the Jordan arc $J_{j}$. We have $g_{D}(z,\infty)=0$, and from the
continuity in $\overline{\mathbb{C}}\setminus E$ of the function $q$
introduced in (\ref{f102a6}), it follows that the two normal derivatives
$\partial/\partial n_{+}$ and $\partial/\partial n_{+}$ of $g_{D}(\cdot
,\infty)$ to both sides of the Jordan arc $J_{j}$ at $z$ are equal in modulus.

Since $g_{D}(\cdot,\infty)\geq0$, and $g_{D}(z,\infty)=0$ for all $z\in J_{j}%
$, it follows also that the signs of the two normal derivatives are equal.
Putting both conclusions together proves equality (\ref{f102a3}) for all $z\in
J_{j}$, $j\in I$, and consequently the $S-$property on $K\setminus E$ in the
sense of Definition \ref{d102a} is proved.\bigskip
\end{proof}

\subsection{\label{s103}Proofs of Results from Section \ref{s4}}

\qquad The central result of Section \ref{s4} is Theorem \ref{t41a}, the
Structure Theorem. As a further result, we have Theorem \ref{t41b}, which
addresses a special topological property of the minimal set $K_{0}(f,\infty)
$, and Theorem \ref{t42a}, which is the analog of Theorem \ref{t41b} for
Problem $(\mathcal{R},\infty^{(0)})$.\medskip

\subsubsection{\label{s1031}Proof of Theorem \ref{t41a}}

\qquad In the proof of Theorem \ref{t41a} we start from Theorem \ref{t101b},
which covers the special case that the function $f$ in Problem $(f,\infty)$ is
algebraic, and which therefore implies that the set $E_{0}$ of non-polar
singularities is finite. The proof of Theorem \ref{t41a} can then be seen as a
lifting of the special result in Theorem \ref{t101b} to the general situation.
In this process the Propositions \ref{p102b}, \ref{p102c}, and \ref{p102d}
from the last subsection will play a decisive role.\medskip

\begin{proof}
[Proof of Theorem \ref{t41a}]It is immediate that the extremal domain
$D_{0}(f,\infty)=\overline{\mathbb{C}}\setminus K_{0}(f,\infty)$ is
elementarily maximal in the sense of Definition \ref{d71a0}. By $E_{00}\subset
K_{0}(f,\infty)$ we denote the subset of all $z\in\partial K_{0}(f,\infty)$
that satisfy condition (i) in Definition \ref{d71a0}, i.e., for each $z\in
E_{00},$\ there exists at least one meromorphic continuation of the function
$f$ out of the domain $D_{0}(f,\infty)$ that has a non-polar singularity at
$z$. By $E_{0}$ we then denote the polynomial-convex hull of $E_{00}$, i.e.,
\begin{equation}
E_{0}:=\widehat{E_{00}}.\label{f103a1}%
\end{equation}
(for a definition of the polynomial-convex hull see Definition \ref{d111b} in
Subsection \ref{s1101}, further below.) Since $K_{0}(f,\infty)$ is
polynomial-convex, we have $E_{0}\subset K_{0}(f,\infty)$.

In the sequel we assume that $E_{0}\cap K_{0}(f,\infty)\neq\emptyset$; for
otherwise Theorem \ref{t41a} is trivial.

Let $U_{n}\subset\mathbb{C}$, $n=1,2,\ldots$, be a sequence of open set with
simply connected components and smooth boundaries $\partial U_{n}$ such that
\begin{equation}
E_{0}\subset U_{n+1}\subset\overline{U}_{n+1}\subset U_{n}\,\ \text{and
\ \ }E_{0}=\bigcap_{n=1}^{\infty}U_{n}.\label{f103a2}%
\end{equation}
We define $E_{n}:=\overline{U}_{n}$ for $n\in\mathbb{N}$. It is immediate that
all sets $E_{n}$ are polynomial-convex, and each of these sets consists only
of finitely many components. The minimal set $K_{0}=K_{0}(f,\infty)$ defines a
connectivity relation on the components of each set $E_{n}$; we say that two
components of $E_{n}$ are connected in $K_{0}$, if they are connected in
$K_{0}\cup E_{n}$. Let $E_{jn}$, $j=1,\ldots,j_{n}$, be the components of
$E_{n}$ with respect to the connectivity in $K_{0}$, i.e.,
\begin{equation}
E_{n}=E_{1n}\cup\ldots\cup E_{j_{n},n},\label{f103a3}%
\end{equation}
and each $E_{jn}$ is connected in $K_{0}$.

In each component $E_{jn}$, $j=1,\ldots,j_{n}$, we select points $z_{jl}$ in
the following manner: For a given $n\in\mathbb{N}$ and any $m\in\mathbb{N}$ we
choose $j_{n}$ families of points
\begin{equation}
Z_{jnm}=\{z_{jl}\}_{l=1}^{m_{j}},\text{ }j=1,\ldots,j_{n}\text{,
\ with\ \ }Z_{nm}:=\bigcup_{j=1}^{j_{n}}Z_{jnm}\text{, \ \ }m=\sum
_{j=1}^{j_{n}}m_{j}\text{,}\label{f103a4}%
\end{equation}
such that
\begin{equation}
Z_{nm}\subset Z_{n,m+1}\text{ \ \ and \ \ }\bigcap_{M=1}^{\infty}%
\overline{\bigcup_{m\geq M}Z_{nm}}=\partial E_{n},\label{f103a5}%
\end{equation}
i.e., the sets $Z_{jnm}$, $j=1,\ldots,j_{n}$, are asymptotically
($m\rightarrow\infty$) dense in each $\partial E_{jn}$, $j=1,\ldots,j_{n}$.
Associated with each point set $Z_{nm}$, we define the algebraic function
\begin{equation}
f_{nm}(z):=\sum_{j=1}^{j_{n}}\left[  \prod_{l=1}^{m_{j}}(1-\frac{z_{jl}}%
{z})\right]  ^{1/m_{j}}\text{ \ for \ }n,m=1,2,\ldots\label{f103a6}%
\end{equation}

Let now $K_{nm}:=K_{0}(f_{nm},\infty)\subset\mathbb{C}$ be the minimal set for
Problem $(f_{nm},\infty)$, $D_{nm}:=\overline{\mathbb{C}}\setminus K_{nm}$ the
corresponding extremal domain $D_{0}(f_{nm},\infty)$, $g_{nm}$ the Green
function $g_{D_{nm}}(\cdot,\infty)$ in $D_{nm}$, and let $q_{nm}$ be the
function
\begin{equation}
q_{nm}(z)=\left(  2\frac{\partial}{\partial z}g_{nm}(z,\infty)\right)
^{2}\label{f103a7}%
\end{equation}
that is defined analogously to (\ref{f102a6}), and this definition appears
also already earlier in (\ref{f52b}). Since $f_{nm}$ is an algebraic function,
we know from Theorem \ref{t101b} that $q_{nm}$ is a rational function. With
Lemma \ref{l102a}, it follows from (\ref{f103a7}) that $K_{nm} $ possesses the
$S-$property on $K_{nm}\setminus Z_{nm}$ in the sense of Definition
\ref{d102a}.

For $m\in\mathbb{N}$ sufficiently large, each component $E_{jn}$ in
(\ref{f103a3}) contains elements of the set $Z_{nm}$. The definition of the
function $f_{nm}$ together with the definition of the components $E_{jn}$ in
(\ref{f103a3}) then imply that components of $E_{n}$ are connected in $K_{nm}
$ for $m$ sufficiently large if they are also connected in $K_{0}%
=K_{0}(f,\infty)$, i.e., the connectivity defined by $K_{nm}$ is coarser than
that defined by $K_{0}$. Further, it follows from the minimality of
$\operatorname{cap}(K_{nm})=\operatorname{cap}(K_{0}(f_{nm},\infty))$ that the
connectivity defined by $K_{nm}$ is also only minimally coarser in the sense
of Definition \ref{d102b} than that defined by $K_{0}$. Notice that the
connectivity relation defined by $K_{0}$ on $E_{n}$\ stands in the background
of the definition of the function $f_{nm}$.

We will now investigated in a first step what happens with the functions
$f_{nm}$, $g_{nm}$, and $q_{nm}$ if $m\rightarrow\infty$. In a second step we
then will also consider the limits for $n\rightarrow\infty$.

From Proposition \ref{p102d} we deduce the upper estimates
\begin{equation}
\left|  q_{nm}(z)\right|  \leq\frac{2}{\operatorname{dist}(z,E_{n})^{2}%
}\left(  \log(3\,r)+\log\frac{1}{\operatorname{cap}(K_{nm})}\right)
\label{f103a8}%
\end{equation}
for the functions $q_{nm}$; they hold for all $z\in\{|z|\leq r\}\setminus
E_{n}$, all $m\geq m_{0}$, and $r>0$ sufficiently large. Notice that the
assumption $E_{0}\cap K_{0}(f,\infty)\neq\emptyset$ at the beginning of this
proof implies that $E_{0}$ contains at least two different components that are
connected in $K_{0}(f,\infty)$, which then further implies that
$\operatorname{cap}(K_{nm})\geq c_{0}>0$ for all $m\geq m_{0}$.

From (\ref{f103a8}) together with Montel's Theorem and the fact that all
$q_{nm}$ are analytic outside of $E_{n}$ (cf. Theorem \ref{t101b}, part (b)),
we deduce that there exists an infinite sequence $N_{n}\subset\mathbb{N}$ such
that
\begin{equation}
\lim_{m\rightarrow\infty,\text{ }m\in N_{n}}q_{nm}(z)=:q_{n}(z)\label{f103a9}%
\end{equation}
holds locally uniformly in $\overline{\mathbb{C}}\setminus E_{n}$.

For the Green functions $g_{nm}$ and the sets $K_{nm}$ we now prove the
existence of limits that correspond to limit (\ref{f103a9}). Using the same
techniques as applied in the proof of Lemma \ref{l92b} after (\ref{f92e1}) and
combining this with (\ref{f103a9}) and (\ref{f103a7}), we deduce that the
limit
\begin{equation}
\lim_{m\rightarrow\infty,\text{ }m\in N_{n}}g_{nm}(z)=:g_{n}(z)\label{f103a10}%
\end{equation}
exists locally uniformly in $\mathbb{C}\setminus E_{n}$. From the two
relations in (\ref{f103a5}) together with the fact that all points in each set
$Z_{jnm}\subset E_{jn}$, $j=1,\ldots,j_{n}$, are connected by a subcontinuum
in $K_{nm}$ and that these continua contain only regular boundary points, it
follows that
\begin{equation}
g_{n}(z)=0\text{ \ \ for all \ \ }z\in E_{n}.\label{f103a11}%
\end{equation}
With the two definitions
\begin{equation}
K_{n}:=\bigcap_{M=1}^{\infty}\widehat{\bigcup_{m\geq M,\text{ }m\in N_{n}%
}K_{nm}}\text{ \ \ and \ \ }D_{n}:=\overline{\mathbb{C}}\setminus
K_{n},\label{f103a12}%
\end{equation}
it further follows from Lemma \ref{l114b} from Subsection \ref{s1104}, further
below, that the function $g_{n}$ introduced in (\ref{f103a10})\ is the Green
function of the domain $D_{n}$, i.e., we have
\begin{equation}
g_{n}=g_{D_{n}}(\cdot,\infty).\label{f103a13}%
\end{equation}

The two limits (\ref{f103a9}) and (\ref{f103a10}) together imply that
analogously to \linebreak(\ref{f103a7}) we also have the relation
\begin{equation}
q_{n}(z)=\left(  2\frac{\partial}{\partial z}g_{n}(z,\infty)\right)
^{2}\text{ \ \ for \ \ }z\in\overline{\mathbb{C}}\setminus E_{n}%
,\label{f103a14}%
\end{equation}
from which we deduce again with Lemma \ref{l102a} that the set $K_{n}$
possesses the $S-$pro\-perty on $K_{n}\setminus E_{n}$ in the sense of
Definition \ref{d102a}.

Like the sets $K_{nm}$, so also the set $K_{n}$ possesses the property that
the connectivity relation defined on the components of $E_{n}$ by $K_{n}$ is
minimally coarser in the sense of Definition \ref{d102b} than that defined by
$K_{0}=K_{0}(f,\infty)$. Indeed, both\ aspects of the property ''minimally
coarser'' carry over from $K_{nm}$ to $K_{n}$. With exactly the same
techniques as those used in the proof of Lemma \ref{l92d}, we show that all
connectivities defined by the sets $K_{nm}$ lead to identical connectivities
defined by the set $K_{n}$. Consequently, also the new connectivity relation
is coarser than that defined by the set $K_{0}$. On the other hand, as a
consequence of convergence $\lim_{m}\operatorname{cap}(K_{nm}%
)=\operatorname{cap}(K_{n})$, which follows from (\ref{f103a10}), it then
further follows that the connectivity defined on $E_{n}$ by $K_{n}$ is also
only minimally coarser in the sense of Definition \ref{d102b} than that
defined by $K_{0}$.

In a second step we investigate limits for $n\rightarrow\infty$, where we can
largely apply the same techniques as those just used after (\ref{f103a8}) for
the investigation of limits with $m\rightarrow\infty$, only that now the
boundary behavior of the Green function can be complicated by irregular points
on $\partial E_{0}$. Notice that the sets $E_{n}$, have been constructed with
nice boundaries. We overcome these new difficulties by using Lemma \ref{l114a}
from Subsection \ref{s1104}, further below.

From the definition of the sets $E_{n}$ after (\ref{f103a2}) we know that
\begin{equation}
E_{0}=\bigcap_{n}E_{n}.\label{f103b1}%
\end{equation}
Analogously to (\ref{f103a8}), we deduce from Proposition \ref{p102d} that
\begin{equation}
\left|  q_{n}(z)\right|  \leq\frac{2}{\operatorname{dist}(z,E_{n})^{2}}\left(
\log(3\,r)+\log\frac{1}{\operatorname{cap}(K_{n})}\right) \label{f103b2}%
\end{equation}
for all $z\in\{|z|\leq r\}\setminus E_{n}$, $n\in\mathbb{N}$, and $r>0$
sufficiently large, for the functions $q_{n}$ from (\ref{f103a9}), which
satisfy (\ref{f103a14}). With the same argumentation as used after
(\ref{f103a8}), we conclude that $\operatorname{cap}(K_{n})\geq c_{0}>0$ for
all $n\in\mathbb{N}$, which shows that the right-hand side of (\ref{f103b2})
can be made independent of $n$. From (\ref{f103b2}) it then follows as before
in (\ref{f103a9}) that there exists an infinite sequence $N\subset\mathbb{N}$
such that the limit
\begin{equation}
\lim_{n\rightarrow\infty,\text{ }n\in N}q_{n}(z)=:\widetilde{q}%
(z)\label{f103b3}%
\end{equation}
exists locally uniformly for $z\in\overline{\mathbb{C}}\setminus E_{0}$, and
the function $\widetilde{q}$ is analytic in $\overline{\mathbb{C}}\setminus
E_{0}$. With the same argumentation as used for the deduction of
(\ref{f103a10}), we now deduce that the limit
\begin{equation}
\lim_{n\rightarrow\infty,\text{ }n\in N}g_{n}(z)=:\widetilde{g}%
(z)\label{f103b4}%
\end{equation}
exists locally uniformly for $z\in\mathbb{C}\setminus E_{0}$.

From (\ref{f103b4}) together with Lemma \ref{l114a} from Subsection
\ref{s1104} and (\ref{f103b1}), we deduce that
\begin{equation}
\widetilde{g}(z)=0\text{ \ \ for quasi every \ \ }z\in E_{0}.\label{f103b5}%
\end{equation}
Using now the definitions
\begin{equation}
\widetilde{K}:=\bigcap_{m=1}^{\infty}\widehat{\bigcup_{n\geq m,\text{ }m\in
N}K_{n}}\text{ \ \ and \ \ }\widetilde{D}:=\overline{\mathbb{C}}%
\setminus\widetilde{K},\label{f103b6}%
\end{equation}
it follows in the same way as after (\ref{f103a12}) with the help of Lemma
\ref{l114b} from Subsection \ref{s1104} that the function $\widetilde{g}$
introduced in (\ref{f103b4}) is the Green functions of the domain
$\widetilde{D}$, i.e., we have
\begin{equation}
\widetilde{g}=g_{\widetilde{D}}(\cdot,\infty).\label{f103b7}%
\end{equation}
Notice that the domain $\widetilde{D}$ may not be regular with respect to
Dirichlet problems, there may exist irregular points (cf. Definition
\ref{d113a} in Subsection \ref{s1103}, further below) on $\partial
\widetilde{D}\cap E_{0}$, and consequently, the equality in (\ref{f103b5})
holds only quasi everywhere on $E_{0}$.

The two limits (\ref{f103b3}) and (\ref{f103b4}) together with relation
(\ref{f103a14}) imply that
\begin{equation}
\widetilde{q}(z)=\left(  2\frac{\partial}{\partial z}g_{\widetilde{D}%
}(z,\infty)\right)  ^{2}\text{ \ \ for \ \ }z\in\overline{\mathbb{C}}\setminus
E_{0}.\label{f103b8}%
\end{equation}

With Lemma \ref{l102a}, we deduce from (\ref{f103b8}) that the set
$\widetilde{K}$ possesses the $S-$pro\-perty on $\widetilde{K}\setminus E_{0}$
in the sense of Definition \ref{d102a}. In the same way, as it has been shown
before in the cases of the sets $K_{nm}$ and the set $K_{n}$, we also now show
that the minimal coarseness of a connectivity relation in the sense of
Definition \ref{d102b} can be carried over from the connectivity relations
defined by the sets $K_{n}$ to the relation defined by the set $\widetilde{K}%
$, i.e., we deduce that the connectivity defined by $K_{n}$ on $E_{0}$ is
minimally coarser in the sense of Definition \ref{d102b} than that defined by
$K_{0}(f,\infty)$.

Since we have assumed at the beginning of the present proof that $E_{0}\cap
K_{0}(f,\infty)\allowbreak\neq\emptyset$, it follows that some components of
$E_{0}$ are connected in $K_{0}(f,\infty)$, and consequently, those components
are also connected in $\widetilde{K}$, which implies that
\begin{equation}
\operatorname{cap}(\widetilde{K})>0.\label{f103b9}%
\end{equation}

Let us summarize what we have proved so far. Starting from the algebraic
functions $f_{nm}$ introduced in (\ref{f103a6}) and using the results proved
in Theorem \ref{t101b} for algebraic functions, we have shown that there
exists a compact set $\widetilde{K}\subset\mathbb{C}$ of positive capacity
with the following properties:

\begin{itemize}
\item[(a)] \noindent$\widetilde{K}$ possesses the $S-$property in the sense of
Definition \ref{d102a} on $\widetilde{K}\setminus E_{0}$.

\item[(b)] \noindent The connectivity relation defined on $E_{0}$ by
$\widetilde{K}$ is minimally coarser in the sense of Definition \ref{d102b}
than that defined by $K_{0}(f,\infty)$.

\item[(c)] \noindent The set $E_{0}\subset\widetilde{K}$ is compact and
polynomial-convex, and its boundary $\partial E_{0}$ contains all non-polar
singularities of the function $f$\ that can be reached by meromorphic
continuations of $f$ from within the extremal domain $D_{0}(f,\infty)$.
\end{itemize}

In the concluding part of the proof, Proposition \ref{p102c} will play a
crucial role. If in Proposition \ref{p102c} we take $K_{0}(f,\infty)$ as $K$
and $\widetilde{K}$ from (\ref{f103b6}) as $K_{1}$, then it is immediate from
(\ref{f103b9}) and the assertions (a), (b), and (c) that all assumptions of
Proposition \ref{p102c} are satisfied, except the assumption (iv).

Since the set $K_{0}(f,\infty)$ is of minimal capacity in the sense of
(\ref{f21a}) in Definition \ref{d21b}, it follows that conclusion
(\ref{f102a5}) of Proposition \ref{p102c} cannot be true. Consequently,
assumption (iv) of Proposition \ref{p102c} has to be false, which proves with
our choice of $K$ and $K_{1}$ that we have
\begin{equation}
K_{0}(f,\infty)=\widetilde{K}.\label{f103b10}%
\end{equation}

Indeed, from the negation of assumption (iv) in Proposition \ref{p102c} we
conclude that $\widetilde{K}\subset K_{0}(f,\infty)$. Identity (\ref{f103b10})
then follows from a combination of the fact that the connectivity relation
defined by $\widetilde{K}$ on $E_{0}$ is coarser than that defined by
$K_{0}(f,\infty)$ (cf. assertion (b)), and the fact that the extremal domain
$D_{0}(f,\infty)$ is elementarily maximal in the sense of Definition
\ref{d71a0} (cf. Proposition \ref{p71a}).

With identity (\ref{f103b10}), all properties proved for the set
$\widetilde{K}$ are now also valid for the minimal set $K_{0}(f,\infty)$.
Hence, we have proved the following four assertions:

\begin{itemize}
\item[$(\alpha)$] The set $K_{0}(f,\infty)\setminus E_{0}$ consists of
critical trajectories of the quadratic differential $\widetilde{q}(z)dz^{2}$
with $\widetilde{q}$ defined in (\ref{f103b8}). In (\ref{f103b8}),
$\widetilde{D}$ is equal to the extremal domain $D_{0}(f,\infty)$ because of
(\ref{f103b10}).
\end{itemize}

Indeed, assertion ($\alpha$) follows immediately from (\ref{f103b8}) and the
definition of trajectories of quadratic differentials in (\ref{f52a}) in
Subsection \ref{s52}.

\begin{itemize}
\item[$(\beta)$] At each $z\in\partial E_{0}$ at least one meromorphic
continuation of the function $f$ out of the extremal domain $D_{0}(f,\infty)$
hits a non-polar singularity.
\end{itemize}

Indeed, assertion ($\beta$) is an immediate consequence of the definition of
the set $E_{0}$ in and before (\ref{f103a1}).

\begin{itemize}
\item[$(\gamma)$] Let $E_{1}$ be set of all zeros of the function
$\widetilde{q}$ from (\ref{f103b8}) on $K_{0}(f,\infty)\setminus E_{0}$. The
set $E_{1}$ is discrete in $K_{0}(f,\infty)\setminus E_{0}$ since the function
$\widetilde{q}$ is analytic in $\overline{\mathbb{C}}\setminus E_{0} $. The
set $K_{0}(f,\infty)\setminus(E_{0}\cup E_{1})$ consists of open, analytic
Jordan arcs, which are trajectories of the quadratic differential
$\widetilde{q}(z)dz^{2}$.
\end{itemize}

Indeed, the first part of assertion ($\gamma$) follows directly from
(\ref{f103b8}). Since $\widetilde{q}(z_{0})\neq0$ for any $z_{0}\in
K_{0}(f,\infty)\setminus(E_{0}\cup E_{1})$, we also have $\frac{\partial
}{\partial z}g_{D_{0}(f,\infty)}(z_{0},\infty)\allowbreak\neq0$, and
consequently, the equation $g_{D_{0}(f,\infty)}(z,\infty)=0$ defines an
analytic Jordan arc in a neighborhood of any point $z_{0}\in K_{0}%
(f,\infty)\setminus(E_{0}\cup E_{1})$, which proves the second part of
assertion ($\gamma$).

\begin{itemize}
\item[$(\delta)$] Let $o(z)$, $z\in E_{1}$, be the order of the zero of
$\widetilde{q}$ at $z$, then the point $z$ is endpoint of $o(z)+2$ different
analytic Jordan arcs in $K_{0}(f,\infty)\setminus(E_{0}\cup E_{1})$.
\end{itemize}

Indeed, assertion ($\delta$) is an immediate consequence of the typical local
structure of trajectories of quadratic differentials in a neighborhood of a
zero of the differential (cf. \cite{Jensen75}, Chapter 8.2, or
\cite{Strebel84}).

\begin{itemize}
\item[$(\varepsilon)$] The function $f$ has meromorphic continuations to each
point of $z\in K_{0}(f,\infty)\setminus E_{0}$ from all sides out of the
extremal domain $D_{0}(f,\infty)$. These continuations lead to exactly two
different function elements at each point $z\in K_{0}(f,\infty)\setminus
(E_{0}\cup E_{1})$, and it leads to more than two different function elements
at each point $\in E_{1}$.
\end{itemize}

Indeed, the first part of assertion ($\varepsilon$) is a consequence of the
definition of the set $E_{0}$ in (\ref{f103a1}). The second part is a
consequence of the minimality (\ref{f21a}) in Definition \ref{d21b} of the set
$K_{0}(f,\infty)$.\medskip

With the four assertions ($\alpha$) - ($\varepsilon$) we have proved more than
is needed for the proof of Theorem \ref{t41a}. The additional results will be
needed in subsequent proofs, most importantly in the proofs of the two
Theorems \ref{t51a} and \ref{t52a}.

Identity (\ref{f41a}) in Theorem \ref{t41a} follows directly from the three
assertions ($\alpha$), ($\beta$), and ($\gamma$). The description of the set
$E_{0}$ in assertions (i) of Theorem \ref{t41a} is practically identical with
assertion ($\beta$). The two remaining assertions (ii) and (iii) in Theorem
\ref{t41a} follow from the two assertions ($\gamma$) and ($\varepsilon$). With
the last sentences the proof of Theorem \ref{t41a} is completed.\bigskip
\end{proof}

\subsubsection{\label{s1032}Proof of Theorem \ref{t41b}}

\qquad The catchword in Theorem \ref{t41b} is local connectedness. Among other
things it will be shown in Theorem \ref{t41b} that the most interesting part
$K_{0}(f,\infty)\setminus E_{0}$ of the minimal set $K_{0}(f,\infty)$ is
locally connected. Local connectedness is an important property in geometric
function theory.\medskip

\begin{proof}
[Proof of Theorem \ref{t41b}]Since we know that the function $\widetilde{q}$
of (\ref{f103b8}) is analytic in $\overline{\mathbb{C}}\setminus E_{0}$, it
follows from $\overline{\mathbb{C}}\setminus E_{0}$ that the Green function
$g_{D_{0}(f,\infty)}(\cdot,\infty)$ is continuous throughout $\overline
{\mathbb{C}}\setminus E_{0}$. From Carath\'{e}odory's theory about the
boundary behavior of Riemann mapping functions (cf. Chapter 9 in
\cite{Pommerenke75}, and there especially Theorem 9.8) we know that the
continuity of the Green function $g_{D_{0}(f,\infty)}(\cdot,\infty)$ is
equivalent to the local connectedness of the set $\partial D_{0}%
(f,\infty)\setminus E_{0}$. From Theorem \ref{t41a} it follows that the set
$\partial D_{0}(f,\infty)\setminus E_{0}$ is equal to $K_{0}(f,\infty
)\setminus E_{0} $, which proves the first half of Theorem \ref{t41b}.

From the local behavior of trajectories of quadratic differentials, as it has
been stated in Lemma \ref{l115a} in Subsection \ref{s1105}, further below,
together with assertion ($\gamma$) at the end of the proof of Theorem
\ref{t41a}, we know that the bifurcation points $z\in E_{1}$ of $K_{0}%
(f,\infty)\setminus E_{0}$ are zeros of the analytic function $\widetilde{q}$
from (\ref{f103b8}), and as such they are of finite order. From the connection
between the zeros of quadratic differentials and the local structure of its
trajectories (see again Lemma \ref{l115a} in Subsection \ref{s1105}), we then
conclude that the finiteness of the order of the zeros of $\widetilde{q}$
implies that each $z\in E_{1}$ can be the endpoint of only finitely many
Jordan arcs $J_{j}$, $j\in I$.\bigskip
\end{proof}

\subsubsection{\label{s1033}Proof of Theorem \ref{t42a}}

\qquad In Theorem \ref{t32a} of Section \ref{s3} it has been shown that the
two Problems $(f,\infty)$ and $(\mathcal{R},\infty^{(0)})$ are equivalent if
the Riemann surface $\mathcal{R}$ is the natural domain of definition for the
function $f$. Because of this equivalence, Theorem \ref{t42a} is practically a
corollary to Theorem \ref{t41a}.\medskip

\begin{proof}
[Proof of Theorem \ref{t42a}]In order to prove the characterization of the
sets $E_{0}$, $E_{1}$, and the family of Jordan arcs $\left\{  J_{j}\right\}
_{j\in I}$ in (\ref{f42a0}) of Theorem \ref{t42a}, we have only to keep in
mind that a non-polar singularity of the function $f$ at a point
$z\in\overline{\mathbb{C}}$ is either a branch point or a transcendental,
essential singularity. In the first case, the point $z$ is an element of
$Br(\mathcal{R})$, and in the second case, it is an element of the relative
boundary $\partial\mathcal{R}$ of the Riemann surface $\mathcal{R}$ over
$\overline{\mathbb{C}}$. With these observations we immediately verify
identity (\ref{f42a}) in Theorem \ref{t42a}.

The two assertions (ii) and (iii) in Theorem \ref{t42a} follow then from the
observation that meromorphic continuations of the function $f$ lead to
different function elements at a point $z\in\overline{\mathbb{C}}$ if, and
only if, the corresponding points $\zeta$ on the Riemann surface $\mathcal{R}$
lie on different sheets of this surface.\bigskip
\end{proof}

\subsection{\label{s104}Proofs of Results from Section \ref{s5}}

\qquad In Section \ref{s5} the Jordan arcs $J_{j}$, $j\in I$, in the minimal
set $K_{0}(f,\infty)$ have been characterized with the help of the
$S-$property and with the help of quadratic differentials. Both concepts are
very similar. These results are especially interesting if the function $f$ has
only finitely many non-polar singularities; Proposition \ref{p53a} and Theorem
\ref{t53a} deal with this situation.\medskip

\subsubsection{\label{s1041}Proof of Theorem \ref{t51a}}

\qquad Most of the work for a proof of Theorem \ref{t51a} has already been
done in the proof of Theorem \ref{t41a}.\medskip

\begin{proof}
[Proof of Theorem \ref{t51a}]In assertion ($\gamma$) at the end of the proof
of Theorem \ref{t41a}, it has been shown that each Jordan arc $J_{j}$, $j\in
I$, is a trajectory of the quadratic differential $\widetilde{q}(z)dz^{2}$
with the function $\widetilde{q}$ defined by (\ref{f103b8}). Because of
(\ref{f103b10}), the domain $\widetilde{D}$ in (\ref{f103b8}) is equal to the
extremal domain $D_{0}(f,\infty)$. With Lemma \ref{l102a}, we then conclude
from (\ref{f103b8}) together with (\ref{f103b10}) that $K_{0}(f,\infty)$
possesses the $S-$property in this sense of Definition \ref{d102a} on
$K_{0}(f,\infty)\setminus E_{0}$, which proves identity (\ref{f51a}) because
of (\ref{f102a3}).\bigskip
\end{proof}

\subsubsection{\label{s1042}Proof of Theorem \ref{t52a}}

\qquad In Theorem \ref{t52a} the Jordan arcs $J_{j}$, $j\in I$, of the minimal
set $K_{0}(f,\infty)$ are characterized by quadratic differentials.\medskip

\begin{proof}
[Proof of Theorem \ref{t52a}]From (\ref{f103b10}) in the proof of Theorem
\ref{t41a} we know that the function $\widetilde{q}$ introduced in
(\ref{f103b8}) and the function $q$ introduced in (\ref{f52b}) are identical.
In assertion ($\alpha$) at the end of the proof of Theorem \ref{t41a}, it has
been proved that the arcs $J_{j}$, $j\in I$, are trajectories of the quadratic
differential $q(z)dz^{2}$.

Identity (\ref{f52c}) in Theorem \ref{t52a} follows from (\ref{f103b8}) for
$z\in\mathbb{C}\setminus E_{0}$, and it follows from the logarithmic pole of
the Green functions $g_{D_{0}(f,\infty)}(\cdot,\infty)$ at infinity. From
(\ref{f103b3}), it further follows that the function $\widetilde{q}=q$ is
analytic in $\overline{\mathbb{C}}\setminus E_{0}$. Thus, it only remains to
prove that the function $\widetilde{q}$ from (\ref{f103b8}) is meromorphic at
isolated points of $E_{0}$. Actually, we shall prove slightly more and show
that $\widetilde{q}$ has at most of simple pole at an isolated point of
$E_{0}$.

If $z\in E_{0}$ is simultaneously an isolated point of $E_{0}$ and of the
minimal set $K_{0}(f,\infty)$, then the Green function $g_{D_{0}(f,\infty
)}(\cdot,\infty)$ is harmonic in a neighborhood of $z$, and it follows from
(\ref{f103b8}) that $\widetilde{q}$ is analytic at $z$.

Let us now assume that $z_{0}\in E_{0}$ is an isolated point of $E_{0},$ but
not of the minimal set $K_{0}(f,\infty)$. From assertion ($\alpha$) at the end
of the proof of Theorem \ref{t41a}, we know that in a neighborhood of $z_{0}$
the set $K_{0}(f,\infty)\setminus\{z_{0}\}$ consists only of trajectories of
the quadratic differential $\widetilde{q}(z)dz^{2}$. From the history of the
definition of the set $\widetilde{K}$ before (\ref{f103b6}), it then follows
that only a finite number, let say $k_{0}\in\mathbb{N}$, of these Jordan arcs
$J_{j}$, $j\in I$, in $K_{0}(f,\infty)\setminus E_{0}$ have $z_{0}$ as their
endpoint. From this observation and the local structure of quadratic
differentials, which has been reviewed in Lemma \ref{l115a} in Subsection
\ref{s1105}, further below, we conclude that near the point $z_{0}$ the
function $\widetilde{q}$ of (\ref{f103b8}) has the local development
\begin{equation}
\widetilde{q}(z)=q_{0}(z-z_{0})^{k_{0}-2}+\ldots,\text{ \ \ }q_{0}%
\neq0,\label{f104a}%
\end{equation}
which shows that the function $\widetilde{q}=q$ is meromorphic at each
isolated point of $E_{0}$, and poles are at most of order $1$.\bigskip
\end{proof}

\subsubsection{\label{s1043}Proof of Theorem \ref{t53a}}

\qquad In Theorem \ref{t53a} the special case has been considered that the set
$E_{0}$ of Theorem \ref{t41a} is finite, which leads to a rational quadratic
differential $q(z)dz^{2}$ for the determination of the Jordan arcs $J_{j}$,
$j\in I$, in $K_{0}(f,\infty)\setminus E_{0}$\ in the sense of Theorem
\ref{t52a}. Algebraic functions provide typical examples for this situation.

In Proposition \ref{p53a} it has been stated that with the set $E_{0}$ also
the two sets $E_{1}$ and $E_{2}$ in Theorem \ref{t41a} and in Definition
\ref{d53b}, respectively, are finite.\medskip

\begin{proof}
[Proof of Proposition \ref{p53a}]If $E_{0}$ is finite, then all points of
$E_{0}$ are isolated, and we know from Theorem \ref{t52a} that the function
$q$ from (\ref{f52b}) is meromorphic throughout $\overline{\mathbb{C}}$, which
implies that $q$ is a rational function. Together with (\ref{f52c}) we then
further know that its numerator degree is by 2 degrees smaller than its
denominator degree. Let $m$ and $n$ denote the numerator and denominator
degrees, respectively. Since it has been shown in Theorem \ref{t52a} that $q$
has a most simple poles, which are all contained in $E_{0}$, it follows that
\begin{equation}
m+2=n\leq\operatorname*{card}(E_{0}).\label{f104b1}%
\end{equation}
From the local structure of the trajectories of quadratic differentials, which
has been reviewed in Lemma \ref{l115a} in Subsection \ref{s1105}, further
below, we know that at each bifurcation point $z$ of$\ K_{0}(f,\infty
)\setminus E_{0}$, the function $q$ has a zero of order
\begin{equation}
\operatorname*{ord}(z)=\operatorname*{i}(z)-2\label{f104b2}%
\end{equation}
with $\operatorname*{ord}(z)$ denoting the order of the zero at $z$ and
$\operatorname*{i}(z)$ denoting the bifurcation index of Definition
\ref{d53a}. Since $E_{1}$ is the set of all bifurcation points of
$K_{0}(f,\infty)$, it follows from (\ref{f104b1}) and (\ref{f104b2}) that the
set $E_{1}$ is finite.

From the definition of critical points of a Green function in Definition
\ref{d53b}, it follows rather immediately that at such points several level
lines of the Green function intersect. Consequently, the function $q$ has a
zero at each critical point $z$ of the Green function $g_{D_{0}(f,\infty
)}(\cdot,\infty)$, but this also follows directly from (\ref{f52b}), and more
precisely, we have
\begin{equation}
\operatorname*{ord}(z)=2\,\operatorname*{j}(z),\label{f104b3}%
\end{equation}
where $\operatorname*{ord}(z)$ denotes again the order of the zero at $z$, and
$j(z)$ is the degree of the critical point introduced in Definition
\ref{d53b}. From (\ref{f104b1}) and (\ref{f104b3}), it then follows that the
set $E_{2} $ is finite.\bigskip
\end{proof}

\begin{proof}
[Proof of Theorem \ref{t53a}]Most work for the proof of Theorem \ref{t53a} has
already been done in the proof of Proposition \ref{p53a}. From there we know
that $q$ is a rational function. All zeros and poles of the function $q$ on
the minimal set $K_{0}(f,\infty)$ are contained in $E_{0}\cup E_{1}$. The
poles are at most of order $1$, and they appear at a point $z\in E_{0}$ if $z$
is endpoint of only one Jordan arc in $K_{0}(f,\infty)\setminus(E_{0}\cup
E_{1})$. For such points $z$, we have the bifurcation index $\operatorname*{i}%
(z)=1$. After putting the information from (\ref{f104b1}) and (\ref{f104b2})
together, we get the first product in (\ref{f53b}).

The second product in (\ref{f53b}) follows from (\ref{f104b3}) and the
observation that in the domain $D_{0}(f,\infty)$ the function $q$ is analytic
and has zeros only at the critical points of the Green function $g_{D_{0}%
(f,\infty)}(\cdot,\infty)$ and at $\infty$. The order of these zeros is given
by (\ref{f104b3}) and (\ref{f52c}).\medskip
\end{proof}

\subsection{\label{s105}Proofs of Results from Section \ref{s7}}

\qquad The main results of Section \ref{s7} are contained in Theorem
\ref{t73a}, where a characterization of the extremal domain $D_{0}(f,\infty)$
has been given with the help of the $S-$property, and in Theorem \ref{t74a},
where several geometric estimates have been formulated with the help of
convexity.\medskip

\subsubsection{\label{s1051}Proof of Theorem \ref{t73a}}

\qquad Most of the work for the proof of Theorem \ref{t73a} has already been
done by Proposition \ref{p102a}.\medskip

\begin{proof}
[Proof of Theorem \ref{t73a}]Let $D\in\mathcal{D}(f,\infty)$ be an admissible
domain that possesses the $S-$property in the sense of Definition \ref{d71a}.
We assume that $D$ is different from the extremal domain $D_{0}(f,\infty)$,
i.e.,
\begin{equation}
D\neq D_{0}(f,\infty).\label{f105a1}%
\end{equation}

It is an immediate consequence of Theorem \ref{t41a} that the extremal domain
$D_{0}(f,\allowbreak\infty)$ is elementarily maximal in the sense of
Definition \ref{d71a0}. From Proposition \ref{p102a} together with assumption
(\ref{f105a1}), we then conclude that
\begin{equation}
\operatorname{cap}(\partial D)<\operatorname{cap}(\partial D_{0}%
(f,\infty)).\label{f105a2}%
\end{equation}
But the last inequality contradicts the minimality (\ref{f21a}) in Definition
\ref{d21b}, which proves Theorem \ref{t73a}.\bigskip
\end{proof}

\subsubsection{\label{s1052}Proof of Theorem \ref{t74a}}

\qquad All results in Theorem \ref{t74a} are basically a consequence of the
fact that the capacity is a monotonically decreasing functions under
orthogonal projections, a result which has been reviewed in Lemma \ref{l111c}
in Subsection \ref{s1101}, further below.\medskip

\begin{proof}
[Proof of Theorem \ref{t74a}]Let
\begin{equation}
L=L_{z_{0},v}:=\{\text{\thinspace}z_{0}+v\,t\text{\ }|\text{\ }t\in
\mathbb{R}\text{\thinspace}\}\text{, \ \ }z_{0},v\in\mathbb{C}\text{,
\ }|v|=1\text{,}\label{f105b1}%
\end{equation}
be an arbitrary line in $\mathbb{C}$, and denote by $H_{\pm}=H_{\pm}^{L}$ the
two complementary half-planes of $L$, i.e., $\mathbb{C}\setminus L=H_{+}\cup
H_{-}=H_{+}^{L}\cup H_{-}^{L}$, and
\begin{equation}
H_{\pm}^{L}:=\{\text{ }z\in\mathbb{C}\text{\ }|\text{\ }\pm\operatorname{Im}%
(\overline{v}(z-z_{0}))>0\text{\ }\}\label{f105b2}%
\end{equation}
with $z_{0}$ and $v$ the same parameters as those used in (\ref{f105b1}).

By $\varphi_{L}:\mathbb{C}\longrightarrow\mathbb{C}$ we denote the orthogonal
projection (\ref{f111d3}) on $L$ out of $H_{+}$. On $L\cup H_{-}$,
$\varphi_{L}$ is the identity.

Before we come to the individual proofs of the five assertions of Theorem
\ref{t74a}, we assemble and prove several preparatory assertions, in which
$E_{0}$ is the compact set introduced in Theorem \ref{t41a}.

\begin{itemize}
\item[(a)] Let $D$ be an admissible domain for Problem $(f,\infty)$, i.e.,
$D\in\mathcal{D}(f,\infty)$, and set $K:=\overline{\mathbb{C}}\setminus D$. If
the line $L$ is such that $K\subset L\cup H_{-}^{L}$, $K\cap L\neq\emptyset$,
and that the function $f$ has a meromorphic continuation out of $H_{+}^{L}$
into a neighborhood of each $z\in K\cap L$, then for every line $\widetilde
{L}$, which is parallel to $L$, and for which $f$ has a meromorphic
continuation throughout $H_{+}^{\widetilde{L}}$, the compact set
$\widetilde{K}:=\varphi_{\widetilde{L}}(K)$ and the domain $\widetilde
{D}:=\overline{\mathbb{C}}\setminus\widetilde{K}$ are admissible for Problem
$(f,\infty)$, i.e., we have
\begin{equation}
\widetilde{K}=\varphi_{\widetilde{L}}(K)\in\mathcal{K}(f,\infty)\text{ \ \ and
\ \ }\widetilde{D}=\overline{\mathbb{C}}\setminus\widetilde{K}\in
\mathcal{D}(f,\infty).\label{f105b3}%
\end{equation}
By $\varphi_{\widetilde{L}}$ we have denoted the orthogonal projection
associated with $\widetilde{L}$ in accordance to (\ref{f111d3}).
\end{itemize}

Indeed, assertion (a) and especially (\ref{f105b3}) follows directly from
Proposition \ref{p91a} together with the specific properties of the orthogonal
projection (\ref{f111d3}) and the assumptions made in assertion (a).

\begin{itemize}
\item[(b)] If the situation of assertion (a) is given, and if we have
$H_{+}^{L}\subsetneq H_{+}^{\widetilde{L}}$, then it follows that
\begin{equation}
\operatorname{cap}(\widetilde{K})\leq\operatorname{cap}(K).\label{f105b4}%
\end{equation}
and a strict inequality holds in (\ref{f105b4}) if, and only if,
\begin{equation}
\operatorname{cap}(K\cap H_{+}^{\widetilde{L}})>0.\label{f105b5}%
\end{equation}

\end{itemize}

Indeed, assertion (b) follows directly from Lemma \ref{l111c}. For the
necessity of condition (\ref{f105b5}) one also needs Lemma \ref{l111b}.

\begin{itemize}
\item[(c)] If the line $L$ is such that
\begin{equation}
E_{0}\subset L\cup H_{-}^{L},\label{f105c1}%
\end{equation}
then we have
\begin{equation}
K_{0}(f,\infty)\subset L\cup H_{-}^{L}.\label{f105c2}%
\end{equation}

\end{itemize}

Indeed, if (\ref{f105c1}) holds true, but (\ref{f105c2}) is false, then we
concluded from Theorem \ref{t41a} that $K_{0}(f,\infty)\cap H_{+}^{L}%
$\ contains at least an open piece of on of the Jordan arcs $J_{j}$, $j\in I$,
from (\ref{f41a}) in Theorem \ref{t41a}. With Lemma \ref{l111a} in Subsection
\ref{s1101}, further below, we then deduce that
\begin{equation}
\operatorname{cap}(K_{0}(f,\infty)\cap H_{+}^{L})>0,\label{f105c3}%
\end{equation}
which together with the two assertions (a), (b), and assumption (\ref{f105c1})
further implies that there exists a line $\widetilde{L}$, like that in
assertion (a), such that
\begin{equation}
\varphi_{\widetilde{L}}(K_{0}(f,\infty))\in\mathcal{K}(f,\infty
),\label{f105c5}%
\end{equation}
and, as in assertion (b), we further have
\begin{equation}
\operatorname{cap}(\varphi_{\widetilde{L}}(K_{0}(f,\infty
)))<\operatorname{cap}(K_{0}(f,\infty)).\label{f105c6}%
\end{equation}
Inequality (\ref{f105c6}) together with (\ref{f105c5}) contradicts the
minimality (\ref{f21a}) of the set $K_{0}(f,\infty)$ in Definition \ref{d21b},
which completes the proof of assertion (c).

\begin{itemize}
\item[(d)] Let $\operatorname{Ex}(K)\subset K$ denote the set of extreme
points of a compact set $K\subset\mathbb{C}$. We have
\begin{equation}
\operatorname{Ex}(K_{0}(f,\infty))\subset E_{0}.\label{f105d1}%
\end{equation}

\end{itemize}

Indeed, for each $z\in\operatorname{Ex}(K_{0}(f,\infty))$ there exists a
straight line $L$ such that $L\cap K_{0}(f,\infty)=\{z\}$ and $K_{0}%
(f,\infty)\subset L\cup H_{-}^{L}$. From assertion (a) and condition (iii) in
Definition \ref{d21b}, we then conclude that any meromorphic continuation of
the function $f$ out of $H_{+}^{L}$ has a non-polar singularity at $z$. From
the last conclusion we deduce (\ref{f105d1}), but also the slightly stronger
assertion (e), which is formulated next.

\begin{itemize}
\item[(e)] Let $L$ be a straight line with the property that $K_{0}%
(f,\infty)\subset L\cup H_{-}^{L}$. If $K_{0}(f,\infty)\cap L\neq\emptyset$,
then we also have $\operatorname{Ex}(K_{0}(f,\infty))\cap L\neq\emptyset$, and
at each point $z\in\operatorname{Ex}(K_{0}(f,\infty))\cap L$ the restriction
of the function $f$ to $H_{+}^{L}$ has a non-polar singularity.
\end{itemize}

For the proof of the assertions (iv) and (v) in Theorem \ref{t74a} we need a
refinement of assertion (e).

\begin{itemize}
\item[(f)] Let $L$ be a straight line with the property that
$\operatorname{cap}(K_{0}(f,\infty)\cap H_{+}^{L})=0$. Then the function $f$
is single-valued in $H_{+}^{L}\setminus K_{0}(f,\infty)$. If further
$K_{0}(f,\infty)\cap L\neq\emptyset$, then we also have $\operatorname{Ex}%
(K_{0}(f,\infty)\setminus H_{+}^{L})\cap L\neq\emptyset$, and the restriction
of the function $f$ to $H_{+}^{L}\setminus K_{0}(f,\infty)$ has a non-polar
singularity at each point $z\in\operatorname{Ex}(K_{0}(f,\infty)\setminus
H_{+}^{L})\cap L$.
\end{itemize}

Indeed, the first part of assertion (f) follows directly from the fact that
$H_{+}^{L}\setminus K_{0}(f,\infty)\subset D_{0}(f,\infty)$.

From the assumption that $\operatorname{cap}(K_{0}(f,\infty)\cap H_{+}^{L})=0$
together with Lemma \ref{l111a} from Subsection \ref{s1101}, further below, we
conclude that the set $K_{0}(f,\infty)\cap H_{+}^{L}$ contains no continuum,
it is totally disconnected, and consequently, we have $K_{0}(f,\infty)\cap
H_{+}^{L}\subset E_{0}$ as a consequence of condition (iii) in Definition
\ref{d21b}.

If $K_{0}(f,\infty)\cap L\neq\emptyset$, then it is immediate that we also
have $\operatorname{Ex}(K_{0}(f,\infty)\setminus H_{+}^{L})\cap L\neq
\emptyset$. With the first part of assertion (f), we then conclude in the same
way as in the proof of assertion (e) that the meromorphic continuation of the
function $f$ out of the domain $H_{+}^{L}\setminus K_{0}(f,\infty)$ has a
non-polar singularity at every $z\in\operatorname{Ex}(K_{0}(f,\infty)\setminus
H_{+}^{L})\cap L$.\medskip

We now come to the individual proofs of the five assertions in Theorem
\ref{t74a}, where we will the assertions (a) - (f).\medskip

(ii) Assertion (ii) is an immediate consequence of assertion (c).\medskip

(i) \ We prove assertion (i) indirectly. Let us assume that there exists a
convex and compact set $K\subset\mathbb{C}$ such that the function $f$ has a
single-valued meromorphic continuation throughout $\overline{\mathbb{C}%
}\setminus K$ and that further
\begin{equation}
K_{0}(f,\infty)\setminus K\neq\emptyset.\label{f105d2}%
\end{equation}
From (\ref{f105d2}) and the convexity of $K$ it follows that also
\begin{equation}
\operatorname{Ex}(K_{0}(f,\infty))\setminus K\neq\emptyset.\label{f105d3}%
\end{equation}
With assertion (e) we then conclude that the meromorphic continuation of the
function $f$ out of $\overline{\mathbb{C}}\setminus(K_{0}(f,\infty)\cup
K)\subset D_{0}(f,\infty)$ has a non-polar singularity at each $z\in
\operatorname{Ex}(K_{0}(f,\infty))\setminus K$, which contradicts the
assumption that the function $f$ is meromorphic throughout $\overline
{\mathbb{C}}\setminus K$. Hence, assertion (i) is proved.\medskip

(iii) Assertion (iii) can be proved like assertion (i), only that now we have
to use assertion (f) instead of assertion (e). We give more details since some
of the conclusions will also be used in the proof of assertion (v), further below.

Assertion (iii) will be proved indirectly. We assume that $K\subset\mathbb{C}$
is a convex and compact set, $E\subset\mathbb{C}\setminus K$ a set that is
relatively compact in $\mathbb{C}\setminus K$, $\operatorname{cap}(E)=0$, the
function $f$ has a meromorphic and single-valued continuation throughout
$\overline{\mathbb{C}}\setminus(K\cup E)$, and
\begin{equation}
(K_{0}(f,\infty)\setminus E)\setminus K=K_{0}(f,\infty)\setminus(K\cup
E)\neq\emptyset.\label{f105d4}%
\end{equation}
From (\ref{f105d4}) it follows that also
\begin{equation}
\operatorname{Ex}(K_{0}(f,\infty)\setminus(K\cup E))\setminus(K\cup
E)\neq\emptyset.\label{f105d5}%
\end{equation}
Since $f$ is single-valued and meromorphic throughout $D_{0}(f,\infty
)=\overline{\mathbb{C}}\setminus K_{0}(f,\infty)$, it follows from assertion
(f) that $f$ has to have a non-polar singularity at each $z\in
\operatorname{Ex}(K_{0}(f,\infty)\setminus(K\cup E))\setminus(K\cup E)$, which
contradicts the assumption that $f$ is meromorphic throughout $\overline
{\mathbb{C}}\setminus(K\cup E)$.\medskip

(iv) Let $K_{min}$ be the intersection of all compact sets $K\subset
\mathbb{C}$ that satisfy the assumptions made in assertion (iii). With the
usual tools of planar topology one can show that $K_{min}$ can also be
represented as a denumerable intersection of such sets $K$. Like these sets
$K$, so the set $K_{min}$ is also convex and compact, and further it follows
from the assumptions made in assertion (iii) together with Lemma \ref{l111b}
in Subsection \ref{s1101} that
\begin{equation}
\operatorname{cap}(K_{0}(f,\infty)\setminus K_{min})=0.\label{f105d6}%
\end{equation}
We define $E_{min}:=K_{0}(f,\infty)\setminus K_{min}$, and with this
definition, assertion (iv) is proved.\medskip

(v) \ Notice that in assertion (iii) we can choose $K=K_{min}$. With the same
argumentation as used after (\ref{f105d5}), we show that the meromorphic
continuation of the function $f$ out of $D_{0}(f,\infty)\setminus K_{min}$ has
a non-polar singularity at each $z\in\operatorname{Ex}(K_{0}(f,\infty)\cap
K_{min})$. Form the definition of $K_{min}$ as the intersection of all compact
sets $K\subset\mathbb{C}$ that satisfy the assumptions made in assertion
(iii), it follows that $K_{0}(f,\infty)\subset K_{min}$. Hence, we have proved
that $\operatorname{Ex}(K_{min})\subset E_{0}$.

From $\operatorname{cap}(E_{min})=0$ together with Lemma \ref{l111a} from
Subsection \ref{s1101}, it follows that the set $E_{min}$ is totally
disconnected, and consequently, it follows from the Structure Theorem
\ref{t41a} that $E_{min}\subset E_{0}$, which completes the proof of assertion
(v).\bigskip
\end{proof}

\section{\label{s110}Some Lemmas from Potential Theory and Geometric Function
Theory}

\qquad In the present section we assemble definitions and lemmas concerning
basic properties and tools from potential theory and from geometric function
theory. These tools have been used at several places in the sections above. It
is hoped that by its concentration in a separate section the flow of
argumentation at earlier places has not been interrupted by argumentations of
a rather different flavor, or by references to the literature together with
the often necessary reformulations and adaptations of results. A separate
compilation is also more convenient and economic with respect to a unified
terminology, which unfortunately is not typical for the whole spectrum of the
literature in this area. As general references to potential-theoretic results
we have used \cite{SaTo}, \cite{Ransford}, and sometimes also \cite{Landkof}.
Towards the end of the present section results become more specific, and some
of them require rather technical proofs, which could not be found in the
literature with the required specific orientation.

We start with topics related to the (logarithmic) capacity, continue then with
logarithmic potentials, Green functions, some special results related to
sequences of compact sets, and at last with some remarks on trajectories of
quadratic differentials. In the penultimate subsection, Carath\'{e}odory's
Theorem about kernel convergence will be an important piece.

\subsection{\label{s1101}Notations and Basic Properties of Capacity}

\qquad The (logarithmic) capacity $\operatorname*{cap}\left(  \cdot\right)  $
is a set function defined on capacitable subsets of $\mathbb{C}$, which
include Borel sets (cf. \cite{Ransford}, Chapter 5 or \cite{SaTo}, Chapter
I.1). For a compact set $K\subset\mathbb{C}$ a definition with a strong
intuitive flavor can be based on the principle of minimal energy: Let $\mu$ be
a probability measure in $\mathbb{C}$; its energy is defined as
\begin{equation}
I(\mu):=\int\int\log\frac{1}{|z-v|}d\mu(z)d\mu(v).\label{f111a1}%
\end{equation}
The capacity of the compact set $K\subset\mathbb{C}$\ can then be defined as
\begin{equation}
\operatorname*{cap}\left(  K\right)  :=\exp\left(  -\inf_{\mu}I(\mu)\right)
,\label{f111a2}%
\end{equation}
where the infimum extends over all probability measures $\mu$ with
$\operatorname*{supp}\left(  \mu\right)  \subset K$.\medskip

A special role is played by sets $E\subset\mathbb{C}$ of capacity zero, which
are also known as polar sets. The property of being a set of capacity zero is
invariant under M\"{o}bius transforms, and thanks to this property, the notion
of 'capacity zero' can be extended to the whole Riemann sphere $\overline
{\mathbb{C}}$.

\begin{definition}
\label{d111a}A property is said to hold quasi everywhere (written in short as
'qu.e.') on a set $S\subset\overline{\mathbb{C}}$ if it holds for every $z\in
S\setminus E$ with $E$ a set of (outer) capacity zero (cf. \cite{SaTo},
Chapter I.1).
\end{definition}

The capacity is monoton with respect to an ordering by inclusions (cf.
\cite{Ransford}, Theorem 5.1.2). In our investigations we have needed some
upper and lower estimates, which are formulated in the next lemma (cf.
\cite{Ransford}, Theorems 5.3.2, 5.3.4, and 5.3.5).

\begin{lemma}
\label{l111a}(i) Let $m(\cdot)$ denote the planar Lebesgue measure in
$\mathbb{C}$. For any compact set $K\subset\mathbb{C}$ we have
\begin{equation}
\sqrt{m(K)/\pi}\leq\operatorname*{cap}\left(  K\right)  \leq
\operatorname*{diam}(K)/2.\label{f111b1}%
\end{equation}
(ii) For a continuum $V\subset\mathbb{C}$ we have
\begin{equation}
\operatorname*{diam}(K)/4\leq\operatorname*{cap}\left(  V\right)
\leq\operatorname*{diam}(K)/2.\label{f111b2}%
\end{equation}

\end{lemma}

As a set function, the capacity is not additive, and does also not possess one
of the usual subadditivity properties as a weaker substitute for the failing
additivity. However, sets of capacity zero are an exception in this respect
(cf. \cite{Ransford}, Theorem 5.1.4).

\begin{lemma}
\label{l111b}Let $K,E\subset\mathbb{C}$ be capacitable sets, and assume that
$\operatorname*{cap}\left(  E\right)  =0$. Then we have
\begin{equation}
\operatorname*{cap}\left(  K\setminus E\right)  =\operatorname*{cap}\left(
K\cup E\right)  =\operatorname*{cap}\left(  K\right)  .\label{f111c1}%
\end{equation}
The denumerable union of capacitable sets of capacity zero is again a set of
capacity zero.\medskip
\end{lemma}

Another property, which has been relevant in our investigations, concerns
radial projections $\varphi_{r}:\mathbb{C}\longrightarrow\mathbb{C}$ onto a
given disk $D:=\left\{  \text{ }|z|\leq r\text{ }\right\}  $, $r>0$, and
orthogonal projections on a line $L:=\left\{  \,z_{0}+v\,t\,|\,z_{0}%
,v\in\mathbb{C},\text{ }|v|=1,\text{ }t\in\mathbb{R}\,\right\}  $ from one
side. Let the radial projection $\varphi_{r}$ be defined by
\begin{equation}
z\mapsto\varphi_{r}(z):=\min(r,|z|)\frac{z}{|z|},\label{f111d1}%
\end{equation}
and the orthogonal projection $\varphi_{L}$ be defined by
\begin{equation}
z\mapsto\varphi_{L}(z):=\left\{
\begin{array}
[c]{lll}%
z_{0}+\operatorname{Re}((z-z_{0})\overline{v})v\smallskip & \text{ \ \ \ if
\ \ \ } & \operatorname{Im}((z-z_{0})\overline{v})>0,\\
z & \text{ \ \ else. \ } &
\end{array}
\right. \label{f111d3}%
\end{equation}

\begin{lemma}
\label{l111c}(cf. \cite{Pommerenke92}, Chapter 9.3, formula (11)) For any
capacitable set $K\subset\mathbb{C}$ and radial projection $\varphi_{r}$
defined in (\ref{f111d1}), we have
\begin{equation}
\operatorname*{cap}\left(  \varphi_{r}(K)\right)  \leq\operatorname*{cap}%
\left(  K\right)  ,\label{f111d2}%
\end{equation}
and for the orthogonal projection $\varphi_{L}$ defined in (\ref{f111d3}), we
also have
\begin{equation}
\operatorname*{cap}\left(  \varphi_{L}(K)\right)  \leq\operatorname*{cap}%
\left(  K\right)  .\label{f111d4}%
\end{equation}
We have a strict inequality in (\ref{f111d2}) or (\ref{f111d4})\ if
$\operatorname*{cap}(K\setminus\allowbreak\varphi_{r}\allowbreak(K))>0$ or
$\operatorname*{cap}(K\setminus\allowbreak\varphi_{L}(K))>0$,
respectively.\medskip
\end{lemma}

The capacity of a set depends only on the outer boundary of this set, which
will become clear from the next definition and the follow-on lemma.\smallskip

\begin{definition}
\label{d111b}For a bounded set $S\subset\mathbb{C}$ the polynomial-convex hull
$\widehat{S}$ (also denoted by $\operatorname*{Pc}(S)$) is defined as the
union of $\overline{S}$ with all bounded components of $\overline{\mathbb{C}%
}\setminus\overline{S}$. The set $\partial\widehat{S}$ is call the outer
boundary, and $\Omega_{S}=\overline{\mathbb{C}}\setminus\widehat{S}$ the outer
domain of $S$. A compact set $K\subset\mathbb{C}$ is call polynomial-convex if
$K=\widehat{K}$.\smallskip
\end{definition}

The notion 'polynomial-convex hull' hints to the possibility to define this
hull by polynomial inequalities. We have
\begin{equation}
\widehat{S}=\{\text{ }z\in\mathbb{C}\text{\ }|\text{ }|p(z)|\leq
||p||_{S}\text{\ for all }p\in\mathcal{P}\text{\ }\}\label{f111e1}%
\end{equation}
where $\mathcal{P}$ denotes the set of all polynomials and $||\cdot||_{S}$ the
uniform norm on $S$.\smallskip

\begin{lemma}
\label{l111d}(cf. \cite{Ransford}, Theorem 5.1.2) For all compact sets
$K\subset\mathbb{C}$ we have
\begin{equation}
\operatorname*{cap}(K)=\operatorname*{cap}(\widehat{K}).\label{f111e2}%
\end{equation}

\end{lemma}

A special property of polynomial-convex sets is the fact that their complement
is always a domain. The next lemma addresses a similar topic, but under
different circumstances.\smallskip

\begin{lemma}
\label{l111e}Let $S\subset\mathbb{C}$ be a set of capacity zero and
$D\subset\overline{\mathbb{C}}$ a domain, then $D\setminus S$ is connected. If
in addition $S$ is assumed to be closed in $D$, then $D\setminus S$ is a domain.
\end{lemma}

\begin{proof}
The lemma has a certain degree of immediate evidence since sets of capacity
zero are totally disconnected. However, a formal proof as to take care of the
topological difficulties in one or the other way. We will use the tools
provided by Lemma \ref{l91c} in Subsection \ref{s91}, further above.

The assertion that $D\setminus S$ is connected will be proved indirectly, and
for this purpose we assume that the opposite holds true. Then there exist two
disjoint open sets $O_{1},O_{2}\subset\overline{\mathbb{C}}$ with $D\setminus
S\subset O_{1}\cup O_{2}$ and $O_{j}\cap(D\setminus S)\neq\emptyset$ for
$j=1,2$. The set $\widetilde{K}:=D\setminus(O_{1}\cup O_{2}) $ is closed in
$D$ and we have $\widetilde{K}\subset S$.

Let $z_{j}\in O_{j}\cap D$, $j=1,2$, be two points, and $\gamma_{0}$ a Jordan
arc connecting $z_{1}$ with $z_{2}$ in $D$, and let further $U\subset D$ be a
small, open, simply-connected neighborhood of $\gamma_{0}$ with $\overline
{U}\subset D$. The arc $\gamma_{0}$ can be extended to a closed Jordan curve
$\gamma_{1}$ in $\mathbb{C}$, and correspondingly $U$ can be extended to a
ring domain $R\subset\mathbb{C}$ with $\gamma_{1}\subset R$ separating the two
components $A_{1}$ and $A_{2}$ of $\overline{\mathbb{C}}\setminus R$. This
extension can be done in such a way that $R\cap\widetilde{K}=U\cap
\widetilde{K}$.

It is immediate that each Jordan curve $\gamma\subset R$ that separates
$A_{1}$ from $A_{2}$ has to intersect the compact set $K:=R\cap\widetilde{K}$,
for otherwise the two sets $O_{1}\cap D$ and $O_{2}\cap D$ would be connected.

After these preparations, we apply the tools offered in Lemma \ref{l91c},
which then shows that there exists a continuum $V\subset K$ which is not
reduced to a single point, and consequently we have
\begin{equation}
\operatorname{cap}(S)\geq\operatorname{cap}(\widetilde{K})\geq
\operatorname{cap}(V)>0\text{,}\label{f111f1}%
\end{equation}
which contradicts the assumption that $\operatorname{cap}(S)$.

If $S$ is closed in $D$, then $D\setminus S$ is open, and consequently it is a
domain.\medskip
\end{proof}

\subsection{\label{s1102}Logarithmic Potentials}

\qquad Let $\mu$ be a (Borel) measure with compact $\operatorname*{supp}%
\left(  \mu\right)  \subset\mathbb{C}$. The (logarithmic) potential of the
measure $\mu$ is defined as
\begin{equation}
p(\mu;z):=\int\log\frac{1}{|z-x|}d\mu(x).\label{f112a1}%
\end{equation}
It is a superharmonic function in $\mathbb{C}$, and it is continuous quasi
everywhere in $\mathbb{C}$ for every measure $\mu$ (cf. \cite{Landkof},
Chapter III, Theorem 3.6). In the fine topology it is even continuous
throughout $\mathbb{C}$, but in our investigations, the concept of fine
topology has not been used. We shall address subtle questions about continuity
only in connection with the Green function further below in Subsection
\ref{s1103}.\medskip

Let $\{\mu_{n}\}_{n\in\mathbb{N}}$ be a weakly convergent\ sequence of
measures with limit measure $\mu_{0}$; this is written as
\begin{equation}
\mu_{n}\overset{\ast}{\longrightarrow}\mu_{0}\text{ \ \ as \ \ }%
n\rightarrow\infty.\label{f112b1}%
\end{equation}
With the convergence (\ref{f112b1}) corresponds a specific asymptotic behavior
of the potentials $p(\mu_{n};\cdot)$, $n\in\mathbb{N}$, (cf. \cite{SaTo},
Chapter I.6.9), which is known as the Lower Envelope Theorem.

\begin{theorem}
[Lower Envelope Theorem]\label{t112a}If $\sup(\mu_{n})\subset K$ for all
$n\in\mathbb{N}$ with $K\subset\mathbb{C}$ compact, then from (\ref{f112b1})
it follows that
\begin{equation}
\liminf_{n\rightarrow\infty}p(\mu_{n};z)\geq p(\mu_{0};z)\label{f112b2}%
\end{equation}
for all $z\in\mathbb{C}$, and equality holds in (\ref{f112b2}) quasi
everywhere in $\mathbb{C}$.
\end{theorem}

On a compact set $K\subset\mathbb{C}$ of positive capacity, there uniquely
exists an equilibrium measure $\omega_{K}$ (cf. \cite{Ransford}, Chapter 3.3),
which is the probability measure on $K$ that minimizes the energy
(\ref{f111a2}). Its potential has a typical behavior on $K$ (cf.
\cite{Ransford}, Theorem 3.3.4), we have
\begin{equation}
p(\omega_{K};z)\left\{
\begin{array}
[c]{lll}%
=-\log\operatorname*{cap}\left(  K\right)  \smallskip & \text{ \ for quasi
every} & z\in\widehat{K}\\
>-\log\operatorname*{cap}\left(  K\right)  & \text{ \ for all \ } & z\in
\Omega_{K},
\end{array}
\right. \label{f112c1}%
\end{equation}
where $\Omega_{K}=\overline{\mathbb{C}}\setminus\widehat{K}$ is the outer
domain and $\widehat{K}$ its polynomial-convex hull of $K$. Both objects have
been introduced in Definition \ref{d111b}.$\smallskip$

In potential theory a special role is played by measures of finite energy,
i.e., measures $\mu$ with $I(\mu)<\infty$ and $I(\cdot)$ defined by
(\ref{f111a1}). For instance, we have the following result (\cite{Ransford},
Theorem 3.2.3).

\begin{lemma}
\label{l112a}For any measure $\mu$ of finite energy and any bounded measurable
set $S\subset\mathbb{C}$ with $\operatorname*{cap}\left(  S\right)  =0$, we
have $\mu\left(  S\right)  =0$.

The equilibrium measure $\omega_{K}$ of a compact set $K\subset\mathbb{C}$
with $\operatorname*{cap}\left(  K\right)  >0$ is of finite energy.
\end{lemma}

In potential theory, a number of basic properties are known as principles; a
first one has already been stated in Theorem \ref{t112a}. In our
investigations we have also needed the next one.

\begin{theorem}
[Principle of Domination]\label{t112b}Let $\mu_{1}$ and $\mu_{2}$ be two
(positive) measures with compact support in $\mathbb{C}$, let $\mu_{1}$ be of
finite energy, and let $c\in\mathbb{R}$ be a constant. If the inequality
\begin{equation}
p(\mu_{1};z)\leq p(\mu_{2};z)+c\label{f112g1}%
\end{equation}
holds true for $\mu_{1}$-almost every $z\in\mathbb{C}$, or if it holds true
for quasi every $z\in\operatorname*{supp}(\mu_{1})$, then inequality
(\ref{f112g1}) holds true for all $z\in\mathbb{C}$.
\end{theorem}

\begin{proof}
The theorem has been proved in \cite{SaTo}, Theorem II.3.2, under the
assumption that (\ref{f112g1}) is satisfied $\mu_{1}$-almost everywhere.

If (\ref{f112g1}) holds true for quasi every $z\in\operatorname*{supp}(\mu
_{1})$, then it follows from Lemma \ref{l112b} and from the assumption that
$\mu_{1}$ is of finite energy that inequality (\ref{f112g1}) holds true also
$\mu_{1}$-almost everywhere.\medskip
\end{proof}

The minimum of two potentials can again be represented by a logarithmic potential.

\begin{lemma}
\label{l112c}Let $\mu_{1}$ and $\mu_{2}$ be two (positive) measures, then
there exists a (positive) measure $\mu_{0}$ and a constant $r_{0}\in\mathbb{R}
$ such that
\begin{equation}
\min\left(  p(\mu_{1};\cdot),p(\mu_{2};\cdot)\right)  =r_{0}+p(\mu_{0}%
;\cdot)\label{f112e1}%
\end{equation}
with $\left\|  \mu_{0}\right\|  =\max(\left\|  \mu_{1}\right\|  ,\left\|
\mu_{2}\right\|  )$. If the two measures $\mu_{1}$ and $\mu_{2}$ are of finite
energy, then the same is true for $\mu_{0}$.
\end{lemma}

\begin{proof}
It is rather immediate that the minimum of two superharmonic functions is
again superharmonic. One has only to check the definition of superharmonicity.
The lemma then follows from the Poisson-Jensen Formula ( \cite{Ransford},
Theorem 4.5.1). The determination of $\left\Vert \mu_{0}\right\Vert $ follows
from a consideration of $\min(p(\mu_{1};\cdot),\,p(\mu_{2};\allowbreak\cdot))$
near infinity.\medskip
\end{proof}

A broad variety of manipulations is possible in the class of logarithmic
potentials if one allows signed measures $\sigma$ in (\ref{f112a1}).

A signed measure $\sigma$ is of finite energy, i.e., $I(\sigma)<\infty$, if
and only if each of its two components $\sigma_{+}$ and $\sigma_{-}$
($\sigma=\sigma_{+}-\sigma_{-}$, $\sigma_{+},\sigma_{-}\geq0$) is of finite energy.

In order to keep our notations simple, we speak of logarithmic potentials also
if there is an additive constant, as for instance, is the case on the
right-hand side of (\ref{f112e1}).\smallskip

\begin{lemma}
\label{l112d}Let the two potentials $p_{j}$, $j=1,2$, be given by
\begin{equation}
p_{j}=r_{j}+p(\sigma_{j};\cdot),\text{ \ \ \ }j=1,2,\label{f112f1}%
\end{equation}
with $r_{j}\in\mathbb{R}$ and $\sigma_{j}$, $j=1,2$, signed measures in
$\mathbb{C}$. The functions $p_{3}:=|p_{1}|$, $p_{4}:=\max(p_{1},0)$,
$p_{5}:=\min(p_{1},0)$, $p_{6}:=\max(p_{1},p_{2})$, and $p_{7}:=\min
(p_{1},p_{2})$ then have representations of the same form as in (\ref{f112f1})
with modified constants $r_{j}\in\mathbb{R}$ and signed measures $\sigma_{j}$,
$j=3,\ldots,7$. If the two measures $\sigma_{1}$ and $\sigma_{2}$ are of
finite energy, then the same is true for the five measures $\sigma_{3}%
,\ldots,\sigma_{7}$.
\end{lemma}

\begin{proof}
For the positive and negative components of the two measures $\sigma_{j}$,
$j=1,2$, we write $\sigma_{j+}$ and $\sigma_{j-}$, respectively, i.e., we have
$\sigma_{j}=\sigma_{j+}-\sigma_{j-}$.We consider the potentials $p_{j+}%
=r_{j}+p(\sigma_{j+};\cdot)$, $p_{j-}=p(\sigma_{j-};\cdot)$, $j=1,2 $. Since
we have
\begin{equation}
|p_{1}|=p_{1+}+p_{1-}-2\,\min(p_{1+},p_{1-}),\label{f112f2}%
\end{equation}
representation (\ref{f112f1}) for $p_{3}$ follows directly from Lemma
\ref{l112c}. The representations for $p_{4},\ldots,p_{7}$ follow then as
further consequences since we have $p_{4}=\frac{1}{2}(p_{1}+|p_{1}|)$,
$p_{5}=\frac{1}{2}(p_{1}-|p_{1}|)$, $p_{6}=p_{1}+\max(p_{2}-p_{1},0)$, and
$p_{7}=p_{1}+\min(p_{2}-p_{1},0)$. The conclusion about the finite energy of
the measures $\sigma_{3},\ldots,\sigma_{7}$ follows from the corresponding
conclusion in Lemma \ref{l112c}.\medskip
\end{proof}

An important tool for the work with logarithmic potentials is the balayage
technique (sweeping out of a measure) (cf. \cite{SaTo}, Chapter II.4). In case
of logarithmic potentials, the balayage out of an unbounded domain requires
special attention.

\begin{definition}
\label{d112a}Let $\mu$ be a measure in $\mathbb{C}$.

\begin{itemize}
\item[(i)] For a bounded domain $D\subset\mathbb{C}$ with $\operatorname*{cap}%
\left(  \partial D\right)  >0,$ by $\widehat{\mu}$\ we denote the balayage
measure resulting from sweeping the measure $\mu$ out of the domain $D$;
it\ has its support on $\partial D\cup\operatorname*{supp}(\mu)\setminus D$,
and it is defined by the relation
\begin{equation}
p(\widehat{\mu};z)=p(\mu;z)\label{f112c2}%
\end{equation}
for every $z\in\overline{\mathbb{C}}\setminus\overline{D}$ and for quasi every
$z\in\partial D$. The balayage measure $\widehat{\mu}$ is uniquely determined
by (\ref{f112c2}) if we assume in addition to (\ref{f112c2}) that
$\widehat{\mu}(Ir(\partial D))=0$, where $Ir(\partial D)$ is the set of
critical points of $\partial D$ that will be introduced in Definition
\ref{d113a} in the next subsection.

\item[(ii)] For an unbounded domain $D\subset\overline{\mathbb{C}}$ with
$\infty\in D$ and $\operatorname*{cap}\left(  \partial D\right)  >0,$ the
concept of balayage is the same as in (i) only that relation (\ref{f112c2})
now has the modified form
\begin{equation}
p(\widehat{\mu};z)=p(\mu;z)+c_{1}\label{f112c3}%
\end{equation}
with a constant $c_{1}>0$ given by
\begin{equation}
c_{1}=\int g_{D}(x,\infty)d\mu(x)\label{f112c4}%
\end{equation}
where $g_{D}$ is the Green function in $D$, which will be introduced in the
next subsection.\smallskip
\end{itemize}
\end{definition}

With the help of the balayage technique, we can introduce an additional method
for manipulating logarithmic potentials which in some sense complements the
methods considered in Lemma \ref{l112d}, and which has also been used further
above.\smallskip

\begin{lemma}
\label{l112e}Let the two logarithmic potentials $p_{j}$, $j=1,2$, be given in
the form (\ref{f112f1}) with signed measures $\sigma_{1}$ and $\sigma_{2}$
that are of finite energy, and let further $D\subset\overline{\mathbb{C}} $ be
a (possibly unbounded) open set with connected complement. We define a new
function $p_{0}$ in a piecewise manner by
\begin{equation}
p_{0}(z):=\left\{
\begin{array}
[c]{lcl}%
p_{1}(z) & \text{ \ for \ }\smallskip & z\in D,\\
p_{2}(z) & \text{ \ for \ } & z\in\overline{\mathbb{C}}\setminus D.
\end{array}
\right. \label{f112h1}%
\end{equation}
If we have
\begin{equation}
p_{1}(z)=p_{2}(z)\text{ \ \ \ for quasi every \ \ }z\in\partial
D,\label{f112h2}%
\end{equation}
then there exists a signed measure $\sigma_{0}$ in $\mathbb{C}$ and a constant
$r_{0}\in\mathbb{R}$ such that
\begin{equation}
p_{0}(z)=r_{0}+p(\sigma_{0};z)\text{ \ \ for quasi every \ \ }z\in
\mathbb{C}.\label{f112h3}%
\end{equation}
The measure $\sigma_{0}$ is of finite energy, and in (\ref{f112h3}) we have
equality everywhere in $\overline{\mathbb{C}}\setminus\partial D$. Further we
have
\begin{equation}
\sigma_{0}|_{D}=\sigma_{1}|_{D}\text{ \ \ and \ \ }\sigma_{0}|_{\mathbb{C}%
\setminus\overline{D}}=\sigma_{2}|_{\mathbb{C}\setminus\overline{D}%
}.\label{f112h4}%
\end{equation}

\end{lemma}

\begin{proof}
The function
\begin{equation}
d:=p_{1}-p_{2}=r_{1}-r_{2}+p(\sigma_{1}-\sigma_{2};\cdot)\label{f112h5}%
\end{equation}
has the form (\ref{f112f1}) with defining measure $\sigma_{1}-\sigma_{2}$. If
we apply the balayage technique to the measure $\sigma_{1}-\sigma_{2}$ and
sweep this measure out of the domain $\overline{\mathbb{C}}\setminus
\overline{D}$, than this leads to a balayage measure in $\overline{D}$ which
we denote by $\widehat{\sigma}_{12}$. With an appropriately chosen constant
$r_{12}\in\mathbb{R}$, we have
\begin{align}
& \operatorname*{supp}(\widehat{\sigma}_{12})\subset\overline{D}\text{,
\ \ \ }\widehat{\sigma}_{12}|_{D}=(\sigma_{1}-\sigma_{2})|_{D}\text{,}%
\smallskip\label{f112h6a}\\
& \text{ \ \ }d(z)=r_{12}+p(\widehat{\sigma}_{12};z)\text{ \ \ for quasi every
}\ z\in\partial D\text{,}\label{f112h6b}%
\end{align}
and the inequality in (\ref{f112h6b}) holds also for all\ $z\in D$. The two
statements in (\ref{f112h6a}) and (\ref{f112h6b}) are a consequence of part
(ii) in Definition \ref{d112a}. We define
\begin{equation}
\widehat{d}:=r_{12}+p(\widehat{\sigma}_{12};\cdot).\label{f112h7}%
\end{equation}

Since logarithmic potentials are continuous quasi everywhere in $\mathbb{C}$
(cf. \cite{Landkof}, Chapter III, Theorem 3.6), it follows from (\ref{f112h2})
and (\ref{f112h5}) that $d(z)=0$ for quasi every $z\in\partial D$, and hence,
we deduce from (\ref{f112h6b}) that
\begin{equation}
\widehat{d}(z)=0\text{ \ \ \ for quasi every }\ z\in\partial D,\label{f112h8}%
\end{equation}
and because of (\ref{f112h6a}) further that
\begin{equation}
\widehat{d}(z)=0\text{ \ \ \ for all }\ z\in\overline{\mathbb{C}}%
\setminus\overline{D}.\label{f112h9}%
\end{equation}
From (\ref{f112h1}), (\ref{f112h6b}), (\ref{f112h7}), (\ref{f112h8}), and
(\ref{f112h9}),\ it then follows that
\begin{equation}
p_{0}(z)=p_{2}(z)+\widehat{d}(z)=r_{2}+r_{12}+p(\sigma_{2}+\widehat{\sigma
}_{12};z)\label{f112h10}%
\end{equation}
for all \ $z\in\overline{\mathbb{C}}\setminus\partial D$\ and for quasi
every\ $z\in\partial D$, which proves (\ref{f112h3}) if we set
\begin{equation}
\sigma_{0}:=\sigma_{2}+\widehat{\sigma}_{12}.\label{f112h11}%
\end{equation}
Since $\operatorname*{cap}(\overline{D})>0$ and since $\sigma_{1}-\sigma_{2}$
is of finite energy, it follows from (\ref{f112c3}) and (\ref{f112c4}) that
the measure $\widehat{\sigma}_{12}$\ is of finite energy, and consequently the
same is true for $\sigma_{2}+\widehat{\sigma}_{12}$. The identities in
(\ref{f112h4}) follow from (\ref{f112h6a}).\smallskip
\end{proof}

We close the present subsection with some estimates of the logarithmic energy
(\ref{f111a1}) associated with signed measures. It is important here that the
logarithmic kernel in (\ref{f111a1}) is positive definite for signed measures
$\sigma$ with $\operatorname*{supp}(\sigma)\subset K\subset\mathbb{C}$ if $K$
is a compact set with $\operatorname*{cap}(K)\leq1$. In the next lemma
estimates have been put together that are relevant in this connection.

\begin{lemma}
\label{l112b}(i) Let $K\subset\mathbb{C}$ be a compact set of positive
capacity. For all signed measures $\sigma$ with $\operatorname*{supp}%
(\sigma)\subset K$ we have
\begin{equation}
I(\sigma)\geq\sigma(K)^{2}\log\frac{1}{\operatorname*{cap}(K)}.\label{f112d1}%
\end{equation}

(ii) \ Let $\sigma$ be a signed measure in $\mathbb{C}$ with
\begin{equation}
\sigma(\mathbb{C})=0,\label{f112d2}%
\end{equation}
and let either $\operatorname*{supp}(\sigma)$ be a compact set or let $\sigma$
be a signed measure of finite energy, then we have
\begin{equation}
I(\sigma)\geq0,\label{f112d3}%
\end{equation}
and equality holds in (\ref{f112d3}) if, and only if, $\sigma=0$.
\end{lemma}

\begin{proof}
Part (ii) of the lemma has been proved in \cite{Landkof}, Theorem 1.6.

In a first step of the proof of part (i), we assume the set $K$ is regular
(cf. Definition \ref{d113a}, below). We set $a:=\sigma(K)$, and define
\begin{equation}
\sigma_{0}:=\sigma-a\,\omega_{K}\label{f112d4}%
\end{equation}
with $\omega_{K}$ the equilibrium measure on $K$. Consequently, we have
$\sigma_{0}(\mathbb{C})=0$. From (\ref{f111a1}) it follows that
\begin{equation}
I(\sigma)=I(\sigma_{0})+a^{2}I(\omega_{K})+2a\,I(\sigma_{0},\omega
_{K}),\label{f112d5}%
\end{equation}
where
\begin{equation}
I(\sigma_{0},\omega_{K})=\int\int\log\frac{1}{|z-v|}d\sigma_{0}(z)d\omega
_{K}(v)\label{f112d6}%
\end{equation}
is the mutual energy of the two measures $\sigma_{0}$ and $\omega_{K}$, which
in case of the equilibrium distribution $\omega_{K}$ can be expressed as
\begin{equation}
I(\sigma_{0},\omega_{K})=\int\left[  -g_{\Omega}(z,\infty)-\log
\operatorname*{cap}\left(  K\right)  \right]  d\sigma_{0}(z)\label{f112d7}%
\end{equation}
with the help of Lemma \ref{l113b}, below. In (\ref{f112d7}), $\Omega$ is the
outer domain of $K$.

From the assumption that $K$ is regular, it follows that $g_{\Omega}%
(z,\infty)=0$\ for all $z\in K$ (cf. the properties stated in (\ref{f113a1}),
further below). From $\sigma_{0}(\mathbb{C})=0$ and (\ref{f112d7}), it then
follows that $I(\sigma_{0},\omega_{K})=0$. From part (ii) we know that
$I(\sigma_{0})\geq0$, which together with (\ref{f112d5}) and (\ref{f111a2})
proves (\ref{f112d1}).

If the compact set $K\subset\mathbb{C}$ is not regular, then it can be
approximated from the outside by open sets (cf. \cite{Ransford}, Theorem
5.1.3). This implies that for any $\varepsilon>0$ there exists a compact set
$\widetilde{K}\subset\mathbb{C}$ with $\partial\widetilde{K}$ consisting of
piece-wise analytic arcs, $K\subset\operatorname*{Int}(\widetilde{K})$, and
$\operatorname*{cap}(\widetilde{K})\leq\operatorname*{cap}(K)+\varepsilon$.
Since $\varepsilon>0$ is arbitrary, (\ref{f112d1}) holds also true in the
non-regular case.$\medskip$
\end{proof}

\subsection{\label{s1103}The Green Function}

\qquad By $g_{D}(\cdot,w)$ we denote the Green function in a domain
$D\subset\overline{\mathbb{C}}$ with logarithmic singularity at $w\in D$ (for
a definition see \cite{Ransford}, Chapter 4.4, or \cite{SaTo}, Chapter I.4).
Somewhat different from the usual definitions, we assume that the Green
function $g_{D}(\cdot,w)$ is defined throughout $\overline{\mathbb{C}} $ and
also for domains $D\subset\overline{\mathbb{C}}$ with $\operatorname*{cap}%
\left(  \partial D\right)  =0$. If the domain $D$ has a boundary $\partial D$
of positive capacity, then for $w\in D$ we have
\begin{equation}
g_{D}(z,w)\left\{
\begin{array}
[c]{lll}%
=0\smallskip & \text{ \ for quasi every} & z\in\partial D\\
>0\smallskip & \text{ \ for all \ } & z\in D\\
=0 & \text{ \ for all \ } & z\in\overline{\mathbb{C}}\setminus\overline{D}.
\end{array}
\right. \label{f113a1}%
\end{equation}
If $\operatorname*{cap}\left(  \partial D\right)  =0$ and $w\in D$, then we
define $g_{D}(\cdot,w)\equiv\infty$.$\medskip$

Irregular points of $\partial D$ with respect to solutions of Dirichlet
problems in the domain $D\subset\overline{\mathbb{C}}$ have required special
attention at several places in our investigations. Irregular points are indeed
an interesting topic in potential theory. This type of points can be defined
in many different ways; one of the possibilities is based on the behavior of
the Green function $g_{D}(\cdot,w)$ on $\partial D$ (cf. \cite{Ransford},
Chapter 4.2). We use this approach in the next definition.\smallskip

\begin{definition}
\label{d113a}A point $z\in\partial D$ is irregular with respect to Dirichlet
problems in the domain $D$ (or short: it is an irregular point of $\partial
D$) if $g_{D}(z,w)>0$ for some $w\in D$. The set of all irregular points of
$\partial D$ is denoted by $Ir(\partial D)$.\smallskip
\end{definition}

It follows from the existence of the Riemann mapping function (see also
\cite{Ransford} Theorem 4.2.1) that if $D\subset\overline{\mathbb{C}}$ is a
simply connected domain and $\partial D$ is not reduced to a single point,
then $Ir(\partial D)=\emptyset$.\smallskip

Often we have had to deal with the outer domains $\Omega_{K}$ of a compact set
$K\subset\mathbb{C}$; the irregular points of $\partial\Omega_{K}$ are
elements of $K$. In the next definition we repeat certain aspects of
Definition \ref{d113a}, but with a refined and a partially new orientation.

\begin{definition}
\label{d113a2}Let $K\subset\mathbb{C}$ be a polynomial-convex set of positive
capacity with outer domain $\Omega_{K}$. By $Ir(K)\subset K$ we denote the set
$Ir(\partial\Omega_{K})$ of critical points. This set is broken down into the
two subsets
\begin{equation}
Ir_{I}(K):=Ir(K)\cap\overline{K\setminus Ir(K)}\text{ \ \ and \ }%
Ir_{II}(K):=Ir(K)\setminus(\overline{K\setminus Ir(K)}).\label{f113a2}%
\end{equation}
We further define the set of regular points of $K$ as $Rg(K):=K\setminus
Ir(K)$.

If $\operatorname*{cap}(K)=0$, then we defined $Ir_{II}(K):=Ir(K):=K$ and
$Ir_{I}(K):=Rg(K):=\emptyset$.\smallskip
\end{definition}

We note that the set $Rg(K)$ introduced in Definition \ref{d113a2} is more
comprehensive then the set $Rg(K)\cap\partial\Omega_{K}$ of regular points
with respect to solutions of Dirichlet problems in $\Omega_{K}$. An important
result in potential theory is Kellog's Theorem, which we state here in a
somewhat specialized and at the same time also extended version (cf.
\cite{Ransford}, Theorem 4.2.5 together with Theorem 4.4.9).\smallskip

\begin{lemma}
\label{l113a1}For a polynomial-convex set $K\subset\mathbb{C}$ we have
$\operatorname*{cap}(Ir(K))=0$, and the Green function $g_{\Omega}(\cdot
,w)$\ is continuous in $\mathbb{C}\setminus Ir_{I}(K)$ for every $w\in
\Omega=\Omega_{K}$.
\end{lemma}

As a consequence of the Lemmas \ref{l111a}, \ref{l112a}, and \ref{l113a1}, we
have the following results about irregular points and Green functions.

\begin{lemma}
\label{l113a2}Let $K\subset\mathbb{C}$ be a polynomial-convex set with outer
domain $\Omega=\Omega_{K}$.

\noindent(i) The set $Ir_{II}(K)$ is totally disconnected.

\noindent(ii) We have $\omega_{K}(Ir(K))=0$ for the equilibrium distribution
$\omega_{K}$ on $K$.

\noindent(iii) The Green function $g_{\Omega}(\cdot,\infty)$ is harmonic in
$(\Omega\setminus\{\infty\})\cup Ir_{II}(K)=\mathbb{C}\setminus\allowbreak
\overline{\operatorname*{Rg}(K)}$.

\noindent(iv) We have $\operatorname*{cap}(U_{z}\cap K)>0$ for every open
neighborhood $U_{z}\subset\mathbb{C}$ of a point $z\in\overline
{\operatorname*{Rg}(K)}$, and $\operatorname*{cap}(U_{z}\cap K)=0$ for every
open neighborhood $U_{z}\subset\mathbb{C}\setminus\overline{\operatorname*{Rg}%
(K)}$ of a point $z\in Ir_{II}(K)$.
\end{lemma}

\begin{proof}
The first three assertions are rather immediate. Assertion (iv) follows from
\cite{Ransford}, Theorem 4.2.3 and Theorem 4.2.4.\medskip
\end{proof}

A connection between logarithmic potentials and Green functions is given by
the representation formula in the next lemma (cf. \cite{Ransford}, Theorem
4.4.7 together with Theorem 5.2.1).

\begin{lemma}
\label{l113b}Let $K\subset\mathbb{C}$ be a compact set with
$\operatorname*{cap}\left(  K\right)  >0$, $\Omega=\Omega_{K}$ its outer
domain, and $\omega_{K} $ the equilibrium distribution on $K$. Then for the
Green function $g_{\Omega}(\cdot,\infty)$, we have the representation
\begin{equation}
g_{\Omega}(\cdot,\infty)=-p(\omega_{K};\cdot)+\log\frac{1}{\operatorname*{cap}%
\left(  K\right)  },\label{f113b1}%
\end{equation}
and near infinity we have
\begin{equation}
g_{\Omega}(z,\infty)=\log|z|+\log\frac{1}{\operatorname*{cap}\left(  K\right)
}+\text{O}(\frac{1}{|z|})\text{ \ \ as \ \ }z\rightarrow\infty.\label{f113b2}%
\end{equation}

\end{lemma}

The next result is related to Lemma \ref{l113b}.

\begin{lemma}
\label{l113c}Let $K_{1},K_{2}\subset\mathbb{C}$ be polynomial-convex sets of
positive capacity. Then we have
\begin{equation}
g_{\overline{\mathbb{C}}\setminus K_{1}}(\cdot,\infty)\equiv g_{\overline
{\mathbb{C}}\setminus K_{2}}(\cdot,\infty)\label{f113c2}%
\end{equation}
if, and only if,
\begin{equation}
\operatorname*{cap}\left(  \left(  K_{1}\setminus K_{2}\right)  \cup\left(
K_{2}\setminus K_{1}\right)  \right)  =0,\label{f113c1}%
\end{equation}
i.e., if, and only if, the two sets $K_{1}$ and $K_{2}$ differ only in a set
of capacity zero.
\end{lemma}

\begin{proof}
We assume that identity (\ref{f113c2}) holds true. From (\ref{f113c2})
together with the defining identity (\ref{f113a1}) for the Green function and
the definition of irregular points in Definition \ref{d113a} it follows that
\begin{equation}
\left(  \partial K_{1}\setminus\partial K_{2}\right)  \cup\left(  \partial
K_{2}\setminus\partial K_{1}\right)  \subset\operatorname*{Ir}(K_{1}%
)\cup\operatorname*{Ir}(K_{2}),\label{f113c3}%
\end{equation}
which with Lemma \ref{l113a1} and Lemma \ref{l111b} implies (\ref{f113c1}).

If, on the other hand, (\ref{f113c1}) holds true, then identity (\ref{f113c2})
is an immediate consequence of the defining identity (\ref{f113a1}) for the
Green function and the uniqueness of the Green function.$\medskip$
\end{proof}

The balayage technique of Definition \ref{d112a} can be made more concrete
with the help of the Green function (cf. \cite{SaTo}, Chapter II.4).

\begin{lemma}
\label{l113e}Under the assumptions of Definition \ref{d112a}, we have
\begin{equation}
p(\widehat{\mu};\cdot)=p(\mu;\cdot)-\int g_{D}(\cdot,x)d\mu(x)\label{f113e1}%
\end{equation}
if the domain $D$ is bounded, and otherwise we have
\begin{equation}
p(\widehat{\mu};\cdot)=p(\mu;\cdot)-\int\left[  g_{D}(\cdot,x)-g_{D}%
(x,\infty)\right]  d\mu(x).\label{f113e2}%
\end{equation}

\end{lemma}

Related to Lemma \ref{l113e} is the Riesz Decomposition Theorem (cf., Theorem
3.1 of \cite{SaTo}, Chapter II), and more definite the Poison-Jensen Formula,
which has been used at several places in our investigations.

\begin{theorem}
\label{t113a}(Poison-Jensen Formula) Let $D\subset\overline{\mathbb{C}}$ be a
domain with $\operatorname*{cap}\left(  \partial D\right)  \allowbreak>0$. We
assume that the real-valued function $u$ is subharmonic on $\overline{D}$, not
identical $-\infty$, and possesses an harmonic majorant in $\overline{D} $.
Then there exists a nonnegative measure $\mu$ in $D$ of finite mass such that
\begin{equation}
u(z)=-\iint_{D}g_{D}(z,v)d\mu(z)+\int_{\partial D}u(v)d\widehat{\delta_{z}%
}(v)\text{ \ \ for \ \ }z\in D\label{f113d3}%
\end{equation}
with $\widehat{\delta_{z}}$ denoting the balayage measure on $\partial D$
resulting from sweeping out the Dirac measure $\delta_{z}$ out of $D$.
($\widehat{\delta_{z}}$ is also known as the harmonic measure on $\partial D$
of the point $z\in D$.)
\end{theorem}

\begin{proof}
The theorem has been proved in \cite{Ransford} as Theorem 4.5.1 under stronger
assumptions about the domain $D$ and the function $u$. The more general form
of the theorem given here is the consequence of a combination of the two
Theorems 4.5.1 and 4.5.4 in \cite{Ransford}.
\end{proof}

Analogously to the energy $I(\cdot)$ that has been defined in (\ref{f111a1})
with a logarithmic kernel, one can also define an energy formula with a Green
kernel, i.e., a formula like (\ref{f111a1}) with a Green function as kernel.
The new formula is called Green energy. A systematic investigation of Green
energy and Green capacity together with the associated Green potentials can be
found in \cite{SaTo}, Chapter II.5.

In our investigations, we have needed the property of positive definiteness of
the Green kernel. The result is contain in the next lemma, and it can be seen
as a completion of the material in Lemma \ref{l112b}.

\begin{lemma}
\label{l113d}Let $D\subset\overline{\mathbb{C}}$ be a domain with
$\operatorname*{cap}\left(  \partial D\right)  >0$. For a signed measure
$\sigma$ of finite energy we have
\begin{equation}
\int\int g_{D}(z,v)d\sigma(z)d\sigma(v)\geq0\label{f113d1}%
\end{equation}
and equality holds in (\ref{f113d1}) if, and only if, $\sigma|_{D\cup
Ir(\partial D)}=0$.
\end{lemma}

\begin{proof}
The lemma can be proved like the analogous result in \cite{Landkof}, Theorem
1.6, together with the tools used in \cite{SaTo} for the proof of Lemma 5.4 in
Chapter II.\medskip
\end{proof}

We next come to some results that are connected with the Green formula. Let
$D\subset\overline{\mathbb{C}}$ be a domain with a smooth and non-empty
boundary $\partial D\subset\mathbb{C}$, and let further $u$ and $v$ be two
real $C^{2}-$functions in $D$ with $L^{1}-$integrable second derivatives in
$D$ and $C^{1}$ boundary functions with $L^{1}-$integrable first derivatives
on $\partial D$. Under these assumptions, the Green identity
\begin{equation}
\iint_{D}u\nabla\nabla v\,dm+\iint_{D}\nabla u\nabla v\,dm+\int_{\partial
D}u\frac{\partial}{\partial n}v\,ds=0\label{f113f1}%
\end{equation}
holds true (see Chapter VIII of \cite{Kellogg67} or \cite{HaymanKennedy76},
Theorem 1.9). In (\ref{f113f1}), $\nabla\nabla$ denotes the Laplace operator
$\partial^{2}/\partial x^{2}+\partial^{2}/\partial y^{2}$, $\nabla$ the napla
or gradient operator $(\partial/\partial x,\partial/\partial y)$,
$\partial/\partial n$ the inwardly showing normal derivation on $\partial D$,
$dm$ the area element in $D$, and $ds$ the (positively oriented) line element
on $\partial D$.

If the function $v$ is harmonic in $D$, then obviously the first term in
(\ref{f113f1}) vanishes. The second term is known as the Dirichlet integral of
$u$ and $v$, and we use the abbreviation
\begin{equation}
D_{D}(u,v):=\frac{1}{2\pi}\iint_{D}\nabla u\nabla v\,dm.\label{f113f2}%
\end{equation}
Using the same letter $D$ in one formula with two different meanings is
certainly unlucky, but mix-ups should be avoidable. In comparison to the
assumptions made in (\ref{f113f1})\ , we often relax assumptions for
(\ref{f113f2}); thus, for instance, with the help of an exhaustion technique
we can admit arbitrary domains $D\subset\overline{\mathbb{C}}$. If not
explicitly stated otherwise, then we assume that both functions $u$ and $v$ in
(\ref{f113f2}) have $L^{2}-$integrable first order derivatives almost
everywhere in $D$.

In the special case that both functions $u$ and $v$ are identical, we write
\begin{equation}
D_{D}(u):=D_{D}(u,u)=\frac{1}{2\pi}\iint_{D}\left(  \nabla u\right)
^{2}dm.\label{f113f3}%
\end{equation}

Notice that in (\ref{f113f1}) the Dirichlet integral $2\pi D_{D}(u,v)$ is the
only term that is symmetric in both functions $u$ and $v$. The use of this
fact leads to interesting special cases of the Green identity.\thinspace Thus,
for instance, one gets Formula (1.1) in Chapter II.1 of \cite{SaTo}, which
will also be used in the present subsection; it is the basis for the proof of
Lemma \ref{l113f}, below, after the next paragraph.

Next, we come to several lemmas that are rather immediate consequences on the
Green identity. The first one has been used at several places, where
potentials have been defined in a piecewise manner. With respect to a proof of
this result, we are in the lucky situation that most of the detailed work has
already been done in \cite{SaTo}, Chapter II, where similar results have been proved.

\begin{lemma}
\label{l113f}Let $D\subset\overline{\mathbb{C}}$ be a domain with
$\operatorname*{cap}\left(  \partial D\right)  >0$, $\gamma$ a $C^{1+\delta}%
-$smooth Jordan arc in $D$, $\delta>0$, and $u$ a bounded real-valued function
that is continuous in $\overline{D}\setminus Ir(\partial D)$, harmonic in
$D\setminus\gamma$, and which possesses $L^{1}-$integrable normal derivatives
to both sides of $\gamma$.

If $u(z)=0$ for all $z\in\partial D\setminus Ir(\partial D)$, then we have
\begin{equation}
u(z)=-\frac{1}{2\pi}\int_{\gamma}(\frac{\partial}{\partial n_{-}}%
+\frac{\partial}{\partial n_{+}})u(v)g_{D}(z,v)ds_{v}\text{ \ \ for \ }z\in
D.\label{f113d2}%
\end{equation}
In (\ref{f113d2}), $\partial/\partial n_{+}$ and $\partial/\partial n_{-}$
denote the normal derivation to both sides of $\gamma$, $ds$ is the line
element on $\gamma$, and $g_{D}(\cdot,\cdot)$ the Green function in $D$.
\end{lemma}

\begin{proof}
From Theorem 1.5 in Chapter II of \cite{SaTo} it follows that if we define the
measure $\sigma$ on $\gamma$ by
\begin{equation}
d\sigma(v):=-\frac{1}{2\pi}(\frac{\partial}{\partial n_{-}}+\frac{\partial
}{\partial n_{+}})u(v)ds_{v},\label{f113d22}%
\end{equation}
then the function
\begin{equation}
d(z):=u(z)-\int g_{D}(z,v)d\sigma(v),\text{ \ \ \ \ \ }z\in D,\label{f113d21}%
\end{equation}
is harmonic in $D$. From the assumptions of the lemma and from (\ref{f113a1})
together with Definition \ref{d113a}, we know that $d(z)=0$ for all
$z\in\partial D\setminus Ir(\partial D)$. Since $d$ is bounded in $D$, it
follows from the uniqueness of the Dirichlet problem under the given
circumstances (cf. Theorem 3.1 in the Appendix A of \cite{SaTo}) that $d(z)=0$
for $z\in D$, which proves the lemma.\medskip
\end{proof}

In the next lemmas, properties of Green functions are expressed with the help
of Dirichlet integrals.

\begin{lemma}
\label{l113g}Let $D\subset\overline{\mathbb{C}}$ be a domain with $\infty\in
D$, $r>0$, and $\operatorname*{cap}\left(  \partial D\right)  >0$, then we
have
\begin{equation}
D_{\{|z|<r\}\cap D}(g_{D}(\cdot,\infty))=\log(r)+\log\frac{1}%
{\operatorname*{cap}\left(  \partial D\right)  }+\text{O}(\frac{1}{r})\text{
\ \ as \ \ }r\rightarrow\infty.\label{f113g1}%
\end{equation}

\end{lemma}

\begin{proof}
The lemma is a consequence of the Green identity (\ref{f113f1}). In a first
step we assume that the domain $D$\ has a sufficiently smooth boundary
$\partial D$. If $\partial D$ is $C^{2}$\ smooth, then the Green function
$g_{D}(\cdot,\infty)$ has continuous first order derivatives on $\partial D$
(details can be found, for instance, in the proof of Theorem 4.11 in Chapter
II of \cite{SaTo}). For $r>0$ sufficiently large, we define
\begin{equation}
D_{r}:=D\setminus\{|z|\geq r\}.\label{f113g11}%
\end{equation}
Since $g_{D}(\cdot,\infty)$ is harmonic in $D_{r}$, we deduce from
(\ref{f113f1}) and (\ref{f113f3})\ that
\begin{align}
& D_{\{|z|<r\}\cap D}(g_{D}(\cdot,\infty))=\frac{1}{2\pi}\iint_{D_{r}}(\nabla
g_{D}(\cdot,\infty))^{2}\,dm\nonumber\\
& \text{ \ \ \ \ \ \ \ \ \ \ \ \ \ \ \ \ \ \ \ \ \ \ \ \ \ \ \ \ }=-\frac
{1}{2\pi}\int_{\partial D_{r}}g_{D}(\cdot,\infty)\frac{\partial}{\partial
n}g_{D}(\cdot,\infty)\,ds\label{f113g12}\\
& =-\frac{1}{2\pi}\int_{\partial D}g_{D}(\cdot,\infty)\frac{\partial}{\partial
n}g_{D}(\cdot,\infty)\,ds-\frac{1}{2\pi}\int_{\{|z|=r\}}g_{D}(\cdot
,\infty)\frac{\partial}{\partial n}g_{D}(\cdot,\infty)\,ds.\nonumber
\end{align}

From the smoothness of $\partial D$\ it follows that $\partial D$ is regular,
and consequently that $g_{D}(z,\infty)=0$ for all $z\in\partial D$, which
implies that the first integral in the last line of (\ref{f113g12}) is
identical zero.

From Lemma \ref{l113b} we know that the function
\begin{equation}
\widetilde{g}:=g_{D}(\cdot,\infty)-\log|\cdot|\label{f113g13}%
\end{equation}
is harmonic in $\{|z|>r\}$, and we have $\widetilde{g}(\infty)=-\log
\operatorname*{cap}\left(  \partial D\right)  $. From (\ref{f113g13}) it
follows that for the inward showing normal derivative on the circle
$\{|z|=r\}$ we have
\begin{equation}
\frac{\partial}{\partial n}g_{D}(z,\infty)=\frac{\partial}{\partial
n}\widetilde{g}(z)+\frac{1}{r}\text{ \ \ \ for \ \ }|z|=r.\label{f113g14}%
\end{equation}
It is rather immediate that
\begin{align}
& \frac{1}{2\pi}\medskip\int_{\{|z|=r\}}\widetilde{g}(z)\frac{1}{r}%
\,ds_{z}=\widetilde{g}(\infty)=\log\frac{1}{\operatorname*{cap}\left(
\partial D\right)  },\label{f113g14a}\\
& \frac{1}{2\pi}\int_{\{|z|=r\}}\log|z|\frac{1}{r}\,ds_{z}=\log
(r),\label{f113g14b}%
\end{align}
and since $\widetilde{g}$ is harmonic in $\{|z|>r\}$, we further have
\begin{align}
& \medskip\int_{\{|z|=r\}}\frac{\partial}{\partial n}\widetilde{g}%
(z)\,ds_{z}=0,\label{f113g14c}\\
& ||\frac{\partial}{\partial n}\widetilde{g}||_{\{|z|=r\}}=\text{O}(\frac
{1}{r^{2}})\text{ \ \ \ as \ \ \ }r\rightarrow\infty,\label{f113g14d}%
\end{align}
where $||\cdot||_{\{|z|=r\}}$ denotes the sup-norm on $\{|z|=r\}$.

Using (\ref{f113g13}) through (\ref{f113g14d}), the only remaining integral in
the last line of (\ref{f113g12}) can be transformed in the following way:
\begin{align}
& \text{ }-\frac{1}{2\pi}\medskip\int_{|z|=r}g_{D}(\cdot,\infty)\frac
{\partial}{\partial n}g_{D}(\cdot,\infty)\,ds=\nonumber\\
& \text{ \ \ \ \ \ }\medskip-\frac{1}{2\pi}\int_{|z|=r}(\log|\cdot
|\frac{\partial}{\partial n}\widetilde{g}+\log|\cdot|\frac{1}{r}+\widetilde
{g}\frac{\partial}{\partial n}\widetilde{g}+\widetilde{g}\frac{1}{r})ds=\text{
\ \ \ \ \ \ \ \ }\label{f113g15}\\
& \text{ \ \ \ \ \ \ \ \ \ \ \ \ }\log(r)+\log\frac{1}{\operatorname*{cap}%
\left(  \partial D\right)  }+\text{O}(\frac{1}{r})\text{ \ \ as \ \ }%
r\rightarrow\infty,\nonumber
\end{align}
which then proves (\ref{f113g1}) under the assumption of a sufficiently smooth
boundary $\partial D$.

In the general case, the domain $D$ is exhausted by a nested sequence of
domains $D_{n}$, $n\in\mathbb{N}$, with sufficiently smooth boundaries
$\partial D_{n}$. We assume that
\begin{equation}
\overline{D_{n}}\subset D_{n+1}\subset D\text{ \ \ and \ \ \ }D=\bigcup
_{n}D_{n}.\label{f113g16}%
\end{equation}
By $g_{n}$ we denote the Green function $g_{D_{n}}(\cdot,\infty)$. Because of
(\ref{f113g16}) we have
\begin{equation}
g_{n}(z)\geq g_{n+1}(z)\geq g_{D}(z,\infty)\text{ \ \ for \ \ }z\in
\mathbb{C}.\label{f113g17}%
\end{equation}
From the Harnack principle of monotonic convergence (cf. Theorem 4.10 in
Chapter 0 of \cite{SaTo}) and the assumption that $\operatorname*{cap}\left(
\partial D\right)  >0$, we then deduce that the sequence $\{g_{n}\}$ as well
as their first order derivatives $\nabla g_{n}$ converge locally uniformly in
$D$, i.e., we have
\begin{equation}
\lim_{n\rightarrow\infty}\nabla g_{n}=\nabla g\label{f113g18}%
\end{equation}
locally uniformly in $D$. From the identity (\ref{f113g1}) for the domains
$D_{n}$\ together with (\ref{f113g18}), identity (\ref{f113g1}) then follows
also in the general case.\medskip
\end{proof}

In (\ref{f113a1}), the Green function $g_{D}(\cdot,\cdot)$ has been defined
for the whole Riemann sphere $\overline{\mathbb{C}}$, and one could therefore
consider the extension of the Dirichlet integral in (\ref{f113g1}) from
$\{|z|\leq r\}\cap D$ to the whole disc $\{|z|\leq r\}$. Such an extension
would indeed be without problems if the planar Lebesgue measure of $\partial
D$ were zero. However, $\partial D$ may be of positive planar Lebesgue
measure.\smallskip

The combination of the assertion of Lemma \ref{l113g} for two different
domains yields the next corollary, which has been useful for the comparison of
the capacities of the complements of two domains.\medskip

\begin{corollary}
\label{c113g1}Let $D_{1},D_{2}\subset\overline{\mathbb{C}}$ be two domains
with $\infty\in D_{j}$\ and $\operatorname*{cap}\left(  \partial D_{j}\right)
\allowbreak>0$ for $j=1,2$. Then for $r>0$ we have
\begin{align}
& D_{\{|z|<r\}\cap D_{1}}(g_{D_{1}}(\cdot,\infty))-D_{\{|z|<r\}\cap D_{2}%
}(g_{D_{2}}(\cdot,\infty))=\medskip\label{f113g2}\\
& \text{
\ \ \ \ \ \ \ \ \ \ \ \ \ \ \ \ \ \ \ \ \ \ \ \ \ \ \ \ \ \ \ \ \ \ \ \ }%
\log\frac{\operatorname*{cap}\left(  \partial D_{2}\right)  }%
{\operatorname*{cap}\left(  \partial D_{1}\right)  }+\text{O}(\frac{1}%
{r})\text{ \ \ as \ \ }r\rightarrow\infty.\nonumber
\end{align}

\end{corollary}

\begin{lemma}
\label{l113h}Let the function $u$ be harmonic and bounded in the domain
$D\subset\overline{\mathbb{C}}$ with $\infty\in D$\ and $\operatorname*{cap}%
\left(  \partial D\right)  >0$. Then the Dirichlet integral $D_{D}%
(u,g_{D}(\cdot,\infty))$ exists, and we have
\begin{equation}
\lim_{r\rightarrow\infty}D_{\{|z|<r\}\cap D}(u,g_{D}(\cdot,\infty
))=D_{D}(u,g_{D}(\cdot,\infty))=0.\label{f113g4}%
\end{equation}

\end{lemma}

\begin{proof}
Like in the proof of Lemma \ref{l113g}, in a first step, we assume that the
domain $D$ has a $C^{2}$ smooth boundary $\partial D$. The subdomain $D_{r}$
is again defined by (\ref{f113g11}) for $r>0$ sufficiently large. Analogously
to (\ref{f113g12}), we have the identities
\begin{align}
& D_{D_{r}}(u,g_{D}(\cdot,\infty))\medskip=\frac{1}{2\pi}\iint_{D_{r}}\nabla
u\,\nabla g_{D}(\cdot,\infty)\,dm\nonumber\\
& \text{ \ \ \ \ \ \ \ \ \ \ \ \ \ \ \ \ \ \ \ \ \ \ \ \ }\medskip=-\frac
{1}{2\pi}\int_{\partial D_{r}}u\frac{\partial}{\partial n}g_{D}(\cdot
,\infty)\,ds\label{f113g41}\\
& \text{ \ }=\medskip-\frac{1}{2\pi}\int_{\partial D}u\frac{\partial}{\partial
n}g_{D}(\cdot,\infty)\,ds-\frac{1}{2\pi}\int_{\{|z|=r\}}u\frac{\partial
}{\partial n}g_{D}(\cdot,\infty)\,ds\nonumber\\
& \text{ \ }=-u(\infty)+u(\infty)+\text{O}(\frac{1}{r})=\text{O}(\frac{1}%
{r})\text{ \ \ as \ \ }r\rightarrow\infty.\nonumber
\end{align}
Indeed, the second equality is a consequence of the Green identity
(\ref{f113f1}). The penultimate equality in (\ref{f113g41}) follows from two
observations, which are concerned with the two integrals in the third line of
(\ref{f113g41}). In the first integral the normal derivative $(1/2\pi
)\allowbreak\partial g_{D}(\cdot,\infty)/\partial n$ is the density of the
equilibrium distribution on $\partial D$ (cf. Theorem 4.11 of Chapter II in
\cite{SaTo}), and it defines the balayage measure on $\partial D$ resulting
from sleeping $\delta_{\infty}$ out of $D$, which implies that this first
integral is equal to $-u(\infty)$.

The second integral in the third line of (\ref{f113g41}) extends over the
circle $\{|z|=r\}$. Using the definition of the function $\widetilde{g}$ in
(\ref{f113g13}) together with (\ref{f113g14}) and (\ref{f113g14d}) yields
\begin{align}
-\frac{1}{2\pi}\int_{\{|z|=r\}}u\frac{\partial}{\partial n}g_{D}(\cdot
,\infty)\,ds  & =-\frac{1}{2\pi}\int_{\{|z|=r\}}(u(z)\frac{\partial}{\partial
n}\widetilde{g}(z)+u(z)\frac{1}{r})\,ds_{z}\medskip\nonumber\\
& =\text{O}(\frac{1}{r})+u(\infty)\text{ \ \ \ as \ \ }r\rightarrow
\infty.\label{f113g42}%
\end{align}
The last two observation together prove the penultimate equality in
(\ref{f113g41}).

We note that the two integrals in the third line of (\ref{f113g41}) have
opposite orientations with respect to the two domains $D$\ and $\{|z|>r\}$.
Like in the conclusions after (\ref{f113g13}), and also in (\ref{f113g41}), we
have applied Theorem 3.1 of the Appendix A in \cite{SaTo}.

With (\ref{f113g41}) we have proved (\ref{f113g4}) under the assumption of a
sufficiently smooth boundary $\partial D$. We add that the existence of the
integral $D_{D_{r}}(u,g_{D}(\cdot,\allowbreak\infty))$ follows from the
Cauchy-Schwartz inequality $D_{D_{r}}(u,g_{D}(\cdot,\infty))^{2}\leq D_{D_{r}%
}(u)\allowbreak D_{D_{r}}(g_{D}(\cdot,\infty))$ together with
$\operatorname*{cap}\left(  \partial D\right)  >0$ and the assumed boundedness
of the function $u$.

Like in the proof of Lemma \ref{l113g}, for a general domain $D$ identity
(\ref{f113g4}) follows from exhausting the domain $D$ by a sequence of nested
domains $D_{n}$, $n\in\mathbb{N}$, with sufficiently smooth boundaries
$\partial D_{n}$.\medskip
\end{proof}

In the next two lemmas we prove rather technical results, which have been used
in Subsection \ref{s102}, further above.

\begin{lemma}
\label{l113i}Let $D\subset\overline{\mathbb{C}}$ be a domain with $\infty\in
D$ and $\operatorname*{cap}\left(  \partial D\right)  >0$. Set $D_{r}%
:=\{|z|<r\}\cap D$ with $r>0$ sufficiently large so that $\partial
D\subset\{|z|<r\}$, and let $u$ be a real-valued function that is harmonic in
$\{|z|>r\}$, and let further $\widehat{u}_{r}$ be the solution of the
Dirichlet problem in $D_{r}$ with boundary function
\begin{equation}
\widehat{u}_{r}(z)=\left\{
\begin{array}
[c]{lll}%
0\smallskip & \text{ \ \ \ for \ \ } & z\in\partial D,\\
u(z) & \text{ \ \ \ for \ } & |z|=r.
\end{array}
\right. \label{f113g51}%
\end{equation}
Under these assumptions, we have
\begin{equation}
D_{\{|z|<r\}\cap D}(\widehat{u}_{r},g_{D}(\cdot,\infty))=u(\infty
)+\text{O}(\frac{1}{r})\text{ \ \ as \ \ }r\rightarrow\infty.\label{f113g5}%
\end{equation}

\end{lemma}

\begin{proof}
In a first step, we assume that $D$ has a $C^{2}$ smooth boundary $\partial D
$. Like in (\ref{f113g41}), we deduce from (\ref{f113f1}) that
\begin{align}
& D_{D_{r}}(\widehat{u}_{r},g_{D}(\cdot,\infty))\medskip\label{f113g52}\\
& \text{ \ \ }=-\frac{1}{2\pi}\int_{\partial D}\widehat{u}_{r}\frac{\partial
}{\partial n}g_{D}(\cdot,\infty)\,ds-\frac{1}{2\pi}\int_{\{|z|=r\}}\widehat
{u}_{r}\frac{\partial}{\partial n}g_{D}(\cdot,\infty)\,ds.\nonumber
\end{align}
It follows from the first line in (\ref{f113g51}) that the first integral in
the second line of (\ref{f113g52}) is identical zero. For the second integral
we deduce with the same arguments as applied in (\ref{f113g42}) and with the
use of the last line in (\ref{f113g51}) that
\begin{equation}
-\frac{1}{2\pi}\int_{\{|z|=r\}}\widehat{u}_{r}\frac{\partial}{\partial n}%
g_{D}(\cdot,\infty)\,ds=u(\infty)+\text{O}(\frac{1}{r})\text{ \ \ as
\ \ }r\rightarrow\infty.\label{f113g53}%
\end{equation}
Identity (\ref{f113g5}) follows immediately from (\ref{f113g52}) and
(\ref{f113g53}) for a domain $D$ with a sufficiently smooth boundary $\partial
D$.

For a general domain, identity (\ref{f113g5}) can again be proved by
exhausting the domain $D$ by a sequence of nested domains $D_{n}$,
$n\in\mathbb{N}$, as it has been done in the proof of Lemma \ref{l113g}%
.\medskip
\end{proof}

\begin{lemma}
\label{l113j}Let $D\subset\overline{\mathbb{C}}$ be a domain with
$\operatorname*{cap}\left(  \partial D\right)  >0$, $V\subset\overline{D}$ a
compact set, $\mu$ a positive measure of finite energy with
$\operatorname*{supp}(\mu)\subset V $, and $u$\ a real-valued function defined
by
\begin{equation}
u(z):=h(z)+\int g_{D}(z,v)d\mu(v)\text{ \ \ for \ \ }z\in D\label{f113h1}%
\end{equation}
with $h$ a harmonic and bounded function in $D$. If we assume that
\begin{equation}
u(z)=0\text{ \ \ \ for quasi every \ \ }z\in V,\label{f113h2}%
\end{equation}
then we have
\begin{equation}
D_{D\setminus V}(u)=D_{D}(h)+\iint g_{D}(v,w)d\mu(v)d\mu(w).\label{f113h3}%
\end{equation}
\medskip
\end{lemma}

\begin{proof}
We deduce from Lemma \ref{l113h} that
\begin{equation}
D_{D}(h,g_{D}(\cdot,v))=0\text{ \ \ for \ \ }v\in D,\label{f113h11}%
\end{equation}
since with the help of a Moebius transform any $v\in D$ can be transported to
infinity. If we choose $g_{D}(\cdot,w)$, $w\in D$, instead of the function $h$
in (\ref{f113h11}) and set $D_{r}:=D\setminus\{|z-w|\leq r\}$, for $r>0$
small, then we deduce from Lemma \ref{l113h} that
\begin{equation}
D_{D_{r}}(g_{D}(\cdot,v),g_{D}(\cdot,w))=0.\label{f113h12}%
\end{equation}
With the same argumentation as used after (\ref{f113g12}) and later also in
the proof of Lemma \ref{l113h} after (\ref{f113g41}), we show that
\begin{equation}
\lim_{r\rightarrow\infty}D_{\{|z-w|<r\}}(g_{D}(\cdot,v),g_{D}(\cdot
,w))=g_{D}(w,v).\label{f113h13}%
\end{equation}
Putting (\ref{f113h12}) and (\ref{f113h13}) together proves that
\begin{equation}
D_{D}(g_{D}(\cdot,v),g_{D}(\cdot,w))=g_{D}(v,w)\text{ \ \ for \ \ }v,w\in
D.\label{f113h14}%
\end{equation}

In a strict sense the Dirichlet integrals in (\ref{f113h11}) and
(\ref{f113h14}) exist only as improper integrals, which is reflected in the
removal of small disks around the points $\infty$ and $w$ in (\ref{f113g4})
and (\ref{f113h12}), respectively. There exist techniques to overcome this
specific problem, as for instance, the use of local smoothing techniques at
the singularity of the Green function, which is demonstrated in detail in
\cite{Landkof}, Chapter 1, \S 5.

In the next step of the proof we assume that $D\setminus V$ has a smooth
boundary $\partial(D\setminus V)$. It follows then that the planar Lebesgue
measure of $\partial V$ is zero, i.e., $m(\partial V)=0$, and further that in
(\ref{f113h2}) we have equality for all $z\in V$. By $g$ we denote the Green
potential in (\ref{f113h1}), i.e.,
\begin{equation}
g:=\int g_{D}(\cdot,v)d\mu(v).\label{f113h15}%
\end{equation}

Because of $m(\partial V)=0$, we have
\begin{equation}
D_{D\setminus V}(u)=D_{D}(u),\label{f113h16}%
\end{equation}
and with (\ref{f113h1}) and (\ref{f113h15}), we rewrite the Dirichlet integral
in (\ref{f113h16}) as
\begin{equation}
D_{D\setminus V}(u)=D_{D}(h)+2\,D_{D}(h,g)+D_{D}(g).\label{f113h17}%
\end{equation}
Then we deduce with Fubini's Theorem from (\ref{f113h11}) that
\begin{equation}
D_{D}(h,g)=\int D_{D}(h,g_{D}(\cdot,v))d\mu(v)=0,\label{f113h18}%
\end{equation}
and analogously from (\ref{f113h14}) that
\begin{equation}
D_{D}(g)=\iint D_{D}(g_{D}(\cdot,v),g_{D}(\cdot,w))d\mu(v)d\mu(w)=\iint
g_{D}(v,w)d\mu(v)d\mu(w).\label{f113h19}%
\end{equation}
Putting (\ref{f113h17}), (\ref{f113h18}), and (\ref{f113h19}) together, we
have proved identity (\ref{f113h3}) for the case that $\partial(D\setminus V)$
is sufficiently smooth.

In the general situation, identity (\ref{f113h3}) follows, as in the proof of
Lemma \ref{l113g}, by exhausting the open set $D\setminus V$ by a sequence of
nested open sets with sufficiently smooth boundaries.\medskip
\end{proof}

\subsection{\label{s1104}Sequences of Compact Sets $K_{n}$}

\qquad Let $K_{n}\subset\mathbb{C}$, $n\in\mathbb{N}$, be a sequence of
compact sets of positive capacity. Because of the weak$^{\ast}$-compactness of
the set of probability measures supported on a compact set, we know that any
infinite subsequence $N\subset\mathbb{N}$ contains an infinite subsequence,
which we continue to denote by $N$, such that the equilibrium measures
$\omega_{n}=\omega_{K_{n}}$ of $K_{n}$ converge weakly in $\overline
{\mathbb{C}}$, i.e., there exists a probability measure $\omega_{0}%
=\omega_{0,N}$ in $\overline{\mathbb{C}}$\ with
\begin{equation}
\omega_{n}\overset{\ast}{\longrightarrow}\omega_{0}\text{ \ \ as
\ \ }n\rightarrow\infty\text{, }n\in N.\label{f114a1}%
\end{equation}
If in addition to (\ref{f114a1}) also the limit
\begin{equation}
\lim_{n\rightarrow\infty\text{, }n\in N}\operatorname*{cap}\left(
K_{n}\right)  =:c_{0}>0\label{f114a2}%
\end{equation}
exists and the inequality in (\ref{f114a2})\ holds true, then it follows from
the Lower Envelope Theorem \ref{t112a} of potential theory that for the
sequence of Green functions $g_{\Omega_{n}}(\cdot,\infty)$, $n\rightarrow
\infty$, $n\in N$, which is associated with the compact sets $K_{n}$ via the
outer domains $\Omega_{n}=\Omega_{K_{n}}$, we have the asymptotic relation
\begin{equation}
\limsup_{n\rightarrow\infty\text{, }n\in N}g_{\Omega_{n}}(\cdot,\infty
)\leq-p(\omega_{0};\cdot)-\log\operatorname*{cap}\left(  c_{0}\right)
=:g_{0,N},\label{f114a3}%
\end{equation}
and equality holds in (\ref{f114a3}) quasi everywhere in $\mathbb{C}%
$.$\smallskip$

Like the measure $\omega_{0}=\omega_{0,N}$\ in (\ref{f114a1}), so also the
function $g_{0}=g_{0,N}$ in (\ref{f114a3})\ depends on the subsequence
$N\subset\mathbb{N}$.

For all infinite subsequences $N\subset\mathbb{N}$, for which the two limits
(\ref{f114a1}) and (\ref{f114a2}) exist, the potential $g_{0,N}$ in
(\ref{f114a3}) is well-defined, and it is an immediate consequence of
(\ref{f113a1}) and the inequality in (\ref{f114a3}) that
\begin{equation}
g_{0,N}(z)\geq0\text{ \ \ for all \ \ \ }z\in\mathbb{C}.\label{f114a4}%
\end{equation}

\begin{lemma}
\label{l114a}Let $E\subset\mathbb{C}$ be a compact set, and let $N\subset
\mathbb{N}$ be an infinite subsequence for which the two limits (\ref{f114a1})
and (\ref{f114a2}) exist. The inclusion $E\subset K_{n}$ for all $n\in N$
implies that
\begin{equation}
g_{0,N}(z)=0\text{ \ \ for quasi every \ }z\in E.\label{f114b1}%
\end{equation}

\end{lemma}

\begin{proof}
Without loss of generality we can assume that $\operatorname*{cap}\left(
E\right)  >0$. From $E\subset K_{n}$, $n\in N$, it follows that
\begin{equation}
g_{\Omega_{n}}(z,\infty)\leq g_{\Omega_{E}}(z,\infty)\text{ \ for \ }%
z\in\mathbb{C}\text{, \ }\Omega_{n}=\Omega_{K_{n}},\label{f114b2}%
\end{equation}
(cf. \cite{Ransford}, Corollary 4.4.5), and consequently we have
\begin{equation}
\limsup_{n\rightarrow\infty\text{, }n\in N}g_{\Omega_{n}}(z,\infty)\leq
g_{\Omega_{E}}(z,\infty)\text{ \ for \ }z\in\mathbb{C.}\label{f114b3}%
\end{equation}
From (\ref{f113a1}), we know that $g_{\Omega_{E}}(z,\infty)=0$ for quasi every
$z\in\widehat{E}$ ($\widehat{E}$ denotes the polynomial-convex hull of $E$),
and from (\ref{f114a3}) and (\ref{f114b3}), we then conclude that
$g_{0,N}(z)=0$ for quasi every $z\in\widehat{E}$, which proves (\ref{f114b1}%
).\medskip
\end{proof}

The next lemma has played a key role at several places in the proof of
existence of an extremal domain in Subsection \ref{s92}. The proof of the
lemma relies heavily on Carath\'{e}odory's Kernel Convergence Theorem, which
will be stated just after the next lemma.

\begin{lemma}
\label{l114b}Let $R\subset\overline{\mathbb{C}}$ be a ring domain with
$\infty\in R$, $A_{1}$, $A_{2}\subset\mathbb{C}$ the two components of
$\overline{\mathbb{C}}\setminus R$, let further $N\subset\mathbb{N}$ be an
infinite subsequence for which the two limits (\ref{f114a1}) and
(\ref{f114a2}) exist, and for which therefore the function $g_{0,N}$ in
(\ref{f114a3}) is well-defined, and let further $0<r<\infty$ be an
appropriately chosen constant.

If for each $n\in N$ there exists a continuum $V_{n}\subset K_{n}%
\cap\{\,|z|\leq r\,\}$ that intersects the ring domain $R$, i.e., we have
\begin{equation}
V_{n}\cap A_{j}\neq\emptyset\text{ \ \ \ for \ \ \ }j=1,2,\label{f114f10}%
\end{equation}
then
\begin{equation}
V_{0}:=\overline{\mathbb{C}}\setminus\Omega\left(  \bigcap_{m\in\mathbb{N}%
}\overline{\bigcup_{n\geq m,\text{ }n\in N}V_{n}}\right) \label{f114f11}%
\end{equation}
is a continuum with $V_{0}\subset\{\,|z|\leq r\,\}$ that also intersects $R$,
i.e., we have
\begin{equation}
V_{0}\cap A_{j}\neq\emptyset\text{ \ \ \ for \ \ \ }j=1,2,\label{f114f12}%
\end{equation}
and further we have
\begin{equation}
g_{0,N}(z)=0\text{ \ \ for all \ \ }z\in V_{0}\label{f114c13}%
\end{equation}
where $g_{0,N}$ denotes the function defined in (\ref{f114a3}).\ By
$\Omega(\cdot)$ we denote the outer domain in (\ref{f114f11}).
\end{lemma}

It has already been mentioned that the proof of Lemma \ref{l114b} is based on
Carath\'{e}o\-dory's Kernel Convergence Theorem from geometric function
theory, which establishes an equivalence between an analytic and a geometric
description of the convergence of a sequences of conformal mapping functions
(cf. \cite{Pommerenke75}, Chapter 1.4).

\begin{theorem}
[Carath\'{e}odory's Kernel Convergence Theorem]\label{t114a} Let
\newline$\allowbreak\left\{  \varphi_{n}\right\}  _{n\in\mathbb{N}}$ be a
sequence of univalent functions defined in $\overline{\mathbb{C}}%
\setminus\overline{\mathbb{D}}$ with $\varphi_{n}(\infty)=\infty$ and
\begin{equation}
0<m_{0}\leq\varphi_{n}^{\prime}(\infty)\leq M_{0}<\infty\text{ \ \ for each
\ }n\in\mathbb{N}.\label{f114c2}%
\end{equation}
The sequence of functions $\varphi_{n}$, $n\in\mathbb{N}$, convergence locally
uniformly in $\overline{\mathbb{C}}\setminus\overline{\mathbb{D}}$ to an
univalent function $\varphi_{0}$ if, and only if, the domains
\begin{equation}
D_{N}=\operatorname*{Ker}\left(  \left\{  \varphi_{n}(\overline{\mathbb{C}%
}\setminus\overline{\mathbb{D}})\right\}  _{n\in N}\right)  :=\Omega\left(
\bigcap_{m\in\mathbb{N}}\overline{\bigcup_{n\geq m,\text{ }n\in N}%
\overline{\mathbb{C}}\setminus\varphi_{n}(\overline{\mathbb{C}}\setminus
\overline{\mathbb{D}})}\right) \label{f114c3}%
\end{equation}
are identical for all infinite subsequences $N\subset\mathbb{N}$. In
(\ref{f114c3}), the outer domain is denoted by $\Omega(\cdot)$. The domain
$D_{N}$ associated with the sequence $N\subset\mathbb{N}$ is called kernel of
the sequence of domains $\varphi_{n}(\overline{\mathbb{C}}\setminus
\overline{\mathbb{D}})$, $n\in N$.

If the sequence $\{\varphi_{n}\}_{n\in\mathbb{N}}$ converges locally uniformly
in $\overline{\mathbb{C}}\setminus\overline{\mathbb{D}}$, then the limit
function
\begin{equation}
\varphi_{0}=\lim_{n\rightarrow\infty}\varphi_{n}\label{f114c4}%
\end{equation}
is the Riemann mapping function of $\overline{\mathbb{C}}\setminus
\overline{\mathbb{D}}$ onto the domain $D_{\mathbb{N}}\subset\overline
{\mathbb{C}}$ with $\varphi_{0}(\infty)=\infty$ and $m_{0}\leq\varphi
_{0}^{\prime}(\infty)\leq M_{0}$.
\end{theorem}

\begin{remark}
The restrictions (\ref{f114c2}) placed on $\varphi_{n}^{\prime}(\infty)$ make
sure that degenerated cases are excluded.
\end{remark}

\begin{proof}
[Proof of Lemma \ref{l114b}]Let $\varphi_{n}$ be the Riemann mapping function
from $\overline{\mathbb{C}}\setminus\overline{\mathbb{D}}$ onto $\overline
{\mathbb{C}}\setminus V_{n}$ for each $n\in\mathbb{N}$ with
\begin{equation}
\varphi_{n}(\infty)=\infty\text{ \ \ and \ \ }\varphi_{n}^{\prime}%
(\infty)>0.\label{f114d1}%
\end{equation}
It is immediate that
\begin{equation}
g_{\overline{\mathbb{C}}\setminus V_{n}}(z,\infty)=\log\left|  \varphi
_{n}^{-1}(z)\right|  \text{ \ \ for \ }z\in\overline{\mathbb{C}}\setminus
V_{n},\label{f114d2}%
\end{equation}
and we have $g_{\overline{\mathbb{C}}\setminus V_{n}}(z,\infty)=0$ for all
$z\in V_{n}$ since continua are regular sets. From Lemma \ref{l113b} it
follows that
\begin{equation}
\operatorname*{cap}(V_{n})=\varphi_{n}^{\prime}(\infty).\label{f114d3}%
\end{equation}
From Lemma \ref{l111a} and (\ref{f114d3}), we conclude that the assumptions
$V_{n}\subset\{$\thinspace$|z|\leq r\,\}$ and $V_{n}\cap A_{j}\neq\emptyset$
for $j=1,2$ and $n\in\mathbb{N}$ imply that
\begin{equation}
\operatorname*{dist}(A_{1},A_{2})/4\leq\operatorname*{cap}(V_{n})\leq r\text{
\ for all \ }n\in\mathbb{N.}\label{f114d4}%
\end{equation}
From (\ref{f114d3}) and (\ref{f114d4}), it then follows that the restrictions
(\ref{f114c2}) in Theorem \ref{t114a}\ are satisfied.

From (\ref{f114d4}) and the assumption $V_{n}\subset\{\,|z|\leq r\,\}$ for all
$n\in\mathbb{N}$, it further follows that there exists an infinite subsequence
$N\subset\mathbb{N}$ such that the two limits (\ref{f114a1}) and
(\ref{f114a2}) exist, and therefore also the limit function $g_{0,N}$ in
(\ref{f114a3}) exists with $K_{n}$ replaced by the sets $V_{n}$, and the outer
domain $\Omega_{n}$ by $\overline{\mathbb{C}}\setminus V_{n}$ for each $n\in
N$.

Because of $V_{n}\subset\{\,|z|\leq r\,\}$ for $n\in\mathbb{N}$, we have a
proper limit and locally uniform convergence in $\{\,|z|>r\,\}$ in
(\ref{f114a3}). With (\ref{f114d2}), this implies that the sequence
$\varphi_{n}$, $n\in N$, converges uniformly in a closed neighborhood of
infinity. From the convergence together with the property that the
$\varphi_{n}$ are mapping functions into $\overline{\mathbb{C}}\setminus
V_{n}$ and the property (\ref{f114d1}), we deduce that $\{$\thinspace
$\varphi_{n}$, $n\in N\,\}$ forms a normal family in $\overline{\mathbb{C}%
}\setminus\overline{\mathbb{D}}$, and by Montel's Theorem together with the
convergence in $\{\,|z|>r\,\}$, it then follows that
\begin{equation}
\lim_{n\rightarrow\infty,\text{ }n\in N}\varphi_{n}(z)=:\varphi_{0,N}%
(z)\label{f114d5}%
\end{equation}
holds locally uniformly for $z\in\overline{\mathbb{C}}\setminus\overline
{\mathbb{D}}$.

From Carath\'{e}odory's Kernel Convergence Theorem, we then concluded that the
limit function $\varphi_{0}$ in (\ref{f114d5}) is the Riemann mapping function
of $\overline{\mathbb{C}}\setminus\overline{\mathbb{D}}$ onto the domain
$D_{N}\subset\overline{\mathbb{C}}$ defined by (\ref{f114c3}) with the
subsequence $N$, which has been used in (\ref{f114d5}), and from
(\ref{f114c3}) and (\ref{f114d5}),\ we further know that
\begin{equation}
V_{0}:=\overline{\mathbb{C}}\setminus D_{N}\label{f114d6}%
\end{equation}
is a continuum. From (\ref{f114c3}) we learn that
\begin{equation}
V_{0}=Pc(\bigcap_{m\in\mathbb{N}}\overline{\bigcup_{n\geq m,\text{ }n\in
N}V_{n}})=\overline{\mathbb{C}}\setminus\Omega\left(  \bigcap_{m\in\mathbb{N}%
}\overline{\bigcup_{n\geq m,\text{ }n\in N}V_{n}}\right) \label{f114d7}%
\end{equation}
with $Pc(\cdot)$ denoting the polynomial-convex hull. For the two components
$A_{1}$ and $A_{2}$ of $\overline{\mathbb{C}}\setminus R$ we have
\begin{equation}
A_{j}\cap\bigcap_{m\in\mathbb{N}}\overline{\bigcup_{\substack{n\geq m, \\n\in
N }}V_{n}}=\bigcap_{m\in\mathbb{N}}\overline{(A_{j}\cap\bigcup
_{\substack{n\geq m, \\n\in N }}V_{n})}=\bigcap_{m\in\mathbb{N}}%
\overline{\bigcup_{\substack{n\geq m, \\n\in N }}\left(  A_{j}\cap
V_{n}\right)  },\label{f114d8}%
\end{equation}
$j=1,2$. Since we have assumed that $V_{n}\cap A_{j}\neq\emptyset$ for $j=1,2$
and all $n\in N$, we conclude from (\ref{f114d8}) that
\begin{equation}
A_{j}\cap\bigcap_{m\in\mathbb{N}}\overline{\bigcup_{n\geq m,\text{ }n\in
N}V_{n}}\text{ }\neq\text{ }\emptyset\text{ \ for \ }j=1,2\text{,}%
\label{f114d9}%
\end{equation}
since the intersections are nested. From (\ref{f114d9}) and (\ref{f114d7}), we
immediately get
\begin{equation}
A_{j}\cap V_{0}\neq\emptyset\text{ \ for \ }j=1,2\text{.}\label{f114d10}%
\end{equation}

Since we have assumed $V_{n}\subset K_{n}$, we have
\begin{equation}
g_{\Omega_{n}}(z,\infty)\leq g_{\overline{\mathbb{C}}\setminus V_{n}}%
(z,\infty)\text{ \ for \ }z\in\mathbb{C}\label{f114d11}%
\end{equation}
and all $n\in N$ and\ $\Omega_{n}:=\Omega_{K_{n}}$ (cf. \cite{Ransford},
Corollary 4.4.5). From the convergence (\ref{f114d5}) together with the
identities in (\ref{f114d2}), we conclude that
\begin{equation}
\lim_{n\rightarrow\infty,\text{ }n\in N}g_{\overline{\mathbb{C}}\setminus
V_{n}}(z,\infty)=g_{\overline{\mathbb{C}}\setminus V_{0}}(z,\infty
)\label{f114d12}%
\end{equation}
holds locally uniformly for $z\in\mathbb{C}$. From (\ref{f114d11}) together
with limit relation (\ref{f114c3}) and (\ref{f114d12}), it then follows that
\begin{equation}
g_{0,N}(z)\leq g_{\overline{\mathbb{C}}\setminus V_{0}}(z,\infty)\text{ \ for
quasi every \ }z\in\mathbb{C,}\label{f114d13}%
\end{equation}
where $g_{0,N}$ is the limit function in (\ref{f114d12}). Because of the first
inequality in (\ref{f114d4}), $V_{0}$ is of positive capacity, and therefore
the equilibrium measure $\omega_{V_{n}}$ of $V_{n}$ is of finite energy. With
the principle of domination in Theorem \ref{t112b}, it then follows that the
inequality in (\ref{f114d13}) holds for all $z\in\mathbb{C} $. Since
$g_{\overline{\mathbb{C}}\setminus V_{0}}(z,\infty)=0$ for all $z\in V_{0}$,
conclusion (\ref{f114c13}) of Lemma \ref{l114b} follows from (\ref{f114d13}).
The two other conclusions (\ref{f114f11}) and (\ref{f114f12}) are identical
with (\ref{f114d7}) and (\ref{f114d10}), which completes the proof of the
lemma.\bigskip
\end{proof}

\subsection{\label{s1105}Critical Trajectories of Quadratic Differentials}

\qquad Let $q$ be a function meromorphic in a domain $D\subset\mathbb{C}$. In
(\ref{f52a}), trajectories of the quadratic differential $q(z)dz^{2}$ have
been defined as smooth Jordan arcs $\gamma$ with parametrization
$z:[0,1]\longrightarrow\overline{\mathbb{C}}$ that satisfy the relation
\begin{equation}
q(z(t))\overset{\bullet}{z}(t)^{2}<0\text{ \ \ \ for \ \ }t\in
(0,1).\label{f115a}%
\end{equation}
Like in Section \ref{s52}, we use \cite{Strebel84} and \cite{Jensen75} as
general reference for quadratic differentials.\smallskip

Assertions about the global behavior of trajectories of a quadratic
differential $q(z)dz^{2}$ are difficult to obtain, but the situation is
dramatically different with respect to their local behavior; it depends only
on the local form of the function $q$, and is basically a consequence of the
degree of its poles and zeros. Further, it is not difficult to see that the
qualitative behavior of trajectories is invariant under conformal maps.

All zeros and simple poles of the function $q$ are called \textit{finite
critical points} of the quadratic differential $q(z)dz^{2}$, and trajectories
that end at zeros and poles are called \textit{critical}. In the next lemma we
assemble results about the local behavior of trajectories that have been used
at several places of our analysis, further above. These results are not
difficult to prove, and proofs can be found in \cite{Jensen75}, Chapter
8.2.\medskip

\begin{lemma}
\label{l115a}We consider a quadratic differential $q(z)dz^{2}$, and assume
that $q$ is meromorphic in a domain $D\subset\mathbb{C}$.\smallskip

(i) \ If $q$ is analytic in a neighborhood $U$ of $z_{0}\in D$ and if
$q(z)\neq0$ for all $z\in U$, then all trajectories of $q(z)dz^{2}$ are
laminar in $U$.\smallskip

(ii)\ Let $z_{0}\in D$ be a finite critical point of the quadratic
differential $q(z)dz^{2}$, i.e., at $z_{0}$ the function $q$ has the local
behavior
\begin{equation}
q(z)=q_{0}(z-z_{0})^{l}+\text{O}((z-z_{0})^{l+1})\text{ \ \ as \ \ }%
z\rightarrow z_{0}\text{, \ }q_{0}\neq0,\label{f115b}%
\end{equation}
and let further $U$ be a neighborhood of $z_{0}$\ with $q(z)\neq0,\infty$ for
all $z\in U\setminus\{z_{0}\}$, then $l+2$ trajectories of $q(z)dz^{2}$ end at
the point $z_{0}$, and they form a regular star at $z_{0}$, i.e., all angles
between neighboring trajectories are equal to $2\pi/(l+2)$. All other
(non-critical) trajectories of $q(z)dz^{2}$ are laminar in $U$.\bigskip
\end{lemma}

\bibliographystyle{plain}

\bigskip

\bigskip
\end{document}